\documentclass[11pt, a4paper]{amsart}
\usepackage{amscd,amsmath,amssymb,amsfonts,verbatim}
\usepackage[cmtip, all]{xy}
\usepackage{MnSymbol}
\usepackage[ left=3cm, right=3cm, top = 2.8cm, bottom = 2.8cm ]{geometry}
\usepackage{comment}
\usepackage{tikz-cd}
\newtheorem{thm}{Theorem}[section]
\newtheorem{prop}[thm]{Proposition}
\newtheorem{lem}[thm]{Lemma}
\newtheorem{cor}[thm]{Corollary}

\theoremstyle{definition}

\newtheorem{defn}[thm]{Definition}
\theoremstyle{remark}
\newtheorem{remk}[thm]{Remark}
\newtheorem{remks}[thm]{Remarks}

\newtheorem{exm}[thm]{Example}
\newtheorem{exms}[thm]{Examples}
\newtheorem{notat}[thm]{Notation}
\numberwithin{equation}{section}

{\hfill$\square$\end{defn}}
{\hfill$\square$\end{remk}}
{\hfill$\square$\end{remks}}
{\hfill$\square$\end{exm}}
{\hfill$\square$\end{exms}}
{\hfill$\square$\end{notat}}

\newcommand{\thmref}{Theorem~\ref}
\newcommand{\propref}{Proposition~\ref}
\newcommand{\corref}{Corollary~\ref}

\newcommand{\lemref}{Lemma~\ref}

\newcommand{\sA}{{\mathcal A}}
\newcommand{\sB}{{\mathcal B}}
\newcommand{\sC}{{\mathcal C}}
\newcommand{\sD}{{\mathcal D}}
\newcommand{\sE}{{\mathcal E}}
\newcommand{\sF}{{\mathcal F}}
\newcommand{\sG}{{\mathcal G}}
\newcommand{\sH}{{\mathcal H}}
\newcommand{\sI}{{\mathcal I}}

\newcommand{\sL}{{\mathcal L}}
\newcommand{\sM}{{\mathcal M}}

\newcommand{\sO}{{\mathcal O}}
\newcommand{\sP}{{\mathcal P}}

\newcommand{\sS}{{\mathcal S}}

\newcommand{\sU}{{\mathcal U}}
\newcommand{\sV}{{\mathcal V}}
\newcommand{\sW}{{\mathcal W}}
\newcommand{\sX}{{\mathcal X}}
\newcommand{\sY}{{\mathcal Y}}
\newcommand{\sZ}{{\mathcal Z}}
\newcommand{\A}{{\mathbb A}}

\newcommand{\C}{{\mathbb C}}

\newcommand{\G}{{\mathbb G}}
\renewcommand{\H}{{\mathbb H}}

\newcommand{\N}{{\mathbb N}}
\renewcommand{\P}{{\mathbb P}}
\newcommand{\Q}{{\mathbb Q}}

\newcommand{\T}{{\mathbb T}}

\newcommand{\Z}{{\mathbb Z}}
\renewcommand{\1}{{\mathbb{S}}}

\newcommand{\fm}{{\mathfrak m}}

\newcommand{\fp}{{\mathfrak p}}

\newcommand{\Ker}{{\rm Ker}}

\newcommand{\an}{{\rm an}}

\newcommand{\surj}{\twoheadrightarrow}
\newcommand{\inj}{\hookrightarrow}
\newcommand{\red}{{\rm red}}
\newcommand{\codim}{{\rm codim}}

\newcommand{\Hom}{{\rm Hom}}

\newcommand{\Spec}{{\rm Spec \,}}
\newcommand{\sing}{{\rm sing}}
\newcommand{\Char}{{\rm char}}

\newcommand{\id}{{\operatorname{id}}}

\newcommand{\Sch}{{\operatorname{\mathbf{Sch}}}}

\newcommand{\holim}{\mathop{{\rm holim}}}

\renewcommand{\>}{\rangle}

\newcommand{\Spc}{{\mathbf{Spc}}}
\newcommand{\Sm}{{\mathbf{Sm}}}

\newcommand{\hocolim}{\mathop{{\rm hocolim}}}

\newcommand{\can}{{\operatorname{\rm can}}}

\newcommand{\Sym}{{\operatorname{\rm Sym}}}

\newcommand{\SH}{{\operatorname{\sS\sH}}}

\newcommand{\End}{{\operatorname{\text{End}}}}

\newcommand{\Nis}{{\operatorname{Nis}}}
\newcommand{\et}{{\text{\'et}}}

\newcommand{\ds}{{/\kern-3pt/}}

\newcommand{\compl}{\, \, {\widehat {}}}

\newcommand{\stX}{\operatorname{[X/G]}}

\newcommand{\Proj}{{\operatorname{Proj}}}

\newcommand{\colim}{\mathop{\text{\rm colim}}}

\newcommand{\Mod}{{\mathbf{Mod}}}

\newcommand{\tr}{{\operatorname{tr}}}

\newcommand{\ov}{\overline}

\renewcommand{\dim}{\text{\rm dim}}

\newcommand{\tuborg}{\left\{\begin{array}{ll}}
\newcommand{\sluttuborg}{\end{array}\right.}

\newcommand{\pr}{{\rm pr}}

\newcommand{\reg}{{\rm reg}}

\newcommand{\Et}{{\rm {\bf{Et}}}}
\newcommand{\fr}{{\rm free}}
\newcommand{\Msp}{{\rm Msp}}
\newcommand{\Maps}{{\rm Maps}}
\newcommand{\dst}{{\rm DSt}}
\newcommand{\hkr}{{\rm HKR}}
\newcommand{\pt}{{\rm pt}}
\newcommand{\Qcoh}{{\rm Qcoh}}

\newcommand{\wt}{\widetilde}
\newcommand{\wh}{\widehat}

\newcommand{\Fil}{{\rm fil}}

\newcommand{\Coh}{{\rm Coh}}
\newcommand{\Perf}{{\rm Perf}}
\newcommand{\Map}{{\rm Map}}

\newcommand{\Grp}{{\mathbf{Grp}}}

\newcounter{elno}

\newcounter{elno-abc}   

\newcounter{elno-abc-prime}

\begin{document}
\title[$K$-theory and cyclic homology of quotient stacks]{Thomason's completion for
  $K$-theory and cyclic homology of quotient stacks}
\author{Amalendu Krishna, Ritankar Nath}
\address{Department of Mathematics, South Hall, Room 6607, University of California
  Santa Barbara, CA, 93106-3080, USA.}
\email{amalenduk@math.ucsb.edu}
\address{School of Mathematics, Tata Institute of Fundamental Research,
  Mumbai, 400005, India.} 
\email{ritankar@math.tifr.res.in}


\keywords{Equivariant $K$-theory, Equivariant cyclic homology}        

\subjclass[2020]{Primary 14F43, Secondary 19D55}

\maketitle
\vskip .4cm

\begin{quote}\emph{Abstract.}
  We prove several completion theorems for equivariant $K$-theory
  and cyclic homology of schemes with group action over a field. One of these
  shows that for an algebraic space over a field acted upon by a linear algebraic
  group, the derived completion of equivariant $K'$-theory at the
  augmentation ideal of the representation ring of the group coincides with the
  ordinary $K'$-theory of the bar construction associated to the group action.
  This provides a solution to Thomason's completion problem. For action with finite
  stabilizers, we show that the equivariant $K$-theory and cyclic homology have
  non-equivariant descriptions even without passing to their completions. As an
  application, we describe all equivariant Hochschild and
  other homology groups for such actions.
\end{quote}
\setcounter{tocdepth}{1}
\tableofcontents

\vskip .4cm

\section{Introduction}\label{sec:Intro}
Equivariant $K$-theory of algebraic varieties equipped with actions of
linear algebraic groups was introduced by Thomason \cite{Thomason-Orange}.
Its significance is premised on two key observations: (1) it is an important
tool for describing the ordinary $K$-theory of varieties with group action (cf.
\cite{Krishna-LMS}, \cite{Merkurjev}, \cite{VV}),  and (2) it is often a key
ingredient in the study of representation theory of algebraic groups
(cf. \cite{Chriss-Ginzburg}). A key advantage of equivariant $K$-theory 
is that it does not depend on the underlying variety and the group but
only on the quotient stack resulting from the action. This provides a great deal of
flexibility in computing equivariant $K$-theory.

Despite the above advantages, equivariant $K$-theory suffers from the drawback
that there aren't many cohomological tools to compute them.
In order to apply the tools (e.g., motivic cohomology) for computing ordinary
$K$-theory, a fundamental problem in equivariant $K$-theory is to find out
how far it is from the ordinary $K$-theory and if there is a way to approximate it
by ordinary $K$-theory. A solution to this problem in topology
is via the celebrated Atiyah-Segal completion theorem \cite{Atiyah-Segal}.
The latter says roughly that if a compact Lie group $G$ acts on a compact Hausdorff
space $X$, then the (topological) equivariant $K$-theory $K^G(X)$ of $X$
becomes the ordinary (topological) $K$-theory $K(EG \stackrel{G}{\times} X)$ of the
homotopy orbit space $EG \stackrel{G}{\times} X$ after its completion with respect to
the augmentation ideal $I_G$ of the representation ring $R(G)$. 
The objective of this paper is the study of algebraic analogues of this theorem.
There are two approaches to this problem and we shall discuss
them separately. 

\subsection{Equivariant $K$-theory via ordinary $K$-theory of bar construction}
\label{sec:Bar-prob}
Let $k$ be a field and $G$ a linear algebraic group over $k$ acting on a
smooth $k$-scheme $X$. Let $\pi \colon X \to [X/G]$ be the stack quotient map for the
$G$-action. Given this datum, Thomason (cf. \cite[\S~1.3]{Thomason-Orange}) constructed
a simplicial scheme
\begin{equation}\label{eqn:Bar-I}
X_G^{\bullet}  := \left(\cdots 
\stackrel{\rightarrow}{\underset{\rightarrow}\rightrightarrows}
G \times G \times X 
\stackrel{\rightarrow}{\underset{\rightarrow}\to} G \times X
\rightrightarrows X \right)
\end{equation}
whose face maps are built by iterating the action and projection maps
$G \times X \to X$ and which is functorial in $G$ and $X$.
This is known as the bar construction associated to $G$-action on $X$. 
For the action of a compact Lie group $G$ on a compact Hausdorff
space $X$, the geometric realization of $X^\bullet_G$ provides a model for
the homotopy orbit space $EG \stackrel{G}{\times} X$.

In his 1986 paper (cf. \cite[\S~0]{Thomason-Duke-1}),
Thomason asked if the pull-back by
the projection map $\wt{\pi} \colon X^\bullet_G \to [X/G]$ (induced by $\pi$)
would produce a weak equivalence of spectra
\begin{equation}\label{eqn:Bar-I-0}
  \wt{\pi}^* \colon K_G(X)^{\compl}_{I_G}  \xrightarrow{\simeq} K(X^\bullet_G),
\end{equation}
where $K(X^\bullet_G) = \holim_n K(G^n \times X)$ is the ordinary $K$-theory spectrum
of the simplicial scheme $X^\bullet_G$,
$K_G(X)$ is the equivariant $K$-theory spectrum (cf. \S~\ref{sec:Recall}) and 
$K_G(X)^{\compl}_{I_G}$ is a homotopy theoretic completion of  $K_G(X)$
at the augmentation ideal $I_G$ of $R(G)$.
Thomason predicted the existence of the spectrum $K_G(X)^{\compl}_{I_G}$ but did not
construct it. 

In the same paper, Thomason showed that if $k$ is a separably
closed field (and some other fields), then a version of the above
problem has a positive answer if one replaces $K_G(X)$ by
the Bott-inverted $K$-theory with finite coefficients
$K_G(X, {\Z}/{n})[\beta^{-1}]$ (with $n$ prime to $\Char(k)$)
and the conjectured homotopy theoretic $I_G$-completion of the spectrum
$K_G(X, {\Z}/{n})[\beta^{-1}]$
is replaced by algebraic $I_G$-completion of the $R(G)$-modules
$\pi_*(K_G(X, {\Z}/{n})[\beta^{-1}])$.  It was shown later in
\cite{Krishna-Crelle} that using finite coefficients and Bott-inversion are both
necessary in Thomason's theorem. 
In this paper, we shall refer to the problem of showing ~\eqref{eqn:Bar-I-0} is a
weak equivalence as {\sl{Thomason's completion problem}}.

Thomason's completion problem is a very powerful statement as a positive
answer to it would allow one to compute
the equivariant $K$-theory up to completion in terms of the 
ordinary $K$-theory of the \u{C}ech nerve of the group action.
In particular, one would be able to apply the machinery of ordinary
motivic cohomology to compute equivariant $K$-theory just as one does for
ordinary $K$-theory.

One of the main results of this paper provides a solution to Thomason's completion
problem.
In order to state it, let $K'_G(X)$ denote the $K$-theory spectrum of
the $G$-equivariant pseudo-coherent complexes on $X$. In \cite[Chap.~7.3]{Lurie-SAG},
Lurie introduced a simple procedure of derived completion of spectra
which implies in particular that
the derived $I_G$-completion of $K'_G(X)$ postulated by Thomason
indeed exists (as its homotopy groups are compatible with the classical
algebraic $I_G$-completions on the homotopy groups of $K'_G(X)$). We let
$K'_G(X)^{\compl}_{I_G}$ denote this spectrum level completion. Let $K'(-)$ denote the
ordinary $K'$-theory functor and let $K'(-, \Q)$ (resp. $K'_G(-, \Q)$) denote the
rationalized ordinary (resp. equivariant) $K'$-theory functor. All algebraic spaces
will be separated and finite type over the given base field.

We prove the following.

\begin{thm}\label{thm:Main-0}
  Let $k$ be a field and $G$ a linear algebraic group over $k$ acting on an algebraic
  space $X$. Let $m \in k^\times$ be an integer. We now have the following.
  \begin{enumerate}
  \item
    If $G$ is special, then the pull-back map $\wt{\pi}^*$
    induces a weak equivalence of spectra
    \[
    \wt{\pi}^* \colon K'_G(X)^{\compl}_{I_G}  \xrightarrow{\simeq} K'(X^\bullet_G).
\]
\item
  If $G$ is arbitrary, then the pull-back map $\wt{\pi}^*$ induces a weak
  equivalence of spectra
  \[
  \wt{\pi}^* \colon K'_G(X, \Q)^{\compl}_{I_G} \xrightarrow{\simeq} K'(X^\bullet_G, \Q).
  \]
 \item
   If  $G$ is arbitrary and $k$ contains all roots of unity, then
   the pull-back map $\wt{\pi}^*$ induces a weak equivalence of spectra
  \[
  \wt{\pi}^* \colon \left(K'_G(X, {\Z}/{m})[\beta^{-1}]\right)^{\compl}_{I_G}
  \xrightarrow{\simeq} K'(X^\bullet_G, {\Z}/m)[\beta^{-1}].
  \]
\end{enumerate}
\end{thm}

\begin{remk}\label{remk:Main-rem-0}
\begin{enumerate}
  \item
    The failure of Galois descent for algebraic $K$-theory of fields
  shows that one should not expect \thmref{thm:Main-0} (2) to be true with integral
  coefficients if $G$ is non-special.
\item
  If $k$ is separably closed, then 
the homotopy groups of $K'_G(X,{\Z}/{m})[\beta^{-1}]$ are
finitely generated $R(G)$-modules, and in particular, commute with Lurie's completion
(cf. \cite[Cor.7.3.6.6]{Lurie-SAG}).
Using this observation, item (3) of the
theorem also follows from Thomason's theorem (the main result of
\cite{Thomason-Duke-1}) if $k$ is separably closed.
Item (3) can therefore be interpreted as an extension of
Thomason's theorem to other fields.
\item
We also remark that one can replace $K'$-theory by $K$-theory (of perfect complexes)
in the theorem if $X$ is regular.
\end{enumerate}
\end{remk}

Recall that given a $G$-equivariant morphism $f \colon X \to Y$ between smooth
$k$-schemes with $G$-actions, it is a hard problem to determine
what kind of weak equivalence between $X$ and $Y$ would imply a weak equivalence
between their equivariant $K$-theory.
This makes the following application of \thmref{thm:Main-0} relevant.

\begin{cor}\label{cor:Main-1}
  Let $k$ be a field and $G$ a linear algebraic group over $k$. Let $f \colon
  X \to Y$ be a $G$-equivariant morphism between smooth $k$-schemes with $G$-actions.
  Assume that $f$ is an $\A^1$-weak equivalence (after ignoring the $G$-action) in the
  sense of Morel-Voevodsky. We then have the following.
  \begin{enumerate}
    \item
    The map $f^* \colon K_G(Y, \Q)^{\compl}_{I_G} \to
    K_G(X, \Q)^{\compl}_{I_G}$ is a weak equivalence.
  \item
    If $G$ is special, the map $f^* \colon K_G(Y)^{\compl}_{I_G} \to
    K_G(X)^{\compl}_{I_G}$ is a weak equivalence.
  \item
    If  $k$ contains all roots of unity and $m \in k^\times$, the map
    $f^* \colon  \left(K_G(Y, {\Z}/{m})[\beta^{-1}]\right)^{\compl}_{I_G} \to
    \left(K_G(X, {\Z}/{m})[\beta^{-1}]\right)^{\compl}_{I_G}$
    is a weak equivalence.
  \end{enumerate}
\end{cor}

\subsection{Equivariant $K$-theory via ordinary $K$-theory of Borel construction}
\label{sec:Borel-prob}
The second approach to algebraic Atiyah-Segal theorem
is via the Borel construction. The Borel equivariant cohomology theory is a
classical concept in topology. Using their $\A^1$-homotopy theory of
schemes, Morel-Voevodsky \cite{Morel-Voevodsky} introduced Borel construction
for linear algebraic groups in algebraic geometry. Totaro \cite{Totaro1}
and Edidin-Graham \cite{EG-Inv} used this construction to define equivariant
motivic cohomology while it was used in \cite{Krishna-Crelle} to
define Borel equivariant $K$-theory.

Let $k$ be a field and $G$ a linear algebraic group over $k$. A Borel construction
for a $G$-action on an algebraic space $X$ is an ind-space
$X_G := \{X \stackrel{G}{\times} U_n\}_{n \ge 1}$, where $\{(V_n, U_n)\}_{n \ge 1}$
(known as an admissible gadget)
is an increasing (in terms of inclusion) sequence of pairs of $k$-schemes
such that $V_i$ is a finite-dimensional rational representation of $G$ and
$U_i \subset V_i$ is $G$-invariant open which is a $G$-torsor over a scheme.
These pairs have the
advantage that the quotients $X \stackrel{G}{\times} U_n$ always exist as
algebraic spaces. 

By its definition, the Borel construction $X_G$ depends
on the choice of the admissible gadget. But its $K'$-theory
$K'(X_G) := \holim_n K'(X \stackrel{G}{\times} U_n)$ does not. We refer to
\S~\ref{sec:Borel-K} for more details.
The $G$-equivariant projection maps $p_n \colon X \times U_n \to X$ define  
a map of ind-stacks $\wt{p} \colon X_G \to [X/G]$ and hence a map
$\wt{p}^* \colon K'_G(X) \to K'(X_G)$. The Borel construction approach to
algebraic Atiyah-Segal theorem asks if Lurie's $I_G$-completion of
$K'_G(X)$ could be described by the ordinary $K'$-theory of $X_G$.
We answer this as follows. 

\begin{thm}\label{thm:Main-2}
Let $k$ be a field and $G$ a linear algebraic group over $k$ acting on an algebraic
  space $X$. Then the pull-back map $\wt{p}^*$ induces a weak equivalence of spectra
    \[
    \wt{p}^* \colon K'_G(X)^{\compl}_{I_G}  \xrightarrow{\simeq} K'(X_G).
    \]
   \end{thm} 

\begin{remk}\label{remk:Main-2-0}
  Some remarks are in order.
  \begin{enumerate}
  \item
    We can replace $K'$-theory by $K$-theory in \thmref{thm:Main-2} if $X$ is
    regular.
  \item
    One may note that our description of the equivariant $K$-theory via
    Borel construction
    holds integrally for all groups in contrast to the description via the bar
    construction. The reason for this is that \thmref{thm:Main-0} is much more
    subtle than \thmref{thm:Main-2} because the admissible gadgets are easier
    to deal with and one has flexibility while choosing them. This isn't true
    with the bar construction. \thmref{thm:Main-2} is in fact a key step in
    the proof of more challenging \thmref{thm:Main-0}.
\item
  Using the $T$-filtrability of smooth projective schemes with torus action
  and \cite[Cor.7.3.6.6]{Lurie-SAG},
  the completion theorems of \cite{Krishna-Crelle} can be recovered
  from \thmref{thm:Main-2}.
\item
  When $k$ is an infinite perfect field, a version of \thmref{thm:Main-2} was
  proven by Carlsson-Joshua \cite{Carlsson-Joshua-Adv}. There are two main differences
  between their result and ours. First, in the completion theorem of
  \cite{Carlsson-Joshua-Adv}, the authors use Carlsson's derived completion
  \cite{Carlsson} whose definition is more involved and 
  which is {\sl{a priori}} different from Lurie's completion.
  In particular, it is unclear if it coincides with the $I_G$-completion on the
  integral homotopy groups of the $K$-theory spectrum 
  (except when the underlying group is a split torus in which case the two
  completions coincide). 

Second, the authors of \cite{Carlsson-Joshua-Adv} fix an
  ambient group (which is a product of general linear groups) and assume other
  groups (including $G$) to be subgroups of this fixed
  ambient group in the construction of Borel $K$-theory and in their main
  theorems. As a result, their right hand side of $\wt{p}^*$ depends on
  the fixed choice of this ambient group. It is as such not functorial in
  $G$. In contrast, \thmref{thm:Main-2} does not rely on any such assumption.
\end{enumerate}
  \end{remk}

\subsection{Completion of equivariant $KH$-theory}\label{sec:KH-Intro}
Our third result extends Theorems~\ref{thm:Main-0} and ~\ref{thm:Main-2}
to the equivariant homotopy $K$-theory under some restrictions.
Before we state it, recall that
 a linear algebraic group $G$ over a field $k$ is called $nice$ 
if it is an extension of a finite linearly reductive
group scheme over $k$ by a group scheme over $k$ of multiplicative type (cf.
\cite[\S~2]{Hoyois-Krishna}). The category of nice group schemes includes
diagonalizable groups and finite constant groups whose orders are invertible
in $k$. For an algebraic space $X$ over $k$
with an action of a linear algebraic group $G$,
let $KH_G(X)$ denote the non-connective equivariant homotopy $K$-theory spectrum
for $X$ (cf. \cite[\S~5]{Krishna-Ravi}).
We prove the following.

\begin{thm}\label{thm:Main-3}
  Let $k$ be a field and $G$ a linear algebraic group over $k$ acting on a $k$-scheme
  $X$. Assume that either \\
  \hspace*{.5cm} $({\rm A})$  $X$ is a toric variety over $k$ on which $G$ acts as the
  dense torus, or \\
  \hspace*{.5cm} $({\rm B})$ $\Char(k) = 0$ and $G$ acts on $X$ with
  nice stabilizers. \\
Then the following hold.
  \begin{enumerate}
  \item
    The maps of spectra 
  \[
  (i) \ \wt{\pi}^* \colon KH_G(X, \Q)^{\compl}_{I_G} \to KH(X^\bullet_G, \Q) \ \ 
  \mbox{and} \ \ (ii) \ \wt{p}^* \colon
  KH_G(X)^{\compl}_{I_G} \to KH(X_G)
  \]
  are weak equivalences.
  Moreover, $(i)$ holds integrally if $G$ is special.
\item
  If $k$ contains all roots of unity and $m \in k^\times$, then the maps of spectra
  \[
  (i) \ \wt{\pi}^* \colon (KH_G(X, {\Z}/m)[\beta^{-1}])^{\compl}_{I_G} \to
    KH(X^\bullet_G, {\Z}/m)[\beta^{-1}] \ \ 
    \mbox{and}
    \]
    \[
    \ \ (ii) \ \wt{p}^* \colon
  (KH_G(X, {\Z}/m)[\beta^{-1}])^{\compl}_{I_G} \to KH(X_G, {\Z}/m)[\beta^{-1}]
  \]
  are weak equivalences.
  \end{enumerate}
\end{thm}

Since the Quillen-Thomason $K$-theory agrees with its $\A^1$-invariant version with
finite coefficients prime to the characteristic, we get the following completion
theorem for the equivariant $K$-theory (of perfect complexes).

\begin{cor}\label{cor:Main-4}
Assume that the hypotheses of \thmref{thm:Main-3} are satisfied and $m \in k^\times$.  
 Then the following hold.
 \begin{enumerate}
   \item
    The map $\wt{p}^* \colon K_G(X, {\Z}/m)^{\compl}_{I_G} \to K(X_G, {\Z}/m)$
    is a weak equivalence.
  \item
    If $G$ is special, 
 $\wt{\pi}^* \colon K_G(X, {\Z}/m)^{\compl}_{I_G} \to K(X^\bullet_G, {\Z}/m)$
    is a weak equivalence.
 \item
    Item (2) of \thmref{thm:Main-3} is true also $K$-theory.
   \end{enumerate}
  \end{cor}

\vskip .2cm

\begin{remk}\label{remk:Main-4-0}
  A version  (where the $I_G$-completion is replaced by Carlsson's completion)
  of the Borel part (equivalence of $ \wt{p}^*$ in (1)(ii)) in \thmref{thm:Main-3}
  is claimed in \cite{CJP} for action
  of any linear algebraic group $G$ over an infinite perfect field $k$ on any
  normal quasi-projective $k$-scheme $X$. However, the key steps 
  of \cite{CJP} are its Lemma~2.8 and
  Proposition~2.9, whose correct proofs are not yet known.   
\end{remk}

\subsection{Completions at other ideals of $R(G)$}\label{sec:Max-other}
The Atiyah-Segal completion theorem describes the completion of equivariant
$K$-theory at the augmentation ideal. In order to fully recover equivariant
$K$-theory from the ordinary $K$-theory, this completion alone is however not 
enough and one instead needs to describe the completions of equivariant
$K$-theory at all maximal ideals of $R(G)$ in terms of ordinary $K$-theory.
This approach of describing equivariant algebraic $K$-theory was first studied by
Edidin-Graham \cite{EG-Duke}.

Let $k$ be an algebraically closed field of characteristic zero and
$G$ a linear algebraic group acting on
an algebraic space $X$. Let $g \in G(k)$ be a semisimple element and let
$\Psi \subset G(k)$ be the conjugacy class of $g$. We let
$R_k(G) = R(G) \otimes_{\Z} k$ and let $\fm_\Psi \subset R_k(G)$
denote the maximal ideal defined by $\Psi$. Let $X^g$ denote the fixed point
subspace of $X$ for the action of $g$ and let $Z_g \subset G$ denote the centralizer
of $g$. Let $K'_G(-, k)$ (resp. $K'(-, k)$) denote the equivariant (resp.
ordinary) $K$-theory functor with $k$-coefficients. 
The following result addresses the completion problem at an arbitrary  maximal ideal of
$R_k(G)$.

\begin{thm}\label{thm:Main-5}
  There are natural weak equivalences of spectra
  \[
  \wt{\pi}^* \colon K'_G(X,k)^{\compl}_{\fm_{\Psi}} \xrightarrow{\simeq}
   K'((X^g)^\bullet_{Z_g}, k); \ \ \ 
\wt{p}^* \colon K'_G(X,k)^{\compl}_{\fm_{\Psi}} \xrightarrow{\simeq} K'((X^g)_{Z_g}, k).
  \]
 \end{thm}

We also prove the analogous statement for the homotopy $K$-theory when $G$ acts
on $X$ with nice stabilizers.

\subsection{Thomason's completion of equivariant homology theories}
\label{sec:Hom**}
We now turn our attention to equivariant homology theory
(i.e., Hochschild, negative cyclic, cyclic and periodic cyclic homology)
of schemes over a field $k$ with group action.
Recall that these homology theories for commutative $k$-algebras can be computed
using explicit Hochschild complexes. Using their
{\'e}tale descent property for $k$-schemes,
one can also compute them for arbitrary schemes. However, the
equivariant Hochschild (and other) homology of a scheme with a group action is
defined abstractly as the Hochschild (and other)
homology theory of the dg category of equivariant
perfect complexes. As such, they are very difficult to compute in general,
just like equivariant $K$-theory.
It is therefore natural to ask if Thomason's completion
theorem holds for these equivariant
homology theories so that they could be computed using ordinary
homology theory. We settle this question as well.

For a $k$-scheme $X$, let $HH(X)$ (resp. $HC^{-}(X), HC(X), HP(X)$)
denote the Hochschild (resp. negative cyclic, cyclic, periodic cyclic) homology
of $X$ relative to $k$. For the action of a group $G$ on a $k$-scheme $X$, we let
$HH_G(X)$ (resp. $HC^{-}_G(X), HC_G(X), HP_G(X)$) denote the equivariant Hochschild
(resp. negative cyclic, cyclic, periodic cyclic) homology of $X$ relative to $k$.
These are objects of the derived category of $k$-vector spaces. We can therefore
consider them as spectra via the Eilenberg-MacLane functor.
Our result is the following.

\begin{thm}\label{thm:Main-6}
  Let $k$ be an algebraically closed field of characteristic zero.
  Let $G$ be a reductive group over $k$ acting on a smooth $k$-scheme $X$.
  Let $g \in G(k)$ be a semisimple element and let
  $\Psi \subset G(k)$ be the conjugacy class of $g$.
  Then for any smooth surjective map $p \colon U \to [{X^g}/{Z_g}]$ with
  $U \in \Sch_k$, one has a natural weak equivalence of spectra
  \begin{equation}\label{eqn:Homology-main-0*}
 \wt{\pi}^* \colon HH_G(X)^{\compl}_{\fm_\Psi} \xrightarrow{\simeq}
    \holim_{n} HH(U_n).
  \end{equation}
  The same is also true for negative cyclic, cyclic and periodic cyclic homology. 
\end{thm}

Note that $\holim_{n} HH(U_n)$ coincides with $HH((X^g)^\bullet_{Z_g})$
(the Hochschild homology of the bar construction associated to $Z_g$-action on
$X^g$) if we let $U = X^g$.  \thmref{thm:Main-6} therefore is a more general
result than \thmref{thm:Main-5}.
Note also that the special case of  \thmref{thm:Main-6} when $g$ is the identity
element of $G$ is an analogue of part (2) of \thmref{thm:Main-0} for
equivariant homology theories.

Recall that equivariant homology theories do not satisfy equivariant {\'e}tale
descent. Using that they do satisfy {\'e}tale descent for schemes, a direct
application of \thmref{thm:Main-6} however is the following.

\begin{cor}\label{cor:Main-10}
  Let $k$ be an algebraically closed field of characteristic zero.
  Let $G$ be a reductive group over $k$ acting on a smooth $k$-scheme $X$.
  Let $f \colon X' \to X$ be a $G$-equivariant {\'e}tale cover. Then the
  pull-back map
  \[
  f^* \colon HH_G(X)^{\compl}_{I_G} \to \holim_n HH_G(X'_n)^{\compl}_{I_G}  
  \]
  is a weak equivalence, where $\{X'_n\}_{n \ge 1}$ is the \u{C}ech nerve of $f$.
  In other words, $HH_G(X)^{\compl}_{I_G}$ satisfies equivariant {\'e}tale descent.
  The same is also true for negative cyclic, cyclic and periodic cyclic homology.
\end{cor}

\subsection{Results for action with finite stabilizers}
\label{sec:Hom***}
In the final set of results, we consider the special case of group action in which
the group $G$ acts on a scheme $X$ with finite stabilizers.
Equivalently, $[X/G]$ is a Deligne-Mumford stack.
In this case, we prove a stronger result, namely, that the
equivariant $KH$-theory of $X$ coincides with the ordinary $KH$-theory of the Borel
and bar constructions associated to the $G$-action on the inertia scheme $I_X$
even without passing to the completion. We prove similar results for the equivariant
homology theories. More precisely, our result is the following.

\begin{thm}\label{thm:Main-7}
  Let $k$ be an algebraically closed field of characteristic zero. 
  Let $G$ be a linear algebraic group over $k$ acting on a $k$-scheme $X$ with
  finite stabilizers. Let $I_X$ denote the inertia scheme. We now have the following.
  \begin{enumerate}
  \item
    There are natural weak equivalences of spectra
    \[
\wt{\Upsilon}^G_X \colon KH_G(X,k) \xrightarrow{\simeq} KH((I_X)^\bullet_G,k); \ \ \
\Upsilon^G_X \colon KH_G(X,k) \xrightarrow{\simeq} KH((I_X)_G,k).
    \]
 \item
    If $G$ is reductive and $X$ is smooth over $k$, then  for any smooth surjective
    map $p \colon U \to [{I_X}/{G}]$ with $U \in \Sch_k$,
 there is a natural weak equivalence of spectra
    \[
    \wt{\Upsilon}^G_X \colon HH_G(X) \xrightarrow{\simeq} \holim_n
    HH(U_n).
    \]
    The same also holds for negative cyclic, cyclic and periodic cyclic homology.
  \end{enumerate}
  \end{thm}

For the periodic cyclic homology, we prove the above result even if $X$ is
a singular scheme (cf. \thmref{thm:HP-sing}).

\vskip .2cm

Recall that a morphism of stacks $f \colon \sY \to \sX$ is called isovariant
{\'e}tale if it is an {\'e}tale map of stacks such that the canonical map
$I_\sY \to I_\sX \times_\sX \sY$ is an isomorphism (cf. \cite[\S~2.5]{Thomason-Duke-2}).
Recall also that a presheaf of spectra $F$ on the category of stacks is
said to satisfy isovariant {\'e}tale descent if given any isovariant {\'e}tale
cover $f \colon \sY \to \sX$, the canonical map $F(\sX) \to \holim_n F(\sY_n)$ is a
weak equivalence, where $\{\sY_n\}_{n \ge 1}$ is the \u{C}ech nerve of $f$.
Using the fact that homology theories satisfy {\'e}tale descent for schemes,
\thmref{thm:Main-7} implies the following.

\begin{cor}\label{thm:Main-8}
  Let $k$ be an algebraically closed field of characteristic zero and let
  $G$ be a reductive group over $k$. Let $f \colon X' \to X$ be a $G$-equivariant
  map between schemes over $k$ with $G$-action with finite stabilizers.
  Assume that the induced
  map of quotient stacks $\sX' \to \sX$ is an isovariant {\'e}tale cover.\
  Then the pull-back map $HH_G(X) \to \holim_n HH_G(X'_n)$ is a weak equivalence,
  where $\{X'_n\}_{n \ge 1}$ is the \u{C}ech nerve of $f$.
  The same is also true for negative cyclic, cyclic and periodic cyclic homology.
\end{cor}

\vskip .2cm

From \thmref{thm:Main-7}, we also obtain the following
computation of all equivariant homology groups. To state it, we let
$~_{s_{\le m}}\Omega^\bullet_{I_\sX}$ and $~_{s_{\ge m}}\Omega^\bullet_{I_\sX}$ denote the
stupid truncations of the de Rham complex $\Omega^\bullet_{I_\sX}$ of
$I_{\sX} = [{I_X}/G]$. Let $H^*_{dR}(I_\sX) = \H^*_\et(I_\sX, \Omega^\bullet_{I_\sX})$. 

\begin{thm}\label{thm:Main-9}
Let $k$ be an algebraically closed field of characteristic zero.
Let $G$ be a reductive group over $k$ acting on a smooth $k$-scheme $X$ with
finite stabilizers. For any integer
$n \ge 0$, we then have natural isomorphisms of $k$-vector spaces
  \[
  \pi_n(HH_G(X)) \xrightarrow{\cong} {\underset{0 \le i \le d}\bigoplus}
  H^{i-n}_\et(I_\sX, \Omega^i_{{I_\sX}/k});
  \]
  \[
  \pi_n(HC^{-}_G(X)) \xrightarrow{\cong} {\underset{0 \le i \le d-n}\bigoplus}
  H^{n+2i}_\et(I_\sX, ~_{s_{\ge i+n}}\Omega^\bullet_{I_\sX}).
  \]
  If the $G$-action on $X$ is proper, then we also have
  \[
  \pi_n(HC_G(X))  \xrightarrow{\cong} {\underset{-d \le i \le n}\bigoplus}
  H^{n-2i}_\et(I_\sX, ~_{s_{\le n-i}}\Omega^\bullet_{I_\sX});
  \]
  \[
  \pi_n(HP_G(X)) \xrightarrow{\cong} {\underset{i \in \Z}\prod} H^{n+2i}_{dR}(I_{\sX}).
  \]
 \end{thm}

When $G$ is finite, \thmref{thm:Main-9} for Hochschild homology can be easily derived
from \cite[Prop.~3.1]{Baranovsky} which uses a different method. It can also be derived
from \cite[Thm.~4.9]{ACH}. For separated Deligne-Mumford stacks, \thmref{thm:Main-9}
for Hochschild homology is asserted in \cite[Prop.~2.13]{HLP}
(but see Remark~\ref{remk:Homology-ET-6}).

\subsection{Future direction}\label{sec:Future}
We indicate a few directions in which the results of this paper may be extended.
First, we would like to know if the second part of \thmref{thm:Main-3} could be
extended to positive characteristic and to more general group actions.
Our proof uses equivariant resolution of singularities which is a hindrance
in positive characteristic. If we work with finite (but prime to characteristic)
coefficients, this problem
can be resolved if a version of alteration could be found for stacks but we are not
aware of its existence.

Second, we would like to appropriately extend our results on homology theories
in characteristic zero to similar results for their topological counterparts
(i.e., topological Hochschild, cyclic and periodic cyclic homology with
$\Z_p$-coefficients) in characteristic $p > 0$. In particular, we would like to
know if these topological  homology theories satisfy the equivariant (isovariant)
{\'e}tale descent for Deligne-Mumford stacks.
These questions will be taken up in a separate project.

\subsection{Layout of the paper}\label{sec:Layout}
The most difficult among our results on equivariant $K$-theory is \thmref{thm:Main-0}.
This should not be surprising to a reader who is familiar with
Thomason's proof of his completion
theorem in \cite{Thomason-Duke-1} (which is a very special case of  
\thmref{thm:Main-0}) which spreads into too many ingenious but very intricate steps.
We explain our strategy for proving \thmref{thm:Main-0} which is
completely different from Thomason's proof (and the latter does not work in
the general case we consider).

First of all, it is not hard to realize that there is simply no direct way to compare
the equivariant $K$-theory of $X$ with the ordinary $K$-theory of its bar
construction $X^\bullet_G$.
This is in sharp contrast to the Borel construction (which is an inductive
limit of open subschemes of equivariant vector bundles over $X$ and hence allows
one to use homotopy invariance and localization to directly
relate its ordinary $K$-theory with the equivariant $K$-theory of $X$).
We use the machinery of $\A^1$-homotopy theory to overcome this problem.

Our idea is to use \thmref{thm:Main-2} to derive \thmref{thm:Main-0}.
The first step do so is to show that there is a uniquely defined
motivic space $X_G$ (called the motivic Borel space) in the 
$\A^1$-homotopy category of Morel-Voevodsky \cite{Morel-Voevodsky}
which does not depend on the choice of an admissible gadget we make in the  
definition of Borel equivariant $K$-theory (cf. \propref{prop:Bar-ind-Gen}).
This allows us to get rid of the dependency problem (on admissible gadgets)
which is essential in proof of \thmref{thm:Main-0}.

We would now like to construct a map between
$X^\bullet_G$ and $X_G$. Unfortunately, no such map exists at the level of
simplicial schemes. To get around this issue, we introduce another
motivic space, namely, $\wt{X}_G :=X \stackrel{G}{\times} \sE_G$ in the 
$\A^1$-homotopy category, where $\sE_G$ is the universal $G$-torsor over
the classifying motivic space $\sB_G$ for the group $G$, constructed by
Morel-Voevodsky. To be able to work with $\sE_G$, we need $G$ to be special
because maps to $\sB_G$ in the $\A^1$-homotopy category do not represent
Nisnevich local $G$-torsors unless $G$ is special. This explains our dependency on
$G$ being special.

Having constructed $\wt{X}_G$, the next step is to construct a direct map
from $X_G$ to it.  
Using the representability of $K$-theory by an object of
Voevodsky's stable homotopy category $\sS\sH_k$ (cf. \cite{Voev-ICM}),
we then show that this map induces a weak equivalence between their $K$-theory
spectra. Next, we show that the projection $X {\times} \sE_G \to X$ induces a canonical
map from the $I_G$-completion of the equivariant $K$-theory spectrum of $X$ to the
ordinary $K$-theory spectrum of $\wt{X}_G$. At this stage, we invoke
\thmref{thm:Main-2} to conclude that this map is a weak equivalence.

We are now left to show that the $K$-theory of $\wt{X}_G$ is canonically
equivalent to that of the bar construction $X^\bullet_G$.
For this, we construct a simplicial scheme $E^\bullet_G(X)$ which is a $G$-torsor
over $X^\bullet_G$. In particular, the $K$-theory of the latter
can be realized as the equivariant $K$-theory of $E^\bullet_G(X)$.
To conclude, we show that there is a direct $G$-equivariant map
$E^\bullet_G(X) \to X {\times} \sE_G$ which induces a weak equivalence between
the ordinary $K$-theory of $\wt{X}_G$ and the equivariant $K$-theory of
$E^\bullet_G(X)$.

To prove \thmref{thm:Main-5}, our major ingredient is \thmref{thm:Max-Main}.
This is a derived version of the nonabelian completion theorem of Edidin-Graham
\cite{EG-Duke} and is a result of independent interest in equivariant $K$-theory.
We expect this result to have other future applications.
Even if \thmref{thm:Max-Main} is directly inspired by the 
nonabelian completion theorem of Edidin-Graham, its proof can not be derived from that
of the latter. One reason for this is that the homotopy groups in general behave
very poorly when one passes to the completion. In fact, this was the main
difficulty Thomason encountered in extending Atiyah-Segal theorem to algebraic
$K$-theory. This forces us go to through several reduction steps to
prove \thmref{thm:Max-Main}.

We prove \thmref{thm:Main-5} by combining Theorems~\ref{thm:Main-0} and
~\ref{thm:Main-2} with the derived nonabelian completion theorem. We 
prove these results for the homotopy $K$-theory by reduction to the $K'$-theory
by means of an equivariant version of cdh descent, shown in \cite{Hoyois-Krishna}.

To prove the completion theorem for the equivariant homology theories, we
first deal with the Hochschild homology. We prove this case by writing the
equivariant Hochschild homology as the derived global section of To{\"e}n's derived
loop space construction \cite{Toen-DAG}. 
We then use the formal HKR theorem for Hochschild homology by Ben-Zvi and
Nadler \cite{Ben-Nadler} to complete the proof. The proof for
negative cyclic homology is easily reduced to that of Hochschild homology.

The major difficulty lies in proving the completion theorems for cyclic and periodic
cyclic homology because it involves commuting homotopy limits and colimits.
In order to handle this issue, we use a derived HKR-filtration of the
equivariant Hochschild
homology and study the resulting graded pieces using arguments of
Bhatt-Morrow-Scholze \cite[\S~3]{BMS}. We then show that the homotopy orbit spaces
for the $S^1$-action (which is a homotopy colimit) of these graded pieces commute
with latter's homotopy limit. We then use an inductive argument and boundedness of the
derived odd tangent bundle of $[X/G]$ to finish the proof of the completion theorem for
cyclic homology. The result for periodic cyclic homology then follows by a formal
argument.

We prove our results on the equivariant $K$-theory and homology theories for 
action with finite stabilizers by combining Theorems~\ref{thm:Max-Main},
~\ref{thm:Gen-AS-max} and ~\ref{thm:Main-7} with \propref{prop:Fin-support}. 

In \S~\ref{sec:Rep-G}, we establish some important results of independent interest
about the representation ring
and in \S~\ref{sec:Recall*}, we review basic results on equivariant $K$-theory and
homology theories. In \S~\ref{sec:Twisting*} and \S~\ref{sec:Twist-op}, we
complete two key technical steps, namely, the decomposition theorem and
the construction of twisting operator. We complete the proofs of
Theorems~\ref{thm:Main-2} and ~\ref{thm:Main-5} in \S~\ref{sec:Compln*} and 
\S~\ref{sec:Max-com}. We finish the proof of \thmref{thm:Main-0} in
\S~\ref{sec:Thom-Gen**}. We prove \thmref{thm:Main-6} in \S~\ref{sec:Cyc**}.
In \S~\ref{sec:DM-Homology}, we prove \thmref{thm:Main-7} 
and its applications.

\subsection{Notations}\label{sec:Notn}
Throughout the paper, $k$ will denote an arbitrary base field. 
 We shall use the notation $\Sch_k$ for the category of  separated and finite type
 $k$-schemes. An object of $\Sch_k$ will be called a $k$-scheme. The product of two
 schemes $X$ and $Y$ in $\Sch_k$ will be denoted by $X \times Y$.
 We let $\Sm_k$ denote the subcategory of $\Sch_k$ consisting of smooth $k$-schemes.
For any ring extension $k \inj R$ and $X \in \Sch_k$, we let
$X_R = X \times \Spec(R)$. If $X$ is affine, we let $k[X]$ denote its
coordinate ring.

A $k$-group $G$ will mean a linear algebraic group (i.e., a smooth closed subgroup
scheme of a general linear group $GL_n$) over $k$.
We let $\Sch^G_{k}$ (resp. $\Sm^G_{k}$) denote the
subcategory of $\Sch_k$ (resp. $\Sm_k$) consisting of schemes with $G$-action and
$G$-equivariant morphisms. An object $X \in \Sch^G_k$ will be called
$G$-quasi-projective if it is a quasi-projective $k$-scheme with linear $G$-action.
We let $\Spc_k$ denote the category of finite type separated algebraic spaces
(termed as $k$-spaces) over $k$ and let
$\Spc^G_k$ denote the subcategory of $\Spc_k$ consisting of
algebraic spaces with $G$-action and $G$-equivariant morphisms.

We shall let $\sS\sH$ denote the stable homotopy category of
topological $S^1$-spectra and $\1 \in \sS\sH$ denote the sphere spectrum.
If $E$ is a presheaf on a subcategory of $\Sch_k$ with values in $\sS\sH$ and
$L$ is any field of characteristic
zero, we shall let $E(X, L)$ be the unique spectrum having the
property that $\pi_i(E(X, L)) = \pi_i(E(X)) \otimes L$ for every $i \in \Z$. For a
positive integer $n$, $E(X, {\Z}/n)$ will denote the smash product spectrum
$E(X) \wedge {\1}/n$, where ${\1}/n$ is the mod-$n$ Moore spectrum.
Its homotopy groups are connected to those of $E(X)$ via the universal coefficient
theorem. 

All rings in this paper will be unital and commutative. Given a ring $R$, we let
$H_R \in \sS\sH$ denote the Eilenberg-MacLane ring spectrum associated to $R$.
The tensor product of two abelian groups $A$ and $B$ will be written as
$A \otimes B$.

\section{Geometry of representation ring}\label{sec:Rep-G}
Let $k$ be a field. We let ${\Grp}_k$ denote the category of $k$-groups and
group homomorphisms. For $G \in \Grp_k$, we let $R(G)$ denote the
representation ring of $G$. We shall assume all representations of $G$ to be
$k$-rational and finite-dimensional. For any ring $\Lambda$, we let
$R_\Lambda(G) = R(G) {\otimes} \Lambda$. In this section, we collect some
results about the ring $R_k(G)$ and its spectrum $\Spec(R_k(G))$.
When $k$ is the field of complex
numbers, these results were proven in \cite[\S~2]{EG-Adv}
using the representation theory of compact complex Lie groups. 
The results of this section will be used throughout this paper.

\subsection{Finite generation}\label{sec:Fin-gen}
For $G \in \Grp_k$ and $X \in \Sch^G_k$, we let $[X/G]$ denote the stack quotient and
let $K_0([X/G])$ denote the Grothendieck group of 
locally free sheaves on $[X/G]$. We write $\pt = \Spec(k)$ and $BG = [{\pt}/G]$.
We shall let $\ov{k}$ denote an algebraic closure of $k$.
We begin with the following result which should be known to experts.

\begin{lem}\label{lem:RG-inj}
  Let $G \in \Grp_k$ and ${k'}/k$ an algebraic field extension. Then the pull-back
  maps $R(G_l) \to R(G_{k'})$ for finite subextensions $l \subseteq k'$ of $k$ induce
  an isomorphism of rings
  \begin{equation}\label{eqn:RG-inj-0}
    \gamma^G_k \colon {\underset{l \subseteq k', [l:k] < \infty}\colim} \
    R(G_l) \xrightarrow{\cong} R(G_{k'}).
    \end{equation}
\end{lem}
\begin{proof}
  It is clear that $\gamma^G_k$ is a ring homomorphism.
  To prove that this is an isomorphism, we first observe
  that $R(G_l) = {F(G_l)}/{E(G_l)}$ for every $l \subset k'$ if we let $F(G_l)$ denote
    the free abelian group
    on the set of isomorphism classes of representations of $G_l$ and $E(G_l)$ denote
    the subgroup of $F(G_l)$ generated by virtual representations
  $[V'] + [V''] - [V]$ associated to short exact sequences of representations
  $0 \to V' \to V \to V'' \to 0$.
    We next observe using the faithfully flat descent that given any finite collection
    of $G_{k'}$-representations and a finite collection of morphisms (or isomorphisms)
  between them, we can find a finite subextension
  $l \subseteq k'$ such that these two collections are defined over $l$.
  Lastly, we observe that any finite collection of short exact sequences of
  $G_{k'}$-representations is also defined (and is short exact) over a finite
  subextension $l \subseteq k'$. 
  It is easy to check that these observations imply the claim that $\gamma^G_k$ is
  bijective.
\end{proof}

Recall that $G \in \Grp_k$ is called split reductive if it is a connected reductive
$k$-group and admits a maximal torus which splits over $k$.

\begin{lem}\label{lem:Noether-Rep-3}
  Let $G \in \Grp_k$ and ${k'}/k$ a field extension. Then the pull-back map
  $R(G) \to R(G_{k'})$ is injective. This map is an isomorphism if $G$ is
  split reductive.
\end{lem}
\begin{proof}
We first assume that ${k'}/k$ is an algebraic extension. 
By \lemref{lem:RG-inj}, it suffices in this case to prove the lemma when
${k'}/k$ is a finite extension. In the latter case, we have the pull-back and
push-forward maps
  $R(G) \xrightarrow{\iota^*} R(G_{k'})$ and $R(G_{k'}) \cong K_0([{\Spec(k')}/{G}])
\xrightarrow{\iota_*} R(G)$ such that $\iota_* \circ \iota^*$ is
multiplication by $[k':k]$.
Since $R(G)$ is the free abelian group on the set of isomorphism classes of
  irreducible representations of $G$ (cf. \cite[App.~C]{Karpenko-Merkurjev}),
  it follows that $R(G) \to R(G_{k'})$ is injective.

In the general case, we can find field extensions $k \subset l \subset k'$, where
  $l/k$ is a purely transcendental extension and ${k'}/l$ is an algebraic extension.
  If $k$ is infinite, we can write $R(G_{l})$ as the filtered colimit of $K_0([U/G])$,
  where $U \subset \A^n_k$ is an open subscheme for some $n \ge 0$ with the trivial
  $G$-action (cf. \cite[Thm.~2.5]{Khan-JJM}).
  The pull-back maps $R(G) = K_0(BG) \to K_0([{\A^n_k}/G])$ and
  $K_0([{\A^n_k}/G]) \to K_0([U/G])$ are bijective (by the homotopy invariance) and
  surjective, respectively.
  On the other hand, our hypothesis implies that every such $U$ contains $k$-rational
  points. This implies that the pull-back $R(G) \to K_0([U/G])$ admits a section.
We deduce that this map is actually an isomorphism.
  Passing to the limit, we conclude that $R(G) \to R(G_l)$ is an isomorphism.

  If $k$ is not necessarily infinite, we look at the commutative diagram
  \begin{equation}\label{eqn:Noether-Rep-3-0}
    \xymatrix@C1pc{
      R(G) \ar[r] \ar[d] & R(G_{\ov{k}}) \ar[d] \\
      R(G_l) \ar[r] & R(G_{l'}),}
  \end{equation}
  where $l' = l \otimes_k \ov{k}$. It is easy to check that $l'$ is a field which is
  a purely transcendental extension of $\ov{k}$. In particular, the right vertical
  arrow is bijective. We showed above that the horizontal arrows are injective.
  It follows that the left vertical arrow is injective.
  
  If $G$ is split reductive and $T$ is a split maximal torus of $G$ with the
  associated Weyl group $W$, then  \cite[Thm.~4]{Serre} says that
  the canonical map of representation rings induces a ring isomorphism 
\begin{equation}\label{eqn:inv}
R(G) \xrightarrow{\cong} {R(T)}^W.
\end{equation}
Since the map $R(T) \to R(T_{k'})$ is clearly an isomorphism, a similar assertion
also holds for $G$.
\end{proof}

\begin{prop}\label{prop:finiteR}
  Let $G \in \Grp_k$ and let $H$ be a closed subgroup of $G$. Assume that either $H$
  is connected or $\Char(k) = 0$. Then the restriction map of integral 
representation rings $R(G) \to R(H)$ is finite. In particular, $R(H)$ is
Noetherian.
\end{prop}
\begin{proof}
  It is enough to prove the proposition when $G$ is a general linear group over $k$.
  We divide the proof of the latter case into two cases.

  {\bf Case~1:} $k$ is algebraically closed.
    
  If $H$ is a maximal torus $T$ of $G$ and $W$ is the associated Weyl 
 group, then the natural map ${\Z}[N] \to R(T)$ is
  an isomorphism of rings (where $N$ is the character group of $T$) which implies
  that $R(T)$ is a finite type algebra over $\Z$. In particular, it is finite over 
${R(T)}^W$ which is then also a finite type $\Z$-algebra 
  (cf. proof of \cite[Lem.~3.9]{Krishna-Adv}) and hence Noetherian. It follows from
  ~\eqref{eqn:inv} that $R(G)$ is Noetherian.

 Suppose now that $H$ is any closed subgroup of $G$. If $V$ is an irreducible 
representation of $H$ and if the unipotent radical $R_uH$ of $H$ acts 
non-trivially on $V$, then the invariant subspace of $V$ for this action 
is a non-zero subspace (cf. \cite[Cor.~10.5]{Borel}). Since $R_uH$ is normal in $H$,
such a subspace must be a subrepresentation of $G$ inside $V$.
But this contradicts the irreducibility of $V$. Hence
$R_uH$ must act trivially on $V$. Since $R(H)$ is a free abelian group
on the set of irreducible representations of $H$
(cf. \cite[App.~C]{Karpenko-Merkurjev}), this implies that the
map $R(H) \to R(H/{R_uH)}$ is an isomorphism. Hence we can assume that 
$H$ is reductive. 

If $H$ is connected, we choose a maximal
torus $S$ of $H$. Since $k$ is algebraically closed, there is a maximal
torus $T$ of $G$ containing $S$. We thus get a commutative diagram
of representation rings
\[
\xymatrix{
R(G) \ar[r] \ar[d] & R(T) \ar[d] \\
R(H) \ar[r] & R(S).}
\]
We showed above that the horizontal arrows are finite and injective 
maps. Since $S$ is a subtorus of $T$, there is a decomposition $T = S 
\times S'$ and hence the map $R(T) \to R(S)$ is surjective. In particular,
$R(S)$ is finite over $R(G)$. Since $R(G)$ is Noetherian, we conclude that $R(H)$ is
finite over $R(G)$.

Suppose now that $\Char(k) = 0$ and $H$ is not necessarily connected.
We first assume that $k \subset \C$ and look at the commutative diagram
  of rings
  \begin{equation}\label{eqn:Noether-Rep-0}
    \xymatrix@C.8pc{
      R(G) \ar[r] \ar[d] & R(G_{\C}) \ar[d] \\
      R(H) \ar[r] & R(H_{\C}),}
  \end{equation}
  where the horizontal arrows are induced by the change of fields.
  The top horizontal arrow is an isomorphism of Noetherian rings and the
  right vertical arrow is a finite ring homomorphism by
  \cite[Prop.~3.2]{Segal} (see also \cite[Prop.~2.3]{EG-Adv}).
  It follows that $R(H_{\C})$ is finite over $R(G)$. It suffices therefore to
  show that the base change map $R(H) \to R(H_{\C})$ is injective. But this
  follows from \lemref{lem:Noether-Rep-3}.

If $k$ is not necessarily a subfield of $\C$, we choose a finitely generated
  subfield $k' \subset k$ such that $G$ is defined over $k'$. We can write
  $k = \varinjlim_\lambda k_\lambda$,
  where $\{k_\lambda\}$ is the filtered family of all finitely generated subfields of
  $k$ which contain $k'$. We now look at the commutative diagram 
\begin{equation}\label{eqn:Noether-Rep-1}
    \xymatrix@C.8pc{
      R(G_{k_\lambda}) \ar[r] \ar[d] & R(G) \ar[d] \\
      R(H_{k_\lambda}) \ar[r] & R(H).}
\end{equation}

The lower horizontal arrow in ~\eqref{eqn:Noether-Rep-1} is injective for each
$\lambda$ by \lemref{lem:Noether-Rep-3}. It now follows easily that the canonical map
$\varinjlim_\lambda R(H_{k_\lambda}) \to R(H)$ is an isomorphism. Since $R(H)$ is a
finitely generated ring by \cite[Prop.~4.1]{Thomason-Duke-2}, the lower horizontal
arrow in ~\eqref{eqn:Noether-Rep-1} must actually be bijective for some $\lambda$.
On the other hand, the left vertical arrow is a finite ring homomorphism
since each $k_\lambda$ admits an embedding inside $\C$. Since the top horizontal
arrow is a bijection, we deduce that the right vertical arrow is a finite ring
homomorphism.

{\bf Case~2:} $k$ is not algebraically closed.

We consider the commutative diagram of representation rings
\[
\xymatrix{
R(G) \ar[r] \ar[d] & R(G_{\ov{k}}) \ar[d] \\
R(H) \ar[r] & R(H_{\ov k})}
\]
Since the top horizontal arrow is an isomorphism and the right vertical arrow is
a finite ring homomorphism, it follows that $R(H_{\ov k})$ is finite over $R(G)$.
We now use \lemref{lem:Noether-Rep-3} to conclude that $R(H)$ is finite over
$R(G)$.
\end{proof}

\begin{remk}\label{remk:RG-disconnected}
  If $\Char(k) > 0$ and $H^0_{\ov{k}}$ is the identity component of $H_{\ov{k}}$, then
  Brauer's theorem says that $R(H)$ is a finitely generated abelian group.
  Combining this fact with \propref{prop:finiteR} for $H^0_{\ov{k}}$, it should be
  possible to deduce that $R(H)$ is finite over $R(G)$ even if $H$ is not connected.
  But we do not know how to complete the proof.
  \end{remk}

Recall that the augmentation ideal $I_G \subset R(G)$ is the kernel of the rank map
$R(G) \surj \Z$. If $\Lambda$ is any torsion-free ring, the kernel of the augmentation
map $R(G)_\Lambda \surj \Lambda$ will also be called the augmentation ideal.
We shall need the following result about $I_G$.

\begin{lem}\label{lem:Noether-Rep}
  The augmentation ideal is the radical of a finitely generated ideal in $R(G)$.
 \end{lem}
\begin{proof}
We choose a closed embedding $G \inj GL_{n}$ and let
  $R(GL_{n}) \to R(G)$ denote the induced map. Let $I'_G \subset I_G$ denote the
  extension of $I_{GL_{n}}$ under this ring homomorphism. One knows from
  \cite[Cor.~6.1]{EG-RR} that the $I_G$-adic topology on $R(G)$ coincides with its
  $I'_G$-adic topology. This implies that $I_G = \sqrt{I_G} = \sqrt{I'_G}$.
  On the other hand, one knows that $R(GL_{n})$ is Noetherian
 which implies that $I'_G \subset R(G)$ is finitely generated.
\end{proof}

\subsection{Relation with class functions}\label{sec:Class-F}
Let $G \in \Grp_k$.
Recall from \cite[\S~2]{Krishna-Ravi-1} (see also \cite[Chap.~1, \S~1]{GIT}) that an
action of $G$ on an affine $k$-scheme $X$ is equivalent to the map of $k$-algebras
\begin{equation}\label{eqn:action}
\wh{\alpha} : k[X] \to  k[G] {\otimes}_k k[X]
\end{equation}
which satisfies the commutativity of certain diagrams of $k$-algebras.

\begin{defn}\label{defn:inv}
We shall say that an element $f$ of $k[X]$ is $G$-invariant for the above action 
if ${\wh{\alpha}}(f) = 1 {\otimes} f$. 
\end{defn}
It is easy to check that the set of $G$-invariant elements in $k[X]$ is
a $k$-subalgebra. We denote it by ${k[X]}^G$. This is the
coordinate ring of the universal geometric quotient (assuming it exists)
of $X$ for the action of 
$G$ on $X$. In particular, if $k$ is algebraically closed, then ${k[X]}^G$
is the set of invariant elements in $k[X] = \Hom_{\Sch_k}(X, {\A}^1_k)$ for the 
action of $G(k)$ on $k[X]$ given by 
\begin{equation}\label{eqn:action1}
G(k) \times k[X] \to k[X] 
\end{equation}
\[
(g, f) \mapsto f^g,
\]
where $f^g(x) = f(gx)$.

We now specialize to the adjoint (conjugation) action of $G$ on itself.
In terms of the dual action of Hopf algebras, we represent this action by
$k[G] \xrightarrow{\wh{ad}} k[G] {\otimes}_k k[G]$ and write ${k[G]}^G$ as $C[G]$.
If $k$ is algebraically closed, it follows from ~\eqref{eqn:action1} that
$C[G]$ is the ring of class functions on $G$, i.e., those regular functions on $G$ which
take a constant value on a conjugacy class.
In order to understand the geometry of $\Spec(R_k(G))$, we need to connect $R_k(G)$  
with $C[G]$. We do it as follows.

Let $(V, \rho)$ be an $n$-dimensional representation of $G$ given
by the morphism of $k$-groups $G \xrightarrow{\rho} GL(V)$.
Let ${\chi} \colon GL(V) \subset {\rm End}(V) \to {\A}^1_k$ denote the
character morphism described algebraically by the composite map
\begin{equation}\label{eqn:character-0}
k[t] \xrightarrow{\wh{\chi}} k[X_{ij}] \inj k[X_{ij}, 1/{\rm det}],
\end{equation}
\[
{\wh{\chi}}(t) = \stackrel{n}{\underset {i=1}{\Sigma}}{X_{ii}}.
\]
Composing this with $\rho \colon G \to GL(V)$, we get the character map
\begin{equation}\label{eqn:character} 
k[t] \xrightarrow{{\wh{\chi}}_{\rho}} k[G].
\end{equation}

It is easy to check using the standard properties of trace of matrices that the
assignment $(V, \rho) \mapsto {\wh{\chi}}_{{\rho}}(t)$
induces a $k$-algebra morphism $R_k(G) \xrightarrow{{\phi}_G} k[G]$.

\begin{prop}\label{prop:character1}
Let $G$ be a split reductive $k$-group. Then the map $\phi_G$ induces 
an injective homomorphism of $k$-algebras
\begin{equation}\label{eqn:char}
{\phi}_G \colon R_k(G) \to C[G].
\end{equation}
It is an isomorphism if $\Char(k) =0$.
\end{prop}
\begin{proof}
We first have to show that the image of ${\phi}_G$ is contained in $C[G]$. 
For this, it suffices to show for any representation $(V, \rho)$ that
$\wh{ad} \circ {\wh{\chi}}_{\rho} = 1 \otimes 
{\wh{\chi}}_{\rho}$. Since the map $k[G] {\otimes}_k k[G] \to
{\ov k}[G_{\ov k}] {\otimes}_{\ov k} {\ov k}[G_{\ov k}]$ is injective,
it suffices to show that $\wh{ad}_{\ov k}\circ{\wh{\chi}}_{{\rho}_{\ov k}}
= 1 \otimes {\wh{\chi}}_{{\rho}_{\ov k}}$. Thus we can assume that
$k$ is algebraically closed. In this case, we can use ~\ref{eqn:action1} 
to reduce to showing that ${\chi}_{\rho}(ghg^{-1}) = {\chi}_{\rho}(h)$, where
$\chi_\rho \colon G \to \A^1_k$ is induced by
${\wh{\chi}}_{\rho}$. But this is standard.

To show that $\phi_G$ is injective in general and bijective in characteristic zero,
we fix a split maximal torus $T \subset G$ and a Borel $B$ containing $T$. Then $G$
 is uniquely described by a root system over $k$, hence its Weyl
group $W$ is a constant finite group which does not depend on the
base change of $G$ by any field extension of $k$. Moreover, $W$ acts on
$T$ and hence on $R_k(T)$ and $C[T] \cong k[T]$ such that the map ${\phi}_T$
is $W$-equivariant. Thus we have a commutative diagram
\begin{equation}\label{eqn:char1}
\xymatrix{
  R_k(G) \ar[r]^{{\phi}_G} \ar[d]_{{\eta}_R} & C[G] \ar@{.>}[d] \ar@{^{(}->}[r] &
  k[G] \ar[d]^{{\eta}_C} \\
{R_k(T)}^W \ar[r]_{{\phi}^W_T} & {C[T]}^W \ar@{^{(}->}[r] & k[T],}
\end{equation}
where $\eta_R$ and $\eta_C$ are the restriction maps.

One knows from \cite[Thm.~1.13]{Thomason-Duke-2} that the restriction map
$R(G) \to R(T)$ is split injective and \cite[Thm.~4]{Serre} says that
its image is $R(T)^W$. This implies that
$R_k(G) \to R_k(T)$ is split injective whose image lies in $R_k(T)^W$.
If $\Char(k) = 0$, it follows that $R_k(G) \to R_k(T)^W$ is a bijection
(cf. \cite[Lem.~3.1]{Krishna-Adv}). We have thus shown that the left vertical arrow
in ~\eqref{eqn:char1} is injective in general and bijective if $\Char(k) = 0$.
Since $T$ is a split torus, ${\phi}^W_T$ is an isomorphism. 
This already shows that $\phi_G$ is injective. 
The remaining claim that it is bijective if $\Char(k) =0$ follows from
\lemref{lem:character2}.
\end{proof}

\begin{lem}\label{lem:character2}
  Under the assumptions of \propref{prop:character1} and notations of its proof,
  the canonical map $C[G] \to C[T]$ is injective and its image lies in ${C[T]}^W$ if
  $\Char(k) = 0$.
\end{lem}
\begin{proof}
  If $k$ is finitely generated as a field, then it has an embedding into $\C$.
  In this case, we can use
  the inclusion $C[G] \inj C[G_{\C}]$ and the $W$-equivariant inclusion
  $C[T] \inj C[T_{\C}]$ to reduce the proof of the lemma to the case when $k = \C$.
  But this case is well known (cf. \cite[\S~2]{EG-Adv}).
  In general, we observe using the root systems that $G$ and $T$ are defined over 
  $\Q$. We let $G'$ (resp. $T'$) be a form of $G$ (resp. $T$) over
  $\Q$. We write $k = \varinjlim_\lambda k_\lambda$,
  where $\{k_\lambda\}$ is the filtered family of all finitely generated subfields of
  $k$. We let $G_\lambda$ (resp. $T_\lambda$) denote the base change of $G'$ (resp. $T'$)
  via the inclusion $\Q \inj k_\lambda$. 

  Since $\varinjlim_{\lambda} k_\lambda[G_\lambda] \to k[G]$
  is an isomorphism and each $k_\lambda[G_\lambda] \otimes_{k_\lambda} k_\lambda[G_\lambda]
  \to k[G] \otimes_{k} k[G]$ is injective, we see using ~\eqref{eqn:action}
  and Definition~\ref{defn:inv} that $\varinjlim_{\lambda} C[G_\lambda] \to C[G]$
  is also an isomorphism. The same also holds for $T$. In particular,
  $\varinjlim_{\lambda} C[T_\lambda]^W \to C[T]^W$ is an isomorphism. This reduces the
  proof of the lemma to the case of finitely generated fields.
  \end{proof}

\subsection{Results of Edidin-Graham}\label{sec:EG-*}
If $k$ is an algebraically closed field of characteristic zero,
one can prove \propref{prop:character1} for any reductive $k$-group
using a straightforward generalization of \cite[Prop.~2.2]{EG-Adv}.

\begin{prop}\label{prop:character1*}
  Assume that $k$ is algebraically closed of characteristic zero
  and $G$ is a reductive $k$-group. Then ${\phi}_G \colon R_k(G) \to C[G]$
  is an isomorphism.
\end{prop}
\begin{proof}
  Mimicking the proof of \cite[Prop.~2.2]{EG-Adv}, we let $\wh{G}$ denote the set of
  irreducible representations of $G$ and consider the map
  \begin{equation}\label{eqn:character1*-0}
    \phi'_G \colon {\underset{V \in \wh{G}}\bigoplus} V^* \otimes_k V \to k[G],
  \end{equation}
  where on each $V \in \wh{G}$, we have
  $\phi'_G(\lambda \otimes v) = (g \mapsto \lambda(gv))$ for
  $\lambda \in V^*$ and $v \in V$. Then $\phi'_G$ is an isomorphism of
  $G$-representations by \cite[Thm.~27.3.9]{Tauvel}
  if we let $G$ act diagonally on $V^* \otimes V$ and by conjugation on $k[G]$.
  Taking the $G$-invariants on both sides and noting that
  $(V^* \otimes_k V)^G \cong \Hom_G(V,V) \cong k_V \subset R_k(G)$,
  we get an isomorphism
  $\phi'_G \colon {\underset{V \in \wh{G}}\bigoplus} k_V \to C[G]$,
  where $k_V$ denotes the 1-dimensional $k$-vector space spanned by the
  class of $V$ in $R_k(G)$. But one easily checks that $\phi'_G$ coincides with
  $\phi_G$ if we identify ${\underset{V \in \wh{G}}\bigoplus} k_V$ with $R_k(G)$
  (cf. \cite[App.~C]{Karpenko-Merkurjev}).
  \end{proof}

For $G \in \Grp_k$, we let $G_s(k)$ denote the set of $k$-rational points $g \in G$
such that $g$ is a semisimple closed point. For $g \in G_s(k)$ and a
representation $(V, \rho)$ of $G$, we let $\tr_g([V]) = \chi_\rho(g) =
(\chi \circ \rho)(g) \in k$ (cf. ~\eqref{eqn:character-0}). This extends linearly
to a $k$-algebra homomorphism $\tr_g \colon R_k(G) \to k$ (cf. \cite[\S~2.2]{EG-Adv}).
Since the trace of $g$ on the 1-dimensional trivial
representation is $1 \in k^\times$, it follows that $\tr_g$ is surjective. In
particular, $\Ker(\tr_g)$ is a maximal ideal of $R_k(G)$. We denote this by $\fm_g$.
Letting $\Omega(R_k(G))$ denote the max-spectrum of $R_k(G)$, we thus get a function
$\tr_G \colon G(k) \to \Omega(R_k(G))$ given by $\tr_G(g) = \fm_g$.
It is easy to see from the
above discussion that $\tr_G$ is a class function on $G_s(k)$.
If $\Psi$ is the conjugacy class of $g$, we shall use the notations $\fm_g$
and $\fm_\Psi$ interchangeably throughout this paper.

The following are the straightforward consequences of \propref{prop:character1*}.
These were proven in \cite[\S~2]{EG-Adv} when $k = \C$.

\begin{cor}\label{cor:SSC}
  Assume that $k$ is algebraically closed of characteristic zero
  and $G$ is a reductive $k$-group.
  Then a conjugacy class $\Psi$ in $G(k)$ is closed if and only if it is semisimple.
\end{cor}
\begin{proof}
The proof is identical to that of \cite[Prop.~2.4]{EG-Adv} in view of
\propref{prop:character1*} once we observe that \cite[Chap.~1, Cor.~1.2]{GIT} holds
over any algebraically closed field (in fact, over any field).
\end{proof}

\begin{cor}\label{cor:maximal1}
 Assume that $k$ is algebraically closed of characteristic zero
  and $G$ is any $k$-group. Then the correspondence
${\Psi} \mapsto {\fm}_{\Psi}$ gives a bijection between the set of
conjugacy classes in $G_s(k)$ and $\Omega(R_k(G))$.
\end{cor}
\begin{proof}
  The proof is identical to that of \cite[Prop.~2.5]{EG-Adv} in view of 
  \propref{prop:character1*}.
\end{proof}

\begin{cor}\label{cor:maximal2}
Assume the conditions of \corref{cor:maximal1}. 
Let $G \inj H$ be a closed embedding into a $k$-group $H$ and let
${\Psi} = C_H(h)$ be a conjugacy class in $H_s(k)$. Then 
${R_k(G)}_{{\fm}_{\Psi}}$ is a semilocal ring with maximal ideals
$\left \{{\fm}_{{\Psi}_1}, \cdots, {\fm}_{{\Psi}_r}\right \}$ where
${\Psi}_1 \amalg \cdots \amalg {\Psi}_r = {\Psi} \cap G$ is the disjoint
union of some conjugacy classes in $G$. 
\end{cor}
\begin{proof}
  Using Proposition~\ref{prop:finiteR} and Corollary~\ref{cor:maximal1},
  we only need to show that if $\Psi'$ is a
semisimple conjugacy class in $G(k)$ not contained in ${\Psi} \cap G$, then
$\fm_{\Psi}R_k(G) \not\subset \fm_{\Psi'}$. The proof of this (over $\C$) given in
\cite[Prop.~2.6]{EG-Adv} is slightly erroneous. We therefore write out the correct
proof.
We let $\Psi_0$ be the conjugacy class of $\Psi'$ in $H_s(k)$. Then
our assumption implies that $\Psi \neq \Psi_0$. Hence, we can find a virtual
character $f \in R_k(H)$ which vanishes on $\Psi$ but not on $\Psi_0$. The
restriction of such a character on $G$ can not vanish on $\Psi'$.
Equivalently, $f \in \fm_{\Psi}R_k(G) \setminus \fm_{\Psi'}$.
\end{proof}

\section{Equivariant $K$-theory and Hochschild homology}
\label{sec:Recall*}
In this section, we shall review the definitions and basic properties
of equivariant $K$-theory. We shall also recall equivariant Hochschild homology and
other equivariant homology theories derived from the Hochschild homology.

Throughout this paper, the word $stack$ over a field $k$ (equivalently, a $k$-stack)
will mean the quotient stack $[X/G]$ associated to the action of a $k$-group $G$ on a
$k$-space $X$ (cf. \cite[\S~2.4.2]{Laumon}). In particular, all stacks in this
paper will be Noetherian Artin quotient stacks with affine (hence separated)
diagonal. As a consequence, all maps from schemes to stacks will be representable
by schemes. We shall let $c_{[X/G]} \colon [X/G] \to X/G$ denote the coarse moduli space
map whenever it exists. If $[X/G]$ is an algebraic space (equivalently, $c_{[X/G]}$ is an
isomorphism), we shall identify $[X/G]$ with $X/G$ via $c_{[X/G]}$  and use the latter
notation for both.

Recall (cf. \cite[Tag 06WS]{SP})
that if $\pi \colon X \to [X/G]$ is the quotient map, then
the pull-back $\pi^*$  induces a natural equivalence between the
categories of quasi-coherent
sheaves of $\sO_{[X/G]}$-modules (on the Lisse-{\'e}tale site
of $[X/G]$) and $G$-equivariant quasi-coherent sheaves of $\sO_X$-modules on $X$
(on the Zariski site). This equivalence is maintained when we restrict to the
subcategories of coherent sheaves. In this paper, we shall make no distinction between
quasi-coherent sheaves on $[X/G]$ and $G$-equivariant  quasi-coherent sheaves on $X$
via these equivalences.

\subsection{Review of equivariant $K$-theory}\label{sec:Recall}
We fix a field $k$ and a $k$-group $G$ acting on a $k$-space $X$. We let
$\sX = [X/G]$ with the quotient map $\pi \colon X \to \sX$. 
We shall consider the following categories.
\begin{enumerate}
\item
${\Qcoh}_{[X/G]} \ = $ \ abelian category of quasi-coherent sheaves on $[X/G]$. 
\item
  ${\Coh}^0_{[X/G]} \ = $ \ abelian category of coherent sheaves on $[X/G]$.
\item
  ${\Qcoh}^G_{X} \ = $ \ abelian category of $G$-equivariant
  quasi-coherent sheaves on $X$. 
\item
 $\Coh_{[X/G]}  \ = $ \ complicial
bi-Waldhausen category of strictly bounded complexes of coherent $\sO_{[X/G]}$-modules. 
\item
 $\Perf_{[X/G]}  \ = $ \ complicial
bi-Waldhausen category of perfect complexes of $\sO_{[X/G]}$-modules. 
\end{enumerate}

The bi-Waldhausen structure of the last two  categories is given with
respect to the degree wise split monomorphisms as cofibrations and quasi-isomorphisms
as weak equivalences. 

We let $K'([X/G])$ denote the (Thomason-Trobaugh)
$K$-theory spectrum of the bi-Waldhausen category of the cohomologically bounded
complexes of sheaves of $\sO_{[X/G]}$-modules with coherent cohomology 
(cf. \cite[Defn.~3.3]{TT}, \cite[\S~3.1]{Khan-JJM}). We let $K^+([X/G])$ denote  the
$K$-theory spectrum of the bi-Waldhausen category $\Perf_{[X/G]}$.
We let $K([X/G])$ denote the non-connective $K$-theory spectrum
of perfect complexes on $[X/G]$
obtained by applying Schlichting's delooping construction \cite{Schlichting} to
$K^+([X/G])$. There is a natural map of ring spectra $K^+([X/G]) \to K([X/G])$ such
that the induced map $K^+_i([X/G]) \to K_i([X/G])$ is an isomorphism for $i \ge 0$ and
$K^+_i([X/G]) = 0$ for $i < 0$.

For $n \in \N$, let 
\begin{equation}\label{eqn:simplex}
\Delta^n = \Spec\left(
\frac{k[t_0, \cdots , t_n]}{(\sum_i t_i - 1)}\right).
\end{equation}
Recall that $\Delta^{\bullet}$ is a cosimplicial $k$-scheme.
The {\sl homotopy $K$-theory} ($KH$-theory in short) of $[X/G]$ is defined as
\begin{equation}\label{eqn:KH-defn}
KH([X/G]) = \hocolim_{\Delta^n} K([X/G] \times \Delta^n).
\end{equation}
There are canonical maps $K([X/G]) \to KH([X/G]) \to K'([X/G])$, where the first map is
induced by $\Delta^0\in \Delta^{\bullet}$ and the second by the homotopy invariance of
$K'([X/G])$ (cf. \cite[Cor.~3.19]{Khan-JJM}.
These maps are weak equivalences if $[X/G]$ is regular. We refer to
\cite[\S~4C]{Hoyois-Krishna} for basic properties of $KH$-theory of stacks.
Recall (e.g., see the proof of \cite[Lem.~3.1]{Krishna-Ravi})
that for the $G$-action on $X$ as above, the equivariant $K'$-theory spectrum
$K'_G(X)$ is canonically equivalent to the spectrum $K'([X/G])$ via $\pi^*$.
Similarly, we have $\pi^* \colon K^+([X/G]) \xrightarrow{\simeq}  K^+_G(X)$,
$\pi^* \colon K([X/G]) \xrightarrow{\simeq}  K_G(X)$ and $\pi^* \colon KH([X/G])
\xrightarrow{\simeq}  KH_G(X)$.

For a commutative $\Q$-algebra $R$, we let $K([X/G],R) = K([X/G]) \wedge_{\1} H_R$,
where $H_R$ is the Eilenberg-MacLane spectrum of $R$. It
is easy to check that $\pi_i(K([X/G],R)) \cong \pi_i(K([X/G])) \otimes R$.
For an integer $n \ge 1$, we let $K([X/G],{\Z}/n) =  K([X/G]) \wedge_{\1} {\1}/n$,
where ${\1}/n$ is the mod-$n$ Moore spectrum.
We define $K'([X/G],R)$ and $K'([X/G],{\Z}/n)$ (resp. $KH([X/G],R)$ and
$KH([X/G],{\Z}/n)$ in a similar way.

\subsection{Review of equivariant Hochschild homology}\label{sec:Recall-0}
Recall that if $k$ is a field, then a $k$-group $T$
is called diagonalizable if it is of the form $D_k(M) = \Spec(k[M])$, where $M$ is a
finitely generated abelian group and $k[M]$ is its group algebra.
$T$ is said to be of multiplicative type  if
$T_{k_s}$ is diagonalizable, where $k_s$ is a separable closure of $k$.
Recall from \S~\ref{sec:Intro} that a $k$-group $G$ is $nice$
if it is an extension of
a finite linearly reductive group scheme over $k$ by a group scheme over $k$ of
multiplicative type (cf. \cite[\S~2]{Hoyois-Krishna}). The category of nice group
schemes includes diagonalizable groups and finite constant groups whose orders are
invertible in $k$.  We shall need the following result later on in this paper.

\begin{lem}\label{lem:Nice-sub}
  Let $H \inj G$ be a closed embedding of $k$-groups such that $G$ is nice.
  Then so is $H$.
\end{lem}
\begin{proof}
  We consider the short exact sequence of $k$-groups
  \begin{equation}\label{eqn:Nice-sub-0}
    1 \to G' \to G \to G'' \to 1,
  \end{equation}
  where $G'$ is of multiplicative type and $G''$ is a finite linearly reductive
  group scheme over $k$. We let $H'' = {\rm Image}(H \to G'')$ and 
  $H' = \Ker(H \to G'')$. Then $H''$ is a closed subgroup of $G''$ by
  \cite[Chap.~1, Cor.~1.4]{Borel} and we have a short exact sequence of $k$-groups
  \[
  1 \to H' \to H \to H'' \to 1.
  \]
 It suffices therefore to show that
  $H'$ is of multiplicative type and $H''$ is finite linearly reductive.
 That claim that $H'$ is of multiplicative type is elementary because
 it is  a closed subgroup of $G'$. So is the claim that $H''$ is finite.

 To show that $H''$ is linearly reductive,
  we note that the map $BH'' \to BG''$ is a torsor for the homogeneous space
  ${G''}/{H''}$ (cf. \cite[Chap.~2, Thm.~6.8]{Borel}). In particular,
  $BH'' \to BG''$ is {\'e}tale locally finite, and hence finite by faithfully flat
  descent. It follows that the functor $\Qcoh(BH'') \to \Qcoh(BG'')$ is exact.
  Since $\Qcoh(BG'') \to \Qcoh(\pt)$ is exact (by definition of linear
  reductivity), it follows that the functor $\Qcoh(BH'') \to \Qcoh(\pt)$ is exact.
\end{proof}

We shall say that $G \in \Grp_k$ acts on $X \in \Spc_k$ with nice stabilizers
if all stabilizer groups (the fibers of the group scheme $I_X \to X$
in ~\eqref{eqn:inertia}) are nice. 
For a (not necessarily quotient) stack $\sX$, we let $D_{\rm qc}(\sX)$ denote the
unbounded derived category of the complexes of $\sO_{\sX}$-modules with quasi-coherent
cohomology. Let $D(\sX)$ denote the unbounded derived category of quasi-coherent
sheaves on $\sX$. Recall from \cite[\S~1.1]{BZFN} that a stack $\sX$ is called perfect
if it has affine diagonal, $D(\sX)$ is compactly generated, and an object of $D(\sX)$
is compact if and only if it is perfect. We shall use the following result about
perfectness of stacks.

\begin{prop}\label{prop:Perf-stack}
  Let $k$ be a field and let $G$ be a $k$-group. Let 
  $X \in \Spc^G_k$ and let $\sX = [X/G]$ be the stack quotient. Then $\sX$ is perfect
  if either of the following two conditions is satisfied.
\begin{enumerate}
\item
  $\Char(k) = 0$ and $X$ is normal.
\item
  $\sX$ has nice stabilizers.
\end{enumerate}
\end{prop}
\begin{proof}
  If (2) holds, then \cite[Prop.~3.5, Cor.~3.17]{Hoyois-Krishna} together say that
  $D_{\rm qc}(\sX)$ is compactly generated (in particular, $\sO_{\sX}$ is compact) and
  an object of
  $D_{\rm qc}(\sX)$ is compact if and only if it is perfect. It follows now from
  \cite[Thm.~1.2]{HNR} that the canonical functor $D(\sX) \to D_{\rm qc}(\sX)$ is an
  equivalence. It follows that $\sX$ is perfect.
If (1) holds, then we apply \cite[Thm.~3.14]{Hoyois-Krishna},
  \cite[Prop.~2.7]{Krishna-Ravi} and \cite[Thm.~B]{Hall-Rydh-Comp} (see its corollary)
  to derive the same conclusion.
  \end{proof}

\subsubsection{Homology theory of mixed complexes}\label{sec:Mixed*}
Let $k$ be a commutative ring. Recall (cf. \cite{Kassel}) that a mixed complex of
$k$-modules is a datum $\sA = (A_\bullet, d, B)$, where $A_\bullet$ is
a chain complex of $k$-modules,
$d \colon A_n \to A_{n-1}$ ($n \ge 0$) is the differential of $A_\bullet$ and
$B \colon A \to A[-1]$ is a map of complexes (called the Connes operator)
such that $dB + Bd = B^2 = 0$. The structure of a mixed complex on
$(A_\bullet,d)$ is the same thing as the structure of a left module on it over
the dg algebra $\Lambda = k[B]/{(B^2)} = (kB \xrightarrow{0} k)$,
where $B$ is of chain degree $1$ and $d(B) = 0$.

We let ${\sM}ix_k$ denote the category of mixed complexes of $k$-modules and
${\sD\sM}ix_k$ denote the localization of ${\sM}ix_k$ with respect to
quasi-isomorphisms (equivalently, the derived category of the dg algebra $\Lambda$),
where a morphism of mixed complexes is said to be a quasi-isomorphism if it is a
quasi-isomorphism of ordinary chain complexes after forgetting the Connes operator.
In this paper, we shall use homological notations for chain complexes.
For a chain complex $M_\bullet$ in an abelian category, $M_\bullet[n]$ will denote the
chain complex whose $i$-th term is $M_{i-n}$.

Let ${\rm Ch}_k$ denote the dg category of unbounded chain complexes of $k$-modules
and $\sD_k$ its derived category.
We recall (cf. \cite{Kassel}) that for ${\sA} = (A_\bullet, d, B) \in {\sD\sM}ix_k$, the
Hochschild, cyclic, negative cyclic and periodic cyclic homology are
defined to be the following complexes in $\sD_k$.
\begin{enumerate}
  \item
    The Hochschild homology is the complex $HH({\sA}) = (A_\bullet, d)$, the cyclic
    homology is the complex $HC({\sA}) = k \otimes^{\mathbb{L}}_\Lambda {\sA}$ and
    the negative cyclic homology is the complex
    $HC^{-}({\sA}) = {\bf{R}}\Hom_\Lambda(k, {\sA})$.
  \item
    The inclusion $kB \inj \Lambda$ is $\Lambda$-linear which induces a `norm' map
    $N \colon HC({\sA})[1] \to HC^{-}({\sA})$. The homotopy cofiber of this map is the
    periodic cyclic homology $HP(\sA)$. We let
    $T \colon  HC^{-}({\sA}) \to HP(\sA)$ denote the induced map.
\end{enumerate}
These homology theories of mixed complexes
and the canonical maps between them define functors
from ${\sD\sM}ix_k$ to $\sD_k$ and natural transformations between these functors.

Assume now that $k$ is a field and $\sC$ is a $k$-linear symmetric
monoidal dg category. Recall from \cite[\S~1.3]{Keller-JPAA}
(see also \cite[Defn.~2.2.18, Rem.~2.2.19]{Chen}) that the cyclic bar complex $H(\sC)$ 
is the sum-total complex of the bicomplex whose $n$th column (for $n \ge 0$) is the
direct sum of complexes
\begin{equation}\label{eqn:DG-complex-0}
  \Hom_{\sC}(X_n , X_0 ) \otimes \Hom_{\sC}(X_{n-1} , X_{n} ) \otimes
  \Hom_{\sC}(X_{n-2} , X_{n-1}) \otimes \cdots \otimes \Hom_{\sC}(X_0 , X_1),
    \end{equation}
where $X_0, \ldots , X_n$ range through the objects of $\sC$ and the tensor product
is taken in the category of chain complexes of $k$-modules. The differential
of this complex is induced by maps of complexes
\[
d(f_n \otimes \cdots \otimes f_0) = f_{n-1} \otimes \cdots \cdots \otimes f_0f_n +
\stackrel{n}{\underset{i =1}\sum} (-1)^i f_n \otimes \cdots \otimes f_i f_{i-1} \otimes
\cdots \otimes f_0.
\]

The cyclic permutations 
\begin{equation}\label{eqn:DG-complex-1}
  t_n(f_{n-1} \otimes \cdots \otimes f_0) = (-1)^n f_0 \otimes f_{n-1} \otimes \cdots
  \otimes f_1
  \end{equation}
endow $H(\sC)$ with the structure of a cyclic complex in the category of $k$-vector
spaces (cf. \cite[\S~6.1]{Loday}) which in turn makes it a mixed complex
$(H(\sC), d, B)$ (cf. \cite[\S~5.3]{Keller-ICM}). The Hochschild
(resp. cyclic, negative cyclic, periodic cyclic) homology of $\sC$ is defined to be the
corresponding homology of the mixed complex $(H(\sC), d, B)$. The monoidal
structure on $\sC$ uniquely defines a monoidal structure on $HH(\sC)$. Since the
other homology theories are obtained from $HH(\sC)$ using monoidal maps,
we see that all these homology theories of $\sC$ are canonically endowed with
monoidal structures.

Suppose that $\sC$ is compactly generated 
monoidal dg category with the subcategory of
compact objects $\sC_c \subset \sC$ such that $\sC$ is the ind-completion of $\sC_c$.
Under this hypothesis, $\sC$ is a dualizable object in the category ${\rm DGCat}_k$,
i.e., there are canonical dg functors
\begin{equation}\label{eqn:DG-complex-2}
  {\rm Ch}_k \xrightarrow{{\rm coev}} \sC \otimes \sC^{\vee} \xrightarrow{{\rm ev}}
  {\rm Ch}_k
\end{equation}
which satisfy the standard properties of 1-dimensional cobordisms.
Furthermore, there is a canonical weak equivalence (cf. \cite[\S~2.2]{Chen})
\begin{equation}\label{eqn:DG-complex-3}
  \alpha_{\sC} \colon HH({\sC}) \xrightarrow{\simeq} ({\rm ev} \circ {\rm coev})(k).
\end{equation}

We let ${\rm {\bf{dgm}}}_k$ denote the category of all 
$k$-linear symmetric monoidal dg-categories (together with dg-functors) modulo Morita
equivalence (cf. \cite[\S~3.8]{Keller-ICM}). As shown in  \S~5.2 of op. cit. and
\cite[\S~7,9]{BGT-1}, the Waldhausen-Bass construction defines a functor
$K \colon {\rm {\bf{dgm}}}_k \to \sS\sH$ such that $K(\Perf_{\sX})$
(cf. \S~\ref{sec:Ho-st}) coincides with
$K(\sX)$. The cyclic bar complex defines a functor
${\rm DGCat}_k \to {\sD\sM}ix_k$. By \cite[Thm.~5.2]{Keller-ICM}, this functor
factors through ${\rm {\bf{dgm}}}_k$. Composing this with the Eilenberg-MacLane
construction associated to any of the homology functors
defined above on ${\rm DGCat}_k$,
we get the following result due to Blumberg-Gepner-Tabuada \cite{BGT}.
We let $\sS(k)$ denote the set $\{K, KH, HH, HC^{-}, HC, HP\}$ of functors on
${\rm {\bf{dgm}}}_k$.

\begin{thm}\label{thm:H-functor}
For any $F \in \sS(k)$, we have the following.
\begin{enumerate}
\item
  The functor
  \[
  F \colon {\rm {\bf{dgm}}}_k \to \sS\sH
  \]
  has the property that $F(\sC)$ is a ring spectrum for $\sC \in {\rm {\bf{dgm}}}_k$.
\item
  There are natural transformations ${\rm tr} \colon K \to HH$ and ${\rm tr_{cyc}}
  \colon K \to HC^{-}$ such that ${\rm tr} = \iota \circ {\rm tr_{cyc}}$, where
  $\iota \colon HC^{-} \to HH$ is the canonical natural transformation.
\item
  For every $\sC \in {\rm {\bf{dgm}}}_k$, all arrows in the commutative diagram
  \[
  \xymatrix@C1pc{
    K(\sC) \ar[r]^-{\rm tr_{cyc}} \ar[dr]_-{{\rm tr}} & HC^{-}(\sC) \ar[d]^-{\iota} \\
    & HH(\sC)}
  \]
 are morphisms of ring spectra.
\item
  For every $\sC \in {\rm {\bf{dgm}}}_k$, the canonical maps
 $HH(\sC) \to HC(\sC)$ and $ HC^{-}(\sC) \xrightarrow{T} HP(\sC)$ 
 are morphisms of ring spectra.
\end{enumerate}
\end{thm}

\begin{remk}\label{remk:H-functor-0}
  The authors of \cite{BGT} study the multiplicativity of the trace maps from
  $K$-theory to the topological counterparts of our homology theories (i.e.,
  topological Hochschild homology et al.). However,
  this does not cause a problem because the homology theories we defined above receive
  canonical multipticative maps from their topological counterparts.
  \end{remk}

\subsubsection{Homology theory of stacks}\label{sec:Ho-st}
We now specialize to dg categories arising out of quasi-coherent sheaves on
quotient stacks. For this, we assume that $k$ is a field of characteristic zero. 
Let $G$ be a $k$-group and let $X \in \Spc^G_k$. We
let $p \colon X \to \sX = [X/G]$ denote the stack quotient map.
Recall (cf. \cite[\S~3A]{Hoyois-Krishna}, \cite[\S~3.2]{Keller-Doc}) that
$\Perf_{\sX}$ has the canonical structure of a dg-category
whose homotopy category is the derived category $\sD_\Perf(\sX)$ of $\Perf_{\sX}$.
Let $\Perf^{\rm ac}_{\sX}$ denote the full subcategory of $\Perf_{\sX}$
consisting of acyclic perfect complexes on ${\sX}$. Then
$(\Perf_{\sX}, \Perf^{\rm ac}_{\sX})$ is a localizing pair of $k$-linear dg categories
(cf. \cite[\S~2.2, 4.4, 4.6]{Keller-ICM}). We let $\Perf_{dg}(\sX)$ denote the
dg quotient of the inclusion $\Perf^{\rm ac}_{\sX} \inj \Perf_{\sX}$.

We assume that $\sX$ is a perfect stack so that $D(\sX)$ is compactly generated and
its compact objects are same as the perfect complexes.
The Hochschild (resp. cyclic, negative cyclic, periodic cyclic) homology of $\sX$
is defined to be the Hochschild (resp. cyclic, negative cyclic, periodic cyclic)
homology of the $k$-linear dg category $\Perf_{dg}(\sX)$ and is denoted by $HH(\sX)$
(resp. $HC(\sX)$, $HC^{-}(\sX)$, $HP(\sX)$). We shall also write
$HH(\sX)$ (resp. $HC(\sX)$, $HC^{-}(\sX)$, $HP(\sX)$) as $HH_G(X)$ 
(resp. $HC_G(X)$, $HC^{-}_G(X)$, $HP_G(X)$). We shall call these 
`the homology theories' of $\sX$ (equivalently, equivariant homology theories of $X$).

Since the pull-back $f^* \colon H(\Perf_{dg}(\sX)) \to H(\Perf_{dg}(\sX'))$ of
a homology theory,
induced by a morphism of quotient stacks $f \colon \sX' \to \sX$, clearly commutes
with the cyclic permutations, we see that $f^*$ is $\Lambda$-linear
(i.e., equivariant for the $B$-action). In particular, all homology theories are
contravariant functorial on the category of perfect stacks.
By a homology theory of stacks, we shall mean any of the functors
$HH, HC, HC^{-}$ and $HP$ on the category of stacks. 
We shall assume these functors to be taking values in $\sS\sH$ via the
Eilenberg-MacLane functor.

We recall that a derived prestack over $k$ is an $\infty$-functor from
${\bf{Aff}}_k$ to the $\infty$-category of simplicial sets
(equivalently, weakly compact Hausdorff topological spaces) ${\bf{S}}$, where
${\bf{Aff}}_k$ is the opposite category of the $\infty$-category of derived commutative
$k$-algebras (equivalently, simplicial commutative $k$-algebras or commutative
differential non-negatively graded $k$-algebras).
A derived stack over $k$ is a derived prestack which
is a sheaf for the derived {\'e}tale topology on ${\bf{Aff}}_k$
(cf. \cite{Toen-DAG}). We let $\dst_k$ denote the $\infty$-category of derived
stacks over $k$.

An ordinary stack over $k$ is canonically an object of $\dst_k$ and we
shall treat it in this way. We shall consider a simplicial set as a locally constant
sheaf on ${\bf{Aff}}_k$ (in particular, an object of $\dst_k$). One example of this
which we shall use is the standard circle $S^1$, the homotopy push-out of the diagram
$\pt \leftarrow \pt \coprod \pt \to \pt$. All functors between derived stacks and
between $\infty$-categories of sheaves on them will be derived. A commutative
diagram in an  $\infty$-category will mean a diagram which commutes up to a
coherent homotopy.

For the quotient stack $\sX = [X/G]$ as above, the loop space $\sL\sX$ is defined to
be the derived mapping stack $\sL\sX = {{\sM}ap}_{\dst_k}(S^1, \sX)$. Using that
$S^1 = \pt {\underset{(\pt \coprod \pt)}{\coprod^h}} \pt$ and
${{\sM}ap}_{\dst_k}(-, \sX)$ takes
homotopy colimits into homotopy limits in $\dst_k$, we get that
\begin{equation}\label{eqn:Loop-space}
  \xymatrix@C1pc{
    \sL\sX \ar[r]^-{\pi} \ar[d]_-{\pi'} & \sX \ar[d]^-{\Delta} \\
    \sX \ar[r]^-{\Delta} & \sX \times \sX}
\end{equation}
is a homotopy Cartesian diagram in $\dst_k$, where $\Delta$ is the diagonal map.
The constant loop on $\sX$ defines a unique morphism
$s \colon \sX \to \sL\sX$ such that $\pi \circ s = \id_{\sX}$. This is the identity
element of $\sL\sX$ when the latter is considered as a derived group stack over $\sX$
via $\pi$.

It follows from the definition of $\sL\sX$ that
$\pi_0(\sL\sX) = \sI_{\sX} = [{I_X}/{G}]$ is the ordinary inertia stack
of $\sX$. Recall here that $I_X$ is the ordinary fiber product
\begin{equation}\label{eqn:inertia}
  \xymatrix@C1pc{
    I_X \ar[r]^-{s_0} \ar[d]_-{t} & X \ar[d]^-{\Delta_X} \\
    G \times X \ar[r]^-{\phi_X} & X \times X,}
\end{equation}
where $\phi_X (g,x) = (gx,x)$ is the action map and $\Delta_X$ is the diagonal.
All maps in this diagram are $G$-equivariant if we let
$G$ act on $G \times X$ by $g'(g,x) = (g'gg'^{-1}, g'x)$ and on $X \times X$
diagonally. The map $t$ gives a $G$-equivariant closed embedding of $I_X$ into
$G \times X$. On geometric points, this embedding is given by
\begin{equation}\label{eqn:inerta-0}
  I_X = \{(g,x)| gx = x\} \ \ \mbox{and} \ \  g' \cdot (g,x) = (g'gg'^{-1}, g'x).
  \end{equation}

We now recall that $S^1$ is a group stack under the operation of loop rotations
and it acts on ${\sM}ap_{\dst_k}(S^1, \sX)$ by left multiplication on the source of
all maps.
It follows that $\sL\sX$ is equipped with a canonical $S^1$-action which
is natural with respect to maps of stacks.
This implies that $\sO(\sL\sX)$ is canonically a comodule over the coalgebra
$\sO(S^1)$. Since the topological group $S^1$ is the geometric realization of the dual
algebra of $\sO(S^1)$, it follows that $\sO(\sL\sX)$ is endowed with a canonical
$S^1$-action (cf. \cite[\S~3.4]{Ben-Nadler}).  We also note that
\begin{equation}\label{eqn:Loop-space-0}
  \sO(\sL\sX) = \Gamma(\sL\sX, \sO_{\sL\sX}) \simeq \Gamma(\sX, \pi'_*(\sO_{\sL\sX}))
  \simeq \Gamma(\sX, \pi'_* \pi^*(\sO_{\sX}))
\end{equation}
\[
\simeq \Gamma(\sX, \Delta^*
  \Delta_*(\sO_{\sX})) \simeq \pr_* \Delta^* \Delta_* \pr^*(k),
\]
where $\pr \colon \sX \to \pt$ is the projection (cf. \cite[Rem.~1.11]{Khan-JJM}).

On the other hand, letting $\sC$ be the dg category $\sC(\sX)$
of unbounded chain complexes
of quasi-coherent sheaves in ~\eqref{eqn:DG-complex-3} so that
$\sC_c = \Perf_{dg}(\sX)$, one knows that $\sC \simeq \sC^\vee$ and
$\sC \otimes \sC^\vee \simeq \sC(\sX \times \sX)$
(cf. \cite[Thm.~4.7, Cor.~4.8]{BZFN}).
Moreover,
one has ${\rm ev} = \pr_* \Delta^*$ and ${\rm coev} = \Delta_* \pr^*$ so that
\eqref{eqn:DG-complex-3} yields a canonical weak equivalence
(cf. \cite[\S~2.2, Exm.~2.2.20]{Chen}, \cite[Prop.~4.1]{HLP})
\begin{equation}\label{eqn:HH-Loop}
  \alpha_{\sX} \colon HH(\sX) \xrightarrow{\simeq} \sO(\sL\sX).
\end{equation}
This map is natural in $\sX$ and
is equivariant with respect to the action of the Connes operator $B$
(equivalently, the $\Lambda$-module structure) on $HH(\sX)$ 
and the $S^1$-action on $\sO(\sL\sX)$ (cf. \cite[Rem.~4.2]{Ben-Nadler-DAG}).

We remark that the $B$-action on $HH(\sX)$ is in fact equivalent to the action of the
simplicial group $B\Z$. Taking the geometric realizations, one finds that the
$B$-action yields a canonical $S^1$-action on $HH(\sX)$ and $\alpha_\sX$ is then
an $S^1$-equivariant map between (the Eilenberg-MacLane) spectra
(cf. \cite{Hoyois-notes}). In the sequel, we shall interpret the
$S^1$-action on $HH(\sX)$ in this term. Under this identification, we get 
\begin{equation}\label{eqn:Loop-space-1}
  HC(\sX) = HH(\sX)_{hS^1}, \ HC^{-}(\sX) =  HH(\sX)^{hS^1} \ \ \mbox{and} \ \
  HP(\sX) = HH(\sX)^{h\T},
\end{equation}
which are respectively, the homotopy orbit, homotopy fixed point and homotopy
Tate construction spectra for the $S^1$-action on $HH(\sX)$.

\section{A decomposition theorem}\label{sec:Twisting*}
The goal of this section is to prove a decomposition theorem for the
equivariant $K$-theory and equivariant homology theories.
Apart from being of independent interest, this theorem is the first technical step in
the proof of the completion theorems. 

We fix a field $k$. We also fix $G \in \Grp_k$ and $X \in \Spc^G_k$.
Let $\pi \colon X \to [X/G]$ be the quotient map. Recall that 
the center of $G$ is the unique closed subgroup $Z(G)$ having the property that
$Z(G)(A)$ is the center of the group $G(A)$ for all $k$-algebras $A$
(cf. \cite[Prop.~I.1.3, I.1.7]{Borel}, \cite[Tag~0BF6, Lem.~39.8.8]{SP}). Let
$P \subset Z(G)$ be a diagonalizable closed subgroup and let $\wh{P}$ be its
character group so that $P = \Spec(k[\wh{P}])$. 
We shall write the group operation of $\wh{P}$ additively, i.e.,
$(\chi + \chi')(x) = \chi(x) \chi'(x) \in \G_m$, where the latter is the
rank one split torus over $k$. We assume that $P$ acts (via its inclusion in $G$)
trivially on $X$.

Let $\sO_X[\wh{P}]$ be the quasi-coherent sheaf of Hopf algebras over $\sO_X$.
Since $P$ acts trivially on $X$, a $P$-equivariant quasi-coherent sheaf $\sF$ on $X$
is the same thing as a quasi-coherent sheaf $\sF$ on $X$ together with a 
homomorphism $\alpha^P_{\sF} \colon P \to {\rm Aut}_X(\sF)$ between group stacks
over $k$ (cf. \cite[Tag~08JS, Lem.~99.3.2]{SP}).
Equivalently, a $P$-equivariant quasi-coherent sheaf $\sF$ on $X$
is a quasi-coherent sheaf $\sF$ on $X$ endowed with a unique structure of a comodule
over the sheaf of Hopf algebras $\sO_X[\wh{P}]$ which
is compatible with the $\sO_X$-module structure of $\sF$.
This structure of a sheaf of comodules over the sheaf $\sO_X[\wh{P}]$
of Hopf algebras is equivalent
to a map of $\sO_X$-modules $\rho \colon \sF \to \sO_X[\wh{P}] \otimes_{\sO_X} \sF$
satisfying the usual conditions (cf. \cite[\S~2A1]{Krishna-Ravi-1}).
For an element $\chi \in \wh{P}$, we let $\rho_\chi \colon
\sF \to \sO_X[\wh{P}] \otimes_{\sO_X} \sF$ be defined by $\rho_\chi(f) =
\chi \otimes f$, where $U \to X$ is {\'e}tale and $f \in \sF(U)$.
We let $\sF_\chi = \Ker(\rho - \rho_\chi)$.

We next note that if we let $\G_m$ act trivially on $X$, then the scalar multiplication
induced by the $\sO_X$-module structure on $\sF$ defines a unique $\G_m$-action
on $\sF$ given by the homomorphism of group stacks
$\can_{\sF} \colon \G_m \to {\rm Aut}_X(\sF)$ over $k$.
For each character $\chi \colon P \to \G_m$, we let $\sC^{\chi}_X$ be the 
full subcategory of ${\Qcoh}^G_X$ whose objects are given by
\begin{equation}\label{eqn:Decom-0}
  {\rm Obj}(\mathcal{C}^{\chi}_X) =\{\mathcal{F} \in {\Qcoh}^G_X| 
  \alpha^P_{\sF} = \can_{\sF} \circ \chi\} = \{\mathcal{F} \in {\Qcoh}^G_X| 
  \sF = \sF_\chi\}.
\end{equation}
We let $\sD^{\chi}_X$ be the full subcategory of $\sC^{\chi}_X$ consisting of objects
which are $G$-equivariant coherent sheaves on $X$. 
Recall that the coproduct $\coprod_{\lambda \in I} \sC_\lambda$ of a family
$\{\sC_\lambda\}_{\lambda \in I}$ of
categories is the filtering inductive limit of products of finite subfamilies of $I$
(cf. \cite[\S~1, p.~84]{Quillen}).

\begin{lem}\label{lem:definitionoftwist}
We have the following.
\begin{enumerate}
\item
For each $\sF \in {\Qcoh}^G_X$ and $\chi \in \wh{P}$, 
the subsheaf $\sF_{\chi}$ is quasi-coherent and $G$-equivariant. 
\item
The natural map ${\underset{\chi \in \wh{P}}\bigoplus} \sF_{\chi} \to \sF$
is an isomorphism in ${\Qcoh}^G_X$.
\item
  For each $\sF \in {\Qcoh}^G_X$ and $\chi \in \wh{P}$, the subsheaf $\sF_\chi$
  is coherent (resp. locally free) if $\sF$ is so.
\item
$\mathcal{C}^{\chi}_X$ is a full abelian subcategory of ${\Qcoh}^G_X$.
\item
The natural inclusions $\sC^{\chi}_X \inj {\Qcoh}^G_X$
induce equivalences of categories 
\[
{\underset{\chi \in \wh{P}}\prod} 
\sC^{\chi}_X \xrightarrow{\simeq} {\Qcoh}_{[X/G]}; \ \
   {\underset{\chi \in \wh{P}}\coprod} 
\sD^{\chi}_X \xrightarrow{\simeq} {\Coh}^0_{[X/G]}.
\]
\item
  If $f \colon X' \to X$ is a morphism in $\Spc^G_k$ with $P$ acting trivially on $X'$
  and $\chi \in \wh{P}$, then
$f^*(\sC^{\chi}_X) \subset \sC^\chi_{X'}$ and $f_*(\sC^\chi_{X'}) \subset \sC^\chi_X$.
\item
  For each pair $\chi, \chi' \in \wh{P}$, the tensor product in ${\Qcoh}^G_X$ defines a
  functor
  \[
  \sC^{\chi}_X \times \sC^{\chi'}_X \to \sC^{\chi + \chi'}_X.
  \]
\end{enumerate}
\end{lem}
\begin{proof}
Except for (3), the proof of items (1) to (5) is a trivial extension of that of
  \cite[Lem.~4.1]{Krishna-Sreedhar} to quasi-coherent sheaves, where one has to use
  \cite[Exp.~I, Prop.~4.7.3]{SGA3}
  instead of \cite[Lem.~5.6]{Thomason-Inv} in the proofs of (2) and (5).
  The item (3) follows from (2). To prove the first part of item (6), we can
  reduce to the case when both $X$ and $X'$ are affine in which case the claim
  is obvious. The second part of (6) is clear from the definitions.
The item (7) is a direct verification.
\end{proof}

By \lemref{lem:definitionoftwist}(5), we get projection functors
$\pr_\chi \colon {\Qcoh}^G_X \to \sC^\chi_X$ such that the composition
$\sC^\chi_X \inj {\Qcoh}^G_X \xrightarrow{\pr_\chi} \sC^\chi_X$ is isomorphic to the
identity functor. These functors commute with $f^*$ and $f_*$.

We let $\Perf^c_{[X/G]}$ be the full subcategory of $\Perf_{[X/G]}$ whose objects
are strictly bounded perfect complexes $\sF_\bullet$ such that $\sF_n$ is a
$G$-equivariant coherent sheaf on $X$ for each $n$. For $\chi \in \wh{P}$,
we let $\Perf^{c,\chi}_{[X/G]}$ be the full subcategory of $\Perf^c_{[X/G]}$ whose objects
are the chain complexes $\sF_\bullet$ such that $\sF_n \in \mathcal{C}^{\chi}_X$ for
every $n$. We let $\Coh^{\chi}_{[X/G]}$ be the full subcategory of $\Coh_{[X/G]}$
whose objects are the strictly bounded chain complexes $\sF_\bullet$ such that
$\sF_n \in \mathcal{C}^{\chi}_X$ for every $n$.
It is clear that $\Perf^c_{[X/G]}$ and $\Perf^{c,\chi}_{[X/G]}$ are $k$-linear
differential graded full subcategories of $\Perf_{[X/G]}$ with respect to the latter's
cofibrations and weak equivalences.
Similarly, $\Coh^{\chi}_{[X/G]}$ is a differential graded
full subcategory of $\Coh_{[X/G]}$. We let $K^{\chi}([X/G])$ (resp. $K'^{\chi}([X/G])$)
denote the non-connective $K$-theory spectrum of $\Perf^{c,\chi}_{[X/G]}$
(resp. $\Coh^{\chi}_{[X/G]}$) as defined in \cite{BGT-1} and \cite{Schlichting}.
We similarly define $F^{\chi}([X/G])$ for $F \in \{HH, HC^{-}, HC, HP\}$ following
\cite{Keller-ICM}.

\lemref{lem:definitionoftwist} implies that for 
$\sF \in \sC^{\chi}_X$ and $\sG \in \sC^{\chi'}_X$, one has
$\Hom_{[X/G]}(\sF, \sG) = 0$ unless $\chi = \chi'$. It follows that every
map $\alpha \colon \sF \to \sG$ in ${\Qcoh}^G_X$ decomposes
into direct sum of maps $\alpha_\chi \colon \sF_\chi \to \sG_\chi$ where $\chi$ runs
through $\wh{P}$. We deduce from \lemref{lem:definitionoftwist}(5) that
${\underset{\chi \in \wh{P}}\coprod} \Perf^{c,\chi}_{[X/G]} \to \Perf^c_{[X/G]}$
and ${\underset{\chi \in \wh{P}}\coprod} \Coh^{\chi}_{[X/G]} \to \Coh_{[X/G]} $ are
equivalences of dg categories. We denote the induced maps between their
non-connective $K$-theory spectra by $\Phi^{P}_{[X/G]}$ and $\Psi^{P}_{[X/G]}$,
respectively. Recall that a stack $\sY$ is said to have the
resolution property
if every coherent sheaf on $\sY$ is a quotient of a locally free coherent sheaf.
Let $f \colon X' \to X$ be a morphism in $\Spc^G_k$ such that
$P$ acts trivially on $X'$.

We let $\sS'(k) =  \sS(k) \cup \{K'\}$, considered as functors on the
category of quotient stacks over $k$. For $F \in \sS'(k)$ and $\chi \in \wh{P}$,
we let $F^{\chi}([X/G]) = F(\Perf^{\chi}_{[X/G]})$.
For $i \in \Z$ and $F \in \sS'(k)$, we let $F_i(-) = (\pi_i \circ F)(-)$ on
${\rm DGCat}_k$.

\begin{prop}\label{prop:Twist-2}
  We have the following.
  \begin{enumerate}
    \item
      The canonical maps of spectra
      \[
      \hspace*{1cm}
      \Phi^{P}_{[X/G]} \colon \coprod_{\chi \in \wh{P}} K^{\chi}([X/G]) \to K([X/G]) \ \
      \mbox{and} \ \
      \Psi^{P}_{[X/G]} \colon \coprod_{\chi \in \wh{P}} K'^{\chi}([X/G]) \to K'([X/G])
      \]
      are weak equivalences.
    \item
      If $f$ is proper, then the push-forward map on equivariant $K'$-theory
      induces a map $f_* \colon  K'^{\chi}([X'/G])  \to K'^{\chi}([X/G])$ for each $\chi
      \in \wh{P}$ such that $f_* \circ \Psi^{P}_{[X'/G]}$ is homotopic to
      $\Psi^P_{[X/G]} \circ f_*$.
    \item
     If $f$ is flat, then the pull-back map on equivariant $K'$-theory
      induces a map $f^* \colon  K'^{\chi}([X/G])  \to K'^{\chi}([X'/G])$ for each $\chi
      \in \wh{P}$ such that $f^* \circ \Psi^{P}_{[X/G]}$ is homotopic to
      $\Psi^{P}_{[X'/G]} \circ f^*$. 
      \item
        The pull-back
        map on equivariant $K$-theory induces $f^* \colon  K^{\chi}([X/G])  \to
        K^{\chi}([X'/G])$ for each $\chi \in \wh{P}$ such that $f^* \circ
        \Phi^{P}_{[X/G]}$ is homotopic to $\Phi^{P}_{[X'/G]} \circ f^*$.
      \item
        If $f$ is a proper lci (local complete intersection) morphism, then the
        push-forward map on equivariant
        $K$-theory induces $f_* \colon K^{\chi}([X'/G])  \to K^{\chi}([X/G])$ for each
        $\chi \in \wh{P}$ such that $f_* \circ \Phi^{P}_{[X'/G]}$ is homotopic to
      $\Phi^P_{[X/G]} \circ f_*$.
     \item
     For any $\chi, \chi' \in \wh{P}$, the tensor
     product of equivariant quasi-coherent sheaves and the pull-back
     $f^* \colon K^{\chi}([X/G]) \to K^{\chi}([{X'}/G])$ together induce maps of spectra
      \[
      \hspace*{1cm}
      K^{\chi}([X/G]) \wedge_{\1} K^{\chi'}([{X'}/G]) \to K^{\chi + \chi'}([{X'}/G]) \ \
      \mbox{and}
      \]
      \[
      K^{\chi}([X/G]) \wedge_{\1} K'^{\chi'}([{X'}/G]) \to K'^{\chi + \chi'}([{X'}/G]).
      \]
      \end{enumerate}
\end{prop}
\begin{proof}
 Using the decomposition ${\underset{\chi \in \wh{P}}\coprod}
  \Perf^{c,\chi}_{[X/G]} \xrightarrow{\simeq} \Perf^c_{[X/G]}$ and
  \cite[\S~1, (4), \S~2, (8), (9)]{Quillen},
  it follows that the canonical map of spectra
  $\coprod_{\chi \in \wh{P}} K^{\chi}([X/G]) \to K(\Perf^c_{[X/G]})$ is a weak equivalence.
  To prove that $\Phi^P_{[X/G]}$ is a  weak equivalence, it remains therefore to show
  that the inclusion functor $\Perf^c_{[X/G]} \to \Perf_{[X/G]}$ induces a weak
  equivalence between their $K$-theory spectra.
  To show the latter claim, we let $\Perf^{qc}_{[X/G]}$ be the full
  subcategory of $\Perf_{[X/G]}$ whose objects
  are perfect complexes $\sF_\bullet$ such that $\sF_n$ is a
  $G$-equivariant quasi-coherent sheaf on $X$ for each $n$.
We let $\Perf^{c'}_{[X/G]}$ be the full
  subcategory of $\Perf_{[X/G]}$ whose objects
  are perfect complexes $\sF_\bullet$ such that $\sF_n$ is a
  $G$-equivariant coherent sheaf on $X$ for each $n$.

  We look at the inclusions of dg categories
  \begin{equation}\label{eqn:Thomason-lemma*}
   \Perf^c_{[X/G]} \inj \Perf^{c'}_{[X/G]} \inj \Perf^{qc}_{[X/G]} \inj \Perf_{[X/G]}.
    \end{equation}
It is straightforward to check that the first inclusion induces an equivalence
  of derived categories. In particular, it induces a weak equivalence of 
  $K$-theory spectra (cf. \cite[\S~12.3, Prop.~3]{Schlichting}).
  It follows from \cite[Lem.~3.1]{Krishna-Ravi}
  that the map of spectra $K(\Perf^{qc}_{[X/G]}) \to K(\Perf_{[X/G]}) = K([X/G])$ is a weak
  equivalence.

To show that $K(\Perf^{c'}_{[X/G]}) \to K(\Perf^{qc}_{[X/G]})$ is
  a weak equivalence, we let $\phi \colon \sF \surj \sG$ be a surjective map of
  $G$-equivariant quasi-coherent sheaves on $X$ such that $\sG$ is coherent. By
  \cite[Lem.~1.4]{Thomason-Adv}, $\sF$ is the filtered colimit of its $G$-equivariant
  coherent subsheaves. Since $X$ is Noetherian, one of these coherent subsheaves
  (say, $\sF'$) will surject onto $\sG$. We thus get a $G$-equivariant coherent
  subsheaf $\sF' \subset \sF$ such that $\phi$ restricts to a surjection
  $\phi \colon \sF' \surj \sG$. We now apply \cite[Lem.~1.9.5]{TT}
  with $\sA = {\Qcoh}_{[X/G]}$, $\sD = \Coh^0_{[X/G]}$,
  $\sC =$ the category of cohomologically bounded complexes
  in $\sA$ and $D = 0$. It follows that the inclusion
  $\Perf^{c'}_{[X/G]} \inj \Perf^{qc}_{[X/G]}$
  induces an equivalence of their derived categories. We conclude as above
  (cf. \cite[Thm.~1.9.8]{TT}) that this inclusion induces a weak equivalence
  \begin{equation}\label{eqn:Twist-2-0}
    K(\Perf^{c'}_{[X/G]}) \xrightarrow{\simeq} K(\Perf^{qc}_{[X/G]}).
    \end{equation}
    This concludes the proof of the claim that $\Phi^P_{[X/G]}$ is a  weak equivalence.

 An identical argument shows that the natural map of spectra
 \begin{equation}\label{eqn:Twist-2-1}   
   K(\Coh_{[X/G]}) \to K'([X/G])
 \end{equation}
 is a weak equivalence and this easily implies as above
  that $\Psi^P_{[X/G]}$ is a  weak equivalence. This proves item (1).
  The item (3) follows directly from \lemref{lem:definitionoftwist}(6) using
  ~\eqref{eqn:Twist-2-1}.

To prove (4), we let $\Perf^f_{[X/G]} \subset \Perf^{qc}_{[X/G]}$ be the full subcategory 
  consisting of bounded above perfect complexes $\sF_\bullet$ such that each $\sF_n$
  is a quasi-coherent flat $\sO_{[X/G]}$-module. We let $\Perf^{f,\chi}_{[X/G]} \subset
  \Perf^f_{[X/G]}$ be the full subcategory whose objects $\sF_\bullet$ have the property
  that each $\sF_n$ lies in $\sC^\chi_X$. We let $\Perf^{{qc}, \chi}_{[X/G]} \subset
\Perf^{qc}_{[X/G]}$ be the full subcategory whose objects $\sF_\bullet$ have the property
that each $\sF_n$ lies in $\sC^\chi_X$. We then have a commutative diagram
of inclusions of dg categories
\begin{equation}\label{eqn:Twist-2-2} 
  \xymatrix@C.8pc{
    \Perf^{c,\chi}_{[X/G]} \ar[r] \ar[d] & \Perf^{{qc},\chi}_{[X/G]} \ar[d] &
    \Perf^{f,\chi}_{[X/G]} \ar[l] \ar[d] \\
    \Perf^c_{[X/G]} \ar[r] & \Perf^{qc}_{[X/G]} & \Perf^f_{[X/G]}. \ar[l]}
\end{equation}

We showed in the proof of item (1) that the left horizontal arrow on the bottom row
induces a weak equivalence between the $K$-theory spectra. Using this and
\lemref{lem:definitionoftwist}, it follows that the left horizontal arrow on the
top row also induces a weak equivalence between the $K$-theory spectra.
It follows from \cite[Thm.~3.5.5]{Gross-thesis} that 
the right horizontal arrow on the bottom row induces an equivalence of the derived
categories and hence a weak equivalence of the $K$-theory spectra.
An application of \lemref{lem:definitionoftwist} shows again that 
right horizontal arrow on the top row also induces a weak equivalence of the
$K$-theory spectra. It suffices therefore to show that $f^*$ takes
$\Perf^f_{[X/G]}$ to $\Perf^{qc}_{[X'/G]}$ such that $f^*(\Perf^{f,\chi}_{[X/G]}) \subset
\Perf^{\chi}_{[{X'}/G]}$. But this is immediate from  \lemref{lem:definitionoftwist}(6).
This proves item (4).

To prove (5), we let $\Perf^i_{[X/G]} \subset \Perf^{qc}_{[X/G]}$ be the full subcategory 
  consisting of bounded below perfect complexes $\sF_\bullet$ such that each $\sF_n$
  is a quasi-coherent injective $\sO_{[X/G]}$-module. We let
  $\Perf^{i,\chi}_{[X/G]} \subset
  \Perf^i_{[X/G]}$ be the full subcategory whose objects $\sF_\bullet$ have the property
  that each $\sF_n$ lies in $\sC^\chi_X$.  We then have a commutative diagram
of inclusions of dg categories
\begin{equation}\label{eqn:Twist-2-3} 
  \xymatrix@C.8pc{
    \Perf^{c,\chi}_{[X'/G]} \ar[r] \ar[d] & \Perf^{{qc},\chi}_{[X'/G]} \ar[d] &
    \Perf^{i,\chi}_{[X'/G]} \ar[l] \ar[d] \\
    \Perf^c_{[X'/G]} \ar[r] & \Perf^{qc}_{[X'/G]} & \Perf^i_{[X'/G]}. \ar[l]}
\end{equation}

Since ${\Qcoh}^G_{X'}$ has enough injectives, we see that the right horizontal arrow
on the bottom row induces an equivalence of the derived categories and hence a
weak equivalence of the $K$-theory spectra.
It follows as before that the right horizontal arrow
on the top row also induces a weak equivalence of the $K$-theory spectra.
We observed above that
the left horizontal arrows on both rows induce weak equivalences of $K$-theory
spectra. On the other hand, for any $\sF_\bullet$ in $\Perf^i_{[X'/G]}$,
the complex $f_*(\sF_\bullet)$ lies in
$\Perf^i_{[X/G]}$ since $f$ is a proper and lci morphism. Furthermore,
\lemref{lem:definitionoftwist}(6) implies that $f_*(\sF_\bullet)$ lies in
$\Perf^{i, \chi}_{[X/G]}$ if $\sF_\bullet$ lies in $\Perf^{i,\chi}_{[X'/G]}$.
This implies item (5). The proof of item (2) is similar to that of (5).

Finally, to prove (6), we can work with $\Perf^{f,\chi}_{[X/G]}$ instead of
$ \Perf^{\chi}_{[X/G]}$ as shown in the proof of item (4).
The desired claim then follows immediately by combining items
(6) and (7) of \lemref{lem:definitionoftwist} (cf. \cite[\S~3.15]{TT}).
\end{proof}

\begin{cor}\label{cor:Twist-2-4}
For $F \in \{HH, HC\}$, there is a natural weak equivalence of spectra
  \[
  \kappa^P_{[X/G]} \colon \coprod_{\chi \in \wh{P}} F^{\chi}([X/G]) \xrightarrow{\simeq}
  F([X/G])
  \]
  such that item (4) of \propref{prop:Twist-2} holds after substituting
  $F$ for $K$ and the first part of item (6) holds after substituting $F$ for $K$.
  \end{cor}
\begin{proof}
  The proof is identical to that of \propref{prop:Twist-2}. Using that the
  functors $HH$ and $HC$ are continuous (i.e., preserve filtered homotopy colimits),
  the only thing one needs to
 observe is that the canonical functors $\Perf^c_{[X/G]} \inj  \Perf^{qc}_{[X/G]}  \inj
  \Perf_{[X/G]}$ induce weak equivalences between the homology functors. This is
  because we showed in the proof of \propref{prop:Twist-2} that
  they have equivalent derived categories. We can therefore apply
  \thmref{thm:H-functor} which says that each of the homology functors
  factors through Morita equivalence, and hence through derived equivalence
  (cf. \cite[\S~4.2]{Keller-ICM}).
  \end{proof}

By \propref{prop:Twist-2}(4), the assignment $X \mapsto K^{\chi}([X/G])$ is
a functor from $(\Spc^G_k)^{\rm op}$ to $\sS\sH$ for every $\chi \in \wh{P}$. 
Letting $G$ act trivially on each $\Delta^n$ (cf. ~\eqref{eqn:simplex})
and noting that $[X/G] \times \Delta^n \cong [(X \times \Delta^n)/G]$,
it follows that $n \mapsto K^{\chi}([(X \times \Delta^n)/G])$ is a simplicial
diagram in $\sS\sH$. Taking the homotopy colimit and using 
\propref{prop:Twist-2}(1), we get a weak equivalence
$\theta^P_{[X/G]} \colon \hocolim_{\Delta^n} (\coprod_{\chi \in \wh{P}}
K^{\chi}([(X \times \Delta^n)/G])) \xrightarrow{\simeq} 
\hocolim_{\Delta^n} K([X/G] \times \Delta^n)$.
We let $KH^{\chi}([X/G]) = \hocolim_{\Delta^n} K^{\chi}([(X \times \Delta^n)/G])$.
Since homotopy colimits commute with coproducts, we get the following.

\begin{cor}\label{cor:Twist-KH}
  There is a natural weak equivalence of spectra
  \[
  \theta^P_{[X/G]} \colon \coprod_{\chi \in \wh{P}} KH^{\chi}([X/G]) \xrightarrow{\simeq}
  KH([X/G])
  \]
  such that items (4), (5) and (6) of \propref{prop:Twist-2} hold.
\end{cor}

\vskip .3cm

As a special case of \propref{prop:Twist-2}, assume that $G$ is a diagonalizable
$k$-group and it acts trivially on $X$. Then we have the canonical map of
spectra $K(X) \to K([X/G])$, induced by the functor $\Perf(X) \to \Perf([X/G])$
which endows a sheaf of $\sO_X$-modules with the trivial action of $G$.
In particular, we get a natural map
$K_i(X) \otimes R(G) \to K_i([X/G])$. A similar map exists between the $K'$-groups.
The second isomorphism in the following corollary was proven by
Thomason \cite[Lem.~5.6]{Thomason-Inv} when $X$ is affine.

\begin{cor}\label{cor:Twist-3}
 The canonical maps
  \[
K_i(X) \otimes \Z[\wh{P}] \to K_i([X/G]); \ K'_i(X) \otimes \Z[\wh{P}] \to K'_i([X/G]);
\]
\[
KH_i(X) \otimes \Z[\wh{P}] \to KH_i([X/G])
\]
are isomorphisms for all $i \in \Z$.
\end{cor}
\begin{proof}
  We fix $i \in \Z$. By \propref{prop:Twist-2}(1), the map
  $(\Phi^{P}_{[X/G]})_* \colon \bigoplus_{\chi \in \wh{P}} \pi_i(K^{\chi}([X/G])) \to
  K_i([X/G])$ is an isomorphism. To prove the first isomorphism, it
  suffices therefore to show that there is a
  canonical isomorphism $K_i(X) \otimes \Z[\wh{P}] \cong
  \bigoplus_{\chi \in \wh{P}} \pi_i(K^{\chi}([X/G]))$. To that end, we consider the functor
$\phi_\chi \colon {\Qcoh}_X \to \sC^\chi_X$ given by $\phi_\chi(\sF) =
\sF'$, where $\sF'$ is the sheaf $\sF$ on which $G$ acts via
the character $\chi$. Thus for any {\'e}tale
open $U$ over $X$ and $x \in \sF(U)$, we have
$g \cdot x = \chi(g)x \in \sF(U)$. It is clear that $\phi_\chi$ is an equivalence
of categories. In particular, $\phi_\chi$ defines an equivalence of the corresponding
categories of chain complexes which preserves perfect complexes.
Considering the induced map on the $K$-theory spectra, we get that
the map $\phi^*_\chi \colon K(X) \to K^{\chi}({[X/G]})$ is a weak equivalence.
Under this equivalence, the induced map
$\phi^*_\chi \colon K_i(X) \xrightarrow{\cong} K_i(X)$ on the homotopy groups is
given by $\phi^*_\chi(x) = \chi  \otimes x$. The arguments for the second and third
isomorphisms are identical.
\end{proof}

An identical argument implies the following.

\begin{cor}\label{cor:Twist-4}
For $F \in  \{HH, HC\}$, the canonical map $F_i(X) \otimes \Z[\wh{P}] \to F_i([X/G])$
is an isomorphism for every $i \in \Z$.
\end{cor}

\section{The twisting operator}\label{sec:Twist-op}
In this section, we shall define the twisting operator on equivariant $K$-theory
and homology theories and describe its effect when we pass to the derived localization
and completion (in the sense of Lurie) of the equivariant $K$-theory spectra.
The construction of this operator is the second technical step in the proof of the
completion theorems.
We begin with a review of Lurie's derived localization and completion. 

\subsection{Lurie's localization and completion}\label{sec:DLC}
We briefly recall Lurie's construction of derived localization and
completion of modules over a commutative ring spectrum
(also called an $E_\infty$ ring)
and list some basic properties that are relevant to this paper.
We refer the reader to \cite[\S~7.2.3]{Lurie-HA} (for localization) and
\cite[Chap.~7.3]{Lurie-SAG} (for completion) for full details.
Let $R$ be a connective $E_\infty$ ring and let $S \subset \pi_0(R)$ be
a multiplicatively closed subset which contains the unit.
Let $\Mod_R$ denote the (derived) category of $R$-module spectra.
Note that $\Mod_R$ closed under homotopy limits and colimits.

As $\pi_0(R)$ acts on $R$, it acts any object $M \in \Mod_R$.
One says that $M$ is $S$-local if the map
$M \xrightarrow{x} M$ (defining the action of $x$ on $M$)
is a weak equivalence for every $x \in S$.
We let ${\Mod}^{S {\rm -loc}}_R$ be the full subcategory of $S$-local objects in
$\Mod_R$. The $S$-localization functor $S^{-1}(-) \colon \Mod_R \to
{\Mod}^{S {\rm -loc}}_R$
is the left adjoint of the inclusion functor ${\Mod}^{S {\rm -loc}}_R \inj \Mod_R$. For
every $M \in \Mod_R$, the canonical map $\delta_M \colon M \to S^{-1}M$ between
$R$-modules induces an isomorphism $S^{-1}\pi_i(M) \xrightarrow{\cong} \pi_i(S^{-1}M)$
between $\pi_0(R)$-modules for every $i$ (cf. \cite[Prop.~7.2.3.20]{Lurie-HA}).
In particular, $S^{-1}M$ is connective if
$M$ is so. If $\fp \subset \pi_0(R)$ is a prime ideal and $S = \pi_0(R) \setminus
\fp$, we shall write $S^{-1}M$ as $M_{\fp}$.

For $x \in \pi_0(R)$ and $M \in \Mod_R$, we let $M[x^{-1}]$ be the homotopy colimit of
the diagram $(M \xrightarrow{x} M \xrightarrow{x} M \xrightarrow{x} \cdots)$. We then
have the natural map $\delta'_M \colon M \to M[x^{-1}]$. If we let
$M[x^{-1}] \xrightarrow{x} M[x^{-1}]$ be the multiplication map, then it induces the
multiplication map $\pi_*(M[x^{-1}]) \xrightarrow{x} \pi_*(M[x^{-1}])$. Using the
isomorphism $\pi_*(M[x^{-1}]) \cong \pi_*(M)[x^{-1}]$ (since the homotopy groups
commute with filtered colimits), we see that the latter map is an
isomorphism. It follows that the map $M[x^{-1}] \xrightarrow{x} M[x^{-1}]$ is a weak
equivalence. In other words, $M[x^{-1}]$ is $S$-local if we
let $S = \{x^n|n \ge 0\}$.
In particular, $\delta'_M$ uniquely factors through $\delta_M \colon M \to S^{-1}M$.
Let $\phi_M \colon S^{-1}M \to M[x^{-1}]$ be the induced map. By looking at
the induced map between the homotopy groups, we observe the following.

\begin{lem}\label{lem:D-loc-0}
  There is a natural weak equivalence $\phi_M \colon S^{-1}M \xrightarrow{\simeq}
  M[x^{-1}]$.
\end{lem}

If $y \in \pi_0(R)$ is another element and we let $T = \{(xy)^n|n \ge 0\}$, then
one checks by passing to the homotopy groups that $T^{-1}M$ is $S$-local. Hence,
\lemref{lem:D-loc-0} implies that there is a canonical map $M[x^{-1}] \to M[(xy)^{-1}]$.
It follows that $\{M[x^{-1}]| x \in \pi_0(R)\}$ is a filtered family of localizations.
Given a multiplicatively closed subset $S \subset \pi_0(R)$ containing the unit, we
let $M[S^{-1}] = \hocolim_{x \in S} M[x^{-1}]$. The following result generalizes
\lemref{lem:D-loc-0} and provides a description of Lurie's derived localization.

\begin{prop}\label{prop:D-loc-main}
  There is a natural weak equivalence $S^{-1}M \xrightarrow{\simeq} M[S^{-1}]$.
\end{prop}
\begin{proof}
  This follows because the canonical map
  $S^{-1}M \xrightarrow{\simeq} M[S^{-1}]$ induces isomorphism between the
  homotopy groups.
 \end{proof}

We next consider Lurie's derived completion functor.
We let $x \in \pi_0(R)$ be an element and let $M \in \Mod_R$. We say that $M$ is
$x$-nilpotent if the action of $x$ on $\pi_*(M)$ is locally nilpotent:
that is, if for each $i \in \Z$ and $y \in \pi_i(M)$, there exists an integer
$n \ge 0$ such that $x^ny = 0$ in the abelian group $\pi_i(M)$.
If $I \subset \pi_0(R)$ is an ideal, we say that $M$ is $I$-nilpotent if it is
$x$-nilpotent for each $x \in I$. We say that $M$ is $I$-local if for every
$I$-nilpotent $R$-module $N$, the mapping space ${\rm Map}_{\Mod_R}(N,M)$ is
contractible. We say that $M$ is $I$-complete if for every
$I$-local $R$-module $N$, the mapping space ${\rm Map}_{\Mod_R}(N,M)$ is
contractible. We let $\Mod^{I {\rm -comp}}_R$ denote the full subcategory of $\Mod_R$
consisting of $I$-complete $R$-modules.

Suppose there is a finitely generated ideal $J \subset \pi_0(R)$ such that
$\sqrt{J} = \sqrt{I}$. Then the inclusion $\Mod^{I {\rm -comp}}_R \inj \Mod_R$ admits
a left adjoint $(-)^{\compl}_I \colon \Mod_R \to \Mod^{I {\rm -comp}}_R$
(cf. \cite[Rem.~7.3.1.2, Prop.~7.3.1.4]{Lurie-SAG}. We
shall call this the $I$-completion functor.
This functor is described as follows. We let $M \in \Mod_R$. 

For $x \in \pi_0(R)$, there is a canonical weak
equivalence $M^{\compl}_{(x)} \simeq \holim_n {M}/{x^nM}$, where ${M}/{xM}$ is 
the homotopy cofiber of the map $M \xrightarrow{x} M$. Equivalently, $M^{\compl}_{(x)}$
is the homotopy cofiber of the augmentation map $T(M) \to M$, where
$T(M)$ is the homotopy limit of the sequence $(\cdots \xrightarrow{x} M
 \xrightarrow{x} M \xrightarrow{x} M)$ (cf. \cite[Prop.~7.3.2.1]{Lurie-SAG}).
Suppose now that $\sqrt{I} = \sqrt{(x_1, \ldots , x_r)} \subset \pi_0(R)$.
For $i = 0, 1, \ldots , r$, we let $M_0 = M$ and let $M_i$ denote the
$(x_i)$-completion of $M_{i-1}$ for $i \ge 1$.
We then have a canonical weak equivalence (cf. \cite[Prop.~7.3.3.2]{Lurie-SAG})
\begin{equation}\label{eqn:I-comp-0} 
  M^{\compl}_I \simeq M_r \simeq
  \holim_{(i_1, \ldots , i_r) \in \N^r} \frac{M}{(x^{i_1}_1, \ldots , x^{i_r}_r)M},
\end{equation}
where each term inside the homotopy limit is an iteration of homotopy cofibers.
Furthermore, $M$ is $I$-complete if and only if the canonical map
$M \to M^{\compl}_I$ is a weak equivalence.

\begin{prop}\label{prop:Compln_D}
  Assuming $\sqrt{I} = \sqrt{(x_1, \ldots , x_r)}$, the $I$-completion functor has the
  following properties.
\begin{enumerate}
\item
  If $M$ is connective, then so is $M^{\compl}_I$.
\item
  If $\fm \subset \pi_0(R)$ is a maximal ideal which is the radical of a
  finitely generated ideal, then the map
  $M \to M^{\compl}_{\fm}$ factors through the localization $M \to M_{\fm} \to
  M^{\compl}_{\fm}$. Furthermore, the map $M^{\compl}_{\fm} \to ( M_{\fm})^{\compl}_{\fm}$
  is a weak equivalence.
\item
  If $\fp \subset \pi_0(R)$ is a prime ideal, then $\pi_i(M_{\fp}) \cong
  (\pi_i(M))_\fp$ for every $i \in \Z$.
\item
  $M$ is $I$-complete if and only if it is
  $(x_i)$-complete for every $i$.
\item
  If $\sqrt{I}$ is the radical of the extension of a finitely generated ideal
  $I' \subset \pi_0(R')$ under a morphism $R' \to R$ of connective $E_\infty$-rings,
  then for every $R$-module $M$,
  the canonical map $M \to M^{\compl}_{I}$ exhibits $M^{\compl}_{I}$ as the
  $I'$-completion of $M$ (when regarded as an $R'$-module). In particular, $M$ is
  $I'$-complete (when regarded as an $R'$-module) if and only if it is
  $I$-complete.
\item
  $M^{\compl}_I \simeq (M^{\compl}_J)^{\compl}_{(x)}$ if $I = J + (x)$.
\item
  If $\{M_\alpha\}_{\alpha \in J}$ is a diagram in $\Mod_R$ and each $M_\alpha$
  is $I$-complete, then $\holim_{\alpha \in J} M_\alpha$ is also $I$-complete.
  \end{enumerate}
\end{prop}
\begin{proof}
All items except (7) are proven in \cite[\S~7.3.2, 7.3.3, 7.3.4]{Lurie-SAG}. To prove 
(7), we let $M = \holim_{\alpha \in J} M_\alpha$.
  By assumption, the completion map
  $M_\alpha \to \holim_{(i_1, \ldots , i_r) \in \N^r}
  \frac{M_\alpha}{(x^{i_1}_1, \ldots , x^{i_r}_r)M_\alpha}$ is a weak equivalence
  for every $\alpha$ and this map is natural in $\alpha \in J$. By passing to the
  limit, we therefore get 
  \[
  \begin{array}{lll}
  M & \xrightarrow{\simeq} & 
  \holim_{\alpha \in J} \holim_{(i_1, \ldots , i_r) \in \N^r}
  \frac{M_\alpha}{(x^{i_1}_1, \ldots , x^{i_r}_r)M_\alpha} \\
  & \xrightarrow{\simeq} & \holim_{(i_1, \ldots , i_r) \in \N^r} \holim_{\alpha \in J} 
  \frac{M_\alpha}{(x^{i_1}_1, \ldots , x^{i_r}_r)M_\alpha} \\
  & \xrightarrow{\simeq} & \holim_{(i_1, \ldots , i_r) \in \N^r}
  \frac{M}{(x^{i_1}_1, \ldots , x^{i_r}_r)M} \\
  & \xrightarrow{\simeq} & M^{\compl}_I,
  \end{array}
  \]
  where the second weak equivalence follows from the fact that two homotopy limits
  commute with each other (cf. \cite[Thm.~24.9 and \S~31.5]{Chacholski-Scherer})
  and the third weak equivalence follows from the fact that
  a homotopy limit commutes a homotopy cofiber (which coincides with homotopy
  fiber up to a shift). This finishes the proof.
\end{proof}

\subsection{Definition and properties of twisting operator}\label{sec:Twist-K-0}
We fix a field $k$. We let $G \in \Grp_k$ and $X \in \Spc^G_k$.
Let $g \in Z(G)(k)$ be a semisimple element which acts trivially
on $X$. We let $P$ denote the closure in $G$ of the cyclic subgroup
$\<g\> \subset G(k)$ generated by $g$.
Since the stabilizers of closed points are closed in $G$, it
follows that $P$ acts trivially on $X$.
The reader may keep in mind that throughout this section,
the notation $F([X/G],k)$ will simply mean
$F([X/G])$ (as defined in \S~\ref{sec:Ho-st}) for $F \in \{HH, HC^{-}, HC, HP\}$.
But that is not the case with $K$-theory.

\begin{lem}\label{lem:Twist-1}
  $P$ is a closed diagonalizable subgroup of $G$ contained in its center.
\end{lem}
\begin{proof}
  The assertion that $P$ is a closed subgroup of $G$ contained in its center follows
  from \cite[Prop.~I.1.3, I.1.7]{Borel}. To see that $P$ is diagonalizable, consider
  $G$ as a closed subgroup of some $GL_{n}$. As $g$ is a semisimple element of
  $G(k)$, it is so as an element of $GL_{n}(k)$ too. In particular, it is contained
  in a (split) maximal torus $T$ of $GL_{n}$. Intersecting $T$ with $G$, we get
  that $\<g\>$ lies in a diagonalizable closed subgroup of $G$ and therefore so does
  its closure. It follows that $P$ is diagonalizable.
\end{proof}

To define the twisting operator on $G$-equivariant $K$-theory of $X$, we assume
that $\Char(k) = 0$. Recall
that the structure map of $X$ induces a morphism of ring spectra
$K(BG) \to K([X/G])$. This yields a natural map of ring spectra
$\gamma_X \colon K(BG) \wedge_{\1} K([X/G]) \to  K([X/G])$ satisfying the standard
properties of modules over ring spectra, where $BG = [{\pt}/G]$ is the
classifying stack of $G$ (cf. \S~\ref{sec:Fin-gen}). In particular,
we get a natural map of spectra $K(BG) \to \Map(K([X/G]), K([X/G]))$, where
$\Map(A, B)$ is the mapping spectrum (the internal
hom in the category of spectra). Applying the $\pi_0$ functor, we get a natural
homomorphism $\lambda_X \colon R(G) \to [K([X/G]), K([X/G])]: = \End_{\SH}(K([X/G]))
= \Hom_{\SH}(K([X/G]), K([X/G]))$. Giving an element $x \in R(G)$ is equivalent to
giving a map $\1 \xrightarrow{x} K(BG)$ which takes the unit to $x$ at the level of
$\pi_0$. The map $\lambda_X(x) \colon K([X/G]) \to K([X/G])$ is then the composition

\begin{equation}\label{eqn:Multip}
  K([X/G]) \simeq \1 \wedge_{\1} K([X/G]) \xrightarrow{x \wedge \id} K(BG)
  \wedge_{\1} K([X/G]) \xrightarrow{\gamma_X} K([X/G]).
  \end{equation}
By the same token, we get a natural $k$-linear map $\lambda_X \colon R_k(G) \to
\End_{\SH}(K([X/G],k))$.

Since $KH([X/G])$ and $K'([X/G])$ are also modules over the spectrum $K(BG)$, we get
similar $k$-linear maps $\lambda^h_X \colon R_k(G) \to \End_{\SH}(KH([X/G],k))$
and $\lambda'_X \colon R_k(G) \to \End_{\SH}(K'([X/G],k))$ such that
$\lambda_X, \lambda^h_X$ and $\lambda'_X$ are compatible with each other under the 
canonical maps of $K(BG)$-modules $K([X/G],k) \to KH([X/G],k) \to K'([X/G],k)$.
By \thmref{thm:H-functor}, we get similar maps for the homology theories of $[X/G]$.
The ring spectrum structure of $H_k$ defines a homomorphism $\lambda_0 \colon k
\cong \pi_0(H_k) \to \End_{\SH}(H_k)$.

Given $\alpha \in k$ and $\chi \in \wh{P}$, there is a map of spectra
$\wt{\alpha} \colon  K^{\chi}([X/G])\wedge_\1 H_k \to K^{\chi}([X/G])\wedge_\1 H_k$, given
by $\wt{\alpha} = \id_{K^{\chi}([X/G])} \wedge \lambda_0(\alpha)$. As
$K^{\chi}([X/G], k) = K^{\chi}([X/G])\wedge_\1 H_k$ (note that $\Char(k) = 0$),
this yields an element $\wt{\alpha} \in \End_{\SH}(K^\chi([X/G],k))$.
Since $g, g^{-1} \in G(k)$, we have that $\chi(g^{-1}) \in k^\times = \G_m(k)$.
We define the twisting operator $t_g$ on $K^{\chi}([X/G],k)$ by
\begin{equation}\label{eqn:Twist-3-0}
t_g := \wt{\chi(g^{-1})} \colon  K^{\chi}([X/G],k) \to K^{\chi}([X/G],k).
\end{equation}

The twisting operator $t_g$ on
$HH^{\chi}([X/G])$ and $HC^{\chi}([X/G])$ is defined in an analogous way using
\corref{cor:Twist-2-4}.
Using the weak equivalences $\Phi^P_{[X/G]}, \ \Psi^P_{[X/G]}, \kappa^P_{[X/G]}$ and
$\theta^P_{[X/G]}$ given by \propref{prop:Twist-2}, \corref{cor:Twist-2-4} and
\corref{cor:Twist-KH}, we define for $F \in \{K, KH, K', HH, HC\}$, 
the operator $t_g$ on $F([X/G],k)$
to be the coproduct of its restrictions on $F^{\chi}([X/G],k)$
as $\chi$ runs through $\wh{P}$. Note that $F([X/G]) =
F([X/G],k)$ for $F \in \{HH,HC^{-}, HC, HP\}$.

To define $t_g$ on $HC^{-}$ and $HP$, we argue slightly differently because we can
not use \corref{cor:Twist-2-4} directly. 
To define $t_g$ on $HC^{-}([X/G])$, we first note that for
every $\chi \in \wh{P}$, the twisting operator $t_g$ on $HH^{\chi}([X/G])$
is $\Lambda$-linear (cf. \S~\ref{sec:Mixed*}) since it is simply multiplication by
$\chi(g^{-1}) \in \Lambda^\times$. Equivalently, the $t_g$-action on each
 $HH^{\chi}([X/G])$ is $S^1$-equivariant.
We next note that for
every $\chi \in \wh{P}$, the inclusion $HH^{\chi}([X/G]) \to HH([X/G])$ is
clearly $S^1$-equivariant because it is induced by the inclusion of dg categories.
This shows that the decomposition of $HH([X/G])$ given in \corref{cor:Twist-2-4} 
is $S^1$-equivariant. It follows from these two observations that the $t_g$-action
on $HH([X/G])$ is $S^1$-equivariant.
We now apply the $(-)^{hS^1}$ functor to get the operator
$t_g \colon HC^{-}([X/G]) \to HC^{-}([X/G])$. 

To define $t_g$ on $HP([X/G])$, we note that the map
$N \colon  HC([X/G])[1] \to HC^{-}([X/G])$ commutes with $t_g$.
To see this,  we use the commutative diagram
\begin{equation}\label{eqn:Twist-Extra}
  \xymatrix@C1pc{
    \coprod_{\chi \in \wh{P}} HC^{\chi}([X/G])[1] \ar[d]_-{N} \ar[r]^-{\simeq} &
     HC([X/G])[1] \ar[d]^-{N} \\
    \coprod_{\chi \in \wh{P}} (HC^{-})^{\chi}([X/G]) \ar[r] &
    HC^{-}([X/G]),}
\end{equation}
where the left vertical arrow is obtained by taking the coproduct of the norm map on
each $HC^{\chi}([X/G])[1]$.

Since the horizontal arrows commute with the $t_g$-actions by latter's construction
and the one on the top row is a weak equivalence, it suffices to
check that $N \colon HC^{\chi}([X/G])[1] \to (HC^{-})^{\chi}([X/G])$ commutes with
$t_g$. But this is clear since $N$ is $\Lambda$-linear (cf. \S~\ref{sec:Mixed*})
and $t_g$ is multiplication by $\chi(g^{-1}) \in \Lambda^\times$.
It follows that $t_g$ induces unique $k$-linear maps
$t_g \colon HP^{\chi}([X/G]) \to HP^{\chi}([X/G])$ and
$t_g \colon HP([X/G]) \to HP([X/G])$ which are commute with the canonical
maps $HP^{\chi}([X/G]) \to HP([X/G])$.

We let $t^*_g \colon F_*([X/G],k) \to F_*([X/G],k)$ denote the induced maps on the
homotopy groups.

\vskip .2cm

The following results summarize the basic properties of $t_g$.

\begin{prop}\label{prop:Twist-BP}
  Let $F \in \sS'(k)$.
  Let $f \colon X' \to X$ be a morphism in $\Spc^G_k$ such that $P$ acts trivially on
  $X'$. Then the twisting operator $t_g$ satisfies the following.
  \begin{enumerate}
  \item
    $t_g$ commutes with the pull-back maps $f^* \colon F([X/G],k) \to F([{X'}/G],k)$
    except when $F = K'$. If $f$ is flat, the same is also true for $F = K'$.
  \item
    If $f$ is proper, then  $t_g$ commutes with the push-forward map
    $f_* \colon K'([{X'}/G],k) \to K'([X/G], k)$. If $f$ is moreover an lci map,
    the same is also true for $K$-theory and $KH$-theory. 
 \item
   $t_g$ is an automorphism of $F([X/G],k)$ in $\sS\sH$.
 \item
   $t^*_g \colon R_k(G) \to R_k(G)$ is a ring homomorphism.
 \item
   If $k$ is algebraically closed and $G$ is reductive, then the trace map
   $R_k(G) \to HH_0(BG)$ is an isomorphism of $k$-algebras and $HH_i(BG) = 0$ for
   $i \neq 0$.
  \end{enumerate}
\end{prop}
\begin{proof}
The items (1) and (2) are immediate from the construction of $t_g$,
  \propref{prop:Twist-2}, \corref{cor:Twist-2-4},
\corref{cor:Twist-KH} and the fact that $f^*$ and $f_*$ (whenever defined)
are maps of $K(BG)$-modules.
It is clear from the construction of $t_g$ that 
  $t^*_g \colon \pi_i(F^{\chi}([X/G],k)) \to \pi_i(F^{\chi}([X/G],k))$ is
  multiplication
  by the element $\chi(g^{-1}) \in k^\times$ for every $i \in \Z$.
  In particular, $t^*_g$ is an isomorphism. It follows that $t_g \colon
  F^{\chi}([X/G],k) \to F^{\chi}([X/G],k)$ is a weak equivalence. Taking the sum over
  $\chi \in \wh{P}$, we get that $t_g$ is an automorphism of $F([X/G],k)$.
  This proves (3), and (4) is shown in \cite[(4.5)]{Krishna-Sreedhar}.

  To prove (5), note that $HH(BG) \simeq \sO(\sL_{BG}) \simeq \sO([G/G])$ for the
  adjoint action by ~\eqref{eqn:HH-Loop} (cf. \cite[Exm.~3.1.7]{Chen}).
  Since $\sO([G/G])$ coincides with the
  ring of regular functions on the affinization $\Spec(k[G]^{G})$ of $[G/G]$,
  we get that the canonical map $k[G]^{G} \to HH(BG)$ is a weak equivalence
  (in particular, $HH(BG)$ is a discrete $K(BG)$-module).
  But the canonical map $R_k(G) \to k[G]^{G}$ is an isomorphism under our assumption
by \propref{prop:character1}.
\end{proof}

\begin{prop}\label{prop:Twist-BP-0}
  For $F \in \sS'(k)$ and $x \in R_k(G)$, one has a commutative diagram
  \begin{equation}\label{eqn:Twist-BP-1}
    \xymatrix@C1.5pc{
      F([X/G],k) \ar[r]^-{\lambda_X(x)} \ar[d]_-{t_g} & F([X/G],k) \ar[d]^-{t_g} \\
      F([X/G],k) \ar[r]^-{\lambda_X(t^*_g(x))} & F([X/G],k).}
  \end{equation}
  \end{prop}
\begin{proof}
  We shall prove this for $K$-theory as the proof is same for
  $F \in \{HH, HC, K', KH\}$, and
  for $HC^{-}$ and $HP$, it follows directly from
  that of $HH$ and $HC$ since the maps in ~\eqref{eqn:Twist-BP-1} are
  $S^1$-equivariant.

  Assuming $F = K$, we see by 
  ~\eqref{eqn:Multip} that ~\eqref{eqn:Twist-BP-1} is the
  outer square in the diagram
  \begin{equation}\label{eqn:Twist-BP-2}
    \xymatrix@C1.5pc{
    \1 \wedge_{\1} K([X/G],k) \ar[r]^-{x \wedge \id} \ar[d]_-{\id \wedge t_g} &
    K(BG,k) \wedge_{\1} K([X/G],k) \ar[r]^-{\gamma_X} \ar[d]^-{t_g \wedge t_g} &
    K([X/G],k) \ar[d]^-{t_g} \\
    \1 \wedge_{\1} K([X/G],k) \ar[r]^-{t^*_g(x) \wedge \id} &
    K(BG,k) \wedge_{\1} K([X/G],k) \ar[r]^-{\gamma_X} & K([X/G],k).} 
\end{equation}
  It suffices therefore to show that ~\eqref{eqn:Twist-BP-2} is commutative. Since
  its left square commutes merely by the definitions of $t_g$ and $t^*_g$, we only
  need to show that the right square is commutative.

 We now look at the diagram
   \begin{equation}\label{eqn:Twist-BP-3}
    \xymatrix@C1.5pc{
      K^{\chi}(BG,k) \wedge_{\1} K^{\chi'}([X/G],k) \ar[r]^-{\gamma_X}
      \ar[d]_-{t_g \wedge t_g} & K^{\chi + \chi'}([X/G],k) \ar[d]^-{t_g} \\
     K^{\chi}(BG,k) \wedge_{\1} K^{\chi'}([X/G],k) \ar[r]^-{\gamma_X}
     & K^{\chi + \chi'}([X/G],k).}
   \end{equation}
   \propref{prop:Twist-2}(6) implies that the horizontal arrows in this diagram are
   defined. Using \propref{prop:Twist-2}(1), it suffices to check that this is a
   commutative diagram for every $\chi, \chi' \in \wh{P}$. By the construction of
   $t_g$, this reduces to showing that the diagram
   \begin{equation}\label{eqn:Twist-BP-4}
    \xymatrix@C1.5pc{
      H_k \wedge_{\1} H_k \ar[r]^-{\cup} \ar[d]_-{\lambda_0(\chi(g^{-1})) \wedge
        \lambda_0(\chi'(g^{-1}))} & H_k \ar[d]^-{\lambda_0((\chi + \chi')(g^{-1}))} \\
      H_k \wedge_{\1} H_k \ar[r]^-{\cup} & H_k}
   \end{equation}
   is commutative, where each of the horizontal arrows is the product operation of
   $H_k$. But this follows easily because ~\eqref{eqn:Twist-BP-4} is induced by the
   diagram of commutative rings (where the horizontal arrows are the multiplication
   maps)
   \begin{equation}\label{eqn:Twist-BP-5}
    \xymatrix@C1.5pc{
      k \otimes k \ar[r] \ar[d]_-{\chi(g^{-1}) \otimes
        \chi'(g^{-1})} & k \ar[d]^-{(\chi + \chi')(g^{-1})} \\
      k \otimes k \ar[r] & k,}
   \end{equation}
   whose commutativity is clear using the identity
   $(\chi + \chi')(g^{-1}) = \chi(g^{-1}) \cdot \chi'(g^{-1})$.
\end{proof}

Let $I_G = \Ker(R(G) \xrightarrow{{\rm rank}} \Z)$ be the augmentation ideal
(cf. \S~\ref{sec:Fin-gen}). In this paper, we shall write 
$\Ker(R_k(G) \xrightarrow{{\rm rank}} k)$ also as $I_G$ while working with
$K$-theory with $k$-coefficients and homology theories.
Note that $I_G \subset R_k(G)$ is the maximal
ideal $\fm_1$ corresponding to the conjugacy class of the identity element of $G$.

 \begin{prop}\label{prop:Twist-local-compl}
    For $F \in \sS'(k)$, the operator
  $t_{g^{-1}} \colon F([X/G],k) \to F([X/G],k)$ induces weak equivalences
  \begin{equation}\label{eqn:Twist-local-compl-0}
  t_{g^{-1}} \colon  F([X/G],k)_{\fm_g} \xrightarrow{\simeq}  F([X/G],k)_{I_G}; \
  t_{g^{-1}} \colon  F([X/G],k)^{\compl}_{\fm_g} \xrightarrow{\simeq}
  F([X/G],k)^{\compl}_{I_G}.
  \end{equation}
 \end{prop}
 \begin{proof}
    We shall show this for $K$-theory as the proof is the same for all functors.
    We showed in the proof of \propref{prop:Twist-BP} that $t^*_g$ is multiplication
    by $\chi(g^{-1})$ on $K^\chi_0(BG,k)$ for each $\chi \in \wh{P}$.
  One checks using this that $t^*_{g^{-1}} \colon R_k(G)
  \to R_k(G)$ maps $\fm_g$ isomorphically onto $I_G$ and it maps
  $R_k(G) \setminus \fm_g$ isomorphically onto $R_k(G) \setminus I_G$
  (cf. \cite[\S~4.2]{Krishna-Sreedhar} or \cite[\S~6.2]{EG-Adv}).
  We conclude from \propref{prop:Twist-BP-0} that $t_{g^{-1}}$
  induces a weak equivalence
  \begin{equation}\label{eqn:Twist-local-compl-1}
   t_{g^{-1}} \colon \hocolim_{x \in R_k(G) \setminus \fm_g} K([X/G],k)[x^{-1}]
   \xrightarrow{\simeq} \hocolim_{y \in R_k(G) \setminus I_G} K([X/G],k)[y^{-1}].
   \end{equation}
   The first weak equivalence of ~\eqref{eqn:Twist-local-compl-0} now follows
   by \propref{prop:D-loc-main} which says that the left hand side of
   ~\eqref{eqn:Twist-local-compl-1} is $K([X/G],k)_{\fm_g}$ while the right hand side
   is $K([X/G],k)_{I_G}$.

 To prove the weak equivalence of the completions, we first note by
   \propref{prop:finiteR} that $R_k(G)$ is a Noetherian ring. We can thus
   write $\fm_g = (x_1, \ldots , x_r)$. If we let $y_i = t^*_{g^{-1}}(x_i)$, it follows
   that $I_G = (y_1, \ldots , y_r)$. Since $t^*_{g^{-1}}$ is a ring homomorphism on
   $R_k(G)$ by \propref{prop:Twist-BP}, we conclude from \propref{prop:Twist-BP-0}
   that $t_{g^{-1}}$ induces a weak equivalence
   \begin{equation}\label{eqn:Twist-local-compl-2}
     t_{g^{-1}} \colon \holim_{(i_1, \ldots , i_r) \in \N^r}
     \frac{K([X/G],k)}{(x^{i_1}_1, \ldots , x^{i_r}_r)K([X/G],k)}
     \xrightarrow{\simeq} \holim_{(i_1, \ldots , i_r) \in \N^r}
     \frac{K([X/G],k)}{(y^{i_1}_1, \ldots , y^{i_r}_r)K([X/G],k)}.
   \end{equation}
   We now apply ~\eqref{eqn:I-comp-0} to finish the proof.
\end{proof}

 \vskip .2cm

\subsection{Group action with finite stabilizers}\label{sec:Proper-action}
Let $k$ be a field and $G$ a $k$-group acting on 
$X \in \Sch_k$. Let $\phi_X \colon G \times X \to X \times X$ denote the action map
(cf. ~\eqref{eqn:inertia}).
Recall that the $G$-action on $X$ is said to have finite
stabilizers if the stabilizers of geometric points are finite.
Equivalently, the projection map $I_X \to X$ is quasi-finite, where
$I_X \subset G \times X$ is the inertia scheme (cf. ~\eqref{eqn:inertia}).
In other words,  $[X/G]$ is a quasi-Deligne-Mumford (equivalently, Deligne-Mumford
if $\Char(k) = 0$) stack (cf. \cite[Tag 06MC]{SP}). 
Recall as well that the $G$-action is said to have finite
inertia if the projection map $I_X \to X$ is finite and it is called 
proper if $\phi_X$ is a proper map. It is clear that proper action implies that it
has finite inertia, which in turn implies that it has finite stabilizers.

\begin{prop}\label{prop:Fin-support}
  Let $k$ be algebraically closed of characteristic zero. Assume that the
  $G$-action on $X$ has finite stabilizers. Then the following hold.
  \begin{enumerate}
  \item
    There exists a finite set of semisimple conjugacy classes
    $\Sigma^G_X = \{\Psi_1, \ldots , \Psi_n\}$ in $G(k)$ such that for any
    $g \in G(k)$, one has that $X^g \neq \emptyset$ if and only if $g \in \Psi_i$
    for some $1 \le i \le n$.
  \item
   There exists an ideal
    $J \subset R_k(G)$ such that ${R_k(G)}/J$ has finite support containing
    $\Sigma^G_X$ and
    \[
    J(K'_i([X/G],k))  = 0 = J(KH_i([X/G],k)) \ \ \forall \ \ i \in \Z.
    \]
  \end{enumerate}
  \end{prop}
\begin{proof}
  Except for the vanishing of $J(KH_*([X/G],k))$, everything else in the
  proposition is proven in \cite[Prop.~3.6, Lem.~3.9]{Krishna-Sreedhar} when
    $k = \C$ and the $G$-action is proper. But we leave it for the reader to check
    that the proof of op. cit. does not require properness of $G$-action. The only
    condition it uses is that this action has finite stabilizers.
    One also does not require $k = \C$ anymore in view of \propref{prop:finiteR} and
    \corref{cor:maximal1}.

We shall prove the vanishing of $J(KH_*([X/G],k))$ in item (2) by 
induction on the dimension of $X$.
If $X \in \Sm^G_k$, then
    we are already home by the vanishing of $J(K'_*([X/G],k))$
    since the canonical map $KH([X/G],k) \to K'([X/G],k)$ is
    a weak equivalence. If $\dim(X) = 0$, then $X_\red$ must be regular.
    In particular, the claim holds in this case because of the 
    nil-invariance property of equivariant $K'$-theory (well known) and
    equivariant $KH$-theory (follows by \cite[Thm.~6.2]{Hoyois-Krishna}).

To prove the claim when $X$ is arbitrary, we apply
    \cite[Thm.~6.2]{Hoyois-Krishna} once again to assume that $X$ is reduced. We next
    choose a $G$-equivariant resolution of singularities
    (which is well known to exist) $f \colon X' \to X$. We let $Y = X_\sing$ and
    $Y' = f^{-1}(Y)$, each having the reduced closed subscheme structure. It follows
    again from \cite[Lem.~2.5]{Thomason-Orange} that there is a commutative diagram
\begin{equation}\label{eqn:Fin-support-0}
  \xymatrix@C1pc{
    Y' \ar[r]^-{\iota'} \ar[d]_-{f'} & X' \ar[d]^-{f} \\
    Y \ar[r]^-{\iota} & X,}
  \end{equation}
  where all schemes lie in $\Sch^G_k$ and all maps are $G$-equivariant.

  Since $G$ acts on $X$ with finite stabilizers, it must act on $X', Y$ and
  $Y'$ also with finite stabilizers.
 We now apply \cite[Thm.~6.2]{Hoyois-Krishna} one more time to conclude that
  ~\eqref{eqn:Fin-support-0} induces a commutative diagram of spectra
  \begin{equation}\label{eqn:Fin-support-2}
  \xymatrix@C1pc{
    KH([X/G]) \ar[d]_-{f^*} \ar[r]^-{\iota^*} & KH([Y/G]) \ar[d]^-{f'^*} \\
    KH([{X'}/G]) \ar[r]^-{\iota'^*} & KH([{Y'}/G])}
  \end{equation}
  which is homotopy Cartesian.
  The same is true for $KH$-theory with coefficients in $k$
  as well.
  The desired claim is true for $X'$ because it is smooth
  over $k$. The claim is true for $Y$ and $Y'$ by induction on the dimension.
  Using the long exact sequence of homotopy groups resulting from
  ~\eqref{eqn:Fin-support-2}, we deduce the claim for $X$.
  We note here that the ideal $J$ is not independent of $X$.
  \end{proof}

\begin{cor}\label{cor:Fin-support-4}
  Let $k$ be of characteristic zero. Let $G$ act on a
    $k$-scheme $X$ with finite stabilizers. Then we have the following.
    \begin{enumerate}
      \item
    The canonical maps $K'([X/G], L)_{I_G} \to
    K'([X/G],L)^{\compl}_{I_G}$ and $KH([X/G], L)_{I_G} \to KH([X/G],L)^{\compl}_{I_G}$
    are weak equivalences for any field $L$ of characteristic zero.
  \item
    The canonical maps
    $K'([X/G], k)_{\fm_\Psi} \to K'([X/G],k)^{\compl}_{\fm_\Psi}$ and
    $KH([X/G], k)_{\fm_\Psi} \to
    KH([X/G],k)^{\compl}_{\fm_\Psi}$ are weak equivalences for any semisimple
    conjugacy class $\Psi \subset G(k)$.
    \end{enumerate}
  \end{cor}
\begin{proof}
  As $R_L(G)$ is Noetherian by \propref{prop:finiteR}, we can write
   $I_G = (x_1, \ldots , x_r)$. To prove (1), we let
    $M \in \{K'([X/G], L)_{I_G}, KH([X/G], L)_{I_G}\}$. By \propref{prop:Compln_D}(2),
    it suffices to show that $M$ is $(x_i)$-complete for every $i$.
    By \cite[Cor.~7.3.2.2]{Lurie-SAG}, this is equivalent to showing that the
    homotopy limit of the tower $T^\bullet_i := (\cdots \xrightarrow{x_i} M
    \xrightarrow{x_i} M \xrightarrow{x_i} M)$ is weakly contractible.
    By the Milnor exact sequence
    \begin{equation}\label{eqn:Milnor-seq}
      0 \to {\varprojlim}^1 \pi_{n+1}(T^\bullet_i) \to \pi_n(\holim T^\bullet_i) \to 
      \varprojlim \pi_{n}(T^\bullet_i) \to 0,
    \end{equation}
    it suffices to show that the sequence $\pi_n(T^\bullet_i) :=
    (\cdots \xrightarrow{x_i} \pi_n(M)
    \xrightarrow{x_i} \pi_n(M) \xrightarrow{x_i} \pi_n(M))$
    is pro-zero for every $n \in \Z$ and $1 \le i \le r$.
    Since a homotopy group functor commutes with localizations (cf.
    \propref{prop:Compln_D}(3)), we are therefore
    reduced to showing that for every $n \in \Z$, there exists $m \gg 0$
    such that 
    \begin{equation}\label{eqn:Fin-support-3-0}
      I^m_G (K'_n([X/G],L)_{I_G}) = 0 = I^m_G (KH_n([X/G],L)_{I_G}).
      \end{equation}

 For $K'_n([X/G],L)_{I_G}$, ~\eqref{eqn:Fin-support-3-0} is shown in the proof of
    \cite[Lem.~7.3]{Krishna-Adv}. In particular,  ~\eqref{eqn:Fin-support-3-0} holds
    if $X \in \Sm^G_k$. This conclusion for $KH_n([X/G],L)_{I_G}$ in the general case
    is now deduced by simply repeating the proof of \propref{prop:Fin-support}.

To prove (2), we repeat the argument of (1) which reduces us to showing
    that for every $n \in \Z$, there exists $m \gg 0$
    such that 
    \begin{equation}\label{eqn:Fin-support-3-1}
      (\fm_\Psi)^m (K'_n([X/G],k)_{\fm_\Psi}) = 0 = (\fm_\Psi)^m (KH_n([X/G],k)_{\fm_\Psi}).
      \end{equation}
    We shall prove this for $KH_n([X/G],k)$ as the other case is identical.
    We fix an algebraic closure $\ov{k}$ of $k$.
    We let ${X'} = X_{\ov{k}}$ and ${G'} = G_{\ov{k}}$.  If ${k'}/k$ is
    a finite field extension, then the composition of the pull-back and
    push-forward maps $KH_n([X/G]) \to KH_n([{X_{k'}}/{G_{k'}}]) \to
    KH_n([X/G])$ is multiplication by $[k':k]$.
    It follows by passing to the limit (cf. \cite[Thm.~2.22]{Khan-JJM}) that
    the pull-back map $KH_n([X/G], k) \to KH_n([{{X'}}/{{G'}}], k)$ is injective.
    In particular, the pull-back map $\lambda \colon KH_n([X/G], k) \to
    KH_n([{{X'}}/{{G'}}], \ov{k})$ is also injective.

 We let $\fm'_\Psi \subset R_{\ov{k}}(G')$ denote the maximal ideal corresponding
    to $g \in G'_s(\ov{k})$.
    By the definition of $\fm_\Psi$, it is clear that
    $\lambda((\fm_\Psi)^m KH_n([X/G], k))
    \subset (\fm'_\Psi)^m  KH_n([{{X'}}/{{G'}}], \ov{k})$. It suffices therefore to
    show that $(\fm'_\Psi)^m  KH_n([{{X'}}/{{G'}}], \ov{k}) = 0$ for $m \gg 0$.
    We can thus assume that $k$ is algebraically closed. In the latter case,
    the claim follows directly from \propref{prop:Fin-support}(2).
\end{proof}

\section{Borel equivariant $K$-theory}\label{sec:Borel-K}
The goal of this section is to define Borel equivariant $K$-theory of algebraic
spaces with a group action. This will be the target of the derived Atiyah-Segal
completion theorem for equivariant $K$-theory (also called Bredon
equivariant $K$-theory). We begin with some basic results related to group action.

\subsection{Some results on group quotients}\label{sec:Prelim-grp}
We fix a field $k$ and a $k$-group $G$.
We let $\Spc^G_{\fr}/k$ be the full subcategory of $\Spc^G_k$ whose objects are those
algebraic spaces over $k$ on which $G$ acts freely.
We let $\Sch^G_{\fr}/k$ be the full subcategory
of $\Spc^G_{\fr}/k$ consisting of those objects which are schemes.
For $X, Y \in \Spc^G_k$, we consider $X \times Y$ as an object of $\Spc^G_k$
via the diagonal action of $G$. It is clear that if either of $X$ and $Y$ lies in
$\Spc^G_{\fr}/k$, then $X \times Y \in \Spc^G_{\fr}/k$. In this case, the quotient
$(X \times Y)/G$ lies in $\Spc_k$ by \cite[Prop.~22]{EG-Inv}. We shall denote this
quotient space by $X \stackrel{G}{\times} Y$.  We shall use the following special
case of this construction very frequently in this text.

Suppose $G \subset G'$ is a closed embedding of
$k$-groups and $X \in \Spc^G_k$. We let
$X \stackrel{G}{\times} G'$ denote the quotient by the diagonal
action of $G$, where the latter acts on $G'$ by the right multiplication
($g \cdot g' = g'g^{-1}$). If we let $G'$ act trivially on $X$ and
by left multiplication on itself, then one checks that its diagonal action on
$X \times G'$ descends to its action on $X \stackrel{G}{\times} G'$.
That is, $X \stackrel{G}{\times} G' \in \Spc^{G'}_k$. We shall call it the
{\sl Morita space} associated to the $G$-action on $X$ and the
embedding $G \inj G'$.
Letting $X'$ denote this Morita space, one checks that the composition
\begin{equation}\label{eqn:Open-quasi-proj-0}
  \delta_X \colon  X \xrightarrow{\iota} X \times G' \xrightarrow{p} X'
    \end{equation}
  (where $\iota(x) = (x, e)$ and $p$ is the quotient map) is $G$-equivariant if
  we let $G$ act on $X'$ via the inclusion $G \inj G'$.
  Passing to the stack quotients by the $G$-actions, we get a map
  $\wt{\delta}_X \colon [X/G] \to [{X'}/{G}]$.
  Precomposing this with the canonical quotient $[{X'}/G] \to [{X'}/{G'}]$, we get
  ${\rm morita}_X \colon [X/G] \to [{X'}/{G'}]$. We now have the following.

  \begin{lem}\label{lem:Morita*}
    The map ${\rm morita}_X \colon [X/G] \to [{X'}/{G'}]$ is an isomorphism of
    stacks.
  \end{lem}
  \begin{proof}
    See \cite[Lem.~2.4]{Krishna-Sreedhar}.
  \end{proof}

  We shall need the following result
  (see \S~\ref{sec:Notn} for the definition of $G$-quasi-projective schemes).

\begin{lem}\label{lem:Open-quasi-proj}
Let $X \in \Spc^G_k$ be a reduced algebraic space.
 If $G$ is connected and $X$ is a normal $k$-scheme, then it is covered by
  $G$-invariant open subschemes which are $G$-quasi-projective. In general,
  $X$ contains a $G$-invariant dense open subscheme
  $U \subset X$ such that $U$ is a regular and $G$-quasi-projective $k$-scheme.
\end{lem}
\begin{proof}
  Under the assumption of the first part, \cite[Thm.~1]{Brion-Moscow}
  says that $X$ is covered by $G$-invariant open subschemes which are
  quasi-projective $k$-schemes.  It follows from \cite[\S~5.7]{Thomason-Duke-2} that
  these open subschemes must be $G$-quasi-projective.
  
  We now prove the second part.
  By \cite[Prop.~6.6]{Knutson}, $X$ contains a unique dense
  open subspace $X' \subset X$
(called the schematic locus of $X$) having the property that $X'$ is a scheme
and any point $x \in X$ is schematic (i.e., has an affine neighborhood in $X$) if
and only if $x \in X'$ (cf. \cite[Tag 03JG]{SP}).
Since $gX' \subset X$ is again a scheme for every $g \in G$, it follows
that $X'$ must be $G$-invariant. Replacing $X$ by $X'$, we can thus assume that
$X \in \Sch^G_k$. Since $X$ is reduced, its regular locus $X_\reg$ is $G$-invariant
dense open. We can thus assume that $X$ is also regular.

If $G$ is connected, then every connected component of $X$ is $G$-invariant
and the result follows by \cite[Thm.~1]{Brion-Moscow} and
\cite[\S~5.7]{Thomason-Duke-2}. In general,
 we choose an embedding $G \inj GL_{n}$ and let
  $Y = X \stackrel{G}{\times} GL_{n}$ be the associated Morita space with
  $GL_{n}$-action. Since $X$ is regular
  and $GL_{n}$ is smooth over $k$, we see that $X \times GL_{n}$ is regular.
  Since the quotient map $X \times GL_{n} \to Y$ is a $G$-torsor and $G \in \Sm_k$,
  it follows that $Y$ is a regular algebraic space with $GL_n$-action.
  Since $GL_n$ is connected, every connected component of
  $Y$ is a regular $GL_n$-invariant open algebraic subspace.
  In particular, its schematic locus is $GL_n$-invariant and hence contains
  a $GL_n$-invariant dense open subspace which is
  quasi-projective over $k$. We let $W \subset Y$ be the disjoint union of
  these open subschemes. Then $W$ is a $GL_n$-invariant dense open subspace of
  $Y$ which is quasi-projective over $k$.

We now let $p \colon X \times GL_n \to Y$ be the quotient map and let
  $W' = p^{-1}(W)$. As $p \colon W' \to W$ is a $G$-torsor,
\lemref{lem:torsor-affine} implies that $W'$ is quasi-projective over $k$.
As $\iota$ is a closed immersion, it follows that
$U:= (\delta_X)^{-1}(W) = \iota^{-1}(W')$ is a quasi-projective $k$-scheme.
Meanwhile, Lemma~\ref{lem:Morita*} implies that $\delta^{-1}_X(W)$ is $G$-invariant
dense open in $X$. This finishes the proof.
\end{proof}

\begin{lem}\label{lem:torsor-affine}
  Let $H$ be a $k$-group and let $p \colon W' \to W$ be an $H$-torsor in
  $\Sch_k$. Then $p$ is an affine morphism. In particular, it is quasi-projective.
\end{lem}
\begin{proof}
  If $p$ is affine, it is quasi-projective by
  \cite[Prop.~5.1.6]{EGAII}. To show that $p$ is affine, it suffices to show that
  $W'$ is affine assuming $W$ is. By faithfully flat descent, it suffices to show
  that $p$ is {\'e}tale locally affine. But this is clear since $H$ is affine and
  every $H$-torsor is {\'e}tale locally trivial.
\end{proof}

For $X \in \Spc^G_{\fr}/k$, one knows that the coarse moduli space map
$c_{[X/G]} \colon [X/G] \to X/G$ is an isomorphism (cf. \cite[Prop.~22]{EG-Inv}).
That is, $[X/G]$ is an algebraic space.
However, it may not be a scheme even if $X$ is. We shall use the
following result in geometric invariant theory 
to conclude that $[X/G] \cong X/G$ is actually a scheme in certain cases.

\begin{prop}$($\cite[Prop.~7.1]{GIT}$)$\label{prop:EG-quotient}
  Let $\phi \colon  X \to Y$ be a morphism in $\Sch^G_{\fr}/k$ such that $Y/G$ is a
  $k$-scheme and there is a $G$-equivariant $\phi$-ample line bundle $\sL$ on $X$.
  Then $X/G$ is also a $k$-scheme. Furthermore, there exists a unique
  $\phi'$-ample line bundle $\sL'$ on $X/G$ such that $\pi^*_X(\sL') \cong \sL$,
  where $\pi_X \colon X \to X/G$ is the quotient map and
  $\phi' \colon X/G \to Y/G$ is the map induced by $\phi$ on the quotients.
  \end{prop}

\begin{cor}\label{cor:EG-quotient-1}
  Suppose that $G$ is a closed
  subgroup of another $k$-group $G'$ and $X \in \Sch^G_k$ is $G$-quasi-projective
  (e.g., $X$ is normal and quasi-projective) over $k$. Then
  $X \stackrel{G}{\times} G' \in \Sch^{G'}_k$ is a quasi-projective $k$-scheme.
  In particular, it is $G'$-quasi-projective if $X$ is normal.
  \end{cor}
\begin{proof}
  We let $Y = X \stackrel{G}{\times} G'$ and let $\phi \colon X \times G' \to G'$
  denote the projection and let $\phi' \colon Y \to {G'}/G$ denote the induced map
  on quotients. Since $X$ admits a $G$-equivariant ample line bundle, it follows that
$X \times G'$ admits a $G$-equivariant $\phi$-ample line bundle. Since
${G'}/G$ is a quasi-projective $k$-scheme by \cite[Thm.~6.8]{Borel}, we conclude
from \propref{prop:EG-quotient} that $Y \in \Sch^{G'}_k$ and it admits a
$\phi'$-ample line bundle. In particular, it is quasi-projective over $k$.
Moreover, smoothness of $G$ and $G'$ implies that $Y$ is normal if $X$ is
so and hence is $G'$-quasi-projective over $k$ by \cite[\S~5.7]{Thomason-Duke-2}.
\end{proof}

\begin{cor}$($\cite[Prop.~23]{EG-Inv}$)$\label{cor:EG-quotient-2}
Suppose that $U \in \Sch^G_{\fr}/k$ such that $U/G \in \Sch_k$.
Let $X \in \Sch^G_k$ and assume that one of the following conditions hold.
\begin{enumerate}
\item
  $G$ is connected and $X$ is normal.
\item
  $G$ is special.
\item
  $X$ is $G$-quasi-projective.
  \end{enumerate}
Then $X \stackrel{G}{\times} U \in \Sch_k$. It is normal (resp. smooth) if $X$ is.
It is quasi-projective if (3) holds and $U/G$ is quasi-projective. 
\end{cor}
\begin{proof}
  The first claim is already proven in \cite[Prop.~23]{EG-Inv} and the second claim is
  obvious.
If $U/G$ is quasi-projective and condition (3) holds, then
we repeat the proof of \corref{cor:EG-quotient-1} with $G'$ replaced
by $U$ to get the desired claim.
\end{proof}

\subsection{$K$-theory of Borel construction}\label{sec:Borel-S}
Let $k$ be a field and $G \in \Grp_k$. Recall from \cite[\S~3]{Krishna-Crelle} that
a pair $(V,U)$ of smooth $k$-schemes is a {\sl good pair} for $G$ if $V$ is a
representation of $G$ and
$U \subsetneq V$ is a $G$-invariant open subset which lies in $\Sch^G_{\fr}/k$
such that $U/G \in \Sm_k$ and
$U(k) \neq \emptyset$ (this is automatic if $k$ is infinite).
It is known (cf. \cite[Rem.~1.4]{Totaro1}) that a good pair for $G$ always exists.
We also recall the following from \cite[Defn.~3.1]{Krishna-Crelle}.

\begin{defn}\label{defn:Add-Gad}
A sequence of pairs $\rho = {\left(V_i, U_i\right)}_{i \ge 1}$ of 
$k$-schemes is called an {\sl admissible gadget} for $G$, if there exists a
good pair $(V,U)$ for $G$ such that $V_i = V^{\oplus i}, U_1 = U$ and
$U_i \subsetneq V_i$ is $G$-invariant open and
the following are satisfied for each $i \ge 1$.
\begin{enumerate}
\item
$\left(U_i \oplus V\right) \cup \left(V \oplus U_i\right)
\subseteq U_{i+1}$ as $G$-invariant open subsets.
\item
$\codim_{U_{i+2}}\left(U_{i+2} \setminus 
\left(U_{i+1} \oplus V\right)\right) > 
\codim_{U_{i+1}}\left(U_{i+1} \setminus \left(U_{i} \oplus V\right)\right)$.
\item
$\codim_{V_{i+1}}\left(V_{i+1} \setminus U_{i+1}\right)
> \codim_{V_i}\left(V_i \setminus U_i\right)$.
\item
  $U_i \in \Sch^G_{\fr}/k$ and ${U_i}/G \in \Sm_k$. 
\end{enumerate}

We shall say that $\rho = {\left(V_i, U_i\right)}_{i \ge 1}$
is a {\sl good admissible gadget}
for $G$ if it also satisfies the following additional condition.
\\
\hspace*{.5cm} (5) Each ${U_i}/G$ is a smooth quasi-projective $k$-scheme.
\end{defn}

The following result guarantees the existence of admissible gadgets.

\begin{lem}\label{lem:Adm-gadget}
  Assume that $G$ is a special reductive group and let $H \subset G$ be a closed
  subgroup. Then the following hold.
  \begin{enumerate}
    \item
  If $\rho = {\left(V_i, U_i\right)}_{i \ge 1}$ is an admissible gadget for $G$, then
  it is also an admissible gadget for $H$ under the restriction of the
  action from $G$ to $H$.
\item
  If $H$ is reductive and $\rho$ is a good admissible gadget for $G$, then it is a
  good admissible gadget for $H$.
\item
  A good admissible gadget for $G$ exists.
  In particular, every $k$-group admits an admissible gadget.
  \end{enumerate}
\end{lem}
\begin{proof}
  To prove (1), we only need to show that if $\rho = {\left(V_i, U_i\right)}_{i \ge 1}$
  is an admissible gadget for $G$, then each ${U_i}/H$ is a smooth $k$-scheme.
  To that end, we fix $i \ge 1$ and look at the quotient maps
  $U_i \xrightarrow{p} {U_i}/H \xrightarrow{q} {U_i}/G$. We let $\pi = q \circ p$.
  Since $\pi$ is a $G$-torsor of schemes and $G$ is special, we get that there exists
  an affine Zariski open cover $\{W_j\}$ of ${U_i}/G$ such that
  $\pi^{-1}(W_j) \cong W_j \times G$ and $\pi$ is the projection to $W_j$. It follows
  that $q^{-1}(W_j) \cong W_j \times ({G/H})$. As $G/H$ is a quasi-projective
  $k$-scheme (cf. \cite[Thm.~6.8]{Borel}), it follows that ${U_i}/H$ has a Zariski
  open cover by $k$-schemes. We conclude that ${U_i}/H \in \Sch_k$. The smoothness of
  ${U_i}/H$ is directly deduced from that of $U_i$.

To prove (2), we need to show in the proof of (1) that ${U_i}/H$ is quasi-projective
  if $H$ is reductive. To that end, we first observe that $G/H$ is an affine
  $k$-scheme. This follows from \cite[Thm.~A]{Richardson} because we can
  always pass to an algebraic closure of $k$ to check affineness.
  We next note that ${U_i}/H$ is Zariski locally product of ${U_i}/G$ with
  $G/H$, as shown above.  It follows that $q$ is an affine map. In particular, it is
  quasi-projective. We are now done because ${U_i}/G$ is 
  quasi-projective over $k$ by the hypothesis in item (2).

To prove (3), we first define our $\rho$. We choose a faithful 
representation $W$ of $G$ of dimension $n$ and let
$V = W^{\oplus n} \simeq {\rm End}_k(W)$.
If we let $U \subset V$ be the largest $G$-invariant open on which $G$ acts freely,
then the faithfulness implies that $U(k) \neq \emptyset$. We let
  $Z = V \setminus U$. We take $V_i = V^{\oplus i}, U_1 = U$ and
  $U_{i+1} = \left(U_i \oplus V\right) \cup \left(V \oplus U_i\right)$ for $i \ge 1$.
  Setting $Z_1 = Z$ and $Z_{i+1} = U_{i+1} \setminus \left(U_i \oplus V\right)$ for 
$i \ge 1$, one checks that $V_i \setminus U_i = Z^i$ and $Z_{i+1} = Z^i \oplus U$.
In particular, $\codim_{V_i}\left(V_i \setminus U_i\right) = i (\codim_V(Z))$ and 
$\codim_{U_{i+1}}\left(Z_{i+1}\right) = (i+1)d - i(\dim(Z))- d =  i (\codim_V(Z))$,
where $d = \dim(V)$. To show that $\rho =  {\left(V_i, U_i\right)}_{i \ge 1}$ is a
good admissible gadget for $G$, it remains to prove that each ${U_i}/G$ is
a quasi-projective $k$-scheme.

To that end, we let $U'_i \subset V_i$ be the largest $G$-invariant open subscheme
  of $V_i$ where the $G$-action is free. Note that $U'_i \neq \emptyset$ because
  $V$ has a non-empty $G$-invariant open with free $G$-action.
  We let $V^s_i \subset V_i$ be the largest
  $G$-invariant open subscheme such that the $G$-orbit of every closed
  point $x \in V^s_i$
  is closed and the stabilizer group $G_x$ is finite. We then have $G$-invariant
  open immersions $U'_i \subset V^s_i \subset V_i$
  (cf. \cite[\S~4.2]{Morel-Voevodsky}, \cite[Chap.~1, App.~B]{GIT}). We now apply
  \cite[Thm.~1.10]{GIT} (with $\sL = \sO_X$) to conclude that the quotient
  ${V^s_i}/G$ exists as a quasi-projective $k$-scheme.
On the other hand, there is a $G$-invariant open inclusion $U_i \subset U'_i$
  by our choice of $U_i$. It follows that there are open immersions
  ${U_i}/G \subset {U'_i}/G \subset {V^s_i}/G$. In particular, ${U_i}/G$ and
  ${U'_i}/G$ are quasi-projective $k$-schemes.
\end{proof}

Given an {\sl admissible gadget} $\rho = {\left(V_i, U_i\right)}_{i \ge 1}$ for $G$
and $X \in \Spc^G_k$, we let $X^i_G(\rho)$ denote the quotient space
$X \stackrel{G} {\times} U_i$.
We shall often write $X^i_G(\rho)$ as $X^i(\rho)$ if the underlying
group $G$ is fixed throughout the context.
We can say the following about $X^i(\rho)$.

\begin{cor}\label{cor:Mixed-quotient}
  Under any of the three conditions in \corref{cor:EG-quotient-2},
  we have $X^i(\rho) \in \Sch_k$. It is smooth over $k$ if $X$ is so.
  If $\rho$ is a good admissible gadget and $X$ is
  $G$-quasi-projective over $k$, then $X^i(\rho)$ is quasi-projective over $k$.
  
\end{cor}
\begin{proof}
 Follows from \corref{cor:EG-quotient-2}.
\end{proof}

Given an admissible gadget $\rho$ as above and $X \in \Spc^G_k$, 
the locally closed immersions $U_i = U_i \times \{0\} \inj U_{i+1}$ for $i \ge 1$
yield an ind-object $``{\varinjlim_i}" X^i(\rho)$ in $\Spc_k$.
We shall denote the limit of this ind-space (as an {\'e}tale sheaf of sets
on $\Sch_k$) by $X^\bullet_G(\rho)$  and call it
a `motivic Borel space'. If $k = \C$, then $X^\bullet_G(\rho)_{\an}$
is homotopy equivalent to the homotopy orbit space for the $G_{\an}$-action on
$X_{\an}$.

Since the inclusion $U_i \inj U_{i+1}$ has factorization $U_i \inj U_i \oplus
V \inj U_{i+1}$, where the first inclusion is a regular closed immersion and
the second is an open immersion, it follows that $U_i \inj U_{i+1}$ is a regular
immersion (in particular, an lci map). As $X$ is $k$-flat, it follows that
the canonical inclusion
$\alpha_i \colon X^i(\rho) \inj X^{i+1}(\rho)$ is also a regular immersion.
In particular, ${\bf{L}}\alpha^*_i$ sends cohomologically bounded
pseudo-coherent complexes on
$X^{i+1}(\rho)$ to cohomologically bounded
pseudo-coherent complexes on $X^{i}(\rho)$ (cf.
\cite[\S~2.5.1]{TT}). We thus have the pull-back maps
$\alpha^*_i \colon K'(X^{i+1}(\rho)) \to K'(X^{i}(\rho))$ which satisfy the
usual composition law. These maps together define a pro-spectrum
$\{K'(X^{i}(\rho))\}_{i \ge 1}$.
We similarly have a pro-spectrum $\{KH(X^{i}(\rho))\}_{i \ge 1}$.
In order to define the $K'$-theory of motivic Borel spaces using these pro-spectra,
we need to prove the following result.

\begin{prop}\label{prop:Rep-ind}
  Let  $\rho = {(V_i,U_i)}_{i \ge 1}$ and $\rho' = {(V'_i, U'_i)}_{i \ge 1}$ be two
  admissible gadgets for $G$ and let
  $X \in \Spc^G_k$. Then there is a canonical weak equivalence
  \[
  \kappa'_X \colon \holim_i K'(X^{i}(\rho)) \xrightarrow{\simeq}  
  \holim_i K'(X^{i}(\rho')).
  \]
  If $X$ is either
  \begin{enumerate}
  \item
    a toric variety on which $G$ acts as the dense torus or
  \item
   $\Char(k) =0$ and $X \in \Sch^G_k$,
  \end{enumerate}
  then there is a canonical weak equivalence
  \[
  \kappa_X \colon \holim_i KH(X^{i}(\rho)) \xrightarrow{\simeq}  
  \holim_i KH(X^{i}(\rho')).
  \]
The weak equivalences $\kappa_X$ and $\kappa'_X$ are functorial in $X \in \Spc^G_k$.
\end{prop}
\begin{proof}
For $i, j \ge 1$, we consider the algebraic spaces $\sV_{i,j} =
{\left(X \times U_i \times V'_j\right)}/G; \
\sW_{i,j} ={\left(X \times V_i \times U'_j\right)}/G$ and their common open 
subspace $\sU_{i,j} = {\left(X \times U_i \times U'_j\right)}/G$.
For a fixed $i \ge 1$, this yields a sequence ${\left(\sV_{i,j}, 
\sU_{i,j}, f_{i,j}\right)}_{j \ge 1}$, where $\sV_{i,j} \xrightarrow{\pi_{i,j}}
X^i(\rho)$ is a vector bundle, $\sU_{i,j} \subseteq \sV_{i,j}$ is an open
subspace of this vector bundle and $f_{i,j} : \left(\sV_{i,j}, \sU_{i,j}\right)
\to \left(\sV_{i,j+1}, \sU_{i,j+1}\right)$ is the natural map of pairs
of spaces over $X^i(\rho)$. For a fixed $j \ge 1$, we similarly get a 
sequence ${\left(\sW_{i,j}, \sU_{i,j}, g_{i,j}\right)}_{i \ge 1}$ of spaces over
$X^j(\rho')$.

Since $f_{i,j} \colon \sU_{i,j} \to \sU_{i,j+1}$ and
$g_{i,j} \colon \sU_{i,j} \to \sU_{i+1,j}$ are lci maps, we get a pro-spectrum
$\{K'(\sU_{i,j})\}_{i,j}$. Taking the limits, we get the maps
of spectra
\begin{equation}\label{eqn:Rep-ind-0}
  \holim_i K'(X^i(\rho)) \xrightarrow{f'} \holim_{i,j} K'(\sU_{i,j})
  \xleftarrow{g'} \holim_j K'(X^j(\rho')).
 \end{equation}
We similarly have the maps of spectra
\begin{equation}\label{eqn:Rep-ind-1}
  \holim_i KH(X^i(\rho)) \xrightarrow{f} \holim_{i,j} KH(\sU_{i,j})
  \xleftarrow{g} \holim_j KH(X^j(\rho')).
\end{equation}

It is clear that all these maps are natural in $X$. We shall show that
each of these maps is a weak equivalence under the hypotheses of the proposition.
Letting $\kappa'_X = g'^{-1} \circ f'$ and $\kappa_X = g^{-1} \circ f$ will then
finish the proof of the proposition, including the naturality claim.
Using the symmetry between $f$ and $g$ (resp. $f'$ and $g'$), the problem is reduced
to showing that $f$ and $f'$ are weak equivalences under the hypotheses of the
proposition. 

We now begin by showing that $f'$ is a weak equivalence. First of all, we can use the
nil-invariance property of $K'$-theory to assume that $X$ is reduced. We now divide
the proof into several steps.
In the first step, we consider the case when $X$ is a smooth and quasi-projective
$k$-scheme. Then $X$ is $G$-quasi-projective by
\cite[Thm.~2.5]{Sumihiro} and \cite[\S~5.7]{Thomason-Duke-2}.
Furthermore, $X^i(\rho) \in \Sm_k$ by \corref{cor:Mixed-quotient}.
Using the Milnor exact
sequence (cf. ~\eqref{eqn:Milnor-seq}), it suffices to show that the map of
pro-abelian groups $f ' \colon \{K'_n(X^i(\rho))\}_{i \ge 1} \to
\{K'_n(\sU_{i,j})\}_{i,j \ge 1}$ is an isomorphism for all $n \in \Z$.
  But this follows from \cite[Thm.~4.3]{Krishna-Crelle}.
  
  In the second step, we assume that $X$ is a $G$-quasi-projective $k$-scheme.
  In this case, we can find a $G$-equivariant closed embedding
  $\iota \colon X \inj X'$ such that $X' \in \Sch^G_k$ is smooth and
  quasi-projective over $k$. We let $U = X' \setminus X$ so that $U \in \Sch^G_k$
  is also smooth and quasi-projective over $k$.
   We let $\sU^{X'}_{i,j} = {\left(X' \times U_i \times U'_j\right)}/G$ and
  $\sU^U_{i,j} = {\left(U \times U_i \times U'_j\right)}/G$ so that
  $\sU_{i,j} = \sU^{X'}_{i,j} \setminus \sU^U_{i,j}$.  We now look at the commutative
  diagram of homotopy fiber sequences (cf. \cite[Thm.~2.7]{Thomason-Orange})

\begin{equation}\label{eqn:Rep-ind-2}
    \xymatrix@C1pc{
      \holim_i K'(X^i(\rho)) \ar[r] \ar[d]_-{f'_X} & \holim_i K'(X'^i(\rho)) \ar[r]
      \ar[d]^-{f'_{X'}} &  \holim_i K'(U^i(\rho)) \ar[d]^-{f'_U} \\
      \holim_{i,j} K'(\sU_{i,j}) \ar[r] & \holim_{i,j} K'(\sU^{X'}_{i,j})  \ar[r] &
      \holim_{i,j} K'(\sU^U_{i,j}),}
\end{equation}
where $f'_X = f'$ and $f'_{X'}$ (resp. $f'_U$) is the analogous map for $X'$
(resp. $U$).
The upper left horizontal arrow (which is the limit of push-forward maps)
exists and the left square commutes by
\cite[Lem.~5.3]{Krishna-Crelle}. The right square clearly commutes. 
We showed in the previous step that $f'_{X'}$ and $f'_U$ are
weak equivalences. It follows that $f'$ is a weak equivalence.

In the third step, we assume that $X$ is a 0-dimensional algebraic space over $k$.
In this case, $X$ must coincide with its scheme locus. In particular, it is
a 0-dimensional $k$-scheme. It is then necessarily regular and quasi-projective over
$k$. In particular, it is $G$-quasi-projective and we are back to the second step.

In the final step, we let $X \in \Spc^G_k$ be any reduced algebraic space of
positive dimension. We use \lemref{lem:Open-quasi-proj} to find a
  $G$-invariant dense open $U \subset X$ such that $U$ is regular and
$G$-quasi-projective over $k$. We let $Z = X \setminus U$ be the complement with
the reduced algebraic space structure. It is then easy to check that $Z$ is
$G$-invariant (cf. \cite[Lem.~2.5]{Thomason-Orange}).
  We let $\sU^Z_{i,j} = {\left(Z \times U_i \times U'_j\right)}/G$ and
  $\sU^U_{i,j} = {\left(U \times U_i \times U'_j\right)}/G$ so that
  $\sU^Z_{i,j} = \sU_{i,j} \setminus \sU^U_{i,j}$.  We now look at the commutative
  diagram of homotopy fiber sequences 
  \begin{equation}\label{eqn:Rep-ind-3}
    \xymatrix@C1pc{
      \holim_i K'(Z^i(\rho)) \ar[r] \ar[d]_-{f'_Z} & \holim_i K'(X^i(\rho)) \ar[r]
      \ar[d]^-{f'_X} &  \holim_i K'(U^i(\rho)) \ar[d]^-{f'_U} \\
      \holim_{i,j} K'(\sU^Z_{i,j}) \ar[r] & \holim_{i,j} K'(\sU_{i,j})  \ar[r] &
      \holim_{i,j} K'(\sU^U_{i,j}).}
  \end{equation}
  We have shown previously that the right vertical arrow is a weak equivalence.
  The left vertical arrow is a weak equivalence by induction on the dimension of
  $X$. It follows that the middle vertical arrow is a weak equivalence.

We now show that the map $f$ in ~\eqref{eqn:Rep-ind-1} is a weak equivalence if
  $X$ satisfies condition (1) or (2) of the proposition. We shall do this by
  induction on $\dim(X)$.  We let $X' = X_\red$ and
 $\sU^{X'}_{i,j} = {\left(X' \times U_i \times U'_j\right)}/G$.
We look at the commutative diagram
  \begin{equation}\label{eqn:Rep-ind-4}
    \xymatrix@C1pc{
      \holim_i KH(X^i(\rho)) \ar[d]_-{f_X} \ar[r] & \holim_i KH(X'^i(\rho)) \ar[r]
      \ar[d]^-{f_{X'}} & \holim_i K'(X'^i(\rho)) \ar[d]^-{f'_{X'}} \\
      \holim_{i,j} KH(\sU_{i,j}) \ar[r] & \holim_{i,j} KH(\sU^{X'}_{i,j}) \ar[r] & 
      \holim_{i,j} K'(\sU^{X'}_{i,j}),}
    \end{equation}
where $f_X = f$ and $f_{X'}$ is the analogous map for $X'$.
  The horizontal arrows in the left square of this diagram
  are the pull-back maps under the
  inclusion $X' \inj X$ and the ones in the right square are induced by the natural
  transformation from the $KH$-theory to the $K'$-theory presheaf in the category of
  $k$-spaces with flat maps. One consequence of the cdh-descent for algebraic spaces
  (cf. \cite[Thm.~6.2]{Hoyois-Krishna}) is that the horizontal
  arrows in the left square in ~\eqref{eqn:Rep-ind-4} are weak equivalences.

Suppose now that $\dim(X) = 0$. Then $X_\red$ is a regular $k$-scheme and
  this implies that the horizontal arrows in right square in
  ~\eqref{eqn:Rep-ind-4} are weak equivalences. Meanwhile,
  we showed previously in this case that the right-most vertical arrow is a weak
  equivalence. We deduce that all vertical arrows in ~\eqref{eqn:Rep-ind-4} are weak
  equivalences. We shall assume in the rest of the proof that $\dim(X) \ge 1$.

To prove the proposition in general under its condition (2),
  we choose a $G$-equivariant map $p \colon \wt{X} \to X'$ such that
  $\wt{X} \in \Sch^G_k$ and $p$ is a resolution of singularities of $X'$.
  We let $Y = X'_\sing$ with the reduced closed subscheme
  structure and let $Y' = p^{-1}(Y)$. We get a commutative diagram in $\Sch^G_k$:
\begin{equation}\label{eqn:Rep-ind-7}
    \xymatrix@C1pc{
      Y' \ar[r] \ar[d] & \wt{X} \ar[d] \\
      Y \ar[r] & X'.}
 \end{equation}

 For every $i$ and $j$, the diagram ~\eqref{eqn:Rep-ind-7} gives rise to
 the abstract blow-up squares 
  \begin{equation}\label{eqn:Rep-ind-5}
    \xymatrix@C1pc{
      Y'^i(\rho) \ar[r] \ar[d] & \wt{X}^i(\rho) \ar[d] & &
      \sU^{Y'}_{i,j} \ar[r] \ar[d] & \sU^{\wt{X}}_{i,j} \ar[d] \\
      Y^i(\rho) \ar[r] & X'^i(\rho)  & &
      \sU^Y_{i,j} \ar[r] & \sU^{X'}_{i,j}.}
    \end{equation}
of algebraic spaces over $k$. 

By \cite[Thm.~6.2]{Hoyois-Krishna}, the above squares give rise to a commutative
diagram of homotopy fiber sequences of spectra
 \begin{equation}\label{eqn:Rep-ind-6}
    \xymatrix@C.6pc{
      \ \holim_i KH(X'^i(\rho)) \ar[d]_-{f_{X'}} \ar[r] &
      \holim_i KH(\wt{X}^i(\rho))
      \amalg \holim_i KH(Y^i(\rho)) \ar[d]^-{f_{\wt{X}} \amalg f_Y} \ar[r] &
      \holim_i KH(Y'^i(\rho)) \ar[d]^-{f_{Y'}} \\
     \holim_{i,j} KH(\sU^{X'}_{i,j}) \ar[r] &
      \holim_{i,j} KH(\sU^{\wt{X}}_{i,j}) \amalg  \holim_{i,j} KH(\sU^{Y}_{i,j}) \ar[r] &
      \holim_{i,j} KH(\sU^{Y'}_{i,j}).}
\end{equation}

 The middle and the right vertical arrows are weak equivalences by induction
 because $\dim(Y) < \dim(X) > \dim(Y')$ and $\wt{X}$ is regular. It follows that the
 left vertical arrow is a weak equivalence. As the horizontal arrows in the
 left square of ~\eqref{eqn:Rep-ind-4} are weak equivalences, we deduce
 that $f$ is a weak equivalence.

 To prove the proposition under its condition (1), we first note that
 a toric variety over $k$ of minimal dimension is regular. In particular, the
 claim holds for such a variety using the right square in ~\eqref{eqn:Rep-ind-4}.
 We next note that for an arbitrary toric variety $X$ with dense torus $G$, 
 there exists a toric resolution of singularities $p \colon \wt{X} \to X$.
 In particular, there exists a commutative diagram identical to
 ~\eqref{eqn:Rep-ind-7} (with $X'$ replaced by $X$) in $\Sch^G_k$ in which all
 schemes are toric varieties (whose dense tori are
 quotients of $G$), all maps are $G$-equivariant and $\wt{X}$ is regular
 (e.g., see \cite[\S~2]{CHWW}).
 We now repeat the above proof in the characteristic zero case
 to conclude that $f$ is a weak equivalence. 
\end{proof}

We can now define the Borel equivariant $K$-theory as follows.

\begin{defn}\label{defn:K-thry-Borel}
  Let $X \in \Spc^G_k$ be as in \propref{prop:Rep-ind}. We define the
  Borel equivariant $K$-theories of $X$ for $G$-action by
  \[
  K'(X_G) := \holim_i K'(X^i(\rho)); \ \ \mbox{and} \ \
  KH(X_G) := \holim_i KH(X^i(\rho)),
  \]
  where $\rho = (V_i, U_i)_{i \ge 1}$ is any admissible gadget for $G$.
  It follows from \propref{prop:Rep-ind} that $K'(X_G)$ and $KH(X_G)$ are
  well defined objects of $\sS\sH$.
  \end{defn}

We shall use the following formal properties of the equivariant Borel $K$-theory.

\begin{prop}\label{prop:K-thry-Borel-0}
  For $X \in \Spc^G_k$ as in \propref{prop:Rep-ind}, the Borel equivariant
  $K$-theories satisfy the following properties.
  \begin{enumerate}
  \item
  The assignment $X \mapsto K'(X_G)$ is contravariant for lci maps and covariant  
  for proper maps in $\Spc^G_k$.
\item
  The assignment $X \mapsto KH(X_G)$ is contravariant for all maps and covariant for
  proper lci maps in $\Sch^G_k$.
  \end{enumerate}
\end{prop}
\begin{proof}
  This follows directly by applying \cite[Lem.~5.3]{Krishna-Crelle}
  (whose part (1) is applied to the inclusion
  $f\colon X \times U_i \inj X \times U_{i+1}$ and part (2) is applied to
  $f \colon  X^i(\rho) \inj  X^{i+1}(\rho)$), \cite[Prop.~3.18]{TT} and passing to
  the limits.
  \end{proof}

\section{Atiyah-Segal theorem for Borel construction}\label{sec:Compln*}
In this section, we shall prove the Atiyah-Segal completion theorem
which asserts that the Lurie completion of the equivariant $K$-theory coincides with
the Borel equivariant $K$-theory. We begin by constructing the motivic Atiyah-Segal
map. We fix a field $k$ and a $k$-group $G$.
Let $I_G \subset R(G)$ be the (integral) augmentation ideal.

\subsection{Motivic Atiyah-Segal map}\label{sec:MAS-map}
We let $\rho = (V_i, U_i)_{i \ge 1}$ be
an admissible gadget for $G$. If $X$ is a $G$-quasi-projective $k$-scheme,
\corref{cor:Mixed-quotient} and \cite[Lem.~8.1]{Krishna-Crelle}
(whose proof works for any $G$) together imply that
$K'_i(X^i(\rho)) \cong K'_i([{(X \times U_i)}/G])$ is $I_G$-nilpotent for every $i$.
As shown in the proof of \corref{cor:Fin-support-4}, this implies that
$K'(X^i(\rho))$ is $I_G$-complete. Since $I_G$ is the radical of a finitely generated
ideal by \lemref{lem:Noether-Rep}, we get that the completion map
$K'(X^i(\rho)) \to K'(X^i(\rho))^{\compl}_{I_G}$ is a weak equivalence.
Combined with nil-invariance and localization sequence,
\lemref{lem:Open-quasi-proj} implies that this weak equivalence holds for any
$X \in \Spc^G_k$.

We can repeat the proof of \propref{prop:Rep-ind} to
conclude that the completion map $KH(X^i(\rho)) \to KH(X^i(\rho))^{\compl}_{I_G}$
is a weak equivalence if $X$ satisfies one of the conditions of
the proposition. It follows that the canonical maps
$\Phi^i_{[X/G]} \colon K'_G(X) \to K'_G(X \times U_i) \simeq K'(X^i(\rho))$
and $\vartheta^i_{[X/G]} \colon KH_G(X) \to KH_G(X \times U_i) \simeq KH(X^i(\rho))$,
induced by the $G$-equivariant projection $\phi^i \colon X \times U_i \to X$,
factor through the completions $\Phi^i_{[X/G]} \colon K'_G(X)^{\compl}_{I_G} \to
K'(X^i(\rho))$ and $\vartheta^i_{[X/G]} \colon KH_G(X)^{\compl}_{I_G} \to KH(X^i(\rho))$.
Passing to the limit, we get the following.

\begin{prop}\label{prop:AS-map-0}
  Let $X \in \Spc^G_k$ be as in \propref{prop:Rep-ind}. Then the collection
  $(\phi^i)_{i \ge 1}$ induces $K(BG)$-linear maps
  \begin{equation}\label{eqn:AS-map-1}
  \Phi_{[X/G]} \colon K'_G(X)^{\compl}_{I_G} \to K'(X_G); \ \
  \vartheta_{[X/G]} \colon KH_G(X)^{\compl}_{I_G} \to KH(X_G).
  \end{equation}
  The map $\Phi_{[X/G]}$ commutes with lci pull-back and proper push-forward maps
  in $K'$-theory. The map $\vartheta_{[X/G]}$ commutes with arbitrary pull-back and
  proper lci push-forward maps in $KH$-theory.
\end{prop}
\begin{proof}
  We only need to explain the commutativity of
  $\Phi_{[X/G]}$ and $\vartheta_{[X/G]}$ with the 
  pull-back and push-forward maps. To prove this, we note that it suffices to prove
  the commutativity of $\Phi_{[X/G]}$ and $\vartheta_{[X/G]}$ before taking the
  $I_G$-adic completions of their left hand sides. Now, the problem 
  reduces to showing that the maps on equivariant $K'$-theory and $KH$-theory
  induced by the
  flat projections $\phi^i \colon X \times U_i \to X$ have the desired properties
  (cf. \propref{prop:K-thry-Borel-0}). But this is standard
  (cf. \cite[Prop.~3.18]{TT}, \cite[Lem.~5.3]{Krishna-Crelle}).  
\end{proof}

\subsection{Equivariant vs. Borel equivariant $K'$-theory}\label{sec:Compln-thm}
We shall now prove the completion theorem relating equivariant and Borel equivariant
$K'$-theory. This will be done in several steps.
We begin with the following special case.

We let $G = \Spec(k[t_1^{\pm 1}, \ldots , t^{\pm 1}_r])$ be a split
torus of rank $r$ so that $R(G) = \Z[t^{\pm 1}_1, \ldots , t^{\pm 1}_r]$ and
$I_G = (t_1 -1, \ldots , t_r -1)$. We let $G_j = \Spec(k[t_j^{\pm 1}])$ for
  $1 \le j \le r$. For $\chi \in \wh{G}$, we let
$L_{\chi}$ denote the one-dimensional 
representation of $G$ on which it acts via the character $\chi$. We fix
a $\Z$-basis $\{\chi_1, \cdots , \chi_r\}$ of $\wh{G}$, where
$\chi_i$ is considered as a basis element of $\wh{G_i}$. Given
$i \ge 1$, set $V_i = \stackrel{r}{\underset{j = 1} \prod} 
L^{\oplus i}_{\chi_j}$ and $U_i = \stackrel{r}{\underset{j = 1} \prod} 
(L^{\oplus i}_{\chi_j} \setminus \{0\})$.
Then $G$ acts on $V_i$ by $(t_1, \cdots , t_r)(x_1, \cdots , x_r)
= \left(\chi_1(t_1)(x_1), \cdots , \chi_r(t_r)(x_r)\right)$.
It is then easy to see that $\rho = \left(V_i, U_i\right)_{i \ge 1}$ is an 
admissible gadget for $G$ such that ${U_i}/G \simeq \left(\P^{i-1}_k\right)^r$.
The line bundle $L_{\chi_j} \stackrel{G_j}{\times} 
(L^{\oplus i}_{\chi_j}\setminus \{0\}) \to \P^{i-1}_k$ is the line bundle
$\sO(1)$ for each $1 \le j \le r$.
We shall use $\rho$ to prove the completion theorem in this case.
The proof is based on the following observation.

\begin{lem}\label{lem:SIF}
  Assume $G = G' \times \G_m$, where $G'$ is a split torus and $X \in \Spc^G_k$. Let
  $\iota \colon X \to X \times V$ be
  the 0-section embedding, where $V$ is an $n$-dimensional representation of $G$ on
  which the split torus $G'$ acts trivially and $\G_m  = \Spec(k[t^{\pm 1}])$
  acts via $t$. Then the composition
  $K'_G(X) \xrightarrow{\iota_*} K'_G(X \times V) \xrightarrow{\iota^*} K'_G(X)$ is
  homotopic to the map $(1-t^{-1})^n \colon K'_G(X) \to K'_G(X)$.
\end{lem}
\begin{proof}
  We consider the commutative diagram 
  \begin{equation}\label{eqn:SIF-0}
    \xymatrix@C.8pc{
      X \ar[r]^-{\iota} \ar[d]_-{p} & X \times V \ar[d]^-{q_2} \ar[r]^-{q_1} & X
      \ar[d]^-{p}
      \\
    \pt \ar[r]^-{\iota'} & V \ar[r]^-{q} & \pt}
\end{equation}
  in $\Spc^G_k$ in which $p,q$ and $q_i$'s are projections and $\iota'$ is
  the 0-section inclusion. The two squares are
  Cartesian, vertical arrows are flat and
  the horizontal arrows are lci morphisms. It follows that
  $\iota_*(1) = \iota_* \circ p^*(1) = q^*_2 \circ \iota'_*(1)$. Meanwhile,
  one checks directly that $\iota'_*(1) = q^*((1-t^{-1})^n)$
  (cf. \cite[Thm.~2.1]{VV-Inv}). This yields
  $\iota_*(1) =  q^*_2 \circ q^*((1-t^{-1})^n) = q^*_1 \circ p^*((1-t^{-1})^n)$.
  It follows that $\iota_*(1) = (1-t^{-1})^n$. 

We now compute
  \[
  \begin{array}{lll}
    \iota^* \circ \iota_* & \simeq & \iota^* \circ \iota_* \circ \iota^* \circ
    q^*_1 \\
    & {\simeq}^{\dagger} & \iota^* \circ (\iota_*(1) \cdot  q^*_1) \\
    & {\simeq} &  \iota^*(\iota_*(1)) \cdot (\iota ^* \circ q^*_1) \\
    & \simeq & (1-t^{-1})^n \cdot (\iota ^* \circ q^*_1) \\
    & \simeq & (1-t^{-1})^n \cdot (q_1 \circ \iota)^* \\
    & \simeq & (1-t^{-1})^n, \\
  \end{array}
  \]
  where ${\simeq}^{\dagger}$ is a consequence of the projection formula
  (cf. \cite[Prop.~3.17]{TT} or \cite[Prop.~3.7]{Khan-JJM}).
\end{proof}

\begin{lem}\label{lem:Compln-thm-2}
  Let $X \in \Spc^G_k$. Then $\Phi_{[X/G]}$  is a weak equivalence when $G$ is a split
  torus.
\end{lem}
\begin{proof}
  Following the above notations, we let $V_{m,i} = \stackrel{m}{\underset{j = 1}
    \prod} L^{\oplus i}_{\chi_j}$ and $U_{m,i} = \stackrel{m}{\underset{j = 1} \prod} 
  \left(L^{\oplus i}_{\chi_j} \setminus \{0\}\right)$ for $i \ge 1$ and $1 \le m \le r$.
  We let $G$ act on $V_{m,i}$ via the product of its first $m$ factors.
  We claim for $1 \le m \le r$ and $i \ge 1$ that the projection
  $\phi^i_m \colon X \times U_{m,i} \to X$ induces a weak equivalence 
  \begin{equation}\label{eqn:Compln-thm-2-0}
    (\phi^i_m)^* \colon \frac{K'_G(X)}{((t_1-1)^i, \ldots , (t_m-1)^i)K'_G(X)}
    \xrightarrow{\simeq} K'_G(X \times U_{m,i}).
  \end{equation}
  Applying this claim for $m =r$ (and using the identification
  $K'_G(X \times U_{r,i}) \cong  K'(X^i(\rho))$, we
  can conclude the proof of the lemma by passing to the limit as
  $i \to \infty$ and using ~\eqref{eqn:I-comp-0}.
  It remains therefore to prove the claim.

To that end, we look at the homotopy fiber sequence
  \begin{equation}\label{eqn:Compln-thm-2-1}
    K'_G(X) \xrightarrow{\iota_*} K'_G(X \times V_{1,i}) \xrightarrow{(\alpha^i_1)^*}
    K'_G(X \times U_{1,i}),
  \end{equation}
  where $\alpha^i_m \colon X \times U_{m,i} \inj X \times V_{m,i}$ is the open
  inclusion.
  Using the homotopy invariance, \lemref{lem:SIF} and the equality of ideals
  $(t_m-1) = (1- t^{-1}_m) \subset R(G)$, we get ~\eqref{eqn:Compln-thm-2-0} when
  $m =1$. For any $1 \le m < r$, we let $F_m =
  \frac{K'_G(X)}{((t_1-1)^i, \ldots , (t_m-1)^i)K'_G(X)}$.

  We then get
  \[
  \begin{array}{lll}
  \frac{K'_G(X)}{((t_1-1)^i, \ldots , (t_{m}-1)^i, (t_{m+1} -1)^i)K'_G(X)}  
  & \xrightarrow{\simeq} & \frac{F_m}{((t_{m+1}-1)^i)F_m} \\
  & {\xrightarrow{\simeq}}^1 & \frac{K'_G(X \times U_{m,i})}
       {((t_{m+1}-1)^i)K'_G(X \times U_{m,i})} \\
       & {\xrightarrow{\simeq}}^2 & K'_G(X \times U_{m,i} \times
       (L^{\oplus i}_{\chi_{m+1}} \setminus \{0\})) \\
       & \xrightarrow{\simeq} & K'_G(X \times U_{m+1,i}), \\
       \end{array}
  \]
  where ${\simeq}^1$ follows by induction on $m$ and
  ${\simeq}^2$ follows by applying the $m =1$ case to
  $X \times U_{m,i} \in \Spc^G_k$. This proves ~\eqref{eqn:Compln-thm-2-0}
  in general.
\end{proof}

To prove that $\Phi_{[X/G]}$ is a weak equivalence in the general case, we need
a lemma. Suppose we are given an embedding $G \inj G'$ of $k$-groups. Let
$X \in \Spc^{G}_k$ and $X' \in \Spc^{G'}_k$. We let $X \stackrel{G}{\times} G'$ denote
the associated Morita
space with $G'$-action. This datum gives rise to two objects in $\Spc^{G'}_k$. The
first is the Morita space $(X \times X') \stackrel{G}{\times} G'$ associated to the
the diagonal $G$-action on $X \times X'$ (where $G$
acts on $X'$ via the inclusion $G \subset G'$). The second is
$(X  \stackrel{G}{\times} G')  \times X'$ with the diagonal $G'$-action.

\begin{lem}\label{lem:Mixed-iso}
  There is a $G'$-equivariant isomorphism
  \begin{equation}\label{eqn:Mixed-iso-1}
  \theta \colon (X \times X')  \stackrel{G}{\times} G' \xrightarrow{\cong}
  (X \stackrel{G}{\times} G') \times X'
  \end{equation}
  such that the diagram
  \begin{equation}\label{eqn:Mixed-iso-0}
    \xymatrix@C2pc{
    X \times X' \ \ \ar[r]^-{\delta_{(X \times X')}} \ar[d]_-{p} &   
      \ (X \times X')  \stackrel{G}{\times} G' \ar[r]^-{\theta} &
      (X \stackrel{G}{\times} G') \times X' \ar[d]^-{q} \\
    X \ar[rr]^-{\delta_X} & & X \stackrel{G}{\times} G'}
  \end{equation}
  commutes, where $p$ and $q$ are projections and $\delta_{(-)}$ is defined in
  ~\eqref{eqn:Open-quasi-proj-0}.
 \end{lem}
\begin{proof}
  We let $Y =  X \stackrel{G}{\times} G'$ and $Z =  X \times X'$. We define
  $\theta$ by letting $\theta (\tau_1((x,x'), g')) = (\tau_2(x,g'), g'x')$,
  where $\tau_1 \colon Z \times G' \to Z \stackrel{G}{\times} G'$ and
  $\tau_2 \colon X \times G' \to Y$ are the quotient maps.
  To show that $\theta$ is well-defined, we let $(x_1, x'_1, g'_1) =
  g(x, x', g') = (gx, gx', g'g^{-1})$, where $g \in G$. Then we get
  \[
  \begin{array}{lllll}
  \theta (\tau_1((x_1,x'_1), g'_1)) & = & (\tau_2(x_1, g'_1), g'_1x'_1) & = &
  (\tau_2(gx, g'g^{-1}), g'_1gx') \\
  & = & (\tau_2 (g(x, g')), g'_1gx') & = &  (\tau_2(x, g'), g'_1gx') \\
  & = &  (\tau_2(x, g'), g'x') & = & \theta (\tau_1((x,x'), g')). \\
  \end{array}
  \]
  This shows that $\theta$ is well-defined.

We next define $\theta' \colon Y \times X' \to Z \stackrel{G}{\times} G'$
  by $\theta'(\tau_2(x,g'), x') = \tau_1(x, {g'}^{-1}x', g')$. To show that
  $\theta'$ is well-defined, we let $(x_1, g'_1) =
  g(x, g') = (gx, g'g^{-1})$ where $g \in G$. Then we get
\[
  \begin{array}{lllll}
    \theta'(\tau_2(x_1,g'_1), x') & = & \theta'(\tau_2(gx, g'g^{-1}), x') & = &
    \tau_1((gx, (g'g^{-1})^{-1}x'), g'g^{-1}) \\
    & = & \tau_1(gx, g{g'}^{-1}x' g'g^{-1}) & = & \tau_1(g((x,g'^{-1}x'), g')) \\
    & = & \tau_1((x, g'^{-1}x'), g') & = & \theta'(\tau_2(x,g'), x').
  \end{array}
  \]
  One easily verifies that $\theta$ and $\theta'$ are $G'$-equivariant and are
  inverses of each other. The commutativity of ~\eqref{eqn:Mixed-iso-0} is also
  clear.
\end{proof}

We now prove our Atiyah-Segal completion theorem for Borel construction
(note: $\Phi_{[X/G]}$ in the theorem below is denoted by $\wt{p}^*$ in
\thmref{thm:Main-2}).

\begin{thm}\label{thm:AS-Main}
For any $X \in \Spc^G_k$, the map of spectra
  \[
  \Phi_{[X/G]} \colon K'_G(X)^{\compl}_{I_G} \to K'(X_G)
  \]
is a weak equivalence.
\end{thm}
\begin{proof}
  We divide the proof into two steps.
 We first assume that $G$ is a split reductive group which is also special.
  We choose a Borel subgroup $B \subset G$ and a split maximal torus
  $T \subset G$ contained in $B$. Using \lemref{lem:Adm-gadget},
  we can choose a good admissible gadget $\rho = (V_i, U_i)_{i \ge 1}$ for $G$
  which is also an admissible gadget for $B$ and $T$. We consider the commutative
  diagram of spectra
 \begin{equation}\label{eqn:AS-Main-0}
\xymatrix@C1pc{
  K'_G(X)^{\compl}_{I_G}  \ar[d]_-{\Phi_{[X/G]}} \ar[r]^-{\alpha_X} &
  K'_B(X)^{\compl}_{I_B} \ar[d]^-{\Phi_{[X/B]}} \ar[r]^-{\beta_X} &
K'_T(X)^{\compl}_{I_T} \ar[d]^-{\Phi_{[X/T]}} \\
   K'(X_G) \ar[r]^-{\alpha'_X} & K'(X_B) \ar[r]^-{\beta'_X}  &  K'(X_T),}
\end{equation}
where the horizontal arrows are the canonical change of groups maps. 

The ring $R(G)$ is Noetherian as $G$ is special
(see the proof of \propref{prop:finiteR}). Moreover,
$I_B = \sqrt{I_GR(B)}$ and $I_T = \sqrt{I_GR(T)}$ by \cite[Cor.~6.1]{EG-RR}
(see also the proof of \lemref{lem:Noether-Rep}). 
We conclude from \cite[Rem.~7.3.1.2]{Lurie-SAG} and \propref{prop:Compln_D}(5)
that the canonical maps
$K'_B(X)^{\compl}_{I_G} \to K'_B(X)^{\compl}_{I_B}$ and $K'_T(X)^{\compl}_{I_G} \to
K'_T(X)^{\compl}_{I_T}$ are weak equivalences. On the other hand, the
map $K'_B(X) \to K'_T(X)$ is a weak equivalence by \cite[Thm.~1.13]{Thomason-Duke-2}
(see its proof). In particular, the map
$K'_B(X)^{\compl}_{I_G} \to K'_T(X)^{\compl}_{I_G}$
is a weak equivalence. It follows that $\beta_X$ is a weak equivalence.
Since $K'(X \stackrel{B}{\times} U_i)  \to K'(X \stackrel{T}{\times} U_i)$ is a
weak equivalence for each $i$ (cf. \cite[Thm.~1.3]{Thomason-Duke-2}),
we get after passing to the limit and using the Milnor exact sequence
(cf. ~\eqref{eqn:Milnor-seq})  that $\beta'_X$  is also a weak equivalence.

We now look at the commutative diagram
\begin{equation}\label{eqn:AS-Main-1}
\xymatrix@C1pc{
  K'_G(X) \ar[r]^-{\phi^*} \ar[d]_-{p^*_i} & K'_G(X \times G/B) \ar[d]^-{q^*_i}
  \ar[r]^-{\phi_*} & K'_G(X) \ar[d]^-{p^*_i} \\
  K'_G(X \times U_i) \ar[r]^-{\phi^*_i} & K'_G((X \times U_i) \times G/B)
  \ar[r]^-{\phi_{i*}} & K'_G(X \times U_i),}
\end{equation}
where $p_i \colon X \times U_i \to X$ is the projection and
$q_i = p_i \times \id$. The maps $\phi$ and $\phi_i$ are also projections.
The left square clearly commutes and the right square commutes by
\cite[Prop.~3.8]{Khan-JJM}.
Since $\phi \colon X \stackrel{B}{\times} G \cong X \times (G/B) \to X$
is a $G$-equivariant flag bundle, it follows from \cite[Thm.~1.13]{Thomason-Duke-2}
and the projection formula for $\phi$ (cf. \cite[Prop.~3.7]{Khan-JJM} or
\cite[Prop.~3.17]{TT}) that the composite horizontal arrows on the top in
~\eqref{eqn:AS-Main-1} is homotopic to identity. The same also holds for the
composite bottom horizontal arrow. The middle vertical arrow is same as the
map $p^*_i \colon K'_B(X) \to K'(X \stackrel{B}{\times} U_i)$ by
\lemref{lem:Morita*}.

Since ~\eqref{eqn:AS-Main-1} is compatible as $i$ varies through positive integers,
we get, after passing to the limit and completions, a commutative diagram
\begin{equation}\label{eqn:AS-Main-2}
\xymatrix@C1pc{
 K'_G(X)^{\compl}_{I_G}  \ar[d]_-{\Phi_{[X/G]}} \ar[r]^-{\phi^*} &
 K'_B(X)^{\compl}_{I_G} \ar[d]^-{\Phi_{[X/B]}} \ar[r]^-{\phi_*} &
K'_G(X)^{\compl}_{I_G} \ar[d]\ar[d]^-{\Phi_{[X/G]}} \\
   K'(X_G) \ar[r]^-{\phi^*} & K'(X_B) \ar[r]^-{\phi_*}  &  K'(X_G),}
\end{equation}
in which the composite horizontal arrows are homotopic to identity.
Since $\Phi_{[X/B]}$ is a weak equivalence, it follows by looking at the induced
maps on the homotopy groups that $\Phi_{[X/G]}$ is a weak equivalence.

To prove the general case, we choose an embedding $G \inj G'$, where $G'$ is a
general linear group over $k$ and choose a good admissible gadget
$\rho = (V_i, U_i)_{i \ge 1}$ for $G'$. Since $G'$ is special, $\rho$ exists
by \lemref{lem:Adm-gadget}. Furthermore, this is an admissible gadget for $G$ too.
We let $Y = X \stackrel{G}{\times} G'$ be the associated Morita space. We
let $X_i = X \times U_i$. 
We now look at the diagram

\begin{equation}\label{eqn:AS-Main-3}
\xymatrix@C1pc{
  K'_{G'}(Y) \ar[rr]^-{\delta^*_X} \ar[d]_-{q^*_i} & & K'_G(X)
  \ar[d]^-{p^*_i} \\
  K'_{G'}(Y \times U_i) \ar[r]^-{\theta^*} & K'_{G'}(X_i \stackrel{G}{\times} G')
  \ar[r]^-{\delta^*_{X_i}} & K'_G(X_i),}
\end{equation}
where the vertical arrows are induced by the projections.
This diagram is commutative and $\theta^*$ is a weak equivalence
by \lemref{lem:Mixed-iso}. The arrows $\delta^*_X$ and $\delta^*_{X_i}$ are weak
equivalences by \lemref{lem:Morita*}.

Using the compatibility of ~\eqref{eqn:AS-Main-3}
as $i$ varies through positive integers, we can pass to the limit and completions
to get a commutative diagram
\begin{equation}\label{eqn:AS-Main-4}
\xymatrix@C1pc{
  K'_{G'}(Y)^{\compl}_{I_{G'}} \ar[r]^-{\delta^*_X} \ar[d]_-{\Phi_{[Y/{G'}]}} &
  K'_G(X)^{\compl}_{I_{G'}} \ar[d]^-{\Phi_{[X/G]}} \\
  K'(Y_{G'}) \ar[r]^-{\delta^*_X \circ \theta^*} & K'(X_G),}
\end{equation}
in which the horizontal arrows are weak equivalences.
The left vertical arrow is a weak equivalence because $G'$ is special.
It follows that the right vertical arrow is also a weak equivalence.
Finally, we use the weak equivalence $K'_G(X)^{\compl}_{I_{G'}} \xrightarrow{\simeq}
K'_G(X)^{\compl}_{I_{G}}$ to conclude that $\Phi_{[X/G]}$ is a weak equivalence.
This finishes the proof of the theorem.
\end{proof}

\subsection{Equivariant vs. Borel Equivariant KH-theory}\label{sec:KH-Borel}
The following is the completion theorem for equivariant $KH$-theory.

\begin{thm}\label{thm:AS-KH}
  Let $X \in \Sch^G_k$. Assume that $X$ is either
  \begin{enumerate}
  \item
    a toric variety over $k$ on which $G$ acts as the dense torus or
  \item
    $\Char(k) = 0$ and $G$ acts with nice stabilizers.
  \end{enumerate}
  Then the map of spectra 
  \[
  \vartheta_{[X/G]} \colon KH_G(X)^{\compl}_{I_G} \to KH(X_G)
  \]
  is a weak equivalence. If $k$ contains all roots of unity and $m \in k^\times$, then
  \[
   \vartheta_{[X/G]} \colon
  (KH_G(X, {\Z}/m)[\beta^{-1}])^{\compl}_{I_G} \to KH(X_G, {\Z}/m)[\beta^{-1}]
  \]
  is a weak equivalence.
\end{thm}
\begin{proof}  
  By the cdh descent theorem \cite[Thm.~6.2]{Hoyois-Krishna}, we can assume that $X$
  is reduced. In particular, $X$ is regular if $\dim(X) = 0$ and our assertion holds
  by \thmref{thm:AS-Main}. In general, we choose a $G$-equivariant resolution
  (toric resolution in case (1)) of singularities $p \colon \wt{X} \to X$
  and look at the resolution diagram ~\eqref{eqn:Rep-ind-7}.

  Since the stabilizer of any point of $\wt{X}$ is a closed subgroup of the
  stabilizer of its image in $X$, it follows from \lemref{lem:Nice-sub} that
  $G$ acts on $\wt{X}$ and $Y'$ (also on $Y = X_\sing$) with nice stabilizers.
We now apply the cdh-descent theorem and follow
the proof and notations of \propref{prop:Rep-ind} (with $X'$ substituted by $X$).
This yields us a commutative diagram
  \begin{equation}\label{eqn:AS-KH-0}
    \xymatrix@C1pc{
      KH_G(X)^{\compl}_{I_{G}} \ar[d]_-{\vartheta_{[X/G]}} \ar[r] &
      KH_G(\wt{X})^{\compl}_{I_{G}} \amalg KH_G(Y)^{\compl}_{I_{G}}
      \ar[d]^-{(\vartheta_{[{\wt{X}}/G]}
        \amalg \vartheta_{[Y/G]})} \ar[r] & KH_G(Y')^{\compl}_{I_{G}}
        \ar[d]^-{\vartheta_{[Y'/G]}} \\
          KH(X_G) \ar[r] & KH(\wt{X}_G) \amalg KH(Y_G) \ar[r] & KH(Y'_G),}
  \end{equation}
  where all horizontal arrows are the pull-back maps.
  The two rows are homotopy fiber
  sequences by ~\eqref{eqn:Rep-ind-5} and the cdh descent. The middle and the right
  vertical arrows are weak equivalences either by regularity or by induction on the
  dimension. We conclude that $\vartheta_{[X/G]}$ is a weak equivalence.
  The proof that this map is a  weak equivalence for the Bott-inverted $KH$-theory
  is identical.
\end{proof}

\section{Completion at other maximal ideals}
\label{sec:Max-com}
In this section, we shall prove a generalized version of Theorem~\ref{thm:AS-Main} 
(resp. ~\thmref{thm:AS-KH}) which describes the completions of
equivariant $K$-theory (resp. $KH$-theory) with $k$-coefficients
at all maximal ideals of $R_k(G)$ in
terms of Borel equivariant $K$-theory if $k$ is an algebraically closed field
of characteristic zero.
We shall first prove a derived version of the nonabelian completion theorem
of Edidin-Graham \cite{EG-Duke}.
We fix an algebraically closed field $k$ of characteristic zero throughout this
section.

\subsection{Some technical results}\label{sec:tech*}
We collect some technical results that we shall use in the proof of the generalized
completion theorem.

\begin{lem}\label{lem:Derived-CRT}
  Let $E$ be a connective $E_\infty$ ring which is an $H_k$-algebra and let
  $R = \pi_0(E)$. Let
  $I \subset R$ be a finitely generated ideal, $M$ a Noetherian $R$-module and $N$
  a $k$-vector space. Consider the $R$-module $P = M \otimes_k N$. As an $E$-module
  via the canonical map $E \to H_R$, one has then a natural weak equivalence
  $(H_P)^{\compl}_I \simeq H_{(P^{\compl}_I)}$, where $P^{\compl}_I$ is the $I$-adic
  completion
  of $P$ as an $R$-module.
\end{lem}
\begin{proof}
  We let $Q = P^{\compl}_I$, $F = (H_P)^{\compl}_I$ and $I = (x_1, \ldots , x_n)$.
  We first observe that the canonical map
  $\gamma \colon H_{(\pi_0(F))} \to \tau_{\le 0} F$ is a weak equivalence since $F$ is
  connective by \propref{prop:Compln_D}(1).
We next look at the maps
  \begin{equation}\label{eqn:Derived-CRT-0}
    F \xrightarrow{\alpha} \tau_{\le 0} F \xrightarrow{\gamma^{-1}}  H_{(\pi_0(F))}
    \xrightarrow{\beta} H_Q,
  \end{equation}
  where $\alpha$ is the truncation map and $\beta$ is induced by applying the
  Eilenberg-Mac Lane functor to the canonical maps
  $\pi_0(F) = \pi_0((H_P)^{\compl}_I) \to (\pi_0(H_P))^{\compl}_I \cong Q$.
  In particular, all maps in ~\eqref{eqn:Derived-CRT-0} are natural in $E, M$
  and $N$. We need to show that $\alpha$ and $\beta$ are weak equivalences.  
Assume first that $I = (x)$ is a principal ideal.

To prove $\alpha$ is a weak
equivalence, it is equivalent to proving that $\pi_i(F) = 0$ for every $i > 0$.
To prove the latter claim, we look at the Milnor exact sequences
\begin{equation}\label{eqn:Derived-CRT-1}
  0 \to {\varprojlim}^1_n \pi_{i+1}({H_P}/(x^n)) \to \pi_i(F) \to \varprojlim_n
  \pi_i({H_P}/(x^n)) \to 0.
  \end{equation}
Using the exact sequence
\begin{equation}\label{eqn:Derived-CRT-2}
  0 \to {\pi_i(L)}/{(x^n)} \to \pi_i({L}/{(x^n)}) \to ~_{x^n} (\pi_{i-1}(L)) \to 0
\end{equation}
for an $E$-module $L$, we get $\pi_0({H_P}/(x^n)) = {P}/{(x^n)}$,
$\pi_1({H_P}/{(x^n)}) =  ~_{x^n} P$ and $\pi_i({H_P}/{(x^n)}) = 0$ for $i \neq 0,1$.
We thus get $\pi_1(F) = \varprojlim_n ~_{x^n} P$, an exact sequence
\begin{equation}\label{eqn:Derived-CRT-3}
  0 \to  {\varprojlim}^1_n  ~_{x^n} P \to  \pi_{0}(F) \xrightarrow{\beta}
  P^{\compl}_I \to 0
  \end{equation}
and $\pi_i(F) = 0$ for $i \neq 0,1$. 
To show that $\alpha$ and $\beta$ are
weak equivalences, it suffices therefore to show that
the pro-$R$-module $\{_{x^n} P\}$ (whose all bonding maps are multiplication by $x$)
is zero. Since $_{x^n} P \cong (_{x^n} M)\otimes_k N$, it is enough to show that
$\{_{x^n} M\}$ is zero. But this follows at once from our assumption that $M$ is a
Noetherian $R$-module so that one can find $n \gg 0$ such that
$_{x^{n+i}} M = ~_{x^n} M$ for all $i \ge 0$. 

To prove the general case, we write $I = J + (x)$, where $J$ is generated by
at most $n-1$ elements and proceed by induction on the number of generators of $I$.
By induction, we get $(H_P)^{\compl}_J \xrightarrow{\simeq} H_{(P^{\compl}_J)}$, and
the Milnor exact sequence tells us that the canonical map
$H_{(P^{\compl}_J)} \simeq H_{({\varprojlim}_i {P}/{J^i})} \to {\varprojlim}_i H_{({P}/{J^i})}$
  is a weak equivalence. It follows that $(H_P)^{\compl}_J \simeq
  {\varprojlim}_i H_{({P}/{J^i})}$.
We thus get
\begin{equation}\label{eqn:Derived-CRT-4}
(H_P)^{\compl}_I  \simeq \left[(H_P)^{\compl}_J\right]^{\compl}_{(x)} \simeq 
\left[{\varprojlim}_i H_{({P}/{J^i})}\right]^{\compl}_{(x)}
\simeq  {\varprojlim}_j \frac{{\varprojlim}_i H_{({P}/{J^i})}}{(x^j)}
\simeq {\varprojlim}_j {\varprojlim}_i \frac{H_{({P}/{J^i})}}{(x^j)},
\end{equation}
where the first weak equivalence is by \propref{prop:Compln_D}(6) and the last weak
equivalence is by the fact that the homotopy limit commutes with the
homotopy cofiber.

Combining ~\ref{eqn:Derived-CRT-4} with the Fubini theorem
(cf. \cite[Thm.~24.9 and \S~31.5]{Chacholski-Scherer}) and the cofinality theorem 
(cf. \cite[Thm.~6.12]{Dugger} and its proof which is valid for spectra) for
homotopy limits, we get
\begin{equation}\label{eqn:Derived-CRT-5}
(H_P)^{\compl}_I  \simeq
{\varprojlim}_j {\varprojlim}_i \frac{H_{({P}/{J^i})}}{(x^j)} \simeq
{\varprojlim}_{(i,j) \in \N \times \N} \frac{H_{({P}/{J^i})}}{(x^j)}
\simeq {\varprojlim}_{i} \frac{H_{({P}/{J^i})}}{(x^i)}.
\end{equation}
Another application of the Milnor sequence now yields an exact sequence
\begin{equation}\label{eqn:Derived-CRT-6}
  0 \to {\varprojlim}^1_i \pi_{n+1}\left(\frac{H_{({P}/{J^i})}}{(x^i)}\right) \to
  \pi_{n}(F) \to  {\varprojlim}_i
  \pi_{n}\left(\frac{H_{({P}/{J^i})}}{(x^i)}\right) \to 0.
  \end{equation}

On the other hand, applying ~\eqref{eqn:Derived-CRT-2} with $L = H_{({P}/{J^i})}$ gives
us 
\begin{equation}\label{eqn:Derived-CRT-7}
\pi_{i}\left(\frac{H_{({P}/{J^i})}}{(x^i)}\right) = \left\{ \begin{array}{ll}
     {P}/{(J^i, x^i)} & \mbox{if $i = 0$} \\
    ~_{x^i} ({P}/{J^i}) & \mbox{if $i = 1$} \\
    0 & \mbox{otherwise.}
  \end{array}
  \right.
  \end{equation}
Since the pro-$R$-module $\{_{x^i} ({P}/{J^j})\}_{i \ge 1}$ is zero for every $j \ge 1$
as we showed earlier, it follows at once that the pro-$R$-module
$\{_{x^i} ({P}/{J^i})\})\}_{i \ge 1}$ is also zero.
Using ~\eqref{eqn:Derived-CRT-6} and ~\eqref{eqn:Derived-CRT-7}, we get that
$\beta$ is a weak equivalence. We also get $\pi_i(F) = 0$ for $i \neq 0$.
Using the isomorphism $P^{\compl}_I \cong {\varprojlim}_{i}  {P}/{(J^i, x^i)}$, we
conclude from ~\eqref{eqn:Derived-CRT-6} and  ~\eqref{eqn:Derived-CRT-7}
that $(H_P)^{\compl}_I \simeq H_{(P^{\compl}_I)}$. This finishes the proof.
\end{proof}

\begin{lem}\label{lem:Torus-EG}
  Assume that $T$ is a split torus over $k$ and $I \subset R_k(T)$ is an ideal
  such that $\sqrt{I} = \fm_1 \cdots \fm_r$, where each $\fm_i$ is a maximal ideal
  of $R_k(T)$. For any $X \in \Spc^T_k$, the canonical map
  $\tau_X \colon K'_T(X,k)^{\compl}_I \to
  \stackrel{r}{\underset{i =1}\coprod} K'_T(X,k)^{\compl}_{\fm_i}$
  is then a weak equivalence.
\end{lem}
\begin{proof}
We can assume that $X$ is reduced.
We now apply Thomason's generic slice theorem to get a regular $T$-invariant affine
open subscheme $U \subset X$ and a diagonalizable subgroup $T' \subset T$ such
that $T'$ acts trivially and $T'' = T/{T'}$ acts freely on $U$ so that
$W := U/{T''}$ is a regular affine scheme and $U \cong W \times T''$ as an object
    of $\Sch^T_k$. It follows from \lemref{lem:Morita*} that
    $K'_T(U,k) \simeq K'_{T'}(W,k)$. Using Noetherian induction and the
    localization theorem, it suffices to show that $\tau_U$ is a weak equivalence.
    Since the canonical map $K'_T(U,k) \to \holim_n \tau_{\le n} K'_T(U,k)$ is
    a weak equivalence and the derived completion commutes with limits, it suffices
    to show that the map $\tau_{U,n} \colon (\tau_{\le n} K'_T(U,k))^{\compl}_I \to
    \stackrel{r}{\underset{i =1}\coprod} (\tau_{\le n} K'_T(U, k))^{\compl}_{\fm_i}$ is
    a weak equivalence for every $n \ge 0$.

Using the homotopy fiber sequences
    \[
    (H_{G_n})^{\compl}_J \to (\tau_{\le n} K'_T(U,k))^{\compl}_J \to
    (\tau_{\le n-1} K'_T(U,k))^{\compl}_J,
    \]
    where $G_n = \pi_n(K'_T(U,k))$ and $J \subset R_k(T)$ any ideal, we are finally
    reduced to showing that the canonical map
\begin{equation}\label{eqn:Torus-EG-0}
  \tau^n_U \colon  (H_{G_n})^{\compl}_I \to  \stackrel{r}{\underset{i =1}\coprod}
  (H_{G_n})^{\compl}_{\fm_i}
\end{equation}
is a weak equivalence for every $n \ge 0$.

To show that ~\eqref{eqn:Torus-EG-0} is a weak equivalence,
we note that $G_n \cong K'_n(W,k) \otimes_k R_k(T')$ by \corref{cor:Twist-3}
(see also \cite[Lem.~5.6]{Thomason-Inv}) and $R_k(T')$ is a Noetherian
$R_k(T)$-module.
\lemref{lem:Derived-CRT} (with $M = R_k(T')$ and $N = K'_n(W,k)$) therefore says 
that $\tau^n_U$ in ~\eqref{eqn:Torus-EG-0} is a weak equivalence if and only if
the map $H_{(G_n)^{\compl}_I} \to \stackrel{r}{\underset{i =1}\coprod}
H_{(G_n)^{\compl}_{\fm_i}}$ is a weak equivalence. Equivalently, $\tau^n_U$ is a weak
equivalence if and only if the map of $R_k(T)$-modules
$(G_n)^{\compl}_{I} \to \ \stackrel{r}{\underset{i =1}\bigoplus} (G_n)^{\compl}_{\fm_i}$
is an isomorphism. But the latter assertion is clear from the definition of the
completion of modules over a commutative ring.
\end{proof}

We now fix $G \in \Grp_k$ and $g \in G_s(k)$.
We let $Z_g$ denote the centralizer of $g$ in $G$. Let $\Psi \subset G(k)$
denote the conjugacy class of $g$. This hypothesis will remain in place
until we reach \S~\ref{sec:Com-max-K}

\begin{lem}\label{lem:Null-Compln}
  Assume that $g \in Z(G)$ and $X \in \Spc^G_k$ such that $X^g = \emptyset$. Then
  $K'_G(X,k)^{\compl}_{\fm_g}$ is weakly contractible.
\end{lem}
\begin{proof}
  This is easily proved by a routine modification of the proof of
  \cite[Thm.~5.1]{EG-Duke}. We omit the details.
  \end{proof}
  
\begin{cor}\label{cor:Localization-0}
  Assume that $g \in Z(G)$ and $X \in \Spc^G_k$. Then the push-forward map
  $\lambda_X \colon K'_G(X^g,k)^{\compl}_{\fm_g} \to K'_G(X,k)^{\compl}_{\fm_g}$ is a weak
  equivalence.
\end{cor}
\begin{proof}
  Using the localization theorem, this is equivalent to saying that
  $K'_G(X \setminus X^g,k)^{\compl}_{\fm_g}$ is weakly contractible. But this follows
  from \lemref{lem:Null-Compln}.
\end{proof}

\subsection{The derived nonabelian completion theorem}\label{sec:EG-compln}
We first consider a special case. We assume that $G$
is a connected reductive $k$-group such that $Z_g$ is also connected reductive
having the same rank as that of $G$. We also assume that $G$ and $Z_g$ have simply
connected commutator subgroups. We choose a common maximal torus $T$ of $G$ and
$Z_g$ which contains $g$. We let $W$ (resp. $W'$) be the Weyl group of $G$
(resp. $Z_g$) with respect to $T$. We then have a canonical inclusion of groups
$W' = {N_{Z_g}(T)}/T \inj {N_G(T)}/T = W$. By \cite[Lem.~4.6]{EG-Duke},
we have $\Psi \cap T = \{t_0 = g, t_1, \ldots , t_{r-1}\}$ on which $W$ acts
transitively and the stabilizer subgroup of $g$ is $W'$ via the above inclusion. 
We let $X \in \Spc^G_k$.

It is easy to see that the first of the two canonical projection maps of stacks
$[X/T] \to [X/{N_G(T)}] \to [X/G]$ is a $W$-torsor. In particular, $[X/T]$ is
a stack with $W$-action. It follows that $K'_T(X) \simeq K'([X/T])$ is a spectrum
with $W$-action and the restriction map $\rho^G_T \colon K'_G(X) \to K'_T(X)$
is a $W$-equivariant morphism of spectra with respect to the trivial
$W$-action on $K'_G(X)$.
We let ${\rm Nm}^G_T \colon K'_T(X) \to K'_T(X)^{hW}$
denote the norm map, where $A^{hH}$ denotes the homotopy fixed point
spectrum if $H$ is a finite group and  $A$ is a spectrum with $H$-action
(cf. \cite[\S~I.1]{Nikolaus-Scholze}).

For any field $K$, let $H_K$-mod denote the category of $H_K$-module spectra.
One knows that the homotopy category of $H_K$-mod is canonically equivalent to the
derived category $\sD_K$ of $K$-vector spaces
(cf. \cite[Thm.~5.16]{Schwede-Shipley}). Since the latter category is semisimple,
we get that the homotopy category of $H_K$-mod is semisimple.

\begin{lem}\label{lem:G-T-map}
The change of groups map induces a weak equivalence of $K(BG, k)$-module spectra
\begin{equation}\label{eqn:GLn-EG}
  \rho^G_T \colon  K'_G(X,k) \xrightarrow{\simeq} {K'_T(X,k)}^{hW}.
  \end{equation}
\end{lem}
\begin{proof}
  By \cite[Prop.~ A6]{Krishna-Adv}, $\rho^G_T$ induces an isomorphism
  $K'_i([X/G],k)) \otimes_{R_k(G)} R_k(T) \xrightarrow{\cong} K'_i([X/T],k)$ of
  $K_*(BG,k)$-modules for every $i \in \Z$. Using the isomorphism
  $R_k(G) \xrightarrow{\cong} R_k(T)^W$ (cf. ~\eqref{eqn:char1})
  and \cite[Lem.~3.1]{Krishna-Adv},
  we deduce that $\rho^G_T$ induces an isomorphism
  \begin{equation}\label{eqn:G-T-map-0}
    K'_i([X/G],k) \xrightarrow{\cong} (K'_i([X/T],k))^W
  \end{equation}
  for every $i \in \Z$, where $A^H$ denotes the fixed point subgroup of an abelian
  group $A$ equipped with the action of a finite group $H$.
  On the other hand, the homotopy fixed point spectral sequence
  (cf. \cite[Thm.~17.5.4]{Hill}) shows that the canonical map
  $\pi_*({K'([X/T],k)}^{hW}) \xrightarrow{\cong} (K'_*([X/T],k))^W$ is an
  isomorphism. We deduce
  that the canonical maps $K'_G(X,k) \leftarrow {K'_G(X,k)}^{hW} \to
  {K'_T(X,k)}^{hW}$ are weak equivalences. 
\end{proof}

We let $\fm^T_\Psi = \fm_\Psi R_k(T) \subset  R_k(T)$.
Since the homotopy fixed point functor is a homotopy limit
(cf. \cite[\S~I.1]{Nikolaus-Scholze}), it commutes with homotopy fibers
(hence with homotopy cofibers) and homotopy limits.
Using this fact and ~\eqref{eqn:I-comp-0}, we deduce from
\propref{prop:Compln_D}(5) that ~\eqref{eqn:GLn-EG} induces a
weak equivalence 
\begin{equation}\label{eqn:GLn-EG-0}
  \rho^G_T \colon  K'_G(X,k)^{\compl}_{\fm_\Psi} \xrightarrow{\simeq}
      {(K'_T(X,k)^{\compl}_{\fm^T_\Psi})}^{hW}.
  \end{equation}
A similar argument shows that the map
$\rho^{Z_g}_T \colon  K'_{Z_g}(X,k)^{\compl}_{\fm_g} \xrightarrow{\simeq}
{(K'_T(X,k)^{\compl}_{\fm^T_g})}^{hW'}$ is also a weak equivalence.

We now look at the commutative diagram of $H_k$-module spectra
\begin{equation}\label{eqn:GLn-EG-1}
  \xymatrix@C1pc{
    \pi_*\left(K'_G(X,k)^{\compl}_{\fm_\Psi}\right) \ar[r]^-{\rho^G_T}
    \ar[d]_-{\rho^G_{Z_g}} &
       {\left({\pi_*\left(K'_T(X,k)^{\compl}_{\fm^T_\Psi}\right)}\right)}^{W}
       \ar[r]^-{\simeq}
       &   
   {\left(\stackrel{r-1}{\underset{i= 0}\bigoplus}
     \pi_*\left(K'_T(X,k)^{\compl}_{\fm_{t_i}}\right)\right)}^{W}
   \ar[d]^-{p_0} \\
   \pi_*\left(K'_{Z_g}(X,k)^{\compl}_{\fm_g}\right) \ar[rr]^-{\rho^{Z_g}_T} & &
   \left(\pi_*\left(K'_T(X,k)^{\compl}_{\fm_{g}}\right)\right)^{W'}}
\end{equation}
where $p_0$ is induced by the projection to the component
corresponding to $i = 0$.

The right horizontal arrow on the top level is an isomorphism by
\lemref{lem:Torus-EG}.
The other horizontal arrows are isomorphisms because
$\pi_*((-)^{hW}) \cong (\pi_*(-))^W$ as we observed above.  Now we recall that
$W$ acts transitively on the set $\Psi \cap T$ by permuting its elements.
In particular, it acts on the sum $\stackrel{r-1}{\underset{i= 0}\bigoplus}
\pi_*(K'_T(X,k)^{\compl}_{\fm_{t_i}})$ by permuting its factors. Since the stabilizer of
$g$ is $W'$, it follows from \cite[Lem.~3.2]{Vistoli} that
$p_0$ is an isomorphism. We conclude that $\rho^G_{Z_g}$ is also
an isomorphism. Equivalently, the map $\rho^G_{Z_g} \colon K'_G(X,k)^{\compl}_{\fm_\Psi} \to
K'_{Z_g}(X,k)^{\compl}_{\fm_g}$ is a weak equivalence. We have therefore proven the
following.

\begin{lem}\label{lem:GLn-EG-2}
Assume that $G$ is a connected reductive
$k$-group such that $Z_g$ is also connected reductive whose rank is same as
that of $G$. Assume further that $G$ and $Z_g$ have simply connected commutator
subgroups. Then the change of groups map induces a weak equivalence
\[
\rho^G_{Z_g} \colon K'_G(X,k)^{\compl}_{\fm_{\Psi}} \xrightarrow{\simeq}
K'_{Z_g}(X,k)^{\compl}_{\fm_g}.
\]
\end{lem}

We now return to the general case.
Let $X \in \Spc^G_k$ and let $\Psi^g_{X}$ denote the composition of the
three maps
\begin{equation}\label{eqn:Completion-map-g}
  K'_G(X,k)^{\compl}_{\fm_{\Psi}} \xrightarrow{\alpha_X}  K'_{Z_g}(X,k)^{\compl}_{\fm_{\Psi}}
  \xrightarrow{\beta_X} K'_{Z_g}(X,k)^{\compl}_{\fm_g} \xrightarrow{(\lambda_X)^{-1}}
  K'_{Z_g}(X^g,k)^{\compl}_{\fm_g},
  \end{equation}
where $\alpha_X$ is the change of groups map, $\beta_X$ is the completion at the
maximal ideal $\fm_g \in R_k(Z_g)$ and $\lambda_X$ is the weak equivalence of
\corref{cor:Localization-0}.

We shall use the following property of $\Psi^g_{X}$.

\begin{lem}\label{lem:Completion-map-g-0}
Let $\iota \colon Z \to X$ (resp. $j \colon U \to X$) be
a  proper (resp. isovariant flat) map in $\Spc^G_k$. Then the diagram
  \begin{equation}\label{eqn:Completion-map-g-1}
    \xymatrix@C.8pc{
      K'_G(Z,k)^{\compl}_{\fm_{\Psi}} \ar[r]^-{\iota_*} \ar[d]_-{\Psi^g_{Z}} &
      K'_G(X,k)^{\compl}_{\fm_{\Psi}} \ar[r]^-{j^*} \ar[d]^-{\Psi^g_{X}} &
      K'_G(U,k)^{\compl}_{\fm_{\Psi}} \ar[d]^-{\Psi^g_{U}} \\
      K'_{Z_g}(Z^g,k)^{\compl}_{\fm_g} \ar[r]^-{\iota_*} & 
      K'_{Z_g}(X^g,k)^{\compl}_{\fm_g} \ar[r]^-{j^*} &  K'_{Z_g}(U^g,k)^{\compl}_{\fm_g}}
  \end{equation}
  is commutative.
\end{lem}
\begin{proof}
  Since $\alpha_X$ is the change of groups map, it clearly commutes with $\iota_*$
  and $j^*$. Since $\iota_*$ and $j^*$ are $K(BG,k)$-linear, they commute with
  $\beta_X$ because the latter is a completion map. Finally, $\lambda_X$ commutes
  with $\iota_*$ because both are push-forward maps. It commutes with $j^*$ by
  \cite[Prop.~3.8]{Khan-JJM} (see also \cite[Prop.~3.18]{TT}) using
  the isovariance of $j$ (so that $U^g \cong X^g \times_X U$ for all $g \in G$).
  It follows that $(\lambda_X)^{-1}$ commutes with $\iota_*$ and $j^*$. 
\end{proof}

We now prove the main result of \S~\ref{sec:EG-compln}.
An analogue of this for the homotopy groups of
$K'(G,k)$ is the main result of Edidin-Graham \cite{EG-Duke}.
Note however that one can not derive the theorem below from that of
op. cit. because the derived completion does not commute with
homotopy groups. This was in fact the main difficulty
Thomason \cite{Thomason-Duke-1}
encountered while extending the topological
Atiyah-Segal completion theorem to equivariant algebraic $K$-theory.

\begin{thm}\label{thm:Max-Main}
  For any $X \in \Spc^G_k$, there is a natural weak equivalence of spectra
  \begin{equation}\label{eqn:Max-Main-0}
    \Psi^g_{X} \colon K'_G(X,k)^{\compl}_{\fm_\Psi} \xrightarrow{\simeq}
    K'_{Z_g}(X^g, k)^{\compl}_{\fm_g}
  \end{equation}
  which commutes with proper push-forward and isovariant flat pull-back maps in
  $K'$-theory.
  \end{thm}
\begin{proof}
  By \corref{cor:Localization-0}, it suffices to show that $\beta_X \circ \alpha_X$
  is a weak equivalence to prove the first part of the theorem.

By the nil-invariance of $K'$-theory, we can assume that $X$ is a reduced space. 
We now use \lemref{lem:Open-quasi-proj} to choose a $G$-invariant open
subspace $U \subset X$ which is a smooth and $G$-quasi-projective $k$-scheme.
We let $Z = X \setminus U$ with the reduced structure. We let
  $\iota \colon Z \inj X$ and $j \colon U \inj X$ be the inclusions and consider
  the diagram of homotopy fiber sequences
  \begin{equation}\label{eqn:Max-Main-3}
    \xymatrix@C.8pc{
      K'_G(Z,k)^{\compl}_{\fm_\Psi} \ar[r]^-{\iota_*} \ar[d]_-{\Psi^g_{Z}} &
        K'_G(X,k)^{\compl}_{\fm_\Psi} \ar[r]^-{j^*} \ar[d]^-{\Psi^g_{X}} &
          K'_G(U,k)^{\compl}_{\fm_\Psi} \ar[d]^-{\Psi^g_{U}} \\
          K'_{Z_g}(Z^g, k)^{\compl}_{\fm_g} \ar[r]^-{\iota_*} &
          K'_{Z_g}(X^g, k)^{\compl}_{\fm_g} \ar[r]^-{j^*} &
            K'_{Z_g}(U^g, k)^{\compl}_{\fm_g}.}
  \end{equation}
  This diagram is commutative by \lemref{lem:Completion-map-g-0}.
If we can show that $\Psi^g_{U}$ is a weak equivalence, it will follow by a
  Noetherian induction that $\Psi^g_Z$ is a weak equivalence. In particular,
  $\Psi^g_X$ is weak equivalence. It remains therefore to prove the theorem
  when $X$ is smooth and $G$-quasi-projective.

If $G$ is either a torus or satisfies the conditions of \lemref{lem:GLn-EG-2},
  then $\Psi^g_X$ is a weak equivalence by \corref{cor:Localization-0} and
  \lemref{lem:GLn-EG-2}. To prove the general case, we first use
  \cite[Prop.~2.8]{EG-Adv} to find an embedding $G \inj G'$ such that
  $G'$ satisfies the conditions of \lemref{lem:GLn-EG-2} and
  $\Psi = \Psi' \cap G$, where $\Psi'$ is the conjugacy class of $g$ in $G'$.
  We let $Y = X \stackrel{G}{\times} G'$ denote the associated Morita scheme and let
    $\delta \colon Y^g \to Y$ be the inclusion of the fixed point locus for the
    $G'$-action on $Y$. If we let $Z'_g$ denote the centralizer of $g$ in $G'$, then
    the canonical map $\gamma \colon X^g \stackrel{Z_g}{\times} Z'_g \to Y^g$ is a
      $Z'_g$-equivariant isomorphism of smooth schemes with $Z'_g$-action by
      \cite[Lem.~5.6]{EG-Duke}.

We now look at the diagram of (smooth) quotient stacks
      \begin{equation}\label{eqn:Max-Main-4}
        \xymatrix@C.8pc{
          [{X^g}/{Z_g}] \ar[r]^-{\alpha_1} \ar[d] \ar[dr]^-{\cong}_-{\beta_1} &
          [X/{Z_g}] \ar[r]^-{\alpha_2} \ar[dr]^-{\beta_2} & [X/G] \ar[r]
          \ar[dr]^-{\cong}_-{\beta_3} & [{(X \stackrel{G}{\times} G')}/G] \ar[d] \\
             [{(X^g \stackrel{Z_g}{\times} Z'_g)}/{Z_g}] \ar[r] &
             [{(X^g \stackrel{Z_g}{\times} Z'_g)}/{Z'_g}] \ar[r]^-{\alpha'_1}
             \ar[d]_-{\gamma} &
                [{(X \stackrel{G}{\times} G')}/{Z'_g}] \ar[r]^-{\alpha'_2}
                \ar[d]^-{\cong} &
             [{(X \stackrel{G}{\times} G')}/{G'}] \ar[d]^-{\cong} \\
             & [{Y^g}/{Z'_g}] \ar[r] & [Y/{Z'_g}] \ar[r] & [Y/{G'}].}
        \end{equation}

      Each of the arrows in this diagram is either a regular closed immersion
or an isomorphism or a quotient map.
To analyze the induced pull-back maps between various $K'$-theory spectra
and their completions, we let  $\alpha^*_3 \colon K'_{Z_g}(X^g, k)^{\compl}_{\fm_{\Psi'}} \to
K'_{Z_g}(X^g, k)^{\compl}_{\fm_{g}}$ and
$\alpha'^*_3 \colon  K'_{Z'_g}(Y^g, k)^{\compl}_{\fm_{\Psi'}}
\to K'_{Z'_g}(Y^g, k)^{\compl}_{\fm_{g}}$
denote the maps induced by taking completions at $\fm_g$.
We let $\phi \colon X^g \inj X$ and $\phi' \colon Y^g \inj Y$ be the inclusions.

We first note that, thanks to the universal property of completion,
      the composition
    $\alpha^*_3 \circ \alpha^*_1$ on the $K'$-theory spectra is homotopic to the
    composition $K'_{Z_g}(X, k)^{\compl}_{\fm_{\Psi'}} \to
    K'_{Z_g}(X, k)^{\compl}_{\fm_{g}} \to K'_{Z_g}(X^g,k)^{\compl}_{\fm_g}$.
    Similarly, $\alpha'^*_3 \circ \alpha'^*_1$ is homotopic to
    the composition $K'_{Z'_g}(Y,k)^{\compl}_{\fm_{\Psi'}} \to
    K'_{Z'_g}(Y,k)^{\compl}_{\fm_{g}} \to K'_{Z'_g}(Y^g,k)^{\compl}_{\fm_{g}}$.
    We next note that the canonical maps $K'_G(X,k)^{\compl}_{\fm_{\Psi'}} \to
    K'_G(X,k)^{\compl}_{\fm_{\Psi}}$ and 
    $K'_{Z_g}(X, k)^{\compl}_{\fm_{\Psi'}} \to K'_{Z_g}(X, k)^{\compl}_{\fm_{\Psi}}$
are weak equivalences by our choice of the
embedding $G \inj G'$. Furthermore, the pull-back
$\alpha^*_2$ is the same as the change of groups map $\alpha_X$ and
$\alpha^*_3 \circ \alpha^*_1$ is the same as the composition of
$\phi^*$ with the $\fm_g$-completion map $\beta_X$.
Similarly, we have $\beta_Y = \alpha'^*_2$ and
$\phi'^* = \alpha'^*_3 \circ \alpha'^*_1$.

By considering the pull-back maps between the completions of $K'$-theory spectra
induced by ~\eqref{eqn:Max-Main-4}, we therefore obtain a commutative diagram
\begin{equation}\label{eqn:Max-Main-6}
        \xymatrix@C.8pc{
          K'_G(X,k)^{\compl}_{\fm_{\Psi}} \ar[r]^-{\alpha_X} &
          K'_{Z_g}(X, k)^{\compl}_{\fm_{\Psi}} \ar[r]^-{\beta_X} &
          K'_{Z_g}(X, k)^{\compl}_{\fm_{g}} \ar[r]^-{\phi^*} &
          K'_{Z_g}(X^g, k)^{\compl}_{\fm_{g}}  \\
          K'_{G'}(Y,k)^{\compl}_{\fm_{\Psi'}} \ar[r]^-{\alpha_Y}
          \ar[u]^-{\beta^*_3}_-{\simeq} &
          K'_{Z'_g}(Y, k)^{\compl}_{\fm_{\Psi'}} \ar[r]^-{\beta_Y} \ar[u]_-{\beta^*_2} &
          K'_{Z'_g}(Y, k)^{\compl}_{\fm_{g}} \ar[r]^-{\phi'^*}
          \ar[u]_-{\beta^*_2} &
 K'_{Z'_g}(Y^g, k)^{\compl}_{\fm_{g}}. \ar[u]_-{\beta^*_1}^-{\simeq}}
\end{equation}

It is easily checked that $\beta_X \circ \alpha_X$
    (resp. $\beta_Y \circ \alpha_Y$) coincides with the map
    $\rho^G_{Z_g}(X)$ (resp. $\rho^{G'}_{Z'_g}(Y)$) in ~\eqref{eqn:GLn-EG-1}.
    In particular, $\beta_Y \circ \alpha_Y$ is a weak equivalence by
    \lemref{lem:GLn-EG-2}. Since $\beta^*_1$ (via the isomorphism $\gamma$)
    and $\beta^*_3$ are weak equivalences
    by \lemref{lem:Morita*}, we conclude from \lemref{lem:Max-Main-7} that
    $\beta_X \circ \alpha_X$ is a weak equivalence. The commutativity of
    $\Psi^g_X$ with proper push-forward and isovariant flat pull-back is
    shown in \lemref{lem:Completion-map-g-0}.
\end{proof}

\begin{lem}\label{lem:Max-Main-7}
  If $X \in \Sm^G_k$, then the pull-back map
  $\phi^* \colon K'_{Z_g}(X, k)^{\compl}_{\fm_{g}} \to K'_{Z_g}(X^g, k)^{\compl}_{\fm_{g}}$ is
  a weak equivalence.
\end{lem}
\begin{proof}
  By \propref{prop:Compln_D}(2), the composition 
  \[
  K'_{Z_g}(X^g, k)^{\compl}_{\fm_{g}} \xrightarrow{\phi_*}  K'_{Z_g}(X, k)^{\compl}_{\fm_{g}}
  \xrightarrow{\phi^*} K'_{Z_g}(X^g, k)^{\compl}_{\fm_{g}}
  \]
  is a morphism of $K([{X^g}/{Z_g}])_{\fm_g}$-module spectra.
  By \cite[Thm.~2.1]{VV-Inv}, this composition is multiplication by
  $\lambda_{-1}(N^\vee) \in K_0([{X^g}/{Z_g}])$, where $N$ is
  the equivariant normal bundle of the inclusion $X^g \inj X$. On the other hand,
  $\lambda_{-1}(N^\vee) \in (K_0([{X^g}/{Z_g}])_{\fm_g})^\times$ by
  \cite[Thm.~3.3]{EG-Adv}.
 It follows that $\phi^*$ will be
    a weak equivalence if and only if $\phi_*$ is. We now apply 
\corref{cor:Localization-0} to finish the proof.
\end{proof}

\subsection{The generalized completion theorem}\label{sec:Com-max-K}
We continue with the set-up of \S~\ref{sec:Max-com} (in particular, $k = \ov{k}$
and  $G \in \Grp_k$ is arbitrary).
Before we prove the generalizations of Theorems~\ref{thm:AS-Main} and
~\ref{thm:AS-KH} to all maximal ideals of $R_k(G)$,
we need to extend \thmref{thm:Max-Main} to $KH$-theory.

\begin{lem}\label{lem:KH-g-0}
  Assume that $G$ acts on a $k$-scheme $X$ with nice stabilizers and $g \in Z(G)$.
  Let $\iota \colon Y \inj X$ be the inclusion of a $G$-invariant closed subscheme
  such that $(X \setminus Y)^g = \emptyset$. Then the pull-back map
  $\iota^* \colon KH_G(X,k)^{\compl}_{\fm_g} \to KH_G(Y,k)^{\compl}_{\fm_g}$ is a weak
  equivalence.
\end{lem}
\begin{proof}
  We first observe that if $(X \setminus Y)^g = \emptyset$, then the same is also
  true for $(X_\red \setminus Y_\red)^g$. On the other hand, $KH$-theory is
  nil-invariant by \cite[Thm.~6.2]{Hoyois-Krishna}. We can thus assume that both $X$
  and $Y$ are reduced $k$-schemes such that $Y^g = X^g$.
  We shall now prove the lemma by induction on
  $\dim(X)$. If $\dim(X) = 0$, then $X$ and $Y$ must be smooth $k$-schemes
  and the claim follows from \lemref{lem:Max-Main-7}. We now assume
  $\dim(X) > 0$.

  We first consider the case when $X \in \Sm^G_k$. We can assume in this case that
  $X$ is $G$-connected, i.e., $X = GX'$, where $X'$ is a connected component of $X$.
  If $Y$ contains any connected component of $X$, then it must be the same as $X$
  in which case the lemma is trivial. Otherwise,
  $\dim(Y) < \dim(X)$. In particular, the pull-back map $KH_G(Y,k)^{\compl}_{\fm_g}
  \to KH_G(Y^g,k)^{\compl}_{\fm_g}$ is a weak equivalence by induction on $\dim(X)$.
  We now look at the commutative diagram of pull-back maps
  \begin{equation}\label{eqn:KH-g-1}
    \xymatrix@C1pc{
      KH_G(X,k)^{\compl}_{\fm_g} \ar[r]^-{\iota^*} \ar[d] & KH_G(Y,k)^{\compl}_{\fm_g}
      \ar[d] \\
      KH_G(X^g,k)^{\compl}_{\fm_g} \ar[r]^-{\simeq} & KH_G(Y^g,k)^{\compl}_{\fm_g}.}
    \end{equation}

We just showed that the right vertical arrow is a weak equivalence. The left
      vertical arrow is a weak equivalence by \lemref{lem:Max-Main-7} and the
      bottom horizontal arrow is a weak equivalence because $Y^g = X^g$.
      It follows that $\iota^*$ is a weak equivalence.

We next assume that $X$ is singular and let $Z \subset X$ be its singular locus
with the reduced closed subscheme structure.
We also assume that $Y \not\subset Z$.
We choose a $G$-equivariant resolution of singularities
$p \colon X' \to X$ and let $W = Z \times_X X'$. We then get the $G$-equivariant
      abstract blow-up squares
      \begin{equation}\label{eqn:KH-g-2}
        \xymatrix@C1pc{
          W \ar[r]^-{u'} \ar[d]_-{p'} & X' \ar[d]^-{p} & & W' \ar[r]^-{v'}
          \ar[d]_-{q'} & Y' \ar[d]^-{q} \\
          Z \ar[r]^-{u} & X & & Z' \ar[r]^-{v} & Y,}
      \end{equation}
      where $Z' = Z \times_X Y, \ W' = Z' \times_X X'$ and $Y' = Y \times_X X'$.

By  \cite[Thm.~6.2]{Hoyois-Krishna}, we get a commutative diagram of homotopy
      fiber sequences
\begin{equation}\label{eqn:KH-g-3}
        \xymatrix@C1pc{
          KH_G(X,k)^{\compl}_{\fm_g} \ar[r]^-{(p^*, u^*)}  \ar[d]_-{\iota^*} &
          KH_G(X',k)^{\compl}_{\fm_g} \coprod KH_G(Z,k)^{\compl}_{\fm_g}
          \ar[d]^-{(\iota'^*, h^*)}
          \ar[r]^-{u'^* - p'^*} &  KH_G(W,k)^{\compl}_{\fm_g} \ar[d]^-{h'^*} \\
          KH_G(Y,k)^{\compl}_{\fm_g} \ar[r]^-{(q^*, v^*)} &
          KH_G(Y',k)^{\compl}_{\fm_g} \coprod KH_G(Z',k)^{\compl}_{\fm_g} 
          \ar[r]^-{v'^* - q'^*} &  KH_G(W',k)^{\compl}_{\fm_g},}
\end{equation}
where $\iota' \colon Y' \inj X', \ h \colon Z' \inj Z$ and $h' \colon W' \inj W$ are
the inclusions.

It is clear from our assumption and ~\eqref{eqn:KH-g-2} that $(X' \setminus Y')^g =
(Z \setminus Z')^g = (W \setminus W')^g = \emptyset$.
      As $X' \in \Sm^G_k$, it follows that the pull-back map
      $KH_G(X',k)^{\compl}_{\fm_g} \xrightarrow{\iota'^*} KH_G(Y',k)^{\compl}_{\fm_g}$ is a
      weak equivalence. The maps $h^*$ and $h'^*$ are weak equivalences by
      induction on the dimension. It follows that $\iota^*$ is also a weak
      equivalence.

      If $Y \subset Z$, then we must have $(X \setminus Z)^g = \emptyset$. In
      particular, $(X' \setminus W)^g = \emptyset$. Since $X' \in \Sm^G_k$, it
      follows that the map $u'^*$ in the top row of ~\eqref{eqn:KH-g-3} is a weak
      equivalence. But this implies that the map  $u^*$ in
      \[
      KH_G(X,k)^{\compl}_{\fm_g} \xrightarrow{u^*} KH_G(Z,k)^{\compl}_{\fm_g}
      \xrightarrow{w^*} KH_G(Y,k)^{\compl}_{\fm_g}
      \]
      is a weak equivalence, where $w \colon Y \inj Z$ is the inclusion.
      Since $(Z \setminus Y)^g = \emptyset$ and $\dim(Z) < \dim(X)$, we see that
      $w^*$ is a weak equivalence. We conclude that $\iota^* = w^* \circ u^*$ is a
      weak equivalence. This finishes the proof.
\end{proof}

The following result extends \thmref{thm:Max-Main} to $KH$-theory.

\begin{cor}\label{cor:KH-g}
    If $X \in \Sch^G_k$ such that $G$ acts on $X$ with nice
  stabilizers, then there is a natural weak equivalence of spectra
  \begin{equation}\label{eqn:Max-Main-1}
    \Theta^g_{X} \colon KH_G(X,k)^{\compl}_{\fm_{\Psi}} \xrightarrow{\simeq}
    KH_{Z_g}(X^g, k)^{\compl}_{\fm_g}
  \end{equation}
  which commutes with arbitrary pull-back in $KH$-theory.
\end{cor}
  \begin{proof}
    Using \lemref{lem:KH-g-0}, we only need to show that the composite
    $KH_G(X,k)^{\compl}_{\fm_{\Psi}} \xrightarrow{\alpha_X} 
    KH_{Z_g}(X,k)^{\compl}_{\fm_{\Psi}} \xrightarrow{\beta_X}
    KH_{Z_g}(X,k)^{\compl}_{\fm_{g}}$
    is a weak equivalence.
 When $X \in \Sm^G_k$, this follows from \thmref{thm:Max-Main}.
 In general, we conclude the proof by choosing a $G$-equivariant resolution of
 singularities of $X$ as in the proof of \lemref{lem:KH-g-0}, using induction on
 $\dim(X)$ and the homotopy fiber sequence of the top row in ~\eqref{eqn:KH-g-3}.
\end{proof}

We can now prove the main result \S~\ref{sec:Max-com}.

\begin{thm}\label{thm:Gen-AS-max}
  Let $k$ be an algebraically closed field of characteristic zero.
  Let $G \in \Grp_k$ and $g \in G_s(k)$.
    Let $\Psi \subset G(k)$ be the conjugacy class of $g$.
    For any $X \in \Spc^G_k$, we then have a natural weak equivalence of spectra
    \begin{equation}\label{eqn:Gen-AS-max-0}
      \Phi^g_X \colon K'_G(X,k)^{\compl}_{\fm_\Psi} \xrightarrow{\simeq}
      K'((X^g)_{Z_g}, k),
    \end{equation}
    which commutes with proper push-forward and isovariant flat pull-back maps in
    $K'$-theory.
    If $G$ acts on $X$ with nice stabilizers, then there is a natural
    weak equivalence of spectra
    \begin{equation}\label{eqn:Gen-AS-max-00}
      \vartheta^g_X \colon KH_G(X,k)^{\compl}_{\fm_\Psi} \xrightarrow{\simeq}
      KH((X^g)_{Z_g}, k),
    \end{equation}
    which commutes with arbitrary pull-back.
  \end{thm}
\begin{proof}
    We define $\Phi^g_X$ to be the composition
    \begin{equation}\label{eqn:Gen-AS-max-1}
      K'_G(X,k)^{\compl}_{\fm_{\Psi}} \xrightarrow{\Psi^g_X}
      K'_{Z_g}(X^g, k)^{\compl}_{\fm_g} \xrightarrow{t_{g^{-1}}}
      K'_{Z_g}(X^g, k)^{\compl}_{I_{Z_g}} \xrightarrow{\Phi_{[{X^g}/{Z_g}]}}
        K'((X^g)_{Z_g},k);
    \end{equation}
    and $\vartheta^g_X$ to be the composition
     \begin{equation}\label{eqn:Gen-AS-max-11}
      KH_G(X,k)^{\compl}_{\fm_{\Psi}} \xrightarrow{\Theta^g_X}
      KH_{Z_g}(X^g, k)^{\compl}_{\fm_g} \xrightarrow{t_{g^{-1}}}
      KH_{Z_g}(X^g, k)^{\compl}_{I_{Z_g}} \xrightarrow{\vartheta_{[{X^g}/{Z_g}]}}
        KH((X^g)_{Z_g},k).
     \end{equation}
     We now apply
     Propositions~\ref{prop:Twist-BP}, ~\ref{prop:Twist-local-compl},
     ~\ref{prop:AS-map-0},
     Theorems~\ref{thm:AS-Main}, ~\ref{thm:AS-KH}, ~\ref{thm:Max-Main} and
     \corref{cor:KH-g} to finish the proof.
\end{proof}

\section{Borel $K$-theory via $\A^1$-homotopy theory}\label{sec:Borel-hom}
We fix a field $k$, a $k$-group $G$ and $X \in \Spc^G_k$. We begin by describing
Thomason's completion problem which we wish to study next.

\subsection{The simplicial space $X^\bullet_G$}\label{sec:Problem*}
The $G$-action on $X$ gives rise to a simplicial $k$-space 
\begin{equation}\label{eqn:Bar-BG}
X_G^{\bullet}  := \left(\cdots 
\stackrel{\rightarrow}{\underset{\rightarrow}\rightrightarrows}
G \times G \times X 
\stackrel{\rightarrow}{\underset{\rightarrow}\to} G \times X
\rightrightarrows X \right)
\end{equation}
whose face maps $p^i_n \colon G^n \times X \to G^{(n-1)} \times X$ (for $n \ge 1$)
are given by
\begin{equation}\label{eqn:Bar-EG0}
p^i_n\left(g_1, \cdots , g_n, x\right) = \left\{
\begin{array}{ll}
\left(g_2, \cdots , g_n, g_1 x\right) & \mbox{if $i = 0$} \\
\left(g_1, \cdots, g_{i-1}, g_{i+1} g_i, g_{i+2}, \cdots, g_n, x \right) &
\mbox{if $0 < i < n$} \\
\left(g_1, \cdots , g_{n-1}, x \right) & \mbox{if $i = n$}.
\end{array}
\right .
\end{equation}
The degeneracy maps $s^i_n \colon G^n \times X \to G^{(n+1)} \times X$ (for $n \ge 0$)
are given by $s^i_n(g_0, \cdots ,g_{n-1},x) =
(g_0, \ldots g_{i-1}, e, g_i, \ldots , g_{n-1}, x)$.
In the sequel, we shall call $X^\bullet_G$ {\sl{the bar construction}} associated to
the $G$-action on $X$. We let $B^{\bullet}_G =  {(\pt)}_G^{\bullet}$.

One notes that $X^\bullet_G$ is the \u{C}ech nerve of the quotient map
$\pi \colon X \to [X/G]$. In particular, the composition
$\pi \circ p^{i_1}_1 \circ \cdots \circ p^{i_n}_n \colon G^n \times X \to [X/G]$
does not depend on the choice of $(i_1, \ldots , i_n)$ for any given $n \ge 1$. We let
$\wt{\pi}_n \colon G^n \times X \to [X/G]$ denote this composition.
The collection $\{\wt{\pi}_n\}_{n \ge 0}$ defines a
projection map $\wt{\pi} \colon X^\bullet_G \to [X/G]$
from the simplicial space $X^\bullet_G$ to the stack $[X/G]$.
This is a morphism between simplicial algebraic stacks over $k$ if we consider
$[X/G]$ as a constant simplicial stack. It is clear that
each $\wt{\pi}_n$ is a smooth map.

Since the face maps of $X^\bullet_G$ are smooth and the degeneracy maps
are regular closed immersions, we get the pull-back maps of cosimplicial spectra
$\wt{\pi}^* \colon K'_G(X) \simeq K'([X/G]) \to \{[n] \mapsto K'(G^n \times X)\}$.
Passing to the homotopy limits, we get the pull-back maps
\begin{equation}\label{eqn:Simplicial-sheaf}
  \wt{\pi}^* \colon K'_G(X) \to K'(X^\bullet_G) := \holim_{n}
  K'(G^n {\times} X); \ \ \ \wt{\pi}^* \colon K'_G(X,k) \to K'(X^\bullet_G,k).
  \end{equation}
We similarly get the pull-back maps $\wt{\pi}^* \colon K_G(X) \to K(X^\bullet_G)$ and
$\wt{\pi}^* \colon KH_G(X) \to KH(X^\bullet_G)$.

The map $\wt{\pi}^*$ can also be viewed as follows. Given a quasi-coherent sheaf $\sF$
on $[X/G]$, the pull-back $\wt{\pi}^*(\sF)$
defines a quasi-coherent sheaf $\sF^\bullet$ on the simplicial space $X^\bullet_G$
(consult \cite[\S~1.3]{Thomason-Orange} for the definition of quasi-coherent sheaves
on simplicial spaces). In particular, $\sF_0 = \pi^*(\sF)$ is simply the sheaf $\sF$,
with the $G$-action ignored. If $\sF$ is coherent (resp. locally free), so is
$\sF^\bullet$. We thus get an exact functor
$\wt{\pi}^* \colon {\Qcoh}^G_X \to {\Qcoh}_{X^\bullet_G}$ which preserves coherent and
locally free sheaves. In particular, it preserves perfect complexes and
pseudo-coherent
complexes. The induced map between the homotopy limits of the
$K'$-theory spectra is the map $\wt{\pi}^*$ of ~\eqref{eqn:Simplicial-sheaf}.

\vskip .3cm

{\bf{Thomason's completion problem:}}
In \cite[\S~0]{Thomason-Duke-1}, Thomason predicted that $\wt{\pi}^*$ would induce a
weak equivalence of spectra
\begin{equation}\label{eqn:Bar-EG-main-0}
  \wt{\pi}^* \colon K'_G(X)^{\compl}_{I_G} \xrightarrow{\simeq} K'(X_G^{\bullet}).
\end{equation}

Thomason showed that if $k$ is separably closed and
$K'_G(X)$ is replaced by the Bott inverted equivariant $K'$-theory with finite
coefficients prime to $\Char(k)$, this problem has a positive solution.
The problem is otherwise currently unsolved.

\vskip .3cm

Our goal in the next few sections is to solve Thomason's completion problem in
as many cases as possible. In this section, we define motivic Borel spaces,
identify them with $X_G^{\bullet}$ 
and use them to represent Borel equivariant $K$-theory of smooth schemes with
group action in Voevodsky's stable $\A^1$-homotopy category of smooth schemes over
$k$.

\vskip .3cm

We shall use the following 
notations throughout our discussion of Thomason's completion problem.

\begin{enumerate}
\item
${\rm Nis}/k :$ \ The Grothendieck site of smooth schemes over $k$ with
the Nisnevich topology. 
\item
${\rm Shv}\left({\rm Nis}/k\right) :$ \ 
The category of sheaves of sets on ${\rm Nis}/k$.
\item
$\Msp(k) :$ \ 
  The category $\Delta^{\rm op}{\rm Shv}\left({\rm Nis}/k\right)$ of sheaves of
  simplicial sets on ${\rm Nis}/k$ (equivalently, the category of simplicial objects
  in ${\rm Shv}\left({\rm Nis}/k\right)$). 
\item
$\sH(k) :$ \ The unstable $\A^1$-homotopy category of simplicial sheaves
on ${\rm Nis}/k$, as defined in \cite{Morel-Voevodsky}.
\item
$\sH_{\bullet}(k) :$ \ The unstable $\A^1$-homotopy category of pointed 
simplicial sheaves on ${\rm Nis}/k$, as defined in \cite{Morel-Voevodsky}.
\item
$\sS\sH(k) :$ \ The stable $\A^1$-homotopy category of pointed 
simplicial sheaves on ${\rm Nis}/k$, as defined (for instance) in \cite{Voev-ICM}. 
\end{enumerate}
The objects of $\Msp(k)$ will be called motivic spaces.

\subsection{Universal torsor over $B^{\bullet}_G$}\label{sec:Bar*}
Given $X \in \Spc^G_k$, we let $E^{\bullet}_G(X)$ denote the simplicial algebraic
space
\begin{equation}\label{eqn:Bar-EG}
E_G^{\bullet}(X)  := \ \left(\cdots 
\stackrel{\rightarrow}{\underset{\rightarrow}\rightrightarrows}
G \times G \times G \times X 
\stackrel{\rightarrow}{\underset{\rightarrow}\to} G \times G \times X 
\rightrightarrows G \times X\right) 
\end{equation}
with the face maps $d^i_n: G^{n+1} \times X  \to  G^{n} \times X $ ($0 \le i \le n$)
given by the projections
\[
d^i_n(g_0, \cdots , g_n, x) =
(g_0, \cdots , g_{i-1}, \wh{g_i}, g_{i+1}, \cdots , g_n, x),
\]
where $\wh{g_i}$ means that this coordinate is omitted. The degeneracy maps
$s^i_n \colon G^{n+1}\times X  \to  G^{n+2} \times X$ ($0 \le i \le n$) are the various
diagonals on $G$ and the identity map on $X$. In particular, the face maps of
$E_G^{\bullet}(X)$ are smooth and the degeneracy maps are regular closed
immersions. 

$G$ acts on $E_G^{\bullet}(X)$ via the levelwise action
$g\cdot \left(g_0, \cdots , g_n, x) =
(g_0g^{-1}, \cdots ,g_ng^{-1}, gx\right)$. It is easy to check that all the
face and degeneracy maps are $G$-equivariant with respect to this action
so that $E^{\bullet}_G(X)$ is a simplicial object in $\Spc^G_k$. Furthermore, this
action is levelwise free.

It is also easy to verify that there is a natural morphism of simplicial spaces
\begin{equation}\label{eqn:Bar-EG1}
\pi_X \colon E_G^{\bullet}(X) \to X_G^{\bullet} ;
\end{equation}
\[
\pi^n_X \left(g_0, \cdots , g_n, x \right) = \left(g_1g^{-1}_0, g_2 g^{-1}_1,
\cdots , g_n g^{-1}_{n-1}, g_0 x\right)
\]
which makes $X_G^{\bullet}$ the quotient of $E_G^{\bullet}(X)$
by the free $G$-action. Hence, $\pi_X$ is a $G$-torsor of
simplicial spaces. Letting $E^{\bullet}_G =  E_G^{\bullet}(\pt)$, we see that
$E_G^{\bullet}(X) \cong   E_G^{\bullet} {\times} X$ with the diagonal action of
$G$ and the canonical map $E_G^{\bullet} \stackrel{G}{\times} X \to X_G^{\bullet}$
is an isomorphism of simplicial algebraic spaces.
Let $\pi_0 \colon E^{\bullet}_G \to B^{\bullet}_G$ denote the quotient map.
We will see that $\pi_0$ represents the universal $G$-torsor over the classifying
space of $G$ in $\sH(k)$.

\subsection{Bar construction vs. Borel construction}\label{sec:Bar}
We shall assume now that $G$ is a special $k$-group.
Recall from \cite[Chap.~4]{Morel-Voevodsky} that if $G$ is a sheaf of groups on
${\rm Nis}/k$, then a left (resp. right) action of $G$ on a motivic space
$Y$ is a morphism $\mu \colon G \times Y \to Y$ (resp.
$\mu \colon Y \times G \to Y$) such that the usual diagrams commute.
For a $G$-action on $Y$, the quotient $Y/G$ is the motivic space such that
\begin{equation}\label{eqn:P-out} 
\xymatrix@C2pc{
G \times Y \ar[r]^-{\mu} \ar[d]_{p_Y} & Y \ar[d] \\
Y \ar[r]^-{p} & Y/G}
\end{equation}
is a pushout diagram of motivic spaces. In particular, $Y/G$ coincides with
the coarse moduli space (if it exists) for the $G$-action on $Y$ if
$Y \in \Sch^G_k$ (cf. \cite[Defn.~0.1]{GIT}). A (left) action as above is called
(categorically) free if the morphism
$G \times Y \to Y \times Y$ of the form $(g, y) \mapsto (gy, y)$
is a monomorphism. We shall let $\Msp^G_{\rm fr}(k)$ denote the category of
motivic spaces with free $G$-actions.
In the case of free $G$-action, the quotient $Y/G$ has the following additional
properties.

\begin{lem}\label{lem:Free-qt}
  Assume that the $G$-action on $Y$ is free. We then have the following.
  \begin{enumerate}
  \item
  The quotient map
  $p \colon Y \to Y/G$ is a Nisnevich local fibration with fiber $G$, i.e., for every
  smooth Henselian local $k$-algebra $R$, the map $Y(R) \to (Y/G)(R)$ is a Kan
  fibration of simplicial sets with fiber $G(R)$.
\item
  If $Y \in \Sm^G_k$, then $Y/G$ is the ({\'e}tale) sheaf represented by the
  algebraic space $[Y/G]$.
  \end{enumerate}
\end{lem}
\begin{proof}
  By definition of $Y/G$ in ~\eqref{eqn:P-out} and \cite[Exc.~8.2.6]{Weibel-HALG},
  the presheaf quotient map $Y \to (Y/G)^{\rm pr}$ is a pointwise Kan fibration of
  presheaves of simplicial sets on $\Sm_k$ with fiber $G$. In particular,
  the sheaf quotient map $p \colon Y \to Y/G$ is a stalkwise Kan fibration
  (called local fibration in \cite[Chap.~2, Defn.~1.11]{Morel-Voevodsky}) in
  the Nisnevich topology of $\Sm_k$. Since every smooth Henselian local $k$-algebra
  $R$ is the Nisnevich stalk of a point on a smooth $k$-scheme, we conclude that
  the map $Y(R) \to (Y/G)(R)$ is a Kan fibration of
  simplicial sets with fiber $G(R)$. The second part of the lemma follows from
  \cite[Prop.~22]{EG-Inv} which says that the stack quotient $[Y/G]$ is an 
  algebraic space and it coincides with the categorical quotient
  (cf. \cite[Defn.~0.1]{GIT}) under the given hypotheses. 
\end{proof}

One can now generalize the construction of $E^\bullet_G(X)$ as follows.
We  let ${\rm Et}/k$ denote the big {\'e}tale site of $\Spec(k)$.
Given a sheaf of sets $F$ on ${\rm Et}/k$ with a free $G$-action, we consider the
simplicial sheaf of sets

\begin{equation}\label{eqn:Bar-EG-Gen}
E^{\bullet}_F = \left(\cdots 
\stackrel{\rightarrow}{\underset{\rightarrow}\rightrightarrows}
F \times F \times F 
\stackrel{\rightarrow}{\underset{\rightarrow}\to} F \times F 
\rightrightarrows F \right), 
\end{equation}
where the face and the degeneracy maps are defined exactly as they are defined
for $E^\bullet_G(X)$ in ~\eqref{eqn:Bar-EG}.
Then $G$ acts freely on $E^{\bullet}_F$ via its 
diagonal action on each $F^n$.
Letting $B^{\bullet}_F$ denote the quotient, we get a $G$-torsor
$\pi_F \colon E^{\bullet}_F \to B^{\bullet}_F$. Recall here that a $G$-torsor
(or a principal $G$-bundle) over a motivic
space $S$ is a morphism $Y \to S$ of motivic spaces
together with a free $G$-action on $Y$ over $S$
such that the  map $Y/G \to S$ is an isomorphism.

Let $B^{\bullet}_G \xrightarrow{\phi} \sB_G$ be a trivial 
cofibration of motivic spaces with $\sB_G$ fibrant.
Since $G$ is special, it follows from 
\cite[Chap.~4, Prop.~1.15, Lem.~1.18]{Morel-Voevodsky} that there is a 
universal $G$-torsor $\sE_G \to \sB_G$ such that for any sheaf of 
sets $F$ on ${\rm Nis}/k$ with free $G$-action, there are Cartesian squares 
of motivic spaces

\begin{equation}\label{eqn:motivic0}
\xymatrix@C2pc{
  E^{\bullet}_F \ar[r]^-{{\ov{\phi}}_F} \ar[d]_-{\pi_F} & \sE_G \ar[d] &
  E^{\bullet}_G(X) \ar[l]_{{\ov{\phi}}_X} \ar[d]^-{\pi_X} \\
B^{\bullet}_F \ar[r]_-{\phi_F} & \sB_G & X^\bullet_G, \ar[l]^-{\phi_X}}
\end{equation}
where $\phi_{F}$ is well-defined up to a simplicial homotopy. The map
$\phi_X$ is the composition of the projection $X^\bullet_G \to B^{\bullet}_G$ followed by
$\phi$.  The following result is a
consequence of \cite[Chap.~4, Prop.~1.20]{Morel-Voevodsky}.

\begin{lem}\label{lem:Bar-ind}
Let $\rho = \left(V_i, U_i\right)_{i \ge 1}$ be an admissible gadget for $G$
(cf. Definition~\ref{defn:Add-Gad})
and let $F = {\rm colim}_i \  U_i \in {\rm Nis}/k$. Then the maps $\phi_F$ and
$\ov{\phi}_F$ in ~\eqref{eqn:motivic0} are simplicial weak equivalences.
\end{lem}
\begin{proof}
Assume that $\rho$ is given by a good pair $(V,U)$ and let $x \in U$ be
a $k$-rational point (which exists by definition of a good pair).
We first show that $\phi_F$ is a simplicial weak equivalence.
According to the proof of \cite[Chap.~4, Prop.~1.20]{Morel-Voevodsky}, we
need to show that if $E \to S$ is a $G$-torsor over a  smooth Henselian local
$k$-scheme $S$, then the morphism $E \stackrel{G}{\times} U_i \to E/G \cong S$
splits for some $i \gg 0$. Since $G$ is special, any such torsor must be trivial. In
particular, it is the projection $S \times G \to S$.
To find a splitting, it suffices to find a $G$-equivariant
morphism $G \to U_i$. But this is given by the $G$-orbit $Gx \inj U = U_1$.

To show that ${\ov{\phi}}_F$ is a simplicial weak equivalence, we need to show
using \cite[Chap.~3, Lem.~1.11]{Morel-Voevodsky} that for every smooth Henselian
local $k$-algebra $R$, the map $E^{\bullet}_F (R) \to \sE_G(R)$ is a
weak equivalence of simplicial sets.
To that end, we note that the vertical maps in ~\eqref{eqn:motivic0} are local
fibrations with fiber $G$ by \lemref{lem:Free-qt}.
These local fibrations yield for us a commutative diagram of simplicial sets
\begin{equation}\label{eqn:Bar-ind-0}
\xymatrix@C2pc{
G(R) \ar[r] \ar[d] & E^{\bullet}_F(R) \ar[r] \ar[d] & B^{\bullet}_F(R) \ar[d] 
\\
G(R) \ar[r] & \sE_G(R) \ar[r] & \sB_G(R),}
\end{equation}
where the rows are fibration sequences of simplicial sets.
It follows from \cite[Prop.~3.6.1]{Hovey-Book} that the rows remain
fibration sequences upon taking the geometric realizations. 

We have shown above that the right vertical arrow in ~\eqref{eqn:Bar-ind-0} is a
simplicial weak equivalence and hence a weak equivalence of geometric realizations.
The left vertical map is an isomorphism of simplicial sets. 
We conclude from the long exact sequence of homotopy groups of
fibrations that the middle vertical map is a weak equivalence of
geometric realizations. Hence, the map $E^{\bullet}_F (R) \to \sE_G(R)$ is a 
weak equivalence of simplicial sets.
\end{proof}

Let $\rho = \left(V_i, U_i\right)_{i \ge 1}$ be an admissible gadget for $G$
and let $X \in \Sm^G_k$. We let $X_G(\rho)$ denote the motivic space 
${\rm colim}_i  X^i_G(\rho)$ where $X^i_G(\rho) = X \stackrel{G}{\times} U_i$
is a smooth $k$-scheme by \corref{cor:Mixed-quotient} and
the colimit is taken in the category ${\rm Nis}/k$.
We call $X_G(\rho)$ a motivic Borel space.
The following result will be a key step in the solution of 
Thomason's completion problem.

\begin{prop}\label{prop:Bar-ind-Gen}
  There is a canonical isomorphism $X_G(\rho) \cong X^{\bullet}_G$ in $\sH(k)$.
In particular, $X_G(\rho)$ does not depend on the choice of the admissible
gadget $\rho$.
\end{prop}
\begin{proof}
  We let $F = {\rm colim}_i U_i$ as in \lemref{lem:Bar-ind}. Since colimits
  commute with quotient maps and finite products, the canonical map $X_G(\rho) \to
  X \stackrel{G}{\times} F$ is an isomorphism of motivic spaces. 
  We let $\phi^X_F \colon X \stackrel{G}{\times} E^\bullet_F  \to
  X \stackrel{G}{\times} \sE_G$  be the map induced on the quotients by the
  $G$-equivariant map $\ov{\phi}^X_F = (\id_X \times \ov{\phi}_F)
  \colon X {\times} E^\bullet_F  \to
  X {\times} \sE_G$. We first show that $\phi^X_F$ is a simplicial weak equivalence.

For this, we note that $\ov{\phi}_F \colon E^\bullet_F \to \sE_G$
  is a simplicial weak equivalence by \lemref{lem:Bar-ind}. This implies that
  the map $\ov{\phi}^X_F$ is also a simplicial weak equivalence.
We now let $R$ be a smooth Henselian local $k$-algebra
  and consider the commutative diagram of simplicial sets
  \begin{equation}\label{eqn:Bar-ind-Gen-0}
\xymatrix@C2pc{
G(R) \ar[r] \ar[d] &  (X {\times} E^\bullet_F)(R) \ar[r] \ar[d]^{\ov{\phi}^X_F(R)} & 
\left(X \stackrel{G}{\times} E^\bullet_F\right)(R) \ar[d]^-{\phi^X_F(R)} \\
G(R) \ar[r] & (X {\times} \sE_G)(R) \ar[r] &
\left(X \stackrel{G}{\times} \sE_G\right)(R),}
  \end{equation}
  whose rows are fibration sequences of simplicial sets, as shown in the
  proof of \lemref{lem:Free-qt}.
  As in the proof of \lemref{lem:Bar-ind}, we need to show that the right vertical
  arrow is a weak equivalence of simplicial sets. But this follows by
  repeating the argument of the proof of \lemref{lem:Bar-ind} and noting that the left
  and the middle vertical arrows are weak equivalences.

We next observe that since the map $\phi \colon B^\bullet_G \to \sB_G$ is a
  simplicial weak equivalence, a repetition of the above arguments
  shows that $\ov{\phi} \colon E^\bullet_G \to \sE_G$ and
  $\phi_X \colon X^\bullet_G \cong X \stackrel{G}{\times} E^\bullet_G \to
  X \stackrel{G}{\times} \sE_G$
  are simplicial weak equivalences. Combining $\phi^X_F$ and $\phi_X$, we get
  simplicial weak equivalences
  \begin{equation}\label{eqn:Bar-ind-Gen-1}
    \phi_X \colon X^\bullet_G \xrightarrow{\simeq} X \stackrel{G}{\times} \sE_G
    \xleftarrow{\simeq}  X \stackrel{G}{\times} E^\bullet_F \colon \phi^X_F.
    \end{equation}

We now think of $F$ as a constant simplicial sheaf and
consider the $G$-equivariant map of simplicial sheaves $u: F \to E^{\bullet}_F$
given in degree $n$ by $u_n(a) = (a, \cdots , a)$. This in turn gives a map 
$X \stackrel{G}{\times} F \xrightarrow{u_X} 
X \stackrel{G}{\times} E^{\bullet}_F$. To finish the proof of the proposition,
it suffices to show that this map is an $\A^1$-weak equivalence.
 By \cite[Chap.~2, Prop.~2.14]{Morel-Voevodsky}, we can further reduce to showing
that each map $u_{X, n}: X \stackrel{G}{\times} F \to
X \stackrel{G}{\times} F^{n+1}$ is an $\A^1$-weak equivalence.   
In order to do so, it suffices to show that the projection
$X \stackrel{G}{\times} F^{n+1} \xrightarrow{p_{X,n}} 
X \stackrel{G}{\times} F^{n}$ is an $\A^1$-weak equivalence for each $n > 0$. 

To show the last assertion, we observe that the assumption that the base
scheme is a field and $G$ is special guarantees that $\rho$ is a nice admissible
gadget for $G$ in the sense of \cite[Chap.~4, Defn.~2.4]{Morel-Voevodsky}.
We can therefore apply Lemma~2.9 (in Chapter 4) of op. cit. to conclude that
the map $\left(X \times (U_i)^n\right) \stackrel{G}{\times} F \to
X  \stackrel{G}{\times} (U_i)^n$
is an $\A^1$-weak equivalence for each $n, i \ge 1$.
Passing to the colimit as $i \to \infty$, we see from \cite[Prop.~19]{Deligne}
that $p_{X,n}$ is an $\A^1$-weak equivalence. This completes the proof of the
proposition.
\end{proof}

Given a motivic space $Y \in \Msp(k)$, we let $K(Y)$ be the mapping spectrum
\begin{equation}\label{eqn:KGL}
  K(Y) = \Maps_{\sS\sH(k)}(\Sigma^\infty_T(Y_+), KGL_k),
\end{equation}
where $KGL_k \in \sS\sH(k)$ is the
algebraic $K$-theory spectrum (cf. \cite[\S~6.2]{Voev-ICM},
\cite[\S~7.2]{Krishna-Park-AGT}) and $\Sigma^\infty_T(Y_+)$ is the image of $Y$
under the composite functor
\begin{equation}\label{eqn:Suspension}
  \sH(k) \xrightarrow{(-)_+} \sH_\bullet(k) \xrightarrow{\Sigma^\infty_T(-)}
  \sS\sH(k).
\end{equation}
For $Y \in \Sm_k$, this coincides with the Thomason-Trobaugh $K$-theory spectrum
of $Y$. In general, we can write $Y = \hocolim_{\alpha \in I} Y_\alpha$, where
$I$ is a locally small diagram category together with a functor
$F \colon I \to \Sm_k$ such that $Y_\alpha = F(\alpha)$. Since each motivic space is
cofibrant and the $T$-suspension functor $\Sigma^\infty_T(-)$ preserves
homotopy colimits, we get
\begin{equation}\label{eqn:K-colim}
  \begin{array}{lll}
    K(Y) & = & K(\hocolim_{\alpha \in I} Y_\alpha) \\
    & = & \Maps_{\sS\sH(k)}(\Sigma^\infty_T(\hocolim_{\alpha} (Y_\alpha)_+), KGL_k) \\
    & \xrightarrow{\simeq} &
    \Maps_{\sS\sH(k)}(\hocolim_\alpha \Sigma^\infty_T((Y_\alpha)_+), KGL_k) \\
  & \xrightarrow{\simeq} & \holim_\alpha
  \Maps_{\sS\sH(k)}(\Sigma^\infty_T((Y_\alpha)_+), KGL_k) \\
  & \simeq &  \holim_\alpha K(Y_\alpha),
  \end{array}
\end{equation}
where the arrow on the third line is a weak equivalence by \lemref{lem:K-simplex-6}.
If $Y \in \Msp(k)$ is a motivic space equipped with a free $G$-action and $Y/G$
is the quotient motivic space, we let $K_G(Y) := K(Y/G)$. If $Y$ and $Y/G$ are
both smooth schemes, then $K_G(Y)$ coincides with the classical equivariant $K$-theory
of $Y$ (cf. \S~\ref{sec:Recall}).

As a consequence of \propref{prop:Bar-ind-Gen}, we get the following.
This provides another proof of \propref{prop:Rep-ind} in a special case.

\begin{cor}\label{cor:Bar-ind-Gen-2}
If $G$ is special and $X \in \Sm^G_k$, then $K(X_G) \simeq K(X_G(\rho))$
for any admissible gadget $\rho$ for $G$. That is, the Borel equivariant
$K$-theory of $X$ is representable in $\sS\sH(k)$.
\end{cor}

\section{Some results on motivic spaces and their $K$-theory}
\label{sec:Gen-construction} 
In this section, we collect some technical results related to motivic spaces
and their $K$-theory that we shall need for proving our main result on
Thomason's completion problem.
We let $k$ be a field, $G \in \Grp_k$ a special group and $X \in \Sm^G_k$.

\subsection{$K$-theory of quotient motivic spaces}\label{secK-quot}
For $\sF \in \Msp(k)$, we let $\int_\sF$ denote the category whose objects are maps
of sheaves $f \colon Y \to \sF$ with $Y \in \Sm_k$ and the hom set between
$f \colon Y \to \sF$ and $f' \colon Y' \to \sF$ consists of maps
$\phi \colon Y \to Y'$ in $\Sm_k$ such that $f'\circ \phi = f$.
Let $\Delta$ denote the standard simplicial indexing category.
We shall use the following known results.

\begin{lem}\label{lem:K-simplex-6}
Let $\sG$ be a simplicial sheaf on $\Nis/k$ which sends
$[n] \in \Delta$ to $\sG^n$ and let $\sF$ be a discrete sheaf of sets on
$\Nis/k$.  Then the canonical maps of simplicial sheaves
  \[
  \hocolim_{[n] \in \Delta} \sG^n \to \sG \ \mbox{and} \
  \hocolim_{Y \in \int_\sF} Y \to \colim_{Y \in \int_\sF} Y \to \sF
  \]
  are simplicial weak equivalences, where each $\sG^n$
is considered as a constant simplicial sheaf.
\end{lem}
\begin{proof}
  The first weak equivalence is well known (cf. \cite[Rem.~2.1]{DHI}).
  It is also well known that $\colim_{Y \in \int_\sF} Y \to \sF$ is a simplicial weak
  equivalence (cf. \cite[(2.1.1)]{Dugger-Adv}).
  The claim that $\hocolim_{Y \in \int_\sF} Y \to \sF$ is a simplicial weak
  equivalence is a direct consequence (in fact, equivalent to) of
  \cite[Lem.~2.7]{Dugger-Adv}, as explained in \cite[Rem.~3.2.7]{Dugger-Unpub}.
\end{proof}

We let $d \colon \sF_1 \to \sF_2$ be a morphism in $\Msp^G_{\rm fr}(k)$.
For $i = 1,2$, we let $\sI_i$ be a 
subcategory of $\int_{\sF_i}$ such that all objects
$f \colon Y \to \sF_i$ in $\sI_i$ are morphisms in $\Msp^G_{\rm fr}(k)$ and all morphisms
between two objects of $\sI_i$ are $G$-equivariant. 
We assume that for any $f \colon Y \to \sF_1$ in $\sI_1$, the composition
$d \circ f: Y \to \sF_2$ is an object of $\sI_2$ so that $d$ defines a functor
$\psi_d \colon \sI_1 \to \sI_2$. For $i = 1,2$ , the product of a map
$(f \colon Y \to \sF_i) \in \int_{\sF_i}$  with $\id_X$
and subsequent quotient by $G$-action yields 
$\ov{f} \colon  X \stackrel{G}{\times} Y \to X \stackrel{G}{\times} \sF_i$
which is compatible with the morphisms in $\int_{\sF_i}$. Taking the homotopy colimit,
we thus get a map $\ov{\phi}_i \colon
\hocolim_{Y \in \sI_i} X \stackrel{G}{\times} Y \to X \stackrel{G}{\times} \sF_i$.

For every $f \colon Y \to \sF_i$ in $\sI_i$, we let $\pr_Y \colon X \times Y \to X$
be the projection. Then $\pr_Y$ commutes with all maps
$X \times Y \xrightarrow{\id_X \times \alpha} X \times Y'$
with $Y \xrightarrow{\alpha} Y'$ in $\sI_1$ and $\sI_2$. Since
$\psi_d(f) = d \circ f \colon Y \to \sF_2$ is an
object of $\sI_2$ for every $f \colon Y \to \sF_1$ in $\sI_1$, we get a
commutative diagram of spectra
\begin{equation}\label{eqn:K-simplex-0}
  \xymatrix@C2pc{
    \holim_{Y' \in \sI_2} K_G(X \times Y') \ar[r]^-{(d\circ f)^*} & 
    K_G(X \times Y) \\
    & K_G(X) \ar[ul]^-{p^*_2} \ar[u]_-{\pr^*_Y},}
\end{equation}
where $p^*_2$ is obtained by taking the homotopy limit of the diagram
$\pr_{Y'}^* \colon K_G(X) \to K_G(X \times Y')$ with $f' \colon Y' \to \sF_2$ in
$\sI_2$. This diagram is clearly compatible with respect to maps in $\sI_1$.
Passing to the limit over $\sI_1$, we get the pull-back 
$d^* \colon \holim_{Y' \in \sI_2} K_G(X \times Y') \to
\holim_{Y \in \sI_1} K_G(X \times Y)$.

Letting $p^*_1$ denote the map  obtained by taking the homotopy limit of the diagram
$\pr_{Y}^* \colon K_G(X) \to K_G(X \times Y)$ with $f \colon Y \to \sF_1$ in
$\sI_1$, we get the following.

\begin{lem}\label{lem:K-simplex-1}
  The diagram
  \begin{equation}\label{eqn:K-simplex-2}
    \xymatrix@C1pc{
   \holim_{Y' \in \sI_2} K_G(X \times Y') \ar[r]^-{d^*} & 
    \holim_{Y \in \sI_1} K_G(X \times Y) \\
    & K_G(X) \ar[ul]^-{p^*_2} \ar[u]_-{p^*_1}}
\end{equation}
is commutative.
\end{lem}

We now assume that $Y/G \in \Sm_k$ for each $f \colon Y \to \sF_i$ in $\sI_i$.
In this case, the quotient $X \stackrel{G}{\times} Y$ is a smooth $k$-scheme by
\corref{cor:EG-quotient-2} and the pull-back
$K(X \stackrel{G}{\times} Y) \to  K_G(X \times Y)$ by the quotient map
$X \times Y \to X \stackrel{G}{\times} Y$ is a weak equivalence.
Letting $\phi_{i,X} \colon \hocolim_{Y \in \sI_i} X \stackrel{G}{\times} Y \to
X \stackrel{G}{\times} \sF_i$ denote the map induced by $\phi_i$ and $\id_X$, 
we have a homotopy commutative diagram
of motivic spaces
\begin{equation}\label{eqn:K-simplex-3}
    \xymatrix@C1pc{
      \hocolim_{Y \in \sI_1} X \stackrel{G}{\times} Y \ar[r]^-{{\phi}_{1,X}}
      \ar[d]_-{\ov{d}} & X \stackrel{G}{\times} \sF_1 \ar[d]^-{\ov{d}} \\
      \hocolim_{Y' \in \sI_2} X \stackrel{G}{\times} Y' \ar[r]^-{{\phi}_{2,X}} &
      X \stackrel{G}{\times} \sF_2.}
    \end{equation}

Considering the induced maps of $K$-theory spectra, we get the following.

\begin{lem}\label{lem:K-simplex-4}
  One has a homotopy commutative diagram
  \begin{equation}\label{eqn:K-simplex-5}
    \xymatrix@C1pc{
      K_G(X) \ar[r]^-{p^*_2} \ar[dr]_-{p^*_1} &  \holim_{Y' \in \sI_2} K_G(X \times Y')
      \ar[d]^-{d^*} & \holim_{Y' \in \sI_2} K(X  \stackrel{G}{\times} Y')
      \ar[l]_-{\simeq} \ar[d]^-{\ov{d}^*} & K(X  \stackrel{G}{\times} \sF_2)
      \ar[d]^-{\ov{d}^*} \ar[l]_-{{\phi}^*_{2,X}} \\
      & \holim_{Y \in \sI_1} K_G(X \times Y) & \holim_{Y \in \sI_1}
      \ar[l]_-{\simeq} K(X  \stackrel{G}{\times} Y) &
      K(X  \stackrel{G}{\times} \sF_1) \ar[l]_-{{\phi}^*_{1,X}}}
    \end{equation}
  of spectra in which the middle horizontal arrows are weak equivalences.
  \end{lem}
\begin{proof}
  The left triangle in ~\eqref{eqn:K-simplex-5} commutes by \lemref{lem:K-simplex-1},
  the middle square clearly commutes and the right square commutes by
  ~\eqref{eqn:K-simplex-3}.  The middle horizontal arrows are weak equivalences
  by \cite[Thm.~11.4]{Hirsch-Notes} and \cite[\S~6]{Hovey-JPAA} since
  $K_G(X \times Y)$ and $K(X  \stackrel{G}{\times} Y)$ as well
  as $K_G(X \times Y')$ and $K(X  \stackrel{G}{\times} Y')$ are $\Omega$-spectra
  for all $Y \in \sI_1$ and $Y' \in \sI_2$.
\end{proof}

\subsection{Approximation of $X \stackrel{G}{\times} \sE_G$}
\label{sec:EG-approx}
We now look at the universal $G$-torsor $\pi \colon \sE_G \to \sB_G$
(cf. ~\eqref{eqn:motivic0}). If $\sE_G$ is the simplicial sheaf which sends
$[n]$ to $\sE^n_G$, then $\pi_n \colon \sE^n_G \to \sB^n_G$ is
a $G$-torsor for every $n$. Furthermore, the canonical map 
$\hocolim_{[n] \in {\Delta}} \sE^n_G \to \sE_G$ is a simplicial weak equivalence
by \lemref{lem:K-simplex-6}. For any $n \ge 0$ and any map
    $f \colon Y \to \sB^n_G$ of discrete sheaves, we let $\sE^n_G(Y)$
    denote the pull-back of the $G$-torsor $\sE^n_G \to \sB^n_G$ via $f$ and
    let $\pi_n(f) \colon \sE^n_G(Y) \to \sE^n_G$ be the induced $G$-equivariant map.
    It follows that $\sE^n_G(Y)$ is a motivic space with a free $G$-action, which is
    a smooth $k$-scheme if $Y$ is.

    We let $\int'_{\sE^n_G}$ denote the category whose objects are the
    $G$-equivariant maps $\pi_n(f) \colon \sE^n_G(Y)$ \\
    $\to \sE^n_G$ where $(f \colon Y \to \sB^n_G) \in \int_{\sB^n_G}$. The morphisms
    between
    the objects of $\int'_{\sE^n_G}$ are the base change of morphisms in $\int_{\sB^n_G}$.
  Recall that if $F \colon \sC \to \sD$ is a functor between two categories
  and $A \in \sD$, then $A/F$ is the comma category whose objects are maps
  $g \colon A \to F(B)$ in $\sD$ with $B \in \sC$
  and the morphisms between $g$ and $g'$ are maps
  $f \colon B \to B'$ in $\sC$ such that $F(f) \circ g = g'$.

\begin{lem}\label{lem:Simplex-4}
  For every $n \ge 0$, the canonical maps
  \begin{enumerate}
  \item
    \hspace*{2cm}
    $\hocolim_{Y \in \int_{\sB^n_G}}  \sE^n_G(Y) \to \colim_{Y \in \int_{\sB^n_G}} \sE^n_G(Y) \to
    \sE^n_G$;
  \item
\hspace*{2cm} $\hocolim_{Y \in \int_{\sB^n_G}}
  X {\times} \sE^n_G(Y) \to \colim_{Y \in \int_{\sB^n_G}}
  X {\times} \sE^n_G(Y) \to X {\times} \sE^n_G;$
    \item
  \hspace*{2cm} $\hocolim_{Y \in \int_{\sB^n_G}}
  X \stackrel{G}{\times} \sE^n_G(Y) \to \colim_{Y \in \int_{\sB^n_G}}
  X \stackrel{G}{\times} \sE^n_G(Y) \to X \stackrel{G}{\times} \sE^n_G$
  \end{enumerate}
  are simplicial weak equivalences.
\end{lem}
\begin{proof}
  We fix $n \ge 0$ and prove (2) first. 
  We let $T = X \times \sE^n_G$ and let $F \colon \int_{\sB^n_G} \to \int_{T}$ be the
  functor which sends $f \colon Y \to \sB^n_G$ to
  $(\id_X \times \pi_n(f)) \colon X \times \sE^n_G(Y) \to T$.
  Using \lemref{lem:K-simplex-6}, \cite[Thm.~10.8]{Hirsch-Notes} and
  \cite[Tag~09WN, Lem.~4.17.2]{SP}, it suffices to show that the functor
  $F$ is right cofinal as well as homotopy right cofinal. In other words, we need to
  show for every $f \colon Z \to T$ in $\int_{T}$ that the nerve of the comma
  category ${f}/F$ is connected as well as contractible
  (cf. \cite[Defn.~10.6]{Hirsch-Notes}). By \cite[Cor.~2, p.~84]{Quillen} and
  \cite[Tag~09WN, Defn.~4.17.1]{SP}, it suffices to show that ${f}/F$ has an
  initial object. 

  To that end, we let $\pr_X \colon T \to X$ and $\pr_{\sE^n_G} \colon T \to \sE^n_G$
  be the projections.
  We let $f_1 = \pr_X \circ f \colon Z \to X$ and
  $f_2 = \pr_{\sE^n_G} \circ f \colon Z \to \sE^n_G$ so that $f = (f_1, f_2)$.
  Letting $g = \pi_n \circ f_2 \colon Z \to \sB^n_G$, we see that there is a unique
  map $\alpha \colon Z \to \sE^n_G(Z)$ such that the diagram
  \begin{equation}\label{eqn:Simplex-4-0}
    \xymatrix@C2pc{
      Z \ar[r]^-{\alpha} \ar[dr]_-{\id}
      \ar@/^2pc/[rr]^-{f_2} & \sE^n_G(Z) \ar[d]^-{\beta}
      \ar[r]^-{\pi_n(g)} &
      \sE^n_G \ar[d]^-{\pi_n} \\
      & Z \ar[r]^-{g} & \sB^n_G}
  \end{equation}
  commutes and the right square is Cartesian.
  We claim that $\wt{\alpha} = (f_1, \alpha) \colon Z \to X \times \sE^n_G(Z)$ is an
  initial object of ${f}/F$.

To prove the claim, we let $F' \colon \int_{\sB^n_G} \to \int_{\sE^n_G}$ be the functor
  which sends $f \colon Y \to \sB^n_G$ to the pull-back map
  $\pi_n(f) \colon \sE^n_G(Y) \to \sE^n_G$. We first show that
  $\alpha \colon Z \to \sE^n_G(Z)$
  is an initial object of the comma category ${f_2}/{F'}$.
  To prove this, we let $(h \colon Y \to \sB^n_G) \in \int_{\sB^n_G}$ and let
  $\sE^n_G(Y)$ be the pull-back of $h$ via $\pi_n$.
  We let $\alpha' \colon Z \to \sE^n_G(Y)$ be any object of
${f_2}/{F'}$ and look at the commutative diagram
  \begin{equation}\label{eqn:Simplex-4-1}
    \xymatrix@C2pc{
      Z \ar[r]^-{\alpha'} \ar@/^2pc/[rr]^-{f_2} \ar[d]_-{\id} & \sE^n_G(Y)
      \ar[d]^-{\beta'} \ar[r]^-{\pi_n(h)} & \sE^n_G \ar[d]^-{\pi_n} \\
      Z  \ar[r]^-{\beta' \circ \alpha'} \ar@/_2pc/[rr]^-{g}  &
      Y \ar[r]^-{h} & \sB^n_G.}
      \end{equation}

Since the two squares in this diagram commute and the right square is Cartesian,
  it follows that there is a unique morphism $\psi \colon \sE^n_G(Z) \to \sE^n_G(Y)$
  such that the diagram
  \begin{equation}\label{eqn:Simplex-4-2}
    \xymatrix@C2pc{
      Z \ar[r]^-{\alpha} \ar@/^2pc/[rr]^-{\alpha'} \ar[dr]_-{\id} &
      \sE^n_G(Z) \ar[r]^-{\psi} \ar[d]^-{\beta} & \sE^n_G(Y) \ar[d]^-{\beta'}
      \ar[r]^-{\pi_n(h)} & \sE^n_G \ar[d]^-{\pi_n} \\
      & Z  \ar[r]^-{\beta' \circ \alpha'} & Y \ar[r]^-{h} & \sB^n_G}
      \end{equation}
  is commutative in which $h \circ \beta' \circ \alpha' = g$ and
  $\pi_n(h) \circ \psi = \pi_n(g)$ and the two squares are Cartesian.
  Since $\alpha'$ and $\beta'$ are
  fixed, it also follows that $\beta' \circ \alpha' \colon Z \to Y$ is the
  unique morphism in $\int_{\sB^n_G}$ such that $\alpha \circ F'(\beta' \circ \alpha')
  = \alpha \circ \psi = \alpha'$. We have thus shown that $\alpha \colon Z \to
  \sE^n_G(Z)$ is an initial object of ${f_2}/{F'}$.

To finish the proof of the claim, we let
  $\wt{\alpha'} = (f'_1, \alpha') \colon Z \to X \times \sE^n_G(Y)$ be any object of
${f}/F$ such that $\sE^n_G(Y)$ is the pull-back of $(h \colon Y \to \sB^n_G) \in
\int_{\sB^n_G}$ via $\pi_n$.
 Since $\alpha \colon Z \to \sE^n_G(Z)$ is an initial object of ${f_2}/{F'}$,
 we get a unique morphism $\gamma \colon (g \colon Z \to \sB^n_G) \to
 (h \colon Y \to \sB^n_G)$ in $\int_{\sB^n_G}$
 such that both ~\eqref{eqn:Simplex-4-1} and ~\eqref{eqn:Simplex-4-2}
 are commutative and $\gamma = \beta' \circ \alpha'$.
 Letting
 $\wt{\gamma} = (\id_X, \psi) \colon X \times \sE^n_G(Z) \to X \times \sE^n_G(Y)$,
 we see that $\wt{\gamma}$ is the unique morphism
 such that $\wt{\alpha'} = \wt{\gamma} \circ \wt{\alpha}$. 
This proves the claim and hence item (2) of the lemma.
Item (1) is a special case of item (2) where we take $X = \pt$.

To prove item (3), we let $W =  X \stackrel{G}{\times} \sE^n_G \cong T/G$
and let $\ov{F} \colon \int_{\sB^n_G} \to \int_W$ be the functor which sends
$(u \colon Y \to \sB^n_G)$ to the map $\ov{u} \colon
X \stackrel{G}{\times} \sE^n_G(Y) \to X \stackrel{G}{\times} \sE^n_G = W$ induced by
the product
$(\id_X \times \pi_n(\ov{u})) \colon X \times \sE^n_G(Y) \to X \times \sE^n_G$.
We let $f \colon Z \to W$ be an object of $\int_W$. As in the previous cases, it
suffices to show that the comma category ${f}/{\ov{F}}$ has an initial object.

To that end, we let $f_2 \colon Z \to W \to \sB^n_G$ be the composite map
where the second map is the canonical projection. We let $\sE^n_G(Z)$ be the
$G$-torsor over $Z$ obtained by the pull-back of $\pi_n$ by $f_2$. We then note that
the diagram
\begin{equation}\label{eqn:Simplex-4-3}
    \xymatrix@C2pc{
      \sE^n_G(Z) \ar[d]_-{\beta} \ar[r]^-{f'} &  X \times \sE^n_G \ar[r]^-{\pr_2}
      \ar[d]^-{\pi_{n,X}} & \sE^n_G \ar[d]^-{\pi_n} \\
      Z \ar[r]^-{f} & W \ar[r]^-{\pr'_2} & \sB^n_G}
\end{equation}
is commutative in which the squares are Cartesian, where $f_2 = \pr'_2 \circ f$.

Letting $f'_2 = \pi_n(f_2)$ and $f' = (f'_1, f'_2)$, we get a commutative diagram
\begin{equation}\label{eqn:Simplex-4-4}
  \xymatrix@C2pc{
    \sE^n_G(Z) \ar[d]_-{\beta} \ar[r]^-{g'_1} &
    X \times \sE^n_G(Z) \ar[r]^-{g'_2} \ar[d]^-{\beta_X} &
    X \times \sE^n_G  \ar[r]^-{\pr_2}
      \ar[d]^-{\pi_{n,X}} & \sE^n_G \ar[d]^-{\pi_n} \\
      Z \ar[r]^-{g_1} & X \stackrel{G}{\times} \sE^n_G(Z) \ar[r]^-{g_2} &
      W \ar[r]^-{\pr'_2} & \sB^n_G,}
\end{equation}
where $\beta_X$ is the quotient map, $g_2 \circ g_1 = f$ and $g'_2 \circ g'_1 = f'$.
Moreover,
$g'_1 = (f'_1, \id)$, $g'_2 = (\id, f'_2)$ and $g_1$ (resp. $g_2$) is the
map on quotients induced by $g'_1$ (resp. $g'_2$). It is clear that the left square
in ~\eqref{eqn:Simplex-4-4} is Cartesian. It remains to show that $g_1 \colon Z \to 
X \stackrel{G}{\times} \sE^n_G(Z)$ is an initial object of ${f}/{\ov{F}}$. 

We let $h_1 \colon Z \to X \stackrel{G}{\times} \sE^n_G(Y)$ be an
object of ${f}/{\ov{F}}$, where $(\alpha \colon Y \to \sB^n_G) \in \int_{\sB^n_G}$ and
$\sE^n(Y)$ is the pull-back of $\alpha$ via $\pi_n$.
Since $\beta \colon \sE^n_G(Z) \to Z$ is a $G$-torsor and 
$G$-torsors are preserved under pull-back, it follows that there is a commutative
diagram
\begin{equation}\label{eqn:Simplex-4-5}
  \xymatrix@C2pc{
\sE^n_G(Z) \ar[d]_-{\beta} \ar[r]^-{h'_1} &
    X \times \sE^n_G(Y) \ar[r]^-{h'_2} \ar[d]^-{\beta'_X} &
    X \times \sE^n_G  \ar[r]^-{\pr_2}
      \ar[d]^-{\pi_{n,X}} & \sE^n_G \ar[d]^-{\pi_n} \\
      Z \ar[r]^-{h_1} & X \stackrel{G}{\times} \sE^n_G(Y) \ar[r]^-{h_2} &
      W \ar[r]^-{\pr'_2} & \sB^n_G}
\end{equation}
whose squares are Cartesian.

Letting $\delta'$ denote the composite map $\sE^n_G(Z) \xrightarrow{h'_1}
X \times \sE^n_G(Y) \xrightarrow{\pr_2} \sE^n_G(Y)$, we see that
$\delta'$ is a $G$-equivariant map of smooth $k$-schemes with free $G$-actions.
Letting $\delta \colon Z \to Y$ be the induced map between the quotients, we get
a map $\delta \colon (f_2 \colon Z \to \sB^n_G) \to (\alpha \colon Y \to \sB^n_G)$
in ${f}/{\ov{F}}$ such that $h_1 = \ov{F}(\delta) \circ g_1$. It is clear that
$\delta$ is the unique map such that $\ov{F}(\delta) \circ g_1 = h_1$.
This proves (3) and concludes the proof of the lemma.  
\end{proof}

\section{Thomason's problem for special groups}\label{sec:Thom-Main}
In the next two sections, we shall prove our main results on
Thomason's completion problem. Our strategy is to factorize 
Thomason's pull-back map $\wt{\pi}^*$ (cf. ~\eqref{eqn:Simplicial-sheaf})
into a composition of two maps.
We shall then show separately that each of these two maps induces a weak equivalence
after we pass to the completion of equivariant $K$-theory.
Before we do this, we first show that $\wt{\pi}^*$ indeed factors through the
completion.

\subsection{Factorization through the completion}\label{sec:Factor-compln}
We let $k$ be any field, $G$ any $k$-group and $X \in \Spc^G_k$.
The $G$-equivariant projection map $\tau_X \colon E^\bullet_G(X) \to X$
(cf. ~\eqref{eqn:Bar-EG}) induces the pull-back map
$\tau^*_X \colon K'_G(X) \to K'_G(E^\bullet_G(X)) := \holim_{n} K'_G(G^{n+1} \times X)$.
On the other hand, the quotient map
$\pi_X \colon E^\bullet_G(X) \to X^\bullet_G$ yields the pull-back map
\begin{equation}\label{eqn:Simplicial-sheaf-0}
  K'(X^\bullet_G) \simeq \holim_{n} K'(G^n {\times} X)
 \xrightarrow{\pi^*_X}   \holim_{n} K'(G^{n+1} \stackrel{G}{\times} X)
\simeq K'_G(E^\bullet_G(X)).
  \end{equation}
The map $\pi^*_X$ is a levelwise weak equivalence, and hence a weak equivalence
between the homotopy limits (see the proof of \lemref{lem:K-simplex-4}).

\begin{lem}\label{lem:Simplicial-K-0}
  The diagram of spectra
  \begin{equation}\label{eqn:Simplicial-K-1}
    \xymatrix@C1pc{
      K'_G(X) \ar[dr]_-{\tau^*_X} \ar[r]^-{\wt{\pi}^*} & K'(X^\bullet_G)
      \ar[d]^-{\pi^*_X} \\
      & K'_G(E^\bullet_G(X))}
  \end{equation}
  is commutative.
  \end{lem}
\begin{proof}
  We need to show that for every $n \ge 0$, the diagram
  \begin{equation}\label{eqn:Simplicial-K-2}
    \xymatrix@C1pc{
      K'_G(X) \ar[dr]_-{\tau^*_{X,n}} \ar[r]^-{\wt{\pi}^*_n} &
      K'(G^n \times X) \ar[d]^-{\pi^*_{X,n}} \\
      & K'_G(G^{n+1} \times X)}
  \end{equation}
  commutes up to homotopy and is compatible with the pull-back maps between
  the $K$-theory spectra induced by the bonding maps of $E^\bullet_G(X)$ and
  $X^\bullet_G$.

To prove the latter claim, it suffices to show that 
  \begin{equation}\label{eqn:Simplicial-K-2-0}
    \xymatrix@C1pc{
      G^{n+1} \times X \ar[r]^-{\pi_{X,n}} \ar[d]_-{\tau_{X,n}} & G^n \times X
      \ar[d]^-{\wt{\pi}_n} \\
    X \ar[r]^-{\pi} & [X/G]}
  \end{equation}
  is a commutative diagram of flat maps which commutes with the
  bonding maps of $E^\bullet_G(X)$ and $X^\bullet_G$. Since
  $\wt{\pi}_n = \pi \circ p^{i_1}_1 \circ \cdots \circ p^{i_n}_n$
  does not depend on the
  choice of $(i_1, \ldots , i_n)$, it suffices to show for the
  commutativity of ~\eqref{eqn:Simplicial-K-2-0} that
  $\pi \circ p^{1}_1 \circ \cdots \circ p^{n}_n \circ \pi_{X,n} = \pi \circ
  \tau_{X,n}$. But this clear since we have
  $\tau_{X,n}(g_0, \ldots , g_{n}, x) = x$ and
  $p^{1}_1 \circ \cdots \circ p^{n}_n \circ \pi_{X,n}(g_0, \ldots , g_{n}, x) =
  g_0 x$ for any $(g_0, \ldots , g_n, x) \in G^{n+1} \times X$.
  The commutativity of
   ~\eqref{eqn:Simplicial-K-2-0} with the
  bonding maps of $E^\bullet_G(X)$ and $X^\bullet_G$ is checked similarly.
\end{proof}

The map $\tau^*_X$ is clearly $K(BG)$-linear.
Since $\pi^*_X$ is canonically a weak equivalence, the  $K(BG)$-module structure
of $K'(E^\bullet_G(X))$ endows the same on $K'(X^\bullet_G)$ with respect to which
$\pi^*_X$ is $K(BG)$-linear. It follows from \lemref{lem:Simplicial-K-0} that
$\wt{\pi}^*$ is $K(BG)$-linear. Since the $K(BG)$-module structures of
$K'_G(E^\bullet_G(X))$ and $K'(X^\bullet_G)$ factor through the canonical map
$K(BG) \to K(\pt)$ (cf. \cite[Lem.~8.1]{Krishna-Crelle}),
it follows that each $K'_G(G^{n+1} \times X)$ (resp. 
$K'(G^{n} \times X)$) is $I_G$-adically complete. In particular, $K'_G(E^\bullet_G(X))$
and $K'(X^\bullet_G)$ are $I_G$-adically complete by \propref{prop:Compln_D}(7).
It follows that ${\tau^*_X}$ and $\wt{\pi}^*$ factor through
\begin{equation}\label{eqn:Simplicial-K-3}
  \wt{\pi}^* \colon K'_G(X)^{\compl}_{I_G} \to K'(X^\bullet_G); \ \
  \tau^*_X \colon K'_G(X)^{\compl}_{I_G} \to K'_G(E^\bullet_G(X)).
\end{equation}

\vskip .2cm

\subsection{Key step: factorization of $\wt{\pi}^*$}\label{sec:Factor-main}
We shall now show the factorization of $\wt{\pi}^*$ into a composition of two maps.
We assume in this subsection that $k$ is any field, $G \in \Grp_k$ is special and
$X \in \Sm^G_k$.

We let $\rho = (V_i, U_i)_{i \ge 1}$ be an admissible gadget for $G$ and let
$F = \colim_i U_i \in \Msp(k)$. We let $\Delta_F \colon F \to E^\bullet_F$
denote the diagonal map (cf. proof of \propref{prop:Bar-ind-Gen}). The commutative
diagram of $G$-equivariant maps between simplicial ind-schemes
(where $p_X$ and $\pr_2$ are the projections)
\begin{equation}\label{eqn:AST-0}
   \xymatrix@C1pc{
     X \times F \ar[r]^-{\Delta_{F,X}} \ar[dr]_-{p_X} &
     X \times E^\bullet_F \ar[d]^-{\pr_2} \\
    & X}
\end{equation}
(where $\Delta_{F,X} = \id_X \times \Delta_F$)
induces a commutative diagram of spectra
\begin{equation}\label{eqn:AST-1}
  \xymatrix@C1pc{
    K_G(X) \ar[d]_-{\pr^*_2} \ar[dr]^-{p^*_X} & \\
    \holim_{([n], i) \in \Delta \times \N} K_G(X \times (U_i)^{n+1}) \ar[r]^-{\Delta^*_{F,X}} &
    \holim_{i\in \N} K_G(X \times U_i).}
  \end{equation}

Since each $X \stackrel{G}{\times} (U_i)^n$ is a smooth $k$-scheme by
\corref{cor:Mixed-quotient}, we see that $K_G(X \times (U_i)^{n})$ coincides with
$K\left(X \stackrel{G}{\times} (U_i)^n\right)$.
By ~\eqref{eqn:K-colim}, the above diagram
therefore gives rise to the commutative diagram
\begin{equation}\label{eqn:AST-2}
  \xymatrix@C1pc{
    K_G(X) \ar[d]_-{\pr^*_2} \ar[dr]^-{p^*_X} & \\
    K\left(X \stackrel{G}{\times} E^\bullet_F\right) \ar[r]^-{\Delta^*_{F,X}} &
    K\left(X \stackrel{G}{\times} F\right)}
  \end{equation}
of $K$-theories of smooth simplicial $k$-schemes. 

Our next task is to define a pull-back map
$K_G(X) \to K(X \stackrel{G}{\times} \sE_G)$ and relate it to
~\eqref{eqn:AST-2}. 
We fix $n \ge 0$. For every $(f \colon Y \to \sB^n_G)$ in $\int_{\sB^n_G}$, we let
$\pr_Y \colon X \times \sE^n_G(Y) \to X$ be the projection.
As this is a $G$-equivariant map between smooth schemes with $G$-actions,
we get the pull-back map
$\pr^*_Y \colon K_G(X) \to K_G(X \times \sE^n_G(Y))$. Since $G$ acts freely on
$X {\times} \sE^n_G(Y)$ and the quotient
$X \stackrel{G}{\times} \sE^n_G(Y)$ lies in $\Sm_k$ by \corref{cor:EG-quotient-2},
the pull-back map $K(X \stackrel{G}{\times} \sE^n_G(Y)) \to K_G(X \times \sE^n_G(Y))$
is a weak equivalence.
Since the projection maps $\pr_Y$ commute with all maps of the kind $\id_X \times \phi$
with $\phi \in \int'_{\sE^n_G}$, we get canonical maps 
\begin{equation}\label{eqn:K-simplex-12}
\pr^*_{1,n} \colon  K_G(X) \to \holim_{Y \in \int_{\sB^n_G}}
K\left(X \stackrel{G}{\times} \sE^n_G(Y)\right) \xleftarrow{\simeq}
K\left(X \stackrel{G}{\times} \sE^n_G\right),
\end{equation}
where the second map is a weak equivalence by \lemref{lem:Simplex-4} and
~\eqref{eqn:K-colim}.

If we now vary $n$ and apply \lemref{lem:K-simplex-4} to all face and degeneracy maps
(which are all $G$-equivariant) of $\sE_G$, we deduce that the maps $\pr^*_{1,n}$
commute with the maps induced on the $K$-theory spectra by the face and degeneracy
maps of $\sE_G$ as $n$ varies. By passing to the limit, ~\eqref{eqn:K-simplex-12}
therefore yields the maps of spectra
\begin{equation}\label{eqn:K-simplex-7}
\pr^*_{1} \colon  K_G(X) \to \holim_{[n] \in \Delta} \holim_{Y \in \int_{\sB^n_G}}
K\left(X \stackrel{G}{\times} \sE^n_G(Y)\right) \xleftarrow{\simeq}
\holim_{[n] \in \Delta} K\left(X \stackrel{G}{\times} \sE^n_G\right) 
\xleftarrow{\simeq} K\left(X \stackrel{G}{\times} \sE_G\right),
\end{equation}
where the last map is a weak equivalence by \lemref{lem:K-simplex-6} and
~\eqref{eqn:K-colim}.

Recall that $\phi^X_{F} \colon X \stackrel{G}{\times} E^\bullet_F \to
X \stackrel{G}{\times} \sE_G$ is the map induced on the quotients by the
$G$-equivariant map
$(\id_X \times \ov{\phi}_F) \colon X {\times} E^\bullet_F \to X {\times} \sE_G$
(see the proof of \propref{prop:Bar-ind-Gen})).
For any $n \ge 0$, we let
$\sC_{E^n_F}$ denote the category whose objects are the inclusions
$(U_i)^{n+1} \inj E^n_F$ for $i \ge 1$ and morphisms are the canonical inclusions
$(U_i)^{n+1} \inj (U_{i+1})^{n+1}$.

\begin{lem}\label{lem:K-simplex-8}
  The diagram
  \begin{equation}\label{eqn:K-simplex-9}
    \xymatrix@C1pc{
      K_G(X) \ar[r]^-{\pr^*_1} \ar[dr]_-{\pr^*_2} &
      K\left(X \stackrel{G}{\times} \sE_G\right)
      \ar[d]^-{(\phi^{X}_{F})^*} \\
      & K\left(X \stackrel{G}{\times} E^\bullet_F\right)}
  \end{equation}
  is commutative.
\end{lem}
\begin{proof}
  We let $n \ge 0$ and let $d_i \colon [n] \to [n+1]$ be any face map.
  We consider the diagram
  \begin{equation}\label{eqn:K-simplex-10}
    \xymatrix@C1pc{
      K\left(X \stackrel{G}{\times} \sE^n_G\right) \ar[rr]^{(\phi^{X}_{F,n})^*}
      \ar[ddd]_-{d^*_i} & &
      K\left(X \stackrel{G}{\times} E^n_F\right) \ar[ddd]^-{d^*_i} \\
      & K_G(X) \ar[ul]^-{\pr^*_{1,n}} \ar[ur]_-{\pr^*_{2,n}} \ar@{=}[d] & \\
      & K_G(X) \ar[dl]_-{\pr^*_{1,n+1}} \ar[dr]^-{\pr^*_{2,n+1}} & \\
   K\left(X \stackrel{G}{\times} \sE^{n+1}_G\right) \ar[rr]^{(\phi^{X}_{F,n+1})^*} & &
   K\left(X \stackrel{G}{\times} E^{n+1}_F\right).}
  \end{equation}

The top and the bottom triangles in the middle commute by \lemref{lem:K-simplex-4}
  and we showed above in the construction of $\pr^*_{1}$ and $\pr^*_2$ that the left
  and the right faces of ~\eqref{eqn:K-simplex-10} commute. The same holds for
  any degeneracy map $s_i \colon [n+1] \to [n]$. Since ~\eqref{eqn:K-simplex-9} is
  obtained by taking
  the cosimplicial limit of the triangles in the middle of ~\eqref{eqn:K-simplex-10},
  it suffices to show that the diagram
  \begin{equation}\label{eqn:K-simplex-11}
    \xymatrix@C1pc{
      K_G(X) \ar[r]^-{\pr^*_{1,n}} \ar[dr]_-{\pr^*_{2,n}} &
      K\left(X \stackrel{G}{\times} \sE^n_G\right)
      \ar[d]^-{(\phi^{X}_{F,n})^*} \\
      & K\left(X \stackrel{G}{\times} E^n_F\right)}
  \end{equation}
  is commutative for every $n \ge 0$.
  But this easily follows by applying \lemref{lem:K-simplex-4} with
  $\sF_1 = E^n_F, \ \sF_2 = \sE^n_G, \ \sI_1 = \sC_{E^n_F}$, $\sI_2 = \int'_{\sE^n_G}$ and
  $d = \ov{\phi}_{F,n}$.
 \end{proof}

Let $\ov{\phi} \colon E^\bullet_G \to \sE_G$ be the $G$-equivariant map
induced on the $G$-torsors by $\phi \colon B^\bullet_G \to \sB_G$ and let
$\ov{\phi}_X = (\id_X, \ov{\phi}) \colon X \times  E^\bullet_G \to  X \times \sE_G$.
We now note that $\tau^*_X \colon K_G(X) \to K(X \stackrel{G}{\times} E^\bullet_G)$
is defined similarly to $\pr^*_2 \colon  K_G(X) \to
K(X \stackrel{G}{\times} E^\bullet_F)$ where we replace $F$ by $G$
(cf, ~\eqref{eqn:Simplicial-K-1}).
By repeating the argument of \lemref{lem:K-simplex-8} with $F$ replaced
by $G$ (and $\sC_{E^n_F}$ replaced by the category having one object
$G^{n+1}$ and the identity map), we get the following.
 
\begin{lem}\label{lem:K-simplex-13}
  The diagram
  \begin{equation}\label{eqn:K-simplex-14}
    \xymatrix@C1pc{
      K_G(X) \ar[r]^-{\pr^*_1} \ar[dr]_-{\tau^*_X} &
      K\left(X \stackrel{G}{\times}\sE_G\right)
      \ar[d]^-{\ov{\phi}^*_X} \\
      & K\left(X \stackrel{G}{\times} E^\bullet_G\right)}
  \end{equation}
  is commutative.
\end{lem}

\subsection{Main result for special groups}\label{sec:Thom-integral}
We let $k$ be any field and $G$ any $k$-group.
We shall use the following lemma in the proofs of our main results.

\begin{lem}\label{lem:Simplicial-sheaf-0}
  Given a proper map $f \colon Z \to X$ and a flat
  map $g \colon U \to X$ in $\Spc^G_k$, the diagram of spectra
  \begin{equation}\label{eqn:Simplicial-sheaf-1}
    \xymatrix@C1pc{
      K'_G(Z) \ar[r]^-{f_*} \ar[d]_-{\wt{\pi}^*_Z} &
      K'_G(X) \ar[r]^-{g^*} \ar[d]^-{\wt{\pi}^*_X} &  K'_G(U) \ar[d]^-{\wt{\pi}^*_U} \\
      K'(Z^\bullet_G) \ar[r]^-{f_*} &  K'(X^\bullet_G) \ar[r]^-{g^*} &  K'(U^\bullet_G)}
  \end{equation}
  is commutative.
\end{lem}
\begin{proof}
  We first note that $f_*$ on the bottom level
  is defined by \cite[Lem.~5.3]{Krishna-Crelle}.
  The commutativity of the right square is clear because all face and
  degeneracy maps of the bar
  constructions are all smooth or regular closed immersions.
  To see that the left square is also commutative,
  we only need to observe that the square
  \begin{equation}\label{eqn:Simplicial-sheaf-2}
    \xymatrix@C1pc{
      Z^\bullet_G \ar[r]^-{f} \ar[d]_-{\wt{\pi}_Z} & X^\bullet_G \ar[d]^-{\wt{\pi}_X} \\
      [Z/G] \ar[r]^-{f} & [X/G]}
  \end{equation} is levelwise Cartesian in which the vertical arrows are levelwise
  smooth and the horizontal arrows are levelwise proper.
  The desired commutativity therefore follows from \cite[Prop.~3.8]{Khan-JJM}.
\end{proof}

We are now ready to prove the main result of this section.

\begin{thm}\label{thm:K-simplex-15}
 Assume that $G$ is special. Then for any $X \in \Spc^G_k$, the pull-back map
  $\wt{\pi}^* \colon K'_G(X) \to K'(X^\bullet_G)$ induces a weak equivalence of spectra
  \begin{equation}\label{eqn:K-simplex-16}
    \wt{\pi}^* \colon K'_G(X)^{\compl}_{I_G} \xrightarrow{\simeq} K'(X^\bullet_G).
  \end{equation}
  \end{thm}
\begin{proof}
  We first reduce the proof of the theorem to the case where $X$ is a smooth
  quasi-projective $k$-scheme using Noetherian induction.
  If $X$ is a $G$-quasi-projective $k$-scheme,
  we can find a $G$-equivariant closed embedding
  $\iota \colon X \inj X'$, where $X'$ is smooth and $G$-quasi-projective.
  We let $j \colon U \inj X'$ be the open complement
  and consider the diagram of spectra
\begin{equation}\label{eqn:K-simplex-16-0}
  \xymatrix@C1pc{
    K'_G(X)^{\compl}_{I_G} \ar[r]^-{\iota_*} \ar[d]_-{\wt{\pi}^*_X} &
    K'_G(X')^{\compl}_{I_G} \ar[r]^-{j^*} \ar[d]^-{\wt{\pi}^*_{X'}} &
    K'_G(U)^{\compl}_{I_G} \ar[d]^-{\wt{\pi}^*_U} \\
    K'(X^\bullet_G) \ar[r]^-{\iota_*} & K'(X'^\bullet_G) \ar[r]^-{j^*} &
    K'(U^\bullet_G)}
  \end{equation}
which is commutative by \lemref{lem:Simplicial-sheaf-0}.

The top row of the above diagram is a homotopy fiber
sequence because it is obtained by
applying the derived completion functor to the localization sequence of equivariant
$K'$-theory. The bottom row is a homotopy fiber sequence because it is obtained by
applying the homotopy limit functor to the localization sequences of the ordinary
$K'$-theory. The middle and the right vertical arrows are weak equivalences
if the theorem is true for smooth $G$-quasi-projective $k$-schemes. It follows that
${\wt{\pi}^*_X}$ is a weak equivalence.

We now suppose that $X$ is an arbitrary object of $\Spc^G_k$. We can assume $X$ to be
reduced. By
  \lemref{lem:Open-quasi-proj}, we can find a $G$-invariant dense open subspace
  $U \subset X$ which is a $G$-quasi-projective $k$-scheme.
  In particular, the assertion is true for $U$ as shown above.
  We now let $Z = X \setminus U$
  and repeat the above argument using the localization sequence and Noetherian
  induction to conclude that ${\wt{\pi}^*_X}$ is a weak equivalence. It remains now
  to show that the theorem is true when $X$ is a smooth and $G$-quasi-projective
  $k$-scheme. We assume this to be the case in the rest of the proof
  and proceed as follows.

We look at the diagram
    \begin{equation}\label{eqn:K-simplex-17}
      \xymatrix@C2pc{
        & & K_G(X) \ar[d]^-{\pr^*_1} \ar[dr]^-{\tau^*_X} \ar[dl]^-{\pr^*_2}
        \ar[dll]_-{p^*_X} \ar[r]^-{\wt{\pi}^*} &  K(X^\bullet_G) \ar[d]^-{\pi^*_X} \\
        K\left(X \stackrel{G}{\times} F\right) &
        K\left(X \stackrel{G}{\times} E^\bullet_F\right)
        \ar[l]^-{\Delta^*_{F,X}} & K\left(X  \stackrel{G}{\times} \sE_G\right)
        \ar[l]^-{(\phi^X_{F})^*}
        \ar[r]_-{\ov{\phi}^*_X} &  K_G\left(E^\bullet_G(X)\right).}
      \end{equation}
    Using  ~\eqref{eqn:AST-2} together with Lemmas~\ref{lem:Simplicial-K-0},
    ~\ref{lem:K-simplex-8} and ~\ref{lem:K-simplex-13}, it follows that this
    diagram is commutative.

 We next note that all arrows in the above diagram are induced by $G$-equivariant
    maps between schemes with
    $G$-actions where $G$ acts trivially on each of the spaces
    $X \stackrel{G}{\times} F,  \ X \stackrel{G}{\times} E^n_F, \
    X \stackrel{G}{\times} \sE^n_G$ and $X \stackrel{G}{\times} E^n_G$.
    In particular, all spectra on the bottom row are homotopy limits of
    $I_G$-complete $K(BG)$-module spectra. 
It follows from \lemref{lem:Noether-Rep} and \propref{prop:Compln_D}(7) that
    each of the spectra on the bottom row is $I_G$-complete. By
    ~\eqref{eqn:Simplicial-K-3}, the same holds for $K(X^\bullet_G)$. We thus get a
    commutative diagram of $I_G$-complete spectra 
\begin{equation}\label{eqn:K-simplex-18}
      \xymatrix@C2pc{
        & & K_G(X)^{\compl}_{I_G} \ar[d]^-{\pr^*_1} \ar[dr]^-{\tau^*_X} \ar[dl]^-{\pr^*_2}
        \ar[dll]_-{p^*_X} \ar[r]^-{\wt{\pi}^*} &  K(X^\bullet_G) \ar[d]^-{\pi^*_X} \\
        K\left(X \stackrel{G}{\times} F\right) &
        K\left(X \stackrel{G}{\times} E^\bullet_F\right)
        \ar[l]^-{\Delta^*_{F, X}} & K\left(X  \stackrel{G}{\times} \sE_G\right)
        \ar[l]^-{(\phi^X_{F})^*} \ar[r]_-{\ov{\phi}^*_X} &  K_G(E^\bullet_G(X)).}
      \end{equation}

The map $\pi^*_X$ is a weak equivalence by ~\eqref{eqn:Simplicial-sheaf-0}.
The maps $(\phi^X_{F})^*$ and $\ov{\phi}^*_X$ are weak equivalences by
  ~\eqref{eqn:Bar-ind-Gen-1}.
It follows from the proof of \propref{prop:Bar-ind-Gen} that $\Delta^*_{F,X}$ is a weak
equivalence. Finally, we note that $p^*_X$ is the same map as the Atiyah-Segal
completion map $\Phi_{[X/G]}$ of ~\eqref{eqn:AS-map-1}. In particular, it is a
weak equivalence by \thmref{thm:AS-Main}. We conclude that $\wt{\pi}^*$ is a weak
equivalence. This finishes the proof of the theorem.
\end{proof}

\begin{remk}\label{remk:Bott-Thom}
  Using the fact that the Bott-inverted ordinary $K$-theory of smooth $k$-schemes
  is representable in $\sS\sH(k)$ (e.g., see \cite{BEO}),
  one easily checks from its proof that \thmref{thm:K-simplex-15} also
  holds for the Bott-inverted equivariant $K'$-theory with finite coefficients
  (prime to $\Char(k)$) if $k$ contains all roots of unity.
  \end{remk}

\section{Thomason's problem for other groups}\label{sec:Thom-Gen**}
In this section, we shall extend \thmref{thm:K-simplex-15} to other $k$-groups.
For equivariant $K$-theory of spaces with group actions, this is usually
done by reducing to the case of special groups via Morita equivalence.
However, one can not expect a version of Morita equivalence for the integral
$K$-theory of the bar construction on account of its failure to
satisfy {\'e}tale descent. We shall show however that Morita equivalence holds for
the rationalized $K$-theory of the bar construction and use this to deduce 
\thmref{thm:K-simplex-15} for other groups. Note that  one can't expect
\thmref{thm:K-simplex-15} to hold for all $k$-groups with integral coefficients
due to the failure of Galois descent for $K$-theory of fields.

\subsection{Morita equivalence for $K$-theory of bar construction}
\label{sec:Bar-Morita}
Let $k$ be a field and $G$ a $k$-group. We let $p \colon G \times Y \to Y$ be the
  trivial $G$-torsor over $Y \in \Spc_k$ and let $X = G \times Y$. We let $X^\bullet_G$
  denote the bar construction for the $G$-action on $X$ where $G$ acts on itself
  by left multiplication and trivially on $Y$. The projection map $p$
  defines an augmentation $p \colon X^\bullet_G \to Y$ of the simplicial algebraic
  space $X^\bullet_G$. We refer the reader to \cite{BKR} for the homotopic terminology
  that we shall use about simplicial objects in any category.

\begin{lem}\label{lem:K-simplex-19}
    There is a simplicial morphism $q \colon Y \to X^\bullet_G$ such that $p \circ q =
    \id_Y$ and $q \circ p$ is homotopic to $\id_{X^\bullet_G}$.
  \end{lem}
\begin{proof}
    Letting $q_n(y) = (e,y)$ where $e$ is the identity element of $G^{n+1}$, it is
    clear that $q$ is a simplicial map such that $p \circ q = \id_Y$ and
    $q_n = (s^0)^n \circ q_0$ for $n \ge 0$, where $s^0$ is the zeroth degeneracy map
    of $X^\bullet_G$. To construct
    a simplicial homotopy between $q \circ p$ and $\id_{X^\bullet_G}$, we note that
  \[
    \begin{tikzcd}
    G \times X \ar[r,shift left=.75ex,"p^1_1"]
    \ar[r,shift right=.75ex,swap,"p^0_1"] & X \ar[r,"p_0"] & Y
    \end{tikzcd}
    \]
    is a coequalizer diagram. By \cite[Thm.~4.3]{BKR} therefore, it suffices to
    show that $X^\bullet_G$ admits an extra degeneracy. But we leave it for the reader
    to check that the map
    \begin{equation}\label{eqn:K-simplex-20}
      s_n \colon G^n \times X \cong G^n \times (G \times Y) \to G^{n+1} \times
      (G \times Y) \cong G^{n+1} \times X;
    \end{equation}
    \[
    \ \ s_n(g_1, \ldots , g_n, g_0,y) =
      (g_0, g_1, \ldots , g_n, e, y); \ n \ge 0
    \]
    is the required extra degeneracy. 
\end{proof}
    
\begin{lem}\label{lem:K-simplex-21}
   Let $p \colon X = G \times Y \to Y$ be the trivial $G$-torsor as above. Then the
   flat pull-back $p^* \colon K'(Y) \to \holim_{[n] \in \Delta}
   K'(G^n \times X)$ is a weak equivalence.
 \end{lem}
 \begin{proof}
   By \lemref{lem:K-simplex-19}, there are morphisms of simplicial algebraic spaces
   $Y \xrightarrow{q} X^\bullet_G \xrightarrow{p} Y$ such that $p \circ q = \id_Y$ and
   $q \circ p$ is homotopic to $\id_{X^\bullet_G}$. Applying the $K'$-theory functor,
   we get morphisms of cosimplicial spectra $K'(Y) \xrightarrow{p^*} K'(X^\bullet_G)
   \xrightarrow{q^*} K'(Y)$ such that $q^* \circ p^*$ is identity and $p^* \circ q^*$
   is homotopic to identity
   (see the last part of the proof of \cite[Lem~2.9]{Thomason-Inv}).
   It follows that the induced map between the homotopy
   limits $p^* \colon K'(Y) \to \holim_{[n] \in \Delta}
   K'(G^n \times X)$ is a weak equivalence.
\end{proof}

\begin{lem}\label{lem:K-simplex-22}
  Let $p \colon X \to Y$ be any $G$-torsor over $Y \in \Spc_k$. Then the pull-back
  map $p^* \colon K'(Y, \Q) \to  \holim_{[n] \in \Delta} K'(G^n \times X, \Q)$ is a
   weak equivalence. If $\Char(k) \nmid m$ and $k$ contains all roots of unity,
   then the pull-back map
   \[
   p^* \colon K'(Y, {\Z}/m)[\beta^{-1}] \to \holim_{[n] \in \Delta}
   (K'(G^n \times X, {\Z}/m)[\beta^{-1}])
   \]
   is a weak equivalence.
 \end{lem}
\begin{proof}
   We can find an {\'e}tale cover $f \colon Y' \to Y$ such that the pull-back
   $p' \colon X \times_Y Y' \to Y'$ is a trivial $G$-torsor. We let
   $X' = X \times_Y Y'$ and let $f' \colon X' \to X$ be the projection.
   It is straightforward to check that the map $G^n \times X' \to G^n \times X$
   is an {\'e}tale cover for every $n \ge 0$.
   We let $C_\bullet(Y')$ (resp. $C_\bullet(G^n \times X')$) denote the \u{C}ech
   nerve of the {\'e}tale cover $Y' \to Y$ (resp. $G^n \times X' \to G^n \times X$).
    We then get a commutative diagram  of bisimplicial algebraic spaces
\begin{equation}\label{eqn:K-simplex-23}
  \xymatrix@C1pc{
    C_\bullet(X'^\bullet_G) \ar[r]^-{f'} \ar[d]_-{p'} & X^\bullet_G \ar[d]^-{p} \\
    C_\bullet(Y') \ar[r]^-{f} & Y.}
\end{equation}

Applying the $K'$-theory functor, we get a commutative diagram of spectra
\begin{equation}\label{eqn:K-simplex-24}
  \xymatrix@C1pc{
    K'(Y, \Q) \ar[r]^-{p^*} \ar[d]_-{f^*} & \holim_{[n] \in \Delta}
    K'(G^n \times X, \Q)
    \ar[d]^-{f'^*} \\
    \holim_{[m] \in \Delta} K'(Y'^m, \Q) \ar[r]^-{p'^*} &
    \holim_{[n] \in \Delta} \holim_{[m] \in \Delta} K'(G^n \times X'^m, \Q),}
\end{equation}
where $G^n \times X'^m$ (resp. $Y'^m$) is the $m$-th level of
$C_\bullet(G^n \times X')$ (resp. $C_\bullet(Y')$). 
By \cite[Thm.~2.15, Cor.~2.16]{Thomason-ENS}, the rationalized $K'$-theory satisfies
{\'e}tale descent for schemes. Since an algebraic space has a dense
open subspace which is a scheme, it follows from the localization theorem for
$K'$-theory that the rationalized $K'$-theory satisfies {\'e}tale descent for
algebraic spaces (cf. \cite[Thm.~5.1]{Khan-JJM}). In particular, the
left vertical arrow in ~\eqref{eqn:K-simplex-24} is a weak equivalence.
By the same reason, the map $K'(G^n \times X, \Q) \to \holim_{[m] \in \Delta}
K'(G^n \times X'^m, \Q)$ is a weak equivalence for every $n \ge 0$.
It follows by \cite[Thm.~11.4]{Hirsch-Notes} and \cite[\S~6]{Hovey-JPAA} that
the right vertical arrow in ~\eqref{eqn:K-simplex-24} is also a weak equivalence.

We next note that on one hand, the spectrum on the bottom right corner of 
~\eqref{eqn:K-simplex-24} is canonically weakly equivalent to
$\holim_{[m] \in \Delta} \holim_{[n] \in \Delta}  K'(G^n \times X'^m, \Q)$ by
\cite[\S~31.5]{Chacholski-Scherer}. On the other hand, the map
$K'(Y'^m, \Q) \to \holim_{[n] \in \Delta}  K'(G^n \times X'^m, \Q)$ is a weak
equivalence for every $m \ge 0$
by \lemref{lem:K-simplex-21} since $X'^m \to Y'^m$ is a trivial
$G$-torsor. It follows again by \cite[Thm.~11.4]{Hirsch-Notes} and
\cite[\S~6]{Hovey-JPAA} that the bottom horizontal arrow in ~\eqref{eqn:K-simplex-24}
is a weak equivalence. We conclude that the top horizontal arrow is also a weak
equivalence. This finishes the proof of the first part of the lemma.

The proof of the second part is identical using the fact that the Bott-periodic
$K'$-theory with finite coefficients (prime to $\Char(k)$)
satisfies {\'e}tale descent by
\cite{Thomason-ENS} in a special case and \cite[Cor.~13.1]{Levine-Notes}
in the general case (see also \cite{BEO}).
\end{proof}

\begin{lem}\label{lem:Bar-Morita}
   Let $H \subset G$ be a closed subgroup and $X \in \Spc^H_k$. Let
   $Y = X \stackrel{H}{\times} G$ be the associated Morita space.
   Let $G'= G \times H$ act on $X'= X \times G$ by
   $(g,h)\star(x,g') = (hx, gg'h^{-1})$. Then the projection map
   $p_1 \colon X' \to X$ and the $H$-torsor map $p_2 \colon X' \to Y$ induce
   weak equivalences of spectra
   \begin{equation}\label{eqn:Bar-Morita-0}
     p^*_1 \colon K'(X^\bullet_H, \Q) \xrightarrow{\simeq} K'(X'^\bullet_{G'}, \Q) \ \
     \mbox{and} \ \ p^*_2 \colon K'(Y^\bullet_G, \Q) \xrightarrow{\simeq}
     K'(X'^\bullet_{G'}, \Q).
   \end{equation}
If $\Char(k) \nmid m$ and $k$ contains all roots of unity, then the same is also
   true for the Bott-periodic $K'$-theory $K'(-,{\Z}/m)[\beta^{-1}]$.
\end{lem}
\begin{proof}
   For $p, q \ge 0$, we let $A_{p,q} = X \times G^{p+1} \times H^q$.
   We let the face maps $d^I_i \colon A_{p,q} \to A_{p-1,q}$ and
   $d^{II}_i \colon A_{p,q} \to A_{p,q-1}$ be given by
   \begin{equation}\label{eqn:Bar-Morita-1}
     d^I_i(x, g_0, \ldots , g_p, h_1, \ldots , h_{q}) =
     \left\{\begin{array}{ll}
     (x, g_0, \ldots , g_{i-1}, g_{i+1}g_{i}, g_{i+2}, \ldots , g_p, h_1, \ldots , h_q)
     & \mbox{if $0 \le i < p$} \\
     (x, g_0, \ldots, g_{p-1}, h_1, \ldots, h_q)  & \mbox{if $i = p$.}
     \end{array}
     \right.
   \end{equation}
   \[
   d^{II}_i(x, g_0, \ldots , g_p, h_1, \ldots , h_q) =
   \left\{\begin{array}{ll}
   (h_1x, g_0h^{-1}_1, g_1, \ldots , g_p, h_2, \ldots , h_q)
   & \mbox{if $i = 0$} \\
    (x, g_0, \ldots , g_p, h_1, \ldots , h_{i-1}, h_{i+1}h_i, \ldots , h_q)
     & \mbox{if $0 < i < q$} \\
     (x, g_0, \ldots, g_{p}, h_1, \ldots, h_{q-1})  & \mbox{if $i = q$.}
     \end{array}
     \right.
     \]
     The degeneracy maps $s^I_i \colon A_{p,q} \to A_{p+1,q}$ and
     $s^{II}_i \colon A_{p,q} \to A_{p,q+1}$ are given by
     \begin{equation}\label{eqn:Bar-Morita-2}
       s^I_i(x, g_0, \ldots , g_p, h_1, \ldots , h_{q}) =
       \left\{\begin{array}{ll}
      (x, g_0, \ldots , g_{i}, e, g_{i+1}, \ldots , g_p, h_1, \ldots , h_{q})  &
       \mbox{if $0 \le i < p$} \\
       (x, g_0, \ldots , g_{p},e, h_1, \ldots , h_{q})  &
       \mbox{if $i = p$.}
       \end{array}
     \right.
\end{equation}
     \[
s^{II}_i(x, g_0, \ldots , g_p, h_1, \ldots , h_q) = 
\left\{\begin{array}{ll}
 (x, g_0, \ldots , g_{p}, e, h_1, \ldots , h_{q})  &
       \mbox{if $i = 0$} \\
      (x, g_0, \ldots , g_{p}, h_1, \ldots h_i, e, h_{i+1}, \ldots , h_{q})  &
       \mbox{if $0 < i < q$} \\
       (x, g_0, \ldots , g_{p}, h_1, \ldots , h_{q}, e)  &
       \mbox{if $i = q$.}
       \end{array}
     \right.
\]
     
It is a straightforward verification that $A_{\bullet, \bullet} = \{A_{p,q}\}_{p,q \ge 0}$
together with the above face and degeneracy maps is a bisimplicial algebraic space. 
Note that one only needs to check the commutativity between $d^I_0$ and $d^{II}_0$ as
the other identities are automatic by the above definition.

We next note that the maps $p_1$ and $p_2$ define augmentation maps
$\epsilon_1 \colon A_{\bullet, \bullet} \to X^\bullet_H$ and $\epsilon_2 \colon 
A_{\bullet, \bullet} \to Y^\bullet_G$ such that for each $q \ge 0$, the map
$\epsilon_1 \colon A_{\bullet,q} \to X \times H^q$ coincides with the augmentation
$(G \times (X\times H^q))^\bullet_G \to X\times H^q$ corresponding to the trivial
$G$-torsor $G \times (X\times H^q) \to X\times H^q$.
Furthermore, the map $\epsilon_1 \colon A_{0, \bullet} \to X^\bullet_H$ is the
map induced between the bar constructions by the $H$-equivariant map
$G \times X \to X$. 
On the other side,  the map $\epsilon_2 \colon A_{p,\bullet} \to G^p \times Y$
coincides with the augmentation
$(G^p \times (X\times G))^\bullet_H \to G^p \times Y$ corresponding to the $H$-torsor
$G^p \times (X\times G) \to G^p \times Y$. Furthermore, the map
$\epsilon_2 \colon A_{\bullet,0} \to Y^\bullet_G$ is the map induced between the bar
constructions by the $G$-equivariant map $G \times X \to Y$. 
We now apply \lemref{lem:K-simplex-22} to conclude that there are weak equivalences
of spectra
\begin{equation}\label{eqn:Bar-Morita-3}
  \epsilon^*_1 \colon K'(X^\bullet_H, \Q) \xrightarrow{\simeq}
  \holim_{p,q} K'(A_{p,q}, \Q) \xleftarrow{\simeq} K'(Y^\bullet_G, \Q) \colon \epsilon^*_2.
  \end{equation}

On the other hand, if we let $K_{p,q} = K'(A_{p,q}, \Q)$ with the coface and
codegeneracy
maps induced by those of $A_{\bullet, \bullet}$ and let $\Delta_K = \{K_{p,p}\}_{p \ge 0}$
denote the diagonal of the bicosimplicial spectrum $K_{\bullet, \bullet}$, then the
left homotopy cofinality of the diagonal
inclusion $\Delta_K \inj \Delta^2_K$ implies
that this map
induces a weak equivalence on the homotopy limits
(cf. \cite[Thm.~13.7]{Hirsch-Notes}, \cite[Prop.~7.11]{Hirsch-JHRS}).
We thus get weak equivalences
\begin{equation}\label{eqn:Bar-Morita-4}
  \epsilon^*_1 \colon K'(X^\bullet_H, \Q) \xrightarrow{\simeq}
  \holim_{p} K'(A_{p,p}, \Q) \xleftarrow{\simeq} K'(Y^\bullet_G, \Q) \colon \epsilon^*_2.
  \end{equation}
This implies ~\eqref{eqn:Bar-Morita-0} and concludes the proof of the first part of
the lemma because $\Delta_K$ is nothing but the cosimplicial spectrum
$[p] \mapsto K'(X' \times G'^p, \Q)$ obtained by applying the $K'$-theory functor on
the bar construction associated to the $G'$-action on $X'$.
The proof of the second part is identical. 
\end{proof}

\subsection{Main result for other groups}\label{sec:Thom-rational}
The following result extends \thmref{thm:K-simplex-15} to all groups if we use
rational coefficients. We fix a field $k$.

\begin{thm}\label{thm:Bar-K-rational}
  Let $G$ be a $k$-group and $X \in \Spc^G_k$. Then the
  map $\wt{\pi}^* \colon K'_G(X,\Q) \to K'(X^\bullet_G, \Q)$ induces a weak
  equivalence of spectra
  \begin{equation}\label{eqn:Bar-K-rational-0}
    \wt{\pi}^* \colon K'_G(X, \Q)^{\compl}_{I_G} \xrightarrow{\simeq} K'(X^\bullet_G, \Q).
  \end{equation}
If $k$ contains all roots of unity, then the same is also
true for the Bott-periodic $K'$-theory $K'(-;{\Z}/m)[\beta^{-1}]$ for every
integer $m$ invertible in $k$.
 \end{thm}
\begin{proof}
  We choose a closed embedding $G \inj Q$ where $Q$ is a general linear $k$-group and
  let $G' = Q \times G$. We let $Y$ be the Morita space $X \stackrel{G}{\times}Q$ and
 $X' = X \times Q$ with the action of $G'$ as described in
  \lemref{lem:Bar-Morita}. Following the notations of that lemma, we let
  $D_\bullet \inj A_{\bullet, \bullet}$ denote the diagonal simplicial algebraic space
  and consider the commutative diagram
  \begin{equation}\label{eqn:Bar-K-rational-1}
    \xymatrix@C1pc{
      X {\times} E^\bullet_G \ar[r] &
      X \stackrel{G}{\times} E^\bullet_G \ar[rr]^-{\cong} & & X^\bullet_G \\
      X' {\times} E^\bullet_{G'}  \ar[r] \ar[d]_-{\theta_2} \ar[u]^-{\theta_1}
      & X' \stackrel{G'}{\times} E^\bullet_{G'} \ar[d]^-{\theta'_2}
      \ar[u]_-{\theta'_1} \ar[r]^-{\cong} &
      D_\bullet \ar[r]^-{\delta} \ar[ur]^-{\delta_1} 
      \ar[dr]_-{\delta_2} & A_{\bullet, \bullet}
      \ar[u]_-{\epsilon_1} \ar[d]^-{\epsilon_2} \\
     Y {\times} E^\bullet_Q \ar[r] &
      Y \stackrel{Q}{\times} E^\bullet_Q \ar[rr]^-{\cong} & & Y^\bullet_Q,}
    \end{equation}
  in which $\theta_i, \ \theta'_i$ and $\delta_i := \epsilon_i \circ \delta$ are
  morphisms of simplicial spaces and $\epsilon_i$ is a morphism of bisimplicial
  spaces for $i = 1,2$. One also checks that $\theta_i$ is $G'$-equivariant for
  $i = 1,2$, where $G'$ acts on $X {\times} E^\bullet_G$ via $G' \surj G$ and
  on $Y {\times} E^\bullet_Q$ via $G' \surj Q$.

 Applying the $K'$-theory functor, we get a commutative diagram of spectra
  \begin{equation}\label{eqn:Bar-K-rational-2}
    \xymatrix@C1pc{
  K'_G(X {\times} E^\bullet_G, \Q) \ar[d]_-{\theta^*_1} &
  K'\left(X \stackrel{G}{\times} E^\bullet_G, \Q\right)
  \ar[l]_-{\simeq} \ar[d]^-{\theta'^*_1} & &
  K'(X^\bullet_G, \Q) \ar[ll]_-{\simeq} \ar[d]^-{\epsilon^*_1} \ar[dl]_-{\delta^*_1} \\
      K'_{G'}(X' {\times} E^\bullet_{G'}, \Q)
      & K'\left(X' \stackrel{G'}{\times} E^\bullet_{G'}, \Q\right)
      \ar[l]_-{\simeq}  &
      K'(D_\bullet, \Q)  \ar[l]_-{\simeq} & K'(A_{\bullet, \bullet}, \Q) \ar[l]_-{\delta^*}
      \\
     K'_Q(Y {\times} E^\bullet_Q, \Q) \ar[u]^-{\theta^*_2} &
     K'\left(Y \stackrel{Q}{\times} E^\bullet_Q, \Q\right)
     \ar[u]_-{\theta'^*_2} \ar[l]_-{\simeq}
      & & K'(Y^\bullet_Q, \Q)
     \ar[ul]^-{\delta^*_2} \ar[ll]_-{\simeq} \ar[u]_-{\epsilon^*_2} }
    \end{equation}

 The arrows $\delta^*_i, \epsilon^*_i$ and $\delta^*$ in ~\eqref{eqn:Bar-K-rational-2}
 are weak equivalences by \lemref{lem:Bar-Morita}.
In particular, $\theta'^*_i$ and $\theta^*_i$ are weak equivalences.
Furthermore, each arrow
  $\theta^*_i$ (hence $\theta'^*_i$) is $K(BG')$-linear as $\theta_i$ is
  $G'$-equivariant. We thus get a commutative diagram of $K(BG')$-linear maps
  of $K(BG')$-module spectra
   \begin{equation}\label{eqn:Bar-K-rational-3}
    \xymatrix@C1pc{
      K'_G(X, \Q) \ar[r]^-{\theta'^*_1} \ar[d]_{\wt{\pi}^*_X} & K'_{G'}(X', \Q)
      \ar[d]^-{\wt{\pi}^*_{X'}}
      & K'_Q(Y, \Q) \ar[d]^-{\wt{\pi}^*_Y} \ar[l]_-{\theta'^*_2} \\
      K'(X^\bullet_G, \Q) \ar[r]^-{\theta'^*_1} & K'(X'^\bullet_{G'}, \Q) &
      K'(Y^\bullet_Q, \Q) \ar[l]_-{\theta'^*_2}}
    \end{equation}
   in which the horizontal arrows on the bottom row are weak equivalences.
   The horizontal arrows on the top row are weak equivalences 
because the maps of stacks $[{X'}/{G'}] \to [X/G]$ and $[{X'}/{G'}] \to [Y/Q]$
are isomorphisms (cf. \lemref{lem:Morita*}).

As all arrows in
~\eqref{eqn:Bar-K-rational-3} are $K(BG')$-linear and we showed in the
proof of \thmref{thm:K-simplex-15} that all terms on the bottom row are
$I_{G'}$-complete, we obtain a commutative diagram of $I_{G'}$-complete
spectra
\begin{equation}\label{eqn:Bar-K-rational-4}
    \xymatrix@C1pc{
      K'_G(X, \Q)^{\compl}_{I_{G'}} \ar[r]^-{\theta'^*_1} \ar[d]_{\wt{\pi}^*_X} &
      K'_{G'}(X', \Q)^{\compl}_{I_{G'}} \ar[d]^-{\wt{\pi}^*_{X'}}
      & K'_Q(Y, \Q)^{\compl}_{I_{G'}} \ar[d]^-{\wt{\pi}^*_Y} \ar[l]_-{\theta'^*_2} \\
      K'(X^\bullet_G, \Q) \ar[r]^-{\theta'^*_1} & K'(X'^\bullet_{G'}, \Q) &
      K'(Y^\bullet_Q, \Q) \ar[l]_-{\theta'^*_2}}
    \end{equation}
in which all horizontal arrows are weak equivalences.

It follows from \cite[Cor.~6.1]{EG-RR} that $I_{G} = \sqrt{I_{G'}R(G)}$ and
$I_{Q} = \sqrt{I_{G'}R(Q)}$. On the other hand, $I_{G'}$ is the radical of a
finitely generated ideal by \lemref{lem:Noether-Rep}. It follows from
\propref{prop:Compln_D}(5) that the canonical maps
$K'_G(X,\Q)^{\compl}_{I_{G'}} \to  K'_G(X,\Q)^{\compl}_{I_{G}}$
and $K'_Q(Y, \Q)^{\compl}_{I_{G'}} \to K'_Q(Y, \Q)^{\compl}_{I_{Q}}$ are weak equivalences.
In particular, $\wt{\pi}^*_Y$ is a weak equivalence by \thmref{thm:K-simplex-15}.
A diagram chase now shows that $\wt{\pi}^*_{X'}$ and $\wt{\pi}^*_X$ are
also weak equivalences. Equivalently, the map in ~\eqref{eqn:Bar-K-rational-0}
is a weak equivalence. This concludes the proof of the first part of the theorem
and the proof of the second part is identical using Remark~\ref{remk:Bott-Thom}. 
\end{proof}

\begin{cor}\label{cor:KH-Thom*}
  Under the hypotheses of \thmref{thm:AS-KH}, the following hold.
  \begin{enumerate}
  \item
    If $G$ is special, then the map
  \[
  \wt{\pi}^* \colon KH_G(X)^{\compl}_{I_{G}} \to KH(X^\bullet_G)
  \]
  is a weak equivalence.
\item
  If $G$ is arbitrary, then the map
   \[
  \wt{\pi}^* \colon KH_G(X, \Q)^{\compl}_{I_{G}} \to KH(X^\bullet_G, \Q)
  \]
  is a weak equivalence.
\item
  If $G$ is arbitrary, $k$ contains all roots of unity and $m \in k^\times$, then
the map
   \[
   \wt{\pi}^* \colon K_G(X, {\Z}/m)[\beta^{-1}]^{\compl}_{I_{G}} \to
   K(X^\bullet_G, {\Z}/m)[\beta^{-1}]
  \]
  is a weak equivalence.
  \end{enumerate}
\end{cor}
  \begin{proof}
    Using Theorems~\ref{thm:K-simplex-15} and ~\ref{thm:Bar-K-rational},
    we can repeat the arguments of \thmref{thm:AS-KH} verbatim to deduce this
    corollary.
    \end{proof}

\begin{cor}\label{cor:Thom-maximal}
  Assume that $k$ is an algebraically closed field of characteristic zero.
  Let $G$ be a $k$-group
  and $g \in G_s(k)$. Let $\Psi \subset G(k)$ be the conjugacy class of $g$. For
  any $X \in \Spc^G_k$, we then have a weak equivalence of spectra
  \begin{equation}\label{eqn:Thom-maximal-0}
    \Theta^g_X \colon K'_G(X, k)^{\compl}_{\fm_\Psi} \xrightarrow{\simeq}
    K'((X^g)^\bullet_{Z_g}, k).
  \end{equation}
 \end{cor}
\begin{proof}
  We use the maps
  \[
   K'_G(X, k)^{\compl}_{\fm_\Psi} \xrightarrow{\Psi^g_X} 
   K'_{Z_g}(X^g, k)^{\compl}_{\fm_g} \xrightarrow{t_{g^{-1}}}
   K'_{Z_g}(X^g, k)^{\compl}_{I_{Z_g}} \xrightarrow{\wt{\pi}^*}
   K'((X^g)^\bullet_{Z_g}, k),
   \]
   where $\wt{\pi}^*$ is as in ~\eqref{eqn:Bar-EG-main-0} with $X$ (resp. $G$)
   substituted by $X^g$ (resp. $Z_g$).
   The first two maps were shown to be weak equivalences in \thmref{thm:Gen-AS-max}
   and the last map is a weak equivalence by applying \thmref{thm:Bar-K-rational}
   to the $Z_g$-action on $X^g$ and noting that the theorem is valid if we replace
   $\Q$ by any field of characteristic zero.
   \end{proof}

\begin{remk}\label{remk:KH-Thom*}
  One easily checks that \corref{cor:Thom-maximal} also holds for $KH$-theory
  of singular $k$-schemes under the additional hypotheses of \thmref{thm:AS-KH}.
  \end{remk}

{\bf {Proof of \corref{cor:Main-4}:}}
Under our assumptions, the canonical maps $K_G(X; {\Z}/n) \to KH_G(X; {\Z}/n)$
and $K(X_G; {\Z}/n) \to KH(X_G; {\Z}/n)$ are weak equivalences by
\cite[Cor.~3.17, Prop.~7.15]{Hoyois-Krishna}.
The result therefore follows from \thmref{thm:AS-KH} and \corref{cor:KH-Thom*}.
\qed

\vskip .2cm

{\bf{Proof of \corref{cor:Main-1}.}}
If $f \colon X \to Y$ is an $\A^1$-weak equivalence, then $(\id \times f) \colon
G^n \times X \to G^n \times Y$ are also $\A^1$-weak equivalences for all $n$ by
\cite[Lem.~3.2.13]{Morel-Voevodsky}. It follows that the induced maps
$K(G^n \times Y) \to K(G^n \times X)$ are weak equivalences for all $n$. As they are
clearly compatible with the face and degeneracy maps of $X^\bullet_G$ and $Y^\bullet_G$,
we can pass to the limits to conclude that the map
$f^* \colon K(Y^\bullet_G) \to K(X^\bullet_G)$ is a weak equivalence.
We now apply Theorems~\ref{thm:K-simplex-15} and ~\ref{thm:Bar-K-rational} to
finish the proof. 
 \qed

\section{Completion of Hochschild homology}\label{sec:Hoch-prop}
Our goal in the next few sections is to prove the analogue of the
completion theorem (cf. \corref{cor:Thom-maximal})
for the equivariant homology theories. As in the case of $K$-theory, the key
step is Thomason's theorem for completion of the homology theories
at the augmentation ideal. For Hochschild homology,
this can be achieved by means of the formal HKR theorem of Ben-Zvi and
Nadler \cite{Ben-Nadler}, as we shall show in this section.
However, the corresponding result for cyclic homology is very delicate
because it involves interchanging homotopy limit and colimit.
We shall study this question in the next section.

We fix once and for all an algebraically closed field $k$ of characteristic zero.
We shall follow the notations of \S~\ref{sec:Recall*} throughout our discussion of
homology theories. In particular, we shall consider all functors on sheaves to be
derived. We shall use homological notations for chain complexes. 

\subsection{$I_G$-completion of Hochschild homology}\label{sec:H-0*}
For $\sX \in \dst_k$, we let ${\mathbb{L}}_{\sX}$ denote the cotangent complex of $\sX$
and let
$\T_{\sX}[-1] = \Spec({\rm Sym}^\bullet_{\sX}({\mathbb{L}}_{\sX}[1]))$ denote the odd
tangent bundle of $\sX$ with the projection $t \colon \T_{\sX}[-1] \to \sX$
(cf. \cite[\S~1.2.1]{Ben-Nadler}). If $\sX = [X/G]$ where
$G$ is a $k$-group acting on a smooth $k$-scheme $X$, then ${\mathbb{L}}_{\sX}$ is
given by the complex $(\Omega^1_{X/k} \to \mathfrak{g}^* \otimes_k \sO_X)$ of
$G$-equivariant coherent sheaves on $X$ with Cartan differential,
where $\mathfrak{g}$ is the Lie algebra of $G$ (cf. \cite[Exm.~2.1.10]{Chen})
and $\Omega^1_{X/k}$ is placed in degree zero.
The augmentation of ${\rm Sym}^\bullet_{\sX}({\mathbb{L}}_{\sX}[1])$ defines the
zero section $e \colon \sX \to \T_{\sX}[-1]$ of the projection $t$.

For a derived stack $\sX$ and a closed substack $\sZ \inj \sX$, let
$\wh{\sX}_{\sZ}$ denote the formal completion of $\sX$ along $\sZ$ (cf.
\cite[\S~6]{Gaitsgory-Rosenblyum}). It is an ind-stack such that for
$S \in {\bf{Aff}}_k$, the space
$\wh{\sX}_{\sZ}(S)$ consists of maps $S \to \sX$ of presheaves on ${\bf{Aff}}_k$
such that its restriction to
$\pi_0(S)_\red$ factors through the inclusion $\sZ \inj \sX$.
Let $\wh{\T}_{\sX}[-1]$ (resp. $\wh{\sL}\sX$)
denote the formal completion of $\T_{\sX}[-1]$ (resp. ${\sL}\sX$) along its
zero-section (resp. identity section).
The derived Hochschild-Kostant-Rosenberg (HKR) theorem for stacks then
provides a natural morphism $\exp \colon \wh{\T}_{\sX}[-1] \to \wh{\sL}\sX$ which
restricts to an isomorphism between the zero-section
$e(\sX) \inj \wh{\T}_{\sX}[-1] $ and the identity section $s(\sX) \inj \wh{\sL}\sX$
(cf. \cite[\S~6]{Ben-Nadler}). 

From now on in \S~\ref{sec:Hoch-prop}, we shall assume that
$G$ is a reductive $k$-group acting on a smooth $k$-scheme $X$. We let
$\sX = [X/G]$
denote the stack quotient and let $\pi \colon X \to \sX$ be the quotient map.
Let $p \colon U \to \sX$ be a smooth surjective morphism
where $U \in \Sm_k$. Since $\sX$ has
affine diagonal, $p$ is representable by schemes. We let $U_\bullet$ denote the
\u{C}ech nerve of $p$ and let $p_\bullet \colon U_\bullet \to \sX$ be the resulting
map of simplicial stacks. If we take $p$ to be the quotient map $X \to \sX$, then
$U_\bullet$ coincides with $X^\bullet_G$ of ~\eqref{eqn:Bar-BG}.
Since $\sO(\sL\sX)$ is a module over the ring spectrum $HH(BG) \simeq \sO({\sL}(BG))$,
it follows from \thmref{thm:H-functor} that it is a module over the ring
spectrum $K(BG)$. In particular, it has an action of the ring $R_k(G) =
\pi_0(K(BG,k))$. The first step to prove the completion theorem for the Hochschild
homology of $\sX$ is the following.

\begin{lem}\label{lem:H-Compln-0}
  The canonical pull-back map $\sO(\sL\sX) \to \sO(\wh{\sL}\sX)$ induces a weak
  equivalence
  \[
  \sO(\sL\sX)^{\compl}_{I_{G}} \xrightarrow{\simeq} \sO(\wh{\sL}\sX).
  \]
\end{lem}
\begin{proof}
  We let $A$ be a simplicial commutative $k$-algebra and let 
  $S = \Spec(A)$ be the unique object of ${\bf{Aff}}_k$ defined by $A$.
  Let $u \colon S \to \sL\sX$ be a morphism of derived stacks and consider the
  diagram
  \begin{equation}\label{eqn:H-Compln-1}
    \xymatrix@C1pc{
      S \times_{\sL\sX} \sX \ar[r]^-{u'} \ar[d]_-{s'} & \sX \ar[d]^-{s} \ar[rr] &
      & [{\Spec(k(e))}/G] \ar[r] \ar[d]^-{j} &
      \Spec(k(e)) \ar[d]^-{j} \ar[dr]^-{j} &
      \\
      S \ar[r]^-{u} & \sL\sX \ar[r] & {\sL}BG \ar[r]^-{\simeq} & [G/G] \ar[r]^-{v} &
      \Spec(C[G]) \ar[r]^-{\cong} & \Spec(R_k(G)),}
  \end{equation}
  where $e \in G$ is the identity element and $[G/G]$ is the stack quotient for the
  adjoint action of $G$ on itself.

  The arrow $v$ is the affinization morphism (cf. \cite[\S~3]{Ben-Nadler})
  and the bottom horizontal arrow on the right is an isomorphism by
  \propref{prop:Twist-BP}(5). All other horizontal arrows are canonical maps.
  The vertical arrows are canonical inclusions of closed substacks.
  The right square is Cartesian because
  $v$ is the coarse moduli space for the adjoint action, the middle square is
  Cartesian by definition of the map $s$ and the left square is Cartesian by
  construction. It follows that the outer trapezium is Cartesian.

 If we let $w \colon S \to \Spec(R_k(G))$ denote the composite bottom horizontal
  arrow in ~\eqref{eqn:H-Compln-1} and $W = S \times_{\Spec(R_k(G))} \Spec(k(e))$, it
  follows from above that
  $u^{-1}(|\sX|) = |W|$ as topological spaces. But this implies by definition
  of the formal completion (cf. \cite[\S~6]{Gaitsgory-Rosenblyum}) that the 
  square
  \begin{equation}\label{eqn:H-Compln-2}
    \xymatrix@C1pc{
      \wh{S}_{W} \ar[r]^-{u'} \ar[d]_-{q'} & \wh{\sL}\sX \ar[d]^-{q} \\
      S \ar[r]^-{u} & \sL\sX}
  \end{equation}
is homotopy Cartesian.

We now compute
\begin{equation}\label{eqn:H-Compln-3}
\begin{array}{lll}
  q_*(\sO_{\wh{\sL}\sX})(S) & \simeq & u^*q_*(\sO_{\wh{\sL}\sX})(S) \\
  & \simeq & q'_* u'^*(\sO_{\wh{\sL}\sX})(S) \\
  & \simeq & \Map_{\sO_S}(\sO_S, q'_* u'^*(\sO_{\wh{\sL}\sX})) \\
   & {\simeq}^{\dagger} & \Map_{\sO_{\wh{S}_{W}}}(\sO_{\wh{S}_{W}}, u'^*(\sO_{\wh{\sL}\sX})) \\
   & \simeq & \Map_{\wh{S}_{W}}(\sO_{\wh{S}_{W}}, \sO_{\wh{S}_{W}}) \\
  & \simeq & A^{\compl}_{I_{G}} = \sO(S)^{\compl}_{I_{G}},
\end{array}
\end{equation}
where $\Map_{\sO_S}(-,-)$ is the mapping space between sheaves of $\sO_S$-modules
and the weak equivalence ${\simeq}^{\dagger}$ follows by the adjointness.

We finally compute
\begin{equation}\label{eqn:H-Compln-4}
\begin{array}{lll}
\sO(\wh{\sL}\sX) = \Gamma(\wh{\sL}\sX,  \sO_{\wh{\sL}\sX}) & \simeq & 
\Gamma(\sL\sX,  q_*(\sO_{\wh{\sL}\sX})) \\
& = & {\underset{S \in {\bf{Aff}}_k, S \to \sL\sX}\holim}  q_*(\sO_{\wh{\sL}\sX})(S) \\
& {\simeq}^{1} & {\underset{S \in {\bf{Aff}}_k, S \to \sL\sX}\holim}
\sO(S)^{\compl}_{I_{G}} \\
& \simeq & \left({\underset{S \in {\bf{Aff}}_k, S \to \sL\sX}\holim}
\sO(S)\right)^{\compl}_{I_{G}} \\
& \simeq & \sO(\sL\sX)^{\compl}_{I_{G}},
\end{array}
\end{equation}
where ${\simeq}^{1}$ follows from ~\eqref{eqn:H-Compln-3}. This finishes the proof.
\end{proof}

We now look at the commutative diagram
\begin{equation}\label{eqn:H-Compln-5}
  \xymatrix@C1pc{
    & {\T}_{U_\bullet}[-1] \ar[dr]^-{{\rm exp}_1} & \\
    \wh{\T}_{U_\bullet}[-1] \ar[r]^-{{\rm exp}_2}  \ar[d] \ar[ur]^-{\alpha}&
    \wh{\sL}U_\bullet \ar[r]^-{q_\bullet} \ar[d] &  {\sL}U_\bullet \ar[d] \\
    \wh{\T}_{\sX}[-1] \ar[r]^-{{\rm exp}_3} & \wh{\sL}\sX \ar[r]^-{q} & {\sL}\sX,}
  \end{equation}
where the horizontal arrows in the lower square on the right as well as the map
$\alpha$ are induced by the completions and the
vertical arrows are induced by $p_\bullet \colon U_\bullet \to \sX$
associated to the smooth surjective map $p \colon U \to \sX$. The lower left square
is commutative by the naturality of the derived HKR theorem.

We now note that each  ${\T}_{U_n}[-1]$ is complete because 
${\rm Sym}^\bullet_{\sX}({\mathbb{L}}_{U_n}[1])$ is
a bounded complex. By passing to the limit, it follows that $\alpha$ is an
equivalence.
The map ${\rm exp}_1$ in ~\eqref{eqn:H-Compln-5}
is an equivalence by the well known HKR theorem for smooth
$k$-schemes while ${\rm exp}_2$ and ${\rm exp}_3$ are equivalences by the formal HKR
theorem of Ben-Zvi and Nadler (cf. \cite[Thm.~6.9]{Ben-Nadler}).
It follows that $q_\bullet$ is an equivalence.
In other words, ${\sL}U_\bullet$ is complete with
respect to its identity section $U_\bullet \inj {\sL}U_\bullet$.
In particular, the map
\begin{equation}\label{eqn:H-Compln-6}
  \sO({\sL}U_\bullet) \simeq  \holim_n \sO({\sL}U_n) \to \holim_n \sO(\wh{\sL}U_n)
 \simeq \sO(\wh{\sL}U_\bullet) 
  \end{equation}
  is an equivalence.

On the other hand, a computation similar to
  ~\eqref{eqn:H-Compln-4} (substitute $U_n$ for $\sX$) shows that $\sO(\wh{\sL}U_n)$
  is $I_G$-complete as an $R_k(G)$-module.
  We conclude that the pull-back map $p^*_\bullet \colon \sO({\sL\sX}) \to
  \holim_n \sO({\sL}U_n)$ factors through $\sO({\sL\sX})^{\compl}_{I_G}$.
  Using ~\eqref{eqn:HH-Loop}, it follows that the pull-back map
  $p^*_\bullet \colon HH(\sX) \to \holim_n HH(U_n)$ factors through the $I_G$-completion
  of $HH(\sX)$ to induce a canonical map
  \begin{equation}\label{eqn:H-Compln-7}
    p^*_\bullet \colon HH(\sX)^{\compl}_{I_G} \to \holim_n HH(U_n).
  \end{equation}
  This is the analogue for the Hochschild homology of Thomason's completion map
  ~\eqref{eqn:Bar-EG-main-0} in $K$-theory.

  We let
  \begin{equation}\label{eqn:H-limit*}
    F(U_\bullet) = \holim_n F(U_n) \ \ \mbox{for} \ \ F \in \{HH,HC, HC^{-}, HP\}.
  \end{equation}

\begin{lem}\label{lem:B-Inv-0}
    The $S^1$-action on $HH(\sX)$ induces an $S^1$-action on $HH(\sX)^{\compl}_{I_G}$
    such that the completion map $HH(\sX) \to HH(\sX)^{\compl}_{I_G}$ as well as
    $ p^*_\bullet \colon HH(\sX)^{\compl}_{I_G} \to HH(U_\bullet)$ is $S^1$-equivariant.
  \end{lem}
  \begin{proof}
    To see that the $S^1$-action on $HH(\sX)$ carries over to $HH(\sX)^{\compl}_{I_G}$,
    note that the $S^1$-action on $HH(\sX)$ is equivalent to the structure of
    a $\Lambda$-algebra on it (cf. \cite[Cor.~3.14]{Ben-Nadler})
    and there are morphisms of
    $\Lambda$-algebras $R_k(G) \simeq HH(BG) \to HH(\sX)$. In particular,
    multiplication by any element of $R_k(G) \cong \pi_0(HH(BG))$ is
    $\Lambda$-linear (in fact, $ HH(\sX)$-linear).
    By using this observation repeatedly, we see that
    the $S^1$-action of $HH(\sX)$ descends to compatible action of $S^1$ on quotients
    ${HH(\sX)}/{(x^{i_1}_1, \ldots , x^{i_r}_r)HH(\sX)}$ 
    as $(i_1, \ldots , i_r)$ varies over the $r$-tuples of positive integers,
    where $I_G = (x_1, \ldots , x_r)$ (cf. \propref{prop:finiteR}). Passing to the
    limit over these $r$-tuples and using ~\eqref{eqn:I-comp-0}, we get the first
    part of the lemma. The $S^1$-equivariance of $p^*_\bullet$ now easily follows
    using that each $HH(U_n)$ is $I_G$-complete.
\end{proof}

The following result is the analogue
of \thmref{thm:Bar-K-rational} for Hochschild homology. One difference though is
that the result for Hochschild homology is more general in that the map
$U \to \sX$ is allowed
  to be any smooth cover by a scheme  (thanks to the smooth descent
  property of the odd tangent bundle) and not just the quotient map $X \to \sX$
  (in which case $U_\bullet = X^\bullet_G$).

\begin{prop}\label{prop:H-Compln-8}
   The map $p^*_\bullet \colon HH(\sX)^{\compl}_{I_G} \to HH(U_\bullet)$ is a weak
   equivalence.
 \end{prop}
 \begin{proof}
   Using ~\eqref{eqn:H-Compln-4} and ~\eqref{eqn:H-Compln-6}, the problem is
   equivalent to showing that the map $p^*_\bullet \colon \sO(\sL\sX) \to
   \sO({\sL}U_\bullet)$ induces a weak equivalence
   $p^*_\bullet \colon \sO(\wh{\sL}\sX) \to \sO(\wh{\sL}U_\bullet)$.
   To that end, we use  the left square of ~\eqref{eqn:H-Compln-5}
   to get a commutative
   diagram
  \begin{equation}\label{eqn:H-Compln-9} 
    \xymatrix@C1pc{
      \sO(\wh{\sL}\sX) \ar[r]^-{{\rm exp}^*_3} \ar[d]_-{p^*_\bullet}
      & \sO(\wh{\T}_{\sX}[-1]) \ar[d]^-{p^*_\bullet} \\
      \sO(\wh{\sL}U_\bullet) \ar[r]^-{{\rm exp}^*_2} & \sO(\wh{\T}_{U_\bullet}[-1]).}
    \end{equation}

  The horizontal arrows are weak equivalences by the formal HKR theorem
  (cf. \cite[Thm.~6.9]{Ben-Nadler}) while the right vertical arrow is a weak
  equivalence
  by the smooth descent for the completed odd tangent bundle of smooth stacks
  (cf. [op. cit., Prop.~6.4]). This forces the left vertical arrow to be a weak
  equivalence.
\end{proof}

\subsection{$I_G$-completion of negative cyclic homology}\label{sec:Neg-cyc-0}
We shall now use \propref{prop:H-Compln-8} to prove a similar result for
the negative cyclic homology of $\sX$.

\begin{cor}\label{cor:NH-functor}
   The pull-back map $p^*_\bullet \colon HC^{-}(\sX) \to HC^{-}(U_\bullet)$ induces
   a weak equivalence
   \[
   p^*_\bullet \colon HC^{-}(\sX)^{\compl}_{I_G} \xrightarrow{\simeq} HC^{-}(U_\bullet).
   \]
 \end{cor}
\begin{proof}
   Recall that a $\Lambda$-module structure on a complex $\sF_\bullet \in \sD_k$ is
   the same thing as an $S^1$-action on the Eilenberg-MacLane spectrum of
   $\sF_\bullet$, which in turn is the same thing as a functor
   $\sF \colon BS^1 \to \sS\sH$, where $BS^1$ is the classifying category of the
   topological group $S^1$
   (cf. \cite[\S~I.1]{Nikolaus-Scholze}). In this case, $|\sF|^{hS^1}$ coincides with
   the homotopy limit $\holim_{S^1} \sF$ of the functor $\sF$. 
   We now let $\sC$ be the category of $r$-tuples of positive integers
   $(i_1, \ldots , i_r)$ with the unique map
   $(i_1, \ldots , i_r) \to (j_1, \ldots , j_r)$ if $i_m \ge j_m$ for all $m$.
   Then we have a natural transformation
   \begin{equation}\label{eqn:NH-functor-0}
     p^*_\bullet \colon \frac{HH(\sX)}{(x^{i_1}_1, \ldots , x^{i_r}_r)HH(\sX)} \to
     \frac{HH(U_n)}{(x^{i_1}_1, \ldots , x^{i_r}_r)HH(U_n)}
     \end{equation}
   between functors from $BS^1 \times \sC \times \Delta$ to $\sS\sH$ where the
   left-hand side factors through the projection
   $BS^1 \times \sC \times \Delta \to BS^1 \times \sC$.
In particular, for every $\alpha \in ({i_1}, \ldots , {i_r}) \in \sC$ and
$[n] \in \Delta$, this map is $S^1$-equivariant.

If we pass to the limit over $\sC \times \Delta$, ~\eqref{eqn:NH-functor-0}
recovers the $S^1$-equivariant map 
\begin{equation}\label{eqn:NH-functor-1}
p^*_\bullet \colon HH(\sX)^{\compl}_{I_G} \to HH(U_\bullet)
\end{equation}   
of \propref{prop:H-Compln-8}. On the other hand, passing to the limit over $BS^1$, it
produces
\begin{equation}\label{eqn:NH-functor-2}
p^*_\bullet \colon \frac{HC^{-}(\sX)}{(x^{i_1}_1, \ldots , x^{i_r}_r)HC^{-}(\sX)} \to
\frac{HC^{-}(U_n)}{(x^{i_1}_1, \ldots , x^{i_r}_r)HC^{-}(U_n)}.
\end{equation}

Hence, the limit over $BS^1$ of ~\eqref{eqn:NH-functor-1} is a weak equivalence if
and only if the limit over $\sC \times \Delta$ of ~\eqref{eqn:NH-functor-2}
is a weak equivalence (cf. \cite[\S~31.5]{Chacholski-Scherer}).
Using \propref{prop:H-Compln-8}, we therefore get a weak equivalence
\[
p^*_\bullet \colon HC^{-}(\sX)^{\compl}_{I_G} \xrightarrow{\simeq}
\holim_{\Delta} \holim_{\sC}
\frac{HC^{-}(U_n)}{(x^{i_1}_1, \ldots , x^{i_r}_r)HC^{-}(U_n)}.
\]
Since each $HH(U_n)$ is $I_G$-complete as we showed above, it follows that
$HC^{-}(U_n) = \holim_{S^1} HH(U_n)$ is also $I_G$-complete
(cf. \propref{prop:Compln_D}(7)). In particular,
  the right-hand side of the last equivalence coincides with
  $\holim_{n} HC^{-}(U_n)$. This finishes the proof.
\end{proof}

 \subsection{Some results on cyclic homology}\label{sec:Cyc-misc}
 We end this section with the following results related to cyclic homology
 which will be used in the next section.

\begin{lem}\label{lem:Chen-7}
   $HH(\sX)$ and $HC(\sX)$ are homologically bounded below complexes in $\sD_k$.
   In particular, $\holim_n HH(U_n)$ is homologically bounded below.
 \end{lem}
\begin{proof}
  Using ~\eqref{eqn:HH-Loop}, claiming that $HH(\sX)$ is homologically bounded below
  is the same thing as claiming that $\sO(\sL\sX)$ is homologically bounded
  below. Since the latter is the derived push-forward of the structure sheaf of
  $\sL\sX$ under the structure map $p \colon \sL\sX \to \pt$,
  it suffices to show more generally that $p$ is of finite cohomological dimension
  (cf. \cite[Lem.~A.1.6]{HLPR}). By \cite[Rem.~1.22]{Khan-JJM},
  this is equivalent to showing that
  the structure map of $\pi_0(\sL\sX) = \sI_{\sX} = [{I_X}/G]$ is
  of finite cohomological dimension. But this follows from \cite[Prop.~A.1.8]{HLPR}
  under our assumptions.

Since the cyclic homology complex is obtained by taking the sum totalization of
   the cyclic bicomplex arising from the mixed complex structure on $HH(\sX)$, it
   follows that $HC(\sX)$ is also homologically bounded below.
The second part of the lemma now follows from Propositions~\ref{prop:H-Compln-8},
~\ref{prop:Compln_D}(1) and ~\ref{prop:finiteR}.
\end{proof}

We now consider a general situation in which $H$ is a $k$-group which acts on a
$k$-scheme
 $Y$ and let $\sY = [Y/H]$. Assume that $I_H = (y_1, \ldots , y_r) \subset R_k(H)$ and
 $I_s = (y^s_1, \ldots , y^s_r)$ for $s \in \N$. We then have an $S^1$-equivariant
 natural transformation of functors $HH(\sY)^{\compl}_{I_H} \to
 \{{HH(\sY)}/{I_s}\}_{s \in \N}$ on $\N$, where we let ${HH(\sY)}/{I_s} :=
  {HH(\sY)}/{I_sHH(\sY)}$.   In particular, we get a
 natural transformation of functors $({HH(\sY)^{\compl}_{I_H}})_{hS^1} \to
 {\{({HH(\sY)}/{I_s})_{hS^1}\}}_{s \in \N} \simeq {\{{HH(\sY)_{hS^1}}/{I_s}\}}_{s \in \N}$.
 Passing to the homotopy limits, we get a natural map of $R_k(H)$-modules
 \begin{equation}\label{eqn:Chen-0*}
  \alpha_\sX \colon (HH(\sY)^{\compl}_{I_H})_{hS^1} \to HC(\sY)^{\compl}_{I_H}.
   \end{equation}
 
\begin{lem}\label{lem:Chen-0}
 Assume that $H$ is reductive and $Y$ is a smooth
  $k$-scheme. Then the canonical map
   \[
   \alpha_{\sX} \colon (HH(\sY)^{\compl}_{I_H})_{hS^1} \to HC(\sY)^{\compl}_{I_H}
   \]
   is a weak equivalence.
 \end{lem}
\begin{proof}
  Letting $M_s = {HH(\sY)}/{I_s}$, it easily follows by using \lemref{lem:Chen-7}
  and by iterating the long exact homotopy groups sequences associated to
 fiber sequences $M \xrightarrow{x} M \to M/{xM}$
  (where $M$ is any complex of $R_k(H)$-modules)
  that the inverse system of complexes \{$M_s\}_{s \in \N}$ is uniformly homologically
  bounded below. That is, for some $N \ll 0$, the canonical map
  $\tau_{\ge N} M_s \to M_s$ is a weak equivalence for every $s \in \N$.
  By replacing each $M_s$ by $\tau_{\ge N} M_s$, we can therefore assume that
  each $M_s$ is bounded below by $N$. 

  We now look at the commutative diagram
  \begin{equation}\label{eqn:Chen-0-0-1}
    \xymatrix@C1pc{
     (\varprojlim_s M_s)_{hS^1} \ar[r] \ar[d] & (\prod_s M_s)_{hS^1} \ar[r]^-{u}
      \ar[d] & (\prod_s M_s)_{hS^1} \ar[d] \\
       \varprojlim_s (M_s)_{hS^1} \ar[r]  & \prod_s (M_s)_{hS^1} \ar[r]^-{u}
       & \prod_s (M_s)_{hS^1}}
  \end{equation}
  in $\sD_k$, where $u$ is induced by the map on the products which is
  given levelwise by $k_n \mapsto k_n - f_{n+1}(k_{n+1})$, where $f_n \colon
  M_{n} \to M_{n-1}$ is the bonding map.
  Since the transition maps of $\{M_s\}$ are $\Lambda$-linear, it follows that
  $u$ is $\Lambda$-linear. In particular, all maps in ~\eqref{eqn:Chen-0-0-1} are
  defined and the rows are homotopy fiber sequences.
  It suffices therefore to show that the canonical map
  $\alpha \colon (\prod_s M_s)_{hS^1} \to \prod_s (M_s)_{hS^1}$ is a weak equivalence.

  Now, we now recall from \cite[Def.~2.2.13]{Chen}
  that $((M_s)_{hS^1})_n = \left((M_s)_n \oplus (M_s)_{n-2} \oplus \cdots \right)$
  for any $n \in \Z$.
  In particular, we have
  \begin{equation}\label{eqn:Chen-0-0-2}
  \left(\prod_s (M_s)_{hS^1}\right)_n = \prod_s \left((M_s)_{hS^1}\right)_n =
  \prod_s \left((M_s)_n \oplus (M_s)_{n-2} \oplus \cdots \right).
  \end{equation}
  On the other hand, we have
  \begin{equation}\label{eqn:Chen-0-0-3}
  \left(\left(\prod_s M_s\right)_{hS^1}\right) = \left( \prod_s (M_s)_n \oplus \prod_s 
  (M_s)_{n-2} \oplus  \cdots \right).
  \end{equation}
 Since each $M_s$ is zero below $N$, one checks that 
 the right hand sides of ~\eqref{eqn:Chen-0-0-2} and ~\eqref{eqn:Chen-0-0-3}
 coincide. This implies that $\alpha$ is a weak equivalence and concludes the proof
 of the lemma. 
  \end{proof}

\begin{remk}\label{remk:Chen-0-0-3}
  The reader can check easily that our assumptions about the reductivity of
  $G$ and $H$ and the smoothness of $[X/G]$ and $[Y/H]$ are not used in the proofs of
  Lemmas~\ref{lem:Chen-7} and ~\ref{lem:Chen-0} even though these
  suffice for us.
  \end{remk}

\section{Thomason's problem for cyclic homology}\label{sec:Cyc**}
In this section, we shall extend \propref{prop:H-Compln-8} to cyclic and periodic
cyclic homology and use these results to prove
Thomason's completion theorem (cf. \corref{cor:Thom-maximal}) for all
homology theories.
We fix an algebraically closed field $k$ of characteristic zero.

\subsection{Cyclic homology of a \u{C}ech nerve}\label{sec:Cyc-cech}
We let $G$ be a reductive $k$-group and $X$ a smooth
$k$-scheme with $G$-action.
We let $\sX = [X/G]$ and let $p \colon U \to \sX$ be a smooth
surjective map with $U \in \Sm_k$. In this subsection, we shall prove the key
lemma which says that 
$HC(U_\bullet)$ is the homotopy orbit space for the $S^1$-action on $HH(U_\bullet)$.
As the reader may note, this is a delicate statement because it involves
interchanging homotopy limit and colimit.

We let $p_n \colon U_n \to \pt$ denote the structure map and let
${\mathbb{L}}_n = \Sym^\bullet_{U_n}({\mathbb{L}}_{U_n}[1]) \in D(U_n)$ so that
$HH(U_n) = (p_n)_*({\mathbb{L}}_n)$. We let $d_n = \dim(U_n)$.
Since $U_n \in \Sm_k$, there is a natural weak equivalence
${\mathbb{L}}_n \xrightarrow{\simeq} \Sym^\bullet(\Omega^1_{{U_n}/k}[1])$. Each
${\mathbb{L}}_n$ comes equipped with
a Connes operator $B$ which is identified with the de Rham differential
$d \colon \Omega^i_{{U_n}/k} \to \Omega^{i+1}_{{U_n}/k}$.
It follows that ${\mathbb{L}}_n$ is endowed with the canonical 
HKR filtration $\Fil^m_{\hkr}{\mathbb{L}}_n
= \stackrel{d_n}{\underset{i = m}\bigoplus} \Omega^i_{{U_n}/k}[i]$. Moreover,
$\{\Fil^m_{\hkr}{\mathbb{L}}_n\}_{m \ge 0}$ is an
   $S^1$-equivariant
filtration in which the $S^1$-action is given by the Connes operator $B$.
Letting $\Fil^m_{\hkr}HH(U_n) = (p_n)_*(\Fil^m_{\hkr}{\mathbb{L}}_n)
   \simeq$ \
   $\stackrel{d_n}{\underset{i = m}\bigoplus}  (p_n)_*(\Omega^i_{{U_n}/k})[i]$ and using
   that the derived push-forward commutes with homotopy limits, we get
   the following.

\begin{lem}\label{lem:Chen-4}
     There is an  $S^1$-equivariant descending filtration
     $\{\Fil^m_{\hkr}HH(U_n)\}_{m \ge 0}$ of $HH(U_n)$
  whose graded quotients are given by
  ${\rm gr}^m_{\hkr}HH(U_n) = (p_n)_*(\Omega^m_{{U_n}/k})[m]$. Furthermore, the
  canonical map
  \[
  q_n \colon HH(U_n) \to \holim_m \frac{HH(U_n)}{\Fil^m_{\hkr}HH(U_n)}
  \]
  is a weak equivalence.
\end{lem}

\begin{lem}\label{lem:Chen-5}
  For every $m \ge 0$, the canonical map
  \[
  \left(\holim_n \frac{HH(U_n)}{\Fil^m_{\hkr}HH(U_n)}\right)_{hS^1} \to
  \holim_n \left(\left(\frac{HH(U_n)}{\Fil^m_{\hkr}HH(U_n)}\right)_{hS^1}\right)
  \]
  is a weak equivalence.
\end{lem}
\begin{proof}
  By induction on $m$ and the commutative diagram of the homotopy fiber
  sequences
\begin{equation}\label{eqn:Chen-5-0}
  \xymatrix@C.4pc{
    \left(\holim_n {\rm gr}^m_{\hkr}HH(U_n)\right)_{hS^1} \ar[r] \ar[d] &
    \left(\holim_n \frac{HH(U_n)}{\Fil^{m+1}_{\hkr}HH(U_n)}\right)_{hS^1} \ar[r] \ar[d] &
    \left(\holim_n \frac{HH(U_n)}{\Fil^{m}_{\hkr}HH(U_n)}\right)_{hS^1} \ar[d] \\
 \holim_n \left(\left({\rm gr}^m_{\hkr}HH(U_n)\right)_{hS^1}\right) \ar[r] &  
 \holim_n \left(\left(\frac{HH(U_n)}{\Fil^{m+1}_{\hkr}HH(U_n)}\right)_{hS^1}\right)
 \ar[r] &
 \holim_n \left(\left(\frac{HH(U_n)}{\Fil^{m}_{\hkr}HH(U_n)}\right)_{hS^1}\right),}
\end{equation}
it suffices to show that the canonical map
\begin{equation}\label{eqn:Chen-5-1}
  \left(\holim_n {\rm gr}^m_{\hkr}HH(U_n)\right)_{hS^1} \to
  \holim_n \left(\left({\rm gr}^m_{\hkr}HH(U_n)\right)_{hS^1}\right)
\end{equation}
is a weak equivalence for every $m \ge 0$.

We now note that each ${\rm gr}^m_{\hkr}HH(U_n)$ has trivial $S^1$-action
since ${\rm gr}^m_{\hkr}{\mathbb{L}_n}$ is a single term complex
(cf. proof of part (1) of \cite[Cor.~3.4]{BMS}). Since
$\holim_n {\rm gr}^m_{\hkr}HH(U_n)$ is the product totalization of
${\rm gr}^m_{\hkr}HH(U_n)$ over $n$, it follows that the former also has trivial
$S^1$-action. Equivalently, it is a mixed complex with trivial Connes operator.
This implies that there is a canonical weak equivalence
(cf. \cite[\S~5.1.7]{Loday})
\begin{equation}\label{eqn:Chen-5-2} 
\left(\holim_n {\rm gr}^m_{\hkr}HH(U_n)\right)_{hS^1}
\xrightarrow{\simeq} {\underset{i \ge 0}\bigoplus} \left(\holim_n
            {\rm gr}^m_{\hkr}HH(U_n)\right)[2i].
            \end{equation}

On the other hand, the commutativity of homotopy limits with
finite direct sums implies using the definition of the HKR-filtration
that we have a direct sum decomposition
            \begin{equation}\label{eqn:Chen-5-4}
            \begin{array}{lll}
            \holim_n  HH(U_n) & \xrightarrow{\simeq} &
            \left(\holim_n \Fil^{m+1}_{\hkr}HH(U_n)\right) \bigoplus
            \left(\holim_n \frac{HH(U_n)}{\Fil^{m}_{\hkr}HH(U_n)}\right) \\
            & & \bigoplus
            \left(\holim_n {\rm gr}^m_{\hkr}HH(U_n)\right).
            \end{array}
            \end{equation}
            In conjunction with \lemref{lem:Chen-7}, this implies that
            $\holim_n {\rm gr}^m_{\hkr}HH(U_n)$ is homologically bounded below.

We now let $M_\bullet = \holim_n {\rm gr}^m_{\hkr}HH(U_n)$ and
            consider the commutative diagram of complexes
    \begin{equation}\label{eqn:Chen-5-3}        
      \xymatrix@C.8pc{
        \tau_{\ge t}M_\bullet \ar[r] \ar[d] &  {\underset{i \ge 0}\bigoplus}
        (\tau_{\ge t}M_\bullet)[2i] \ar[r] \ar[d] & {\underset{i \ge 0}\prod}
        (\tau_{\ge t}M_\bullet)[2i] \ar[d] \\
        M_\bullet \ar[r] &  {\underset{i \ge 0}\bigoplus}
        M_\bullet[2i] \ar[r] & {\underset{i \ge 0}\prod} M_\bullet[2i],}
    \end{equation}
    where the vertical arrows are the canonical inclusions and the left horizontal
    arrows are the obvious inclusions.

We just showed above that the left vertical arrow in ~\eqref{eqn:Chen-5-3} is
a weak equivalence for some $t \ll 0$. In particular, the other two vertical
arrows
are also weak equivalences (as one easily checks) for $t \ll 0$.
On the other hand, the map
    ${\underset{i \ge 0}\bigoplus} (\tau_{\ge t}M_\bullet)[2i] \to
    {\underset{i \ge 0}\prod} (\tau_{\ge t}M_\bullet)[2i]$ is a weak equivalence since
    $\tau_{\ge t}M_\bullet$ is strictly bounded below.
    We conclude from this and ~\eqref{eqn:Chen-5-2} that
    there is a canonical weak equivalence
\begin{equation}\label{eqn:Chen-5-5} 
\left(\holim_n {\rm gr}^m_{\hkr}HH(U_n)\right)_{hS^1}
\xrightarrow{\simeq} {\underset{i \ge 0}\prod} \left(\holim_n
            {\rm gr}^m_{\hkr}HH(U_n)\right)[2i].
            \end{equation}
Composing this with the natural weak equivalences
\[
\begin{array}{lll}
{\underset{i \ge 0}\prod} \left(\holim_n {\rm gr}^m_{\hkr}HH(U_n)\right)[2i]
& \xrightarrow{\simeq} &
\holim_n \left({\underset{i \ge 0}\prod}{\rm gr}^m_{\hkr}HH(U_n)[2i] \right)
\\
& \xrightarrow{\simeq} & \holim_n \left({\underset{i \ge 0}\bigoplus}
  {\rm gr}^m_{\hkr}HH(U_n)[2i]\right)
\\
& \xrightarrow{\simeq} &
\holim_n \left(\left({\rm gr}^m_{\hkr}HH(U_n)\right)_{hS^1}\right), \\
\end{array}
\]
we conclude that ~\eqref{eqn:Chen-5-1} is a weak equivalence. This
finishes the proof.
\end{proof}

\begin{lem}\label{lem:Chen-6}
  For every $n \ge 0$, the canonical map
  \begin{equation}\label{eqn:Chen-6-0} 
  \left(\holim_m \frac{HH(U_n)}{\Fil^m_{\hkr}HH(U_n)}\right)_{hS^1} \to
 \holim_m \left(\left(\frac{HH(U_n)}{\Fil^m_{\hkr}HH(U_n)}\right)_{hS^1}\right)
 \end{equation}
 is a weak equivalence. The same is also true for the canonical map
  \begin{equation}\label{eqn:Chen-6-1}
  \left(\holim_m \holim_n \frac{HH(U_n)}{\Fil^m_{\hkr}HH(U_n)}\right)_{hS^1} \to
 \holim_m \left(\left(\holim_n \frac{HH(U_n)}{\Fil^m_{\hkr}HH(U_n)}\right)_{hS^1}\right).
 \end{equation}
 \end{lem}
\begin{proof}
 If we let $F_m = \holim_n \frac{HH(U_n)}{\Fil^m_{\hkr}HH(U_n)}$, then 
  \lemref{lem:Chen-7} and ~\eqref{eqn:Chen-5-4} imply that 
  $\{F_m\}_{m \ge 0}$ is a uniformly (in $m$) homologically bounded below sequential
  diagram of mixed complexes in $\sD_k$. We showed in the proof of
  \lemref{lem:Chen-0} that in this case, ~\eqref{eqn:Chen-6-1} is a weak equivalence.
The same argument also applies to ~\eqref{eqn:Chen-6-0}.
\end{proof}

\begin{lem}\label{lem:Chen-2}
  The natural map $\alpha_{U_\bullet} \colon  {HH(U_\bullet)}_{hS^1} \to  HC(U_\bullet)$
  is a weak equivalence.
 \end{lem}
\begin{proof}
  We let $q = \holim_n q_n$ (cf. \lemref{lem:Chen-4}) and look at the diagram
   \begin{equation}\label{eqn:Chen-2-0}
     \xymatrix@C1pc{
       HH(U_\bullet)_{hS^1} \ar@{=}[d]  & & \\
       \left(\holim_n HH(U_n)\right)_{hS^1} \ar[rr]^-{q'} \ar[dd]_-{\alpha_{U_\bullet}} & &
       \left(\holim_n \holim_m \frac{HH(U_n)}{\Fil^m_{\hkr}HH(U_n)}\right)_{hS^1}
       \ar[d]^-{\alpha} \\
       & &
       \left(\holim_m \holim_n \frac{HH(U_n)}{\Fil^m_{\hkr}HH(U_n)}\right)_{hS^1}
       \ar[d]^-{\beta} \\
       \holim_n HH(U_n)_{hS^1} \ar[d]_-{q''} & & 
       \holim_m \left(\left(\holim_n \frac{HH(U_n)}{\Fil^m_{\hkr}HH(U_n)}\right)_{hS^1}\right)
       \ar[d]^-{\gamma} \\
       \holim_n \left(\holim_m \frac{HH(U_n)}{\Fil^m_{\hkr}HH(U_n)}\right)_{hS^1} 
       \ar[d]_-{\tau} & &
       \holim_m \holim_n \left(\left(\frac{HH(U_n)}{\Fil^m_{\hkr}HH(U_n)}\right)_{hS^1}\right)
       \ar[d]^-{\delta} \\
        \holim_n \holim_m  \left(\left(\frac{HH(U_n)}{\Fil^m_{\hkr}HH(U_n)}\right)_{hS^1}\right) \ar@{=}[rr] & &
   \holim_n \holim_m  \left(\left(\frac{HH(U_n)}{\Fil^m_{\hkr}HH(U_n)}\right)_{hS^1}\right).}
    \end{equation}

In this diagram, the arrow $q'$ is induced from $q$ by applying the  functor
   $(-)_{hS^1}$ and $q''$ is induced by first applying the functor $(-)_{hS^1}$ to each
   $q_n$ and then taking limit over $n$. The arrows $\alpha$ and $\delta$ are the
   canonical maps which commute the two homotopy limits, $\beta$ is the canonical map
   to the homotopy limit over $m$, $\tau$ is the canonical map and
   $\gamma$ is obtained by first taking the canonical
   map to the homotopy limit over $n$ (for each $m$) and then
   passing to the limit over $m$. 

To see that ~\eqref{eqn:Chen-2-0} is commutative, we let
   $\pr_n \colon \left(\holim_n HH(U_n)\right)_{hS^1} \to HH(U_n)_{hS^1}$
   and $\pr'_n \colon\holim_n \holim_m
   \left(\left(\frac{HH(U_n)}{\Fil^m_{\hkr}HH(U_n)}\right)_{hS^1}\right) \to
   \holim_m
   \left(\left(\frac{HH(U_n)}{\Fil^m_{\hkr}HH(U_n)}\right)_{hS^1}\right)$ 
   denote the projection maps for $n \ge 0$. We let $\vartheta = \delta \circ \gamma
   \circ \beta \circ \alpha \circ q'$ and $\vartheta' = \tau \circ q'' \circ
   \alpha_{U_\bullet}$. It suffices then to show that $\pr'_n \circ \vartheta =
   \pr'_n \circ \vartheta'$ up to homotopy for each $n \ge 0$. To that end, one first
   easily checks that
   $\pr'_n \circ \vartheta' = \pr'_n \circ \tau \circ q'' \circ \alpha_{U_\bullet} =
   \tau_n \circ q'_n \circ \pr_n$, where
   $q'_n$ is obtained from $q_n$ by applying the functor $(-)_{hS^1}$ and
   $\tau_n \colon \left(\holim_m \frac{HH(U_n)}{\Fil^m_{\hkr}HH(U_n)}\right)_{hS^1}  \to
   \holim_m \left(\frac{HH(U_n)}{\Fil^m_{\hkr}HH(U_n)}\right)_{hS^1}$ is the canonical
   map to the homotopy limit. It remains therefore to show that
   $\pr'_n \circ \vartheta = \tau_n \circ q'_n \circ \pr_n$ for every $n \ge 0$.

To show the last claim, we look at the diagram
   \begin{equation}\label{eqn:Chen-2-1}
     \xymatrix@C1pc{
       \left(\holim_n HH(U_n)\right)_{hS^1} \ar[d]_-{q'} 
       \ar[rr]^-{\pr_n} & & HH(U_n)_{hS^1} \ar[d]^-{q'_n} \\
       \left(\holim_n \holim_m \frac{HH(U_n)}{\Fil^m_{\hkr}HH(U_n)}\right)_{hS^1}
       \ar[d]_-{\alpha} \ar[rr] & & 
\left(\holim_m \frac{HH(U_n)}{\Fil^m_{\hkr}HH(U_n)}\right)_{hS^1}
\ar[d]^-{\id} \\
\left(\holim_m \holim_n \frac{HH(U_n)}{\Fil^m_{\hkr}HH(U_n)}\right)_{hS^1}
       \ar[d]_-{\beta}  \ar[rr] & & 
\left(\holim_m \frac{HH(U_n)}{\Fil^m_{\hkr}HH(U_n)}\right)_{hS^1}
\ar[d]^-{\gamma_n} \\
\holim_m \left(\left(\holim_n \frac{HH(U_n)}{\Fil^m_{\hkr}HH(U_n)}\right)_{hS^1}\right)
\ar[d]_-{\gamma} \ar[rr] & &
\holim_m \left(\frac{HH(U_n)}{\Fil^m_{\hkr}HH(U_n)}\right)_{hS^1} \ar[d]^-{\id} \\
\holim_m \holim_n \left(\left(\frac{HH(U_n)}{\Fil^m_{\hkr}HH(U_n)}\right)_{hS^1}\right)
       \ar[d]_-{\delta} \ar[rr] & & 
       \holim_m \left(\frac{HH(U_n)}{\Fil^m_{\hkr}HH(U_n)}\right)_{hS^1} \ar[d]^-{\id} \\
       \holim_n \holim_m
       \left(\left(\frac{HH(U_n)}{\Fil^m_{\hkr}HH(U_n)}\right)_{hS^1}\right)
       \ar[rr]^-{\pr'_n} & &
       \holim_m \left(\frac{HH(U_n)}{\Fil^m_{\hkr}HH(U_n)}\right)_{hS^1},}
       \end{equation}
   where all horizontal arrows in the middle are the canonical projections and
   the composition of all vertical arrows on the left is $\vartheta$.
   It is clear that all squares in this diagram commute and $\gamma_n  =
   \tau_n$. This proves the claim, and hence the commutativity of
   ~\eqref{eqn:Chen-2-0}.

We now note that the arrows $q'$ and $q''$ in ~\eqref{eqn:Chen-2-0} are weak
   equivalences. The arrows $\alpha$ and $\delta$ are weak equivalences by the
   commutativity property of two homotopy limits. The arrows $\tau$ and
   $\beta$ are weak equivalences by \lemref{lem:Chen-6}. Finally, $\gamma$ is a
   weak equivalence by \lemref{lem:Chen-5}. It follows that $\alpha_{U_\bullet}$ is a
   weak equivalence. This concludes the proof.
\end{proof}

\subsection{$I_G$-completion of cyclic homology}\label{sec:H-cyc*}
We now let $G$ be a reductive $k$-group and $\sX = [X/G]$, where
$X \in \Sm^G_k$. We do not assume that $X$ is quasi-projective. Let
$p \colon U \to \sX$ be a smooth
surjective map with $U \in \Sm_k$. We shall now prove the analogue of
\propref{prop:H-Compln-8} for cyclic and periodic cyclic homologies.
We begin with the following.

\begin{lem}\label{lem:Chen-1}
   For each $n \ge 0$, the spectrum $HC(U_n)$ is $I_G$-complete.
 \end{lem}
\begin{proof}
  Using \propref{prop:finiteR}, we can write
  $I_G = (x_1, \ldots , x_r) \subset R_k(G)$.
  For $s \ge 1$, we let $I_s = (x^s_1, \ldots , x^s_r)$. 
  We now note that $HH(U_n)$ is homologically
  bounded since $U \in \Sm_k$. We showed in the proof
  of \lemref{lem:Chen-7} that this implies that the sequential diagram
  $\{{HH(U_n)}/{I_sHH(U_n)}\}_{s \ge 1}$ is uniformly homologically bounded in $s$.
  It follows therefore from the proof of \lemref{lem:Chen-0} that the canonical map
  \begin{equation}\label{eqn:Chen-1-0}
    \left(HH(U_n)^{\compl}_{I_G}\right)_{hS^1} \to HC(U_n)^{\compl}_{I_G}
  \end{equation}
  is a weak equivalence. The lemma now follows because we showed earlier that
  $HH(U_n)$ is $I_G$-complete (cf. ~\eqref{eqn:H-Compln-6}). 
  \end{proof}
   
By \lemref{lem:Chen-1}, we get a diagram
 \begin{equation}\label{eqn:Chen-3-0}
     \xymatrix@C1pc{
       HC(\sX) \ar[r]^-{p^*_\bullet} \ar[d] & HC(U_\bullet) \ar[d] \\
HC(\sX)^{\compl}_{I_G} \ar@{.>}[ur] \ar[r]^-{p^*_n} & HC(U_n)}
     \end{equation}
 which is compatible with the face and degeneracy maps of
 $\{HC(U_n)\}_{[n] \in \Delta}$ and whose outer square is commutative.
 Passing to the homotopy limit of $p^*_n$, we get a unique map
 $ p^*_\bullet \colon HC(\sX)^{\compl}_{I_G} \to HC(U_\bullet)$ such that the lower triangle
 also commutes for each $n$. But this forces the upper triangle to commute too.

 We can now prove our main results about the $I_G$-completions of cyclic and
 periodic cyclic homologies.

\begin{thm}\label{thm:Chen-3}
   The map
   \[
   p^*_\bullet \colon HC(\sX)^{\compl}_{I_G} \to HC(U_\bullet)
   \]
   is a weak equivalence.
 \end{thm}
\begin{proof}
  We look at the diagram
  \begin{equation}\label{eqn:Chen-3-1}
     \xymatrix@C1pc{
       {(HH(\sX)^{\compl}_{I_G})}_{hS^1} \ar[r]^-{p^*_\bullet} \ar[d]_-{\alpha_\sX} &
       {HH(U_\bullet)}_{hS^1} \ar[d]^-{\alpha_{U_\bullet}} \\
   HC(\sX)^{\compl}_{I_G} \ar[r]^-{p^*_\bullet} & HC(U_\bullet),}
     \end{equation}
which is homotopy commutative by the naturality of the vertical arrows.
The left vertical arrow is a weak equivalence by \lemref{lem:Chen-0} and the
top horizontal arrow is a weak equivalence by \propref{prop:H-Compln-8}. It suffices
therefore to show that the right vertical arrow is a weak equivalence. But this
follows from \lemref{lem:Chen-2}.
\end{proof}

\begin{cor}\label{cor:Chen-8}
  The pull-back map
 $HP(\sX) \to HP(U_\bullet)$ factors through a weak equivalence
\[
   p^*_\bullet \colon HP(\sX)^{\compl}_{I_G} \to HP(U_\bullet).
   \]
  \end{cor}
\begin{proof}
  Using the commutative diagram of the homotopy fiber sequences
  (cf. \thmref{thm:H-functor})
 \begin{equation}\label{eqn:Chen-8-0}
     \xymatrix@C1pc{
 HC(U_n)[1] \ar[r]^-{N} \ar[d] & HC^{-}(U_n) \ar[r]^-{T} \ar[d] &  HP(U_n) \ar[d] \\
     HC(U_n)^{\compl}_{I_G}[1] \ar[r]^-{N} & HC^{-}(U_n)^{\compl}_{I_G} \ar[r]^-{T} &
     HP(U_n)^{\compl}_{I_G}}
       \end{equation}
 and the $I_G$-completeness of $HC(U_n)$ and $HC^{-}(U_n)$, we get that
 each $HP(U_n)$ is
 $I_G$-complete. By passing to the limit and repeating the previous argument for the
 cyclic homology, we get the factorization of the pull-back
 $HP(\sX) \xrightarrow{p^*_\bullet} HP(U_\bullet)$ through the $I_G$-completion.

We now look at the commutative diagram of homotopy fiber sequences
   \begin{equation}\label{eqn:Chen-8-1}
     \xymatrix@C1pc{
       HC(\sX)^{\compl}_{I_G}[1] \ar[d]_-{p^*_\bullet}
       \ar[r]^-{N} & HC^{-}(\sX)^{\compl}_{I_G}
       \ar[r]^-{T} \ar[d]^-{p^*_\bullet} &
       HP(\sX) ^{\compl}_{I_G} \ar[d]^-{p^*_\bullet} \\
       HC(U_\bullet)[1] \ar[r]^-{N} & HC^{-}(U_\bullet) \ar[r]^-{T} &  HP(U_\bullet).}
     \end{equation}
   The left and the middle vertical arrows are weak equivalences by
   \corref{cor:NH-functor} and \thmref{thm:Chen-3}. It follows that the
   right vertical arrow is also a weak equivalence.
\end{proof}

\vskip .2cm

The following result extends \corref{cor:Thom-maximal} to equivariant homology
theories.

\begin{thm}\label{thm:Homology-main}
Let $G$ be a reductive $k$-group acting on a smooth $k$-scheme $X$ and let $\sX
  = [X/G]$. Let $F$ belong to the set $\{HH, HC^{-}, HC, HP\}$ of functors
  on the category of stacks over $k$. Let $g \in G_s(k)$ and let
  $\Psi \subset G(k)$ be the conjugacy class of $g$.
  Then for any smooth surjective map $p \colon U \to [{X^g}/{Z_g}]$ with
  $U \in \Sch_k$, one has a natural weak equivalence of spectra
  \begin{equation}\label{eqn:Homology-main-0}
    \Theta^g_{\sX} \colon F(\sX)^{\compl}_{\fm_\Psi} \xrightarrow{\simeq}
    \holim_{n} F(U_n).
    \end{equation}
\end{thm}
\begin{proof}
  We consider the maps
\begin{equation}\label{eqn:Homology-main-1}
  F(\sX)^{\compl}_{\fm_\Psi} \xrightarrow{\iota^*_g}  F([{X^g}/{Z_g}])^{\compl}_{\fm_g}
  \xrightarrow{t_{g^{-1}}} F([{X^g}/{Z_g}])^{\compl}_{I_{Z_g}}
  \xrightarrow{p^*_\bullet} \holim_{n} F(U_n),
\end{equation}
where $\iota^*_g$ is the pull-back map induced by the closed immersion of
smooth schemes $\iota_g \colon X^g \inj X$.
Since the group of connected components of $Z_g$ is reductive as $\Char(k) = 0$,
it follows from \cite[Thm.~13.19]{Borel} that $Z_g$ is a reductive $k$-group.
We deduce by combining \cite[Thm.~3.2.3, 3.2.10]{Chen} with \propref{prop:H-Compln-8},
  \corref{cor:NH-functor}, \thmref{thm:Chen-3}, \corref{cor:Chen-8} and
  \propref{prop:Twist-local-compl} that all arrows in ~\eqref{eqn:Homology-main-1}
  are weak equivalences. This proves the theorem.
\end{proof}

\section{Results for Deligne-Mumford stacks}\label{sec:DM-Homology}
We let $k$ be an algebraically closed field of characteristic zero and let
$G$ be a $k$-group acting on a $k$-scheme $X$ with finite stabilizers
(cf. ~\S~\ref{sec:Proper-action}). Equivalently, $[X/G]$ is a Deligne-Mumford
stack. In this section,
we shall show that the equivariant $KH$-theory (without passing to any completion)
 of $X$ coincides with the Borel equivariant $KH$-theory of  the inertia scheme $I_X$.
 We shall prove the same statement for the homology theories if $X$ is smooth.
 The corresponding result for $K'$-theory was proven in \cite{Krishna-Sreedhar}
 when $k = \C$.
These results provide simple descriptions of the equivariant $KH$-theory and
homology theories of quotient Deligne-Mumford stacks.

\subsection{$KH$-theory of DM stacks}\label{sec:KH-DM}
We recall for the reader that $KH((I_X)_G, k)$ is the Borel equivariant
homotopy $K$-theory of $I_X$ while
$KH((I_X)^\bullet_G, k) = \holim_n KH(G^n \times I_X)$
is the homotopy $K$-theory of the bar construction for $G$-action on $I_X$.
Our main results for $KH$-theory are the following.

\begin{thm}\label{thm:AS-KH-DM}
 There is a weak equivalence of spectra
    \begin{equation}\label{eqn:AS-KH-DM-0}
      \Upsilon^G_X \colon KH_G(X,k) \xrightarrow{\simeq} KH((I_X)_G, k)
    \end{equation}
    which is natural in $X$.
  \end{thm}
\begin{proof}
  We shall write $KH_G(X,k)$ as $KH([X/G],k)$ throughout the proof.
    By \propref{prop:Fin-support}, there exists an ideal $J \subset R_k(G)$ with
    finite support $\{\fm_1, \ldots , \fm_r\}$ such that
    $J(KH_i(\stX, k)) = 0$ for all $i \in \Z$. If let $\Sigma^G_X =
    \{\Psi'_1, \ldots, \Psi'_n\}$ be as in \propref{prop:Fin-support}(1), then
    we can further assume that $\{\fm_1, \ldots , \fm_r\}$ contains the maximal ideals
    corresponding to $\Sigma^G_X$. 
In particular, $KH_i(\stX, k)$ is an ${R_k(G)}/J$-module.
    We thus get
    \begin{equation}\label{eqn:AS-KH-DM-1}
      \begin{array}{lll}
        KH_i(\stX,k) & \xrightarrow{\cong} & KH_i(\stX,k) \otimes_{R_k(G)} {R_k(G)}/J \\
        & \xrightarrow{\cong} & \stackrel{r}{\underset{i =1}\bigoplus}
          KH_i(\stX,k) \otimes_{R_k(G)} ({R_k(G)}/J)_{\fm_i} \\
        & \xrightarrow{\cong} & \stackrel{r}{\underset{i =1}\bigoplus}    
            KH_i(\stX,k) \otimes_{R_k(G)}  {R_k(G)}/J   \otimes_{R_k(G)}
            {R_k(G)}_{\fm_i} \\
         & \xrightarrow{\cong} & \stackrel{r}{\underset{i =1}\bigoplus}    
              KH_i(\stX,k)  \otimes_{R_k(G)} {R_k(G)}_{\fm_i} \\
          & \xrightarrow{\cong} & \stackrel{r}{\underset{i =1}\bigoplus}    
                KH_i(\stX,k)_{\fm_i} \\
      \end{array}
      \end{equation}

 Since the homotopy groups of derived localizations are the usual localizations
    of homotopy groups, we get that the canonical map
    of spectra
    \begin{equation}\label{eqn:AS-KH-DM-2}
      KH_G(X,k) \to \stackrel{r}{\underset{i =1}\coprod} KH_G(X,k)_{\fm_i}
    \end{equation}
    is a weak equivalence, where the terms on the right hand side are 
    derived localizations (cf. \S~\ref{sec:DLC}).
    Combining this with \corref{cor:Fin-support-4}, we get that 
    the canonical map
    \begin{equation}\label{eqn:AS-KH-DM-3}
      KH_G(X,k) \to \stackrel{r}{\underset{i =1}\coprod} KH_G(X,k)^{\compl}_{\fm_i}
    \end{equation}
    is a weak equivalence.

By \corref{cor:maximal1}, every $\fm_i$ is of the form $\fm_{\Psi_i}$ for a
 unique semisimple conjugacy class $\Psi_i \subset G_s(k)$. If we let
 $g_i \in G_s(k)$ be a representative of $\Psi_i$, then ~\eqref{eqn:AS-KH-DM-3} and
    \thmref{thm:Gen-AS-max} together imply that there is a canonical weak equivalence
    \begin{equation}\label{eqn:AS-KH-DM-4}
      KH_G(X,k) \xrightarrow{\simeq} \stackrel{r}{\underset{i =1}\coprod}
      KH_{Z_{g_i}}(X^{g_i},k)^{\compl}_{I_{Z_{g_i}}}
    \end{equation}
    The canonical map $KH_{Z_{g_i}}(X^{g_i},k)^{\compl}_{I_G} \to
    KH_{Z_{g_i}}(X^{g_i},k)^{\compl}_{I_{Z_{g_i}}}$
    is a weak equivalence 
    by \corref{cor:maximal2}, \propref{prop:finiteR} and
    \propref{prop:Compln_D}(5).
    If $g \in G_s(k) \setminus (\stackrel{r}{\underset{i =1}\cup} \Psi_i)$,
    then $X^{g} = \emptyset$ by
    \propref{prop:Fin-support} since $\Sigma^G_X \subset \{\fm_1, \ldots , \fm_r\}$.
    We can thus write ~\eqref{eqn:AS-KH-DM-4} as
    \begin{equation}\label{eqn:AS-KH-DM-5}
      KH_G(X,k) \xrightarrow{\simeq} {\underset{g \in G_s(k)}\coprod}
      KH_{Z_g}(X^{g}, k)^{\compl}_{I_G},
    \end{equation}
    where the right hand side is a finite coproduct.

Using the decomposition
    $[{I_X}/G] =  {\underset{g \in G_s(k)}\coprod} [{X^g}/{Z_g}]$ of the inertia stack
    (cf. \cite[Lem.~2.6]{Krishna-Sreedhar}), we get
    \begin{equation}\label{eqn:AS-KH-DM-6}
      \begin{array}{lll}
      KH_G(X,k) & \xrightarrow{\simeq} & {\underset{g \in G_s(k)}\coprod}
      KH_{Z_g}(X^{g}, k)^{\compl}_{I_G} \\
      & \xrightarrow{\simeq} &
      \left[{\underset{g \in G_s(k)}\coprod} KH_{Z_g}(X^{g}, k)\right]^{\compl}_{I_G} \\
& \xrightarrow{\simeq} &
      \left[{\underset{g \in G_s(k)}\coprod}
        KH([{X^{g}}/{Z_g}], k)\right]^{\compl}_{I_G} \\
& \xrightarrow{\simeq} & KH([{I_X}/G], k)^{\compl}_{I_G} \\
        & \xrightarrow{\simeq} & KH_G(I_X, k)^{\compl}_{I_G}.
          \end{array}
    \end{equation}

Letting $\Upsilon^G_X$ be the composition of this weak equivalence with the map
    $\vartheta_{{I_X}/G}$ of ~\eqref{eqn:AS-map-1} and applying
    \thmref{thm:AS-KH},
    we conclude that $\Upsilon^G_X$ is a weak equivalence. Furthermore,
    this is natural in
    $X$ since all maps which are used in its definition have been shown
    previously to be natural in $X$.
\end{proof}

\begin{thm}\label{thm:DM-Thom}
  If $X \in \Sch^G_k$and
    the hypotheses of \thmref{thm:AS-KH} are satisfied, then 
we have a weak equivalence of spectra
    \begin{equation}\label{eqn:DM-Thom-1}
      \wt{\Upsilon}^G_X \colon KH_G(X,k) \xrightarrow{\simeq} KH((I_X)^\bullet_G, k)
    \end{equation}
    which is natural in $X$.
\end{thm}
\begin{proof}
  The proof is identical to that of \thmref{thm:AS-KH-DM} except that we take
  $\wt{\Upsilon}^G_X$ to be the composition of the weak equivalence in
  ~\eqref{eqn:AS-KH-DM-6} with the map $\wt{\pi}^*$ of \corref{cor:KH-Thom*}
  instead of the map $\vartheta_{[{I_X}/G]}$ of ~\eqref{eqn:AS-map-1}.
 \end{proof}

\subsection{Homology theories of DM stacks}\label{sec:Hom_DM}
We assume now that $G$ is reductive and $X \in \Sm^G_k$.
We let $\sX = [X/G]$, $I_{\sX} = [{I_X}/G]$ and let $\tau \colon I_\sX \to \sX$ be the
projection. Let $\Sigma^s_G$ be the set of all semisimple conjugacy classes in $G(k)$
and let $\fm_{\Psi} \subset R_k(G)$ denote the maximal ideal corresponding to
$\Psi \in \Sigma^s_G$. In this subsection, we shall write $K(\sX,k)$ simply as
$K(\sX)$. The key to the main results of this subsection is the
following.

\begin{prop}\label{prop:Homology-decom}
  For $F \in \{HH, HC^{-}, HC, HP\}$, the canonical map
  \begin{equation}\label{eqn:Homology-decom-0}
    F(\sX) \to {\underset{\Psi \in \Sigma^s_G}\coprod} F(\sX)^{\compl}_{\fm_{\Psi}}
  \end{equation}
  is a weak equivalence. Moreover, the right hand side is a finite coproduct.
\end{prop}
\begin{proof}
  We have shown in ~\eqref{eqn:AS-KH-DM-2} that the localization map
  \begin{equation}\label{eqn:Homology-decom-1}
    K(\sX) \to {\underset{\Psi \in \Sigma^s_G}\coprod} K(\sX)_{\fm_{\Psi}}
  \end{equation}
  is a weak equivalence. Moreover, the right hand side is a finite coproduct of ring
  spectra. Since the trace maps 
  \begin{equation}\label{eqn:Homology-decom-2}
K(\sX) \to HC^{-}(\sX) \to HH(\sX)
  \end{equation}
  as well as the maps $T \colon HC^{-}(\sX) \to HP(\sX)$  and  $HH(\sX) \to HC(\sX)$
  are maps of ring spectra (cf. \thmref{thm:H-functor}),
  it follows that $F(\sX)$ is a $K(\sX)$-algebra for
  $F \in \{HH, HC^{-}, HC, HP\}$.
  By taking the smash product of ~\eqref{eqn:Homology-decom-1} with $F(\sX)$ over
  $K(\sX)$, we get a weak equivalence
  \[
  F(\sX) \xrightarrow{\simeq} {\underset{\psi \in \Sigma^s_G}\coprod}
  F(\sX) \wedge_{K(\sX)} K(\sX)_{\fm_{\Psi}}.
  \]

Since $K(\sX)_{\fm_{\Psi}}$ is the homotopy colimit of sequences
  $(K(\sX) \xrightarrow{x} K(\sX) \xrightarrow{x} \cdots)$ with
  $x \in R_k(G) \setminus \fm_\Psi$
  and the smash product commutes with homotopy colimits
  (cf. \cite[\S~2.4, p.~38]{Lurie-DAG}), the canonical map
  $F(\sX) \wedge_{K(\sX)} K(\sX)_{\fm_{\Psi}} \to  F(\sX)_{\fm_{\Psi}}$ 
  is a weak equivalence. On the other hand, $(\fm_\Psi)^n K_0(\sX)_{\fm_\Psi} = 0$
  for some $n \gg 0$, as shown in the proof of \corref{cor:Fin-support-4}.
  As $F(\sX)$ is a $K(\sX)$-module, it follows that
  every element of $\fm_\Psi$ acts nilpotently on
  each $\pi_i\left(F(\sX)_{\fm_{\Psi}}\right)$.
  But this implies that the completion map
  \begin{equation}\label{eqn:Homology-decom-7}
    F(\sX)_{\fm_{\Psi}} \to F(\sX)^{\compl}_{\fm_{\Psi}}
  \end{equation}
  is a weak equivalence (cf. \cite[Cor.~7.3.2.2]{Lurie-SAG}). 
  This finishes the proof.
\end{proof}

For every $\Psi \in \Sigma^s_G$, we pick an element $g_\Psi \in \Psi$ and
write $X_\Psi = X^{g_\Psi}$ and $Z_\Psi = Z_{g_\Psi}$. We let $\sX_\Psi = [{X_\Psi}/{Z_\Psi}]$.
Let $p \colon U \to I_\sX$ be a smooth surjective map with $U \in \Sch_k$.
We now look at the sequence of maps
\begin{equation}\label{eqn:Homology-decom-4}
  \begin{array}{lll}
   F(\sX) & \to & {\underset{\Psi \in \Sigma^s_G}\coprod} F(\sX)^{\compl}_{\fm_{\Psi}} \\
   & \xrightarrow{\amalg \iota_\Psi^*} &  {\underset{\Psi \in \Sigma^s_G}\coprod}
   F(\sX_\Psi)^{\compl}_{\fm_{g_\Psi}}   \\
   & \xrightarrow{\amalg t_{g^{-1}_\Psi}} & {\underset{\Psi \in \Sigma^s_G}\coprod}
   F(\sX_\Psi)^{\compl}_{I_{Z_\Psi}}  \\
   & {\simeq} & {\underset{\Psi \in \Sigma^s_G}\coprod}
   F(\sX_\Psi)^{\compl}_{I_G} \\
    & \xrightarrow{\simeq} &
    \left({\underset{\Psi \in \Sigma^s_G}\coprod}  F(\sX_\Psi)\right)^{\compl}_{I_G} \\
    & \xrightarrow{\simeq} & F(I_\sX)^{\compl}_{I_G} \\
    & \xrightarrow{p^*_\bullet} & \holim_{n}
      F(U_n).
    \\
  \end{array}
  \end{equation}

In the above sequence of arrows, the first arrow is the completion map of
\propref{prop:Homology-decom}, the second arrow is the sum of pull-back maps
via the inclusions $X_\Psi \inj X$, the third arrow is the twisting map of
\propref{prop:Twist-local-compl}, the fourth weak equivalence follows from
\corref{cor:maximal2}, \propref{prop:finiteR} and \propref{prop:Compln_D}(5),
the fifth weak equivalence is clear, the sixth weak equivalence follows
from the decomposition $I_\sX = {\underset{\Psi \in \Sigma^s_G}\coprod}
\sX_\Psi$, and the last weak equivalence is from \propref{prop:H-Compln-8}
and its analogue for other homology theories.
Combining these weak equivalences, we get the following analogue of
\thmref{thm:DM-Thom} for homology theories.

\begin{thm}\label{thm:Homology-decom-5}
  Let $G$ be a reductive $k$-group acting on a smooth $k$-scheme $X$ with finite
  stabilizers and let $\sX = [X/G]$. 
  Given any smooth surjective map $p \colon U \to I_\sX$ with $U \in \Sch_k$ and
$F \in \{HH, HC^{-}, HC, HP\}$, there is a weak equivalence of spectra
  \begin{equation}\label{eqn:Homology-decom-6}
 \wt{\Upsilon}_\sX \colon F(\sX) \xrightarrow{\simeq} \holim_{n} F(U_n)
  \end{equation}
  which is natural in $\sX$.
  \end{thm}
\begin{proof}
  The weak equivalence is shown above and the proof of
  its naturality is the same as that
  for $K$-theory in \thmref{thm:DM-Thom}.
  \end{proof}

\subsection{Individual homology groups of DM stacks}\label{sec:DM-pi}
We shall now give a complete description of all homology groups of the quotient stack
$\sX$ of \thmref{thm:Homology-decom-5}.
We begin by recalling some facts about the cotangent complex of
smooth Deligne-Mumford stacks.

Let $\sY$ be a Deligne-Mumford stack over $k$. Recall that there is a natural
augmentation map ${\mathbb{L}}_\sY \to \Omega^1_{{\sY}/k}[0]$. If
$\sY$ is smooth over $k$ (which will be the case below), then this map is a
weak equivalence (cf. \cite[\S~3]{Satriano}).
In particular, $\T_{\sY}[-1] \simeq \Spec({\rm Sym}^\bullet_{\sY}(\Omega^1_{{\sY}/k}[1]))$.
Since each $\Omega^i_{{\sY}/k}$ is locally free which is zero for $i > \dim(\sY)$,
it follows that ${\rm Sym}^\bullet_{\sY}(\Omega^1_{{\sY}/k}[1])$ is a finite complex.
In particular, its augmentation ideal is nilpotent which implies that the
formal completion $\wh{\T}_{\sY}[-1] \to \T_{\sY}[-1]$ along the zero section is
an equivalence. We let ${\Et}_{\sY}$ denote the {\'e}tale site of $\sY$ whose
objects are {\'e}tale maps $f \colon U \to \sY$ with $U \in \Sm_k$ and morphisms are
{\'e}tale maps between schemes.

We write
\begin{equation}\label{eqn:C-complex-1}
  \sH\sH_{\sY} =
{\rm Sym}^\bullet_{\sY}(\Omega^1_{{\sY}/k}[1]) = (\Omega^{d_\sY}_{{\sY}/k}
\xrightarrow{0} \cdots \xrightarrow{0} \Omega^1_{{\sY}/k} \xrightarrow{0}
\sO_{\sY}),
\end{equation}
where we write $d_{\sY} = \dim(\sY)$ and the complex is written homologically
with $\sO_{\sY}$ placed at degree zero.
We let $\Omega^\bullet_{\sY}$ be the (cohomological) de Rham complex of $\sY$ over $k$
(with de Rham differentials) and let
$_{s_{\ge m}}\Omega^\bullet_{\sY}$ and $_{s_{\le m}}\Omega^\bullet_{\sY}$ denote its
stupid truncations. We let $H^n_{dR}(\sY) = \H^n_\et(\sY, \Omega^\bullet_{\sY})$.

We let
\begin{equation}\label{eqn:C-complex-2}
  \sH\sC^{-}_{\sY} = (\sH\sH_{\sY})^{hS^{1}} =
        {\underset{i \in \Z}\bigoplus} \ \Omega^{\ge i}_{\sY}[2i];
        \ \ \sH\sC_{\sY} = (\sH\sH_{\sY})_{hS^{1}} =
                  {\underset{i \ge 0}\bigoplus} \ \Omega^{\le i}_{\sY}[2i];
                  \end{equation}
                  \[
                  \ \ \sH\sP_{\sY} = (\sH\sH_{\sY})^{tS^{1}} =
                            {\underset{i \in \Z}\bigoplus} \ \Omega^{\bullet}_{\sY}[2i].
                            \]
For $F \in \{HH, HC^{-}, HC, HP\}$, we let $\wt{F}_{\sY}$ be the fibrant replacement
of $F$ on the category of presheaves of complexes of $k$-vector spaces on ${\Et}_\sY$.
Then the canonical map $\sH\sH_{\sY}(U) \to \wt{HH}_{\sY}(U) = HH(U)$
is a weak equivalence if
$U \in {\Et}_{\sY}$ is affine as one checks using the hypercohomology spectral
sequence. It follows that the canonical map
$\H_\et(\sY,  \sH\sH_{\sY}) \to \H_\et(\sY,  \wt{HH}_{\sY})$ is a weak equivalence of
hypercohomology spectra (cf. \cite{Thomason-ENS}).
The same also holds for other homology functors.
In the following, $\Gamma(-,-)$ is the global section functor on the category of
chain complexes of {\'e}tale sheaves.

We now prove the following.

\begin{thm}\label{thm:Homology-ET}
  Let $G$ be a reductive $k$-group acting on a smooth $k$-scheme $X$ with finite
  stabilizers and let $\sX = [X/G]$. Then there are canonical weak
  equivalences
  \[
  HH(\sX) \xrightarrow{\simeq} {\bf{R}}\Gamma(I_{\sX}, \sH\sH_{I_{\sX}}); \ \
  HC^{-}(\sX) \xrightarrow{\simeq} {\bf{R}}\Gamma(I_{\sX}, \sH\sC^{-}_{I_{\sX}}).
  \]
  If $\sX$ is separated, then there are canonical weak
  equivalences
  \[
  HC(\sY) \xrightarrow{\simeq} {\bf{R}}\Gamma(I_{\sX}, \sH\sC_{I_{\sX}}); \  \
  HP(\sY) \xrightarrow{\simeq} {\bf{R}}\Gamma(I_{\sX}, \sH\sP_{I_{\sX}}).
  \]
\end{thm}
\begin{proof}
  We have shown in the proofs of \propref{prop:Homology-decom} and
  \thmref{thm:Homology-decom-5} that for
  $F \in \{HH, HC^{-}, HC, HP\}$, there are canonical weak equivalences
  \begin{equation}\label{eqn:Homology-ET-0}
    \theta_{\sX} \colon F(\sX) \xrightarrow{\simeq} F(I_{\sX})_{I_G}
    \xrightarrow{\simeq} F(I_{\sX})^{\compl}_{I_G}.
  \end{equation}
  On the other hand, it follows from \lemref{lem:H-Compln-0} and formal HKR
  theorem (see also ~\eqref{eqn:HH-Loop}) that 
$HH(I_{\sX})^{\compl}_{I_G} \xrightarrow{\simeq} \sO(\wh{\T}_{I_{\sX}}[-1])$. Since
$\wh{\T}_{I_{\sX}}[-1] \to \T_{I_{\sX}}[-1]$ is a weak equivalence, we get that there is a
weak equivalence
\begin{equation}\label{eqn:Homology-ET-1}
\theta'_{I_{\sX}} \colon HH(I_{\sX})^{\compl}_{I_G} \xrightarrow{\simeq}
\sO(\T_{I_{\sX}}[-1]) = {\bf{R}}\Gamma(I_{\sX}, \sH\sH_{I_{\sX}}).
\end{equation}
Letting $\Psi_{\sX} = \theta'_{I_{\sX}} \circ  \theta_{\sX}$, we get a natural weak
equivalence
 \begin{equation}\label{eqn:Homology-ET-2}
   \Psi_{\sX} \colon  HH(\sX) \xrightarrow{\simeq}
       {\bf{R}}\Gamma(I_{\sX}, \sH\sH_{I_{\sX}}).
 \end{equation}

Since the functor $(-)^{hS^{1}}$ commutes with ${\bf{R}}\Gamma(I_{\sX}, (-))$,
 we get
 \begin{equation}\label{eqn:Homology-ET-3}
   \begin{array}{lllll}
   HC^{-}(\sX) & \xrightarrow{\simeq}& HC^{-}(I_{\sX})^{\compl}_{I_G}
   & \simeq & \left(HH(I_{\sX})^{hS^{1}}\right)^{\compl}_{I_G} \\
& \simeq & \left(HH(I_{\sX})^{\compl}_{I_G}\right)^{hS^{1}} & \simeq &
({\bf{R}}\Gamma(I_{\sX}, \sH\sH_{I_{\sX}}))^{hS^{1}} \\
   & \simeq & {\bf{R}}\Gamma(I_{\sX}, (\sH\sH_{I_{\sX}})^{hS^{1}}) 
   & \simeq & {\bf{R}}\Gamma(I_{\sX}, \sH\sC^{-}_{I_{\sX}}), \\
   \end{array}
 \end{equation}
 where the first weak equivalence is from ~\eqref{eqn:Homology-decom-4},
 the third weak equivalence was shown in the proof of
 \corref{cor:NH-functor} and the fourth weak equivalence is from
 ~\eqref{eqn:Homology-ET-1}.

To prove the result for $HP$ when $\sX$ is separated, we note that 
\begin{equation}\label{eqn:Homology-ET-4}
    HP(\sX) \xrightarrow{\simeq} HP(I_{\sX})_{I_G} \simeq
    \left(HH(I_{\sX})^{tS^{1}}\right)_{I_G},
  \end{equation}
where the first weak equivalence is from ~\eqref{eqn:Homology-decom-7} and
~\eqref{eqn:Homology-decom-4}.
 On the other hand, we can write $HH(I_{\sX})^{tS^{1}}$ as the filtered homotopy colimit
  \[
  HH(I_{\sX})^{tS^{1}} \simeq \hocolim_i \left(HH(I_{\sX})^{hS^{1}} \to
  HH(I_{\sX})^{hS^{1}}[2] \to \cdots \to HH(I_{\sX})^{hS^{1}}[2i] \to \cdots \right).
  \]
  \[
  \hspace*{2cm} \simeq \hocolim_i \left(HC^{-}(I_{\sX}) \to
  HC^{-}(I_{\sX})[2] \to \cdots \to HC^{-}(I_{\sX})[2i] \to \cdots \right).
  \]
  Since localization at the maximal ideal $I_G$ is also a homotopy colimit, we
  get
\[
\left(HH(I_{\sX})^{tS^{1}}\right)_{I_G} \simeq
\hocolim_i \left(HC^{-}(I_{\sX})_{I_G}  \to
  HC^{-}(I_{\sX})_{I_G}[2] \to \cdots \to HC^{-}(I_{\sX})_{I_G}[2i] \to \cdots \right)
  \]
 \[
 \hspace*{2cm} \simeq \hocolim_i \left(E_1 \to  E_2 \to \cdots \to E_{2i} \to \cdots
 \right),
  \] 
  where the second weak equivalence follows from ~\eqref{eqn:Homology-ET-3}
  if we let $E_i =  {\bf{R}}\Gamma(I_{\sX}, \sH\sC^{-}_{I_{\sX}})[2i]$.
  As $\sX$ is separated, the projection map $\sI_{\sX} \to \sX$ is
  finite. In particular, $\sI_{\sX}$ is separated and it follows from
  \cite[Lem.~2.15]{HLP} that the functor ${\bf{R}}\Gamma(I_{\sX}, (-))$
  commutes with filtered homotopy colimits. We thus get
\[
  \begin{array}{lllll}
    HP(\sX) & \simeq &
    {\bf{R}}\Gamma(I_{\sX}, \hocolim_i (\sH\sC^{-}_{I_{\sX}})[2i]) 
  & \simeq &  {\bf{R}}\Gamma(I_{\sX}, (\sH\sH_{I_{\sX}})^{tS^{1}}) \\
  & \simeq &
  {\bf{R}}\Gamma(I_{\sX}, \sH\sP_{I_{\sX}}).
 \end{array}
  \]

Finally, we look at the diagram of $H_k$-module spectra
  \begin{equation}\label{eqn:Homology-ET-5}
    \xymatrix@C1pc{
    HC(I_{\sX})_{I_G}[1] \ar[r]^-{N} \ar@{.>}[d] &     
 HC^{-}(I_{\sX})_{I_G} \ar[r]^-{T} \ar[d] & 
 HP(I_{\sX})_{I_G} \ar[d] \\
 {\bf{R}}\Gamma(I_{\sX}, \sH\sC_{I_{\sX}})[1] \ar[r]^-{N} &
 {\bf{R}}\Gamma(I_{\sX}, \sH\sC^{-}_{I_{\sX}}) \ar[r]^-{T}
 & {\bf{R}}\Gamma(I_{\sX}, \sH\sP_{I_{\sX}}).}
  \end{equation}
  The two rows are homotopy fiber sequences and the right square is 
  commutative in which the vertical arrows are weak equivalences. It follows that
  there is a canonical weak equivalence
  $HC(I_{\sX})_{I_G} \xrightarrow{\simeq} {\bf{R}}\Gamma(I_{\sX}, \sH\sC_{I_{\sX}})$.
  Combining this with ~\eqref{eqn:Homology-ET-0}, we get our claim for $HC$, and
  this concludes the proof of the theorem.
\end{proof}

\begin{remk}\label{remk:Homology-ET-6}
  When $G$ is finite, \thmref{thm:Homology-ET}  for Hochschild homology was
  shown in \cite[Prop.~3.1]{Baranovsky} by a different method, and can also be
  derived from \cite[Thm.~4.9]{ACH}.
  For smooth and separated Deligne-Mumford stacks, \thmref{thm:Homology-ET} 
  for Hochschild homology was shown  in \cite[Prop.~2.13]{HLP}
  (and used at several places in the same paper) by a completely
  different method. However, the proof therein appears to have a gap in that
  the authors claim that the composition of the canonical pull-back maps
  $HH(\sX) \to  HH(I_{\sX}) \to {\bf{R}}\Gamma(I_{\sX}, \sH\sH_{I_{\sX}})$ is a weak
  equivalence, which is not the case. The reader can see using
  ~\eqref{eqn:Homology-decom-4} that the map  $\Psi_{\sX}$ in
  ~\eqref{eqn:Homology-ET-2} is not the usual pull-back map.
\end{remk}

The following is a direct consequence of \thmref{thm:Homology-ET}. It provides 
a complete description of all homology theories of the quotient stack $\sX$.  
We let $d = \dim(\sX)$.

\begin{cor}\label{cor:Hom-etale}
For every $n \in \Z$, there are canonical isomorphisms of $k$-vector spaces
  \[
  HH_n(\sX) \xrightarrow{\cong} {\underset{0 \le i \le d}\bigoplus} H^{i-n}_\et(I_\sX,
  \Omega^i_{{I_\sX}/k});
  \]
  \[
  HC^{-}_n(\sX) \xrightarrow{\cong} {\underset{0 \le i \le d-n}\bigoplus}
  H^{n+2i}_\et(I_\sX, ~_{s_{\ge i+n}}\Omega^\bullet_{I_\sX}).
  \]
  If $\sX$ is separated, there are canonical isomorphisms of $k$-vector spaces
  \[
  HC_n(\sX)  \xrightarrow{\cong} {\underset{-d \le i \le n}\bigoplus}
  H^{n-2i}_\et(I_\sX, ~_{s_{\le n-i}}\Omega^\bullet_{I_\sX});
  \]
  \[
  HP_n(\sX) \xrightarrow{\cong} {\underset{i \in \Z}\prod} H^{n+2i}_{dR}(I_{\sX}).
  \]
 \end{cor}

We end this section with generalizations of Theorems~\ref{thm:Homology-main} and
~\ref{thm:Homology-decom-5} for the equivariant periodic cyclic homology for
certain group actions on singular $k$-schemes. The result is the following.

\begin{thm}\label{thm:HP-sing}
  Let $k$ be an algebraically closed field of characteristic zero. Let
  $G$ be a reductive $k$-group acting on a $k$-scheme $X$ and let $\sX = [X/G]$.
  Let $g \in G_s(k)$ and let $\fm_\Psi \subset R_k(G)$ be the maximal ideal
  corresponding to the conjugacy class $\Psi$ of $g$. We then have the following.
  \begin{enumerate}
    \item
If $G$ acts on $X$ with nice stabilizers, then one has a weak equivalence of spectra
  \[
  \vartheta^g_{\sX} \colon HP(\sX)^{\compl}_{\fm_\Psi} \xrightarrow{\simeq}
  \holim_{n} HP((Z_g)^n \times X^g)
  \]
  which is natural in $X$.
\item
  If $G$ acts on $X$ with finite stabilizers, then there is a weak equivalence
  of spectra
  \[
  \Phi_\sX \colon HP(\sX) \xrightarrow{\simeq} \holim_{n}
  HP(G^n \times I_X)
  \]
  which is natural in $X$.
  \end{enumerate}
\end{thm}
\begin{proof}
  One knows by \cite[Exm.~5.2.7]{Elmanto-Sosnilo} that $HP$ is nil-invariant and
  satisfies cdh descent for stacks. We also note that associated to the 
abstract blow-up square ~\eqref{eqn:KH-g-2} (on the left),
one has an abstract blow-up square
  \begin{equation}\label{eqn:HP-sing-0}
    \xymatrix@C1pc{
      G^n \times W \ar[r] \ar[d] & G^n \times X' \ar[d] \\
      G^n \times Z \ar[r] & G^n \times X}
  \end{equation}
  for each $n \ge 0$ which is compatible with the face and degeneracy maps of
  simplicial schemes $\{G^n \times Y\}_{n \ge 0}$, where $Y \in \{X, Z, X', W\}$.
  Using these observations, the first part of the theorem is proven by repeating
  the arguments for the $KH$-theory in \thmref{thm:Gen-AS-max} 
  verbatim (and using \corref{cor:Chen-8}) while the
  second part follows in a similar way. We omit the
  details.
\end{proof}

\vskip .4cm

\noindent\emph{Acknowledgments.}
The second author would like to thank the department of Mathematics at IISc, Bangalore
for hosting him for a part of the work on this paper.


\vskip .4cm







\begin{thebibliography}{99}
  \bibitem{ACH} D. Arinkin, A. Caldararu, M. Hablicsek, {\sl Formality of derived
  intersections and the orbifold HKR isomorphism\/}, J. Algebra, {\bf 540}, (2019),
  100--120. \

\bibitem{Atiyah-Segal} M. Atiyah, G. Segal, {\sl Equivariant $K$-theory and
  completion}, J. Differential Geometry, {\bf 3}, (1969), 1--18. \


\bibitem{BEO} T. Bachmann, E. Elmanto, P. A. {\O}stv{\ae}r, {\sl {\'E}tale motivic
  spectra and Voevodsky's convergence conjecture\/}, J. Eur. Math. Soc., (2024). \
  
\bibitem{Baranovsky} V. Baranovsky, {\sl Orbifold Cohomology as Periodic Cyclic
  Homology\/}, Int. J. Math., {\bf 14}, no.~8, (2003), 791--812. \

  
\bibitem{BKR} M. Barr, J. Kennison, R. Raphael, {\sl Contractible simplicial
  objects\/}, Comment. Math. Univ. Carolin., {\bf 60}, no.~4, (2019), 473--495. \

  \bibitem{BZFN} D. Ben-Zvi, J. Francis, D. Nadler, {\sl Integral transforms and
  Drinfeld centers in derived algebraic geometry\/}, J. Amer. Math. Soc.,
  {\bf 23}, (2010), 3763--3812. \


\bibitem{Ben-Nadler} D. Ben-Zvi, D. Nadler, {\sl Loop spaces and connections\/},
  J. Topol., {\bf 5}, no.~2, (2012), 377--430. \


\bibitem{Ben-Nadler-DAG} D. Ben-Zvi, D. Nadler, {\sl Nonlinear traces\/},
  Derived algebraic geometry, Panor. Synth{\`e}ses, {\bf 55},  
  Soc. Math. France, Paris, (2021), 39--84. \

\bibitem{BMS} B. Bhatt, M. Morrow, P. Scholze, {\sl Topological Hochschild homology
  and integral $p$-adic Hodge theory\/},  Publications math{\'e}matiques de
  l'IH{\'E}S, {\bf 129}, (2019), 199-310. \


\bibitem{BGT-1} A. Blumberg, D. Gepner, G. Tabuada, {\sl  A universal
  characterization of higher algebraic $K$–theory\/}, Geom. Topol., {\bf 17}(2),
  (2013), 733--838. \
  
\bibitem{BGT} A. Blumberg, D. Gepner, G. Tabuada, {\sl Uniqueness of the
  multiplicative cyclotomic trace\/}, Adv. Math., {\bf 260}, (2014), 191--232. \

  
\bibitem{Borel} A. Borel, {\sl Linear Algebraic Groups\/}, Graduate text in 
  Mathematics, Springer-Verlag, Second edition, {\bf 126}, (1991). \

\bibitem{Brion-Moscow} M. Brion, {\sl Algebraic group actions on normal varieties\/},
  Trans. Moscow Math. Soc., {\bf 78}, (2017), 85--107. \

\bibitem{Carlsson} G. Carlsson, {\sl Derived Completions in Stable Homotopy Theory\/},
  J. Pure Appl. Algebra, {\bf 212}, no.~3, (2008), 550--577. \

\bibitem{Carlsson-Joshua-Adv} G. Carlsson, R. Joshua, {\sl Equivariant
  algebraic $K$-Theory,
  $G$-Theory and derived completions\/}, Adv. Math., {\bf 430}, (2023), 109194. \

\bibitem{CJP} G. Carlsson, R. Joshua, P. Pelaez, {\sl Equivariant
  Algebraic $K$-Theory and Derived completions II: the case of Equivariant Homotopy
  $K$-Theory and Equivariant $K$-Theory\/}, arXiv:2404.1319v1 [math.AG], (2024). \

\bibitem{Chacholski-Scherer} W. Chacholski, J. Scherer, {\sl Homotopy theory of
  diagrams\/}, Mem. Amer. Math. Soc., {\bf 155}, (2002), no. 736, x+90 pp. \ 

  
\bibitem{Chen} H. Chen, {\sl Equivariant localization and completion in cyclic
  homology and derived loop spaces\/}, Adv. Math., {\bf 364}, (2020), 107005. \

  \bibitem{Chriss-Ginzburg} N. Chriss, V. Ginzburg, {\sl Representation Theory and
    Complex Geometry\/}, Modern Birkh{\"a}user Classics, Birkh{\"a}user, Boston,
  1997. \
  

\bibitem{CHWW} G. Cortinas, C. Haesemeyer, M. Walker, C. Weibel, {\sl
  The $K$-theory of toric varieties\/}, Trans. Amer. Math. Soc., {\bf 361}, no.~6,
  (2009), 3325--3341. \
  
\bibitem{SP} A. J. de Jong, et al., {\sl The Stacks project\/} electronic reference,
  2005–, available at http://stacks.math.columbia.edu. \

\bibitem{Deligne} P. Deligne, {\sl Voevodsky's Lectures on Motivic Cohomology
2000/2001\/}, The proceedings of Abel
Symposium 2007, The Norwegian mathematical society, {\bf 4}, (2007), 
355-409. \
  
\bibitem{SGA3} M. Demazure, A. Grothendieck, {\sl Sch{\'e}mas en groupes\/},
  SGA 3, Lecture Notes in Math., {\bf 151}, 1970. \


  \bibitem{Dugger-Unpub} D. Dugger,  {\sl Sheaves and Homotopy theory\/},
  available on author's homepage, Department of Mathematics, University of Oregon,
 (1999),  https://pages.uoregon.edu/ddugger/ \
  
\bibitem{Dugger-Adv} D. Dugger, {\sl Universal Homotopy Theories\/}, Adv. Math.,
  {\bf 164}, (2001), 144-176. \

\bibitem{Dugger} D. Dugger, {\sl A primer on homotopy colimit\/},
  available on author's homepage, Department of Mathematics, University of Oregon,
  (2008), https://pages.uoregon.edu/ddugger/ \



\bibitem{DHI} D. Dugger, S. Hollander, D. Isaksen, {\sl Hypercovers and simplicial
  presheaves\/},  Math. Proc. Cambridge Philos. Soc., {\bf 136}, (2004), no. 1,
  9--51. \

  
\bibitem{EG-Inv} D. Edidin, W. Graham, {\sl  Equivariant intersection theory\/}, 
Invent. Math., {\bf 131}, (1998), 595-634. \
  
\bibitem{EG-RR} D. Edidin, W. Graham, 
{\sl Riemann-Roch for equivariant Chow groups\/}, Duke Math. J.,
{\bf 102}, no. 3, (2000), 567--594. \
  
\bibitem{EG-Adv} D. Edidin, W. Graham, {\sl Nonabelian localization in 
equivariant K-theory and Riemann-Roch for quotients\/}, Adv. Math., 
{\bf 198}, no. 2, (2005), 547--582. \

\bibitem{EG-Duke} D. Edidin, W. Graham, {\sl Algebraic cycles and 
completions of equivariant $K$-theory\/}, Duke Math. J., {\bf 144}, no. 3, 
(2008), 489--524. \



\bibitem{Elmanto-Sosnilo} E. Elmanto, V. Sosnilo, {\sl On Nilpotent Extensions of
  $\infty$-Categories and the Cyclotomic Trace\/}, Int. Math. Res. Not. IMRN,
  {\bf 21}, (2022), 16569--16633. \



\bibitem{GIT} J. Fogarty, F. Kirwan, D. Mumford, {\sl Geometric Invariant Theory\/},
  Ergebnisse der Mathematik und ihrer Grenzgebiete, {\bf 34}, 3rd Ed., Springer-Verlag,
  Berlin, 1994. \

\bibitem{Gaitsgory-Rosenblyum} D. Gaitsgory, N. Rosenblyum, {\sl DG indschemes\/},
  in: Perspectives in Representation Theory, Contemp. Math., {\bf 610}, Amer. Math.
  Soc., Providence, RI, 2014, pp.~139--251. \


\bibitem{Gross-thesis} P. Gross, {\sl Vector Bundles as Generators
  on Schemes and Stacks\/}, Ph.D. thesis, Universit{\"a}t D{\"u}sseldorf, Germany,
  2010. \
\bibitem{EGAII} A. Grothendieck, {\sl {\'E}l{\'e}ments de g{\'e}om{\'e}trie
  alg{\'e}brique: II. {\'E}tude globale {\'e}l{\'e}mentaire de quelques classes de
  morphismes\/},  Publications Math{\'e}matiques de l'IHES, {\bf 8},
  (1961), 5--222. \

  

\bibitem{Hall-Rydh-Comp} J. Hall, D. Rydh, {\sl Perfect complexes on algebraic
  stacks\/}, Compos. Math., {\bf 153}, (2017), 2318--2367. \ 


\bibitem{HNR} J. Hall, A. Neeman, D. Rydh, {\sl One positive and two negative results
  for derived categories of algebraic stacks\/},
  Journal of the Institute of Mathematics of Jussieu, {\bf 18}, no.~5, (2018),
  1087--1111. \ 

  
\bibitem{HLP} D. Halpern-Leistner, D. Pomerleano, {\sl Equivariant Hodge theory and
  noncommutative geometry\/}, Geom. Topol., {\bf 24},  no. 5, (2020), 2361--2433. \ 

\bibitem{HLPR} D. Halpern-Leistner, A. Preygel, {\sl Mapping stacks and categorical
  notions of properness\/}, Comp. Math., {\bf 159}, (2023), 530--589. \


  
\bibitem{HKO} J. Heller, A. Krishna, P. {\O}stv{\ae}r, {\sl Motivic homotopy theory
  of group scheme actions\/}, J. Topol., {\bf 8}, (4), (2015), 1202--1236. \
  
\bibitem{Hill} M. Hill, {\sl Equivariant stable homotopy theory \/},
  in `Handbook of Homotopy Thoery', 1st ed., Chapman and Hall/CRC, New York, 2020. \

\bibitem{Hirsch-Notes} P. Hirschhorn, {\sl Notes on homotopy colimits and homotopy
  limits\/} Available at author's homepage https://math.mit.edu/~psh/, 2014. \

\bibitem{Hirsch-JHRS} P. Hirschhorn, {\sl The diagonal of a multicosimplicial
  object\/},  J. Homotopy Relat. Struct., {\bf 12}, (2027), 971--992. \ 

  
\bibitem{Hochschild} G. Hochschild, {\sl Basic Theory of Algebraic Groups and Lie
  Algebras\/}, Graduate Texts in Mathematics, {\bf 75}, Springer-Verlag, 1981. \
  
\bibitem{Hovey-Book} M. Hovey, {\sl Model Categories\/}, Mathematical surveys and
  monograph series, {\bf 63}, Amer. Math. Soc., Providence, (1999). \

\bibitem{Hovey-JPAA} M. Hovey, {\sl Spectra and symmetric spectra in general model
  categories\/}, J. Pure Appl. Algebra, {\bf 165}, no.~1, (2001), 63--127. \

\bibitem{Hoyois-notes} M. Hoyois, {\sl The homotopy fixed points of the circle action
  on Hochschild homology\/}, arXiv:1506.07123v2, [math.KT], (2018). \
  
\bibitem{Hoyois-Doc} M. Hoyois, {\sl Cdh Descent in Equivariant Homotopy
  $K$-Theory\/}, Documenta Math., {\bf 25}, (2020), 457--482. \ 
  
\bibitem{Hoyois-Krishna} M. Hoyois, A. Krishna, {\sl Vanishing theorems for the
  negative $K$-theory of stacks\/}, Annals $K$-theory, {\bf 4}, no. 3, (2019),
  439--472. \ 


\bibitem{Karpenko-Merkurjev} N. Karpenko, A. Merkurjev, {\sl Chow Filtration on
  Representation Rings of Algebraic Groups\/},  International Mathematics Research
  Notices IMRN, {\bf 9}, (2021), 6691--6716. \

  
\bibitem{Kassel} C. Kassel, {\sl Cyclic homology, comodules, and mixed
  complexes\/}, J. Algebra, {\bf 107}, no.~1, (1987), 195--216. \
  
\bibitem{Keller-Doc} B. Keller, {\sl On the cyclic homology of ringed spaces and
  schemes\/}. Doc. Math., {\bf 3} (1998), 231--259. \

\bibitem{Keller-JPAA} B. Keller, {\sl On the cyclic homology of exact categories\/},
  J. Pure Appl. Algebra, {\bf 136}, no.~1, (1999), 1--56. \ 
  
\bibitem{Keller-ICM} B. Keller, {\sl On differential graded categories\/},
  In: International Congress of Mathematicians, {\bf II}, edited by M. Sanz-Sol{\'e}
  et al., European Mathematical Society, Z{\"u}rich, 2007, pp. 151--190. \ 

\bibitem{Khan-JJM} A. Khan, {\sl $K$-theory and $G$-theory of derived algebraic
  stacks\/}, Japan. J. Math., {\bf 17}, (2022), 1--61. \
  

  
\bibitem{Knutson} D. Knutson, {\sl Algebraic Spaces\/}, Lecture Notes in Math.,
  {\bf 203}, Springer-Verlag, Berlin, 1971. \
  





\bibitem{Krishna-Adv} A. Krishna, {\sl Riemann-Roch for equivariant $K$-theory\/},  
Adv. Math., {\bf 262}, (2014), 126--192. \


\bibitem{Krishna-LMS} A. Krishna, {\sl Equivariant $K$-theory and higher Chow groups
  of schemes\/}, Proc. LMS, {\bf 114}, no.~3, (2017), 657--683. \

\bibitem{Krishna-Crelle} A. Krishna, {\em The completion problem for equivariant 
  $K$-theory\/}, J. Reine Angew. Math., {\bf 740}, (2018), 275--317. \

\bibitem{Krishna-Park-AGT}  A. Krishna, J. Park, {\sl Semitopologization in motivic
  homotopy theory and applications\/}, Algebr. Geom. Topol., {\bf 15}, no.~2, (2015),
  823--861. \

\bibitem{Krishna-Ravi-1} A. Krishna, C. Ravi, {\sl Equivariant vector bundles, their
  derived category and $K$-theory on afﬁne schemes\/}, Ann. $K$-theory, {\bf 2}, no. 2,
  (2017), 235-275. \
  
\bibitem{Krishna-Ravi} A. Krishna, C. Ravi, {\sl Algebraic $K$-theory of quotient
  stacks\/}, Ann. $K$-theory, {\bf 3}, no. 2, (2018), 207-233. 

\bibitem{Krishna-Sreedhar} A. Krishna, B. Sreedhar, {\sl Atiyah-Segal theorem for
  Deligne-Mumford stacks and applications\/}, J. Algebraic Geom., {\bf 29}, no. 3
  (2020), 403--470. \ 
  

\bibitem{Laumon} G. Laumon, L. Moret-Bailly, {\sl Champs alg{\'e}briques\/},
 Ergebnisse der Mathematik und ihrer Grenzgebiete, {\bf 39}, Springer, 2000. \

\bibitem{Levine-Notes} M. Levine, {\sl $K$-theory and motivic cohomology of
  schemes\/}, Preprint, available at
  https://faculty.math.illinois.edu/k-theory/0336/mot.pdf. \

\bibitem{Loday} J.-L. Loday, {\sl Cyclic homology\/}, Grundlehren der Mathematischen
  Wissenschaften, {\bf 301}, Springer-Verlag, Berlin, 1992. \

\bibitem{Lurie-DAG} J. Lurie, {\sl Derived Algebraic Geometry\/},
  Ph.D. Thesis, Massachusetts Institute of Technology, Dept. of Mathematics, (2004),
   available at http://hdl.handle.net/1721.1/30144. \

 \bibitem{Lurie-HA} J. Lurie, {\sl Higher Algebra\/}, (2017),
  available at author's homepage. \

\bibitem{Lurie-SAG} J. Lurie, {\sl Spectral Algebraic Geometry\/}, (2018),
  available on author's homepage. \

\bibitem{Merkurjev} A. Merkurjev, {\sl Equivariant $K$-theory\/},
  In: `Handbook of $K$-theory', Springer-Verlag, {\bf 2}, (2005),
pp. 925--954. \

\bibitem{Morel-Voevodsky} F. Morel, V. Voevodsky, {\sl $\A^1$-homotopy theory of
  schemes\/}, Publ. Math. IHES, {\bf 90}, (1999), 45--143. \

\bibitem{Nikolaus-Scholze} T. Nikolaus, P. Scholze, {\sl On topological cyclic
  homology\/}, Acta Math., {\bf 221}, (2018), 203--409. \
  
\bibitem{Quillen} D. Quillen, {\sl Higher algebraic $K$-theory. I\/},
  In: `Algebraic $K$-theory, I: Higher $K$-theories', Lecture Notes in Math.,
  {\bf 341}, Springer, Berlin, 1973. \

\bibitem{Richardson} R. Richardson, {\sl Affine Coset Spaces of Reductive
  Algebraic Groups\/}, Bull. London Math. Soc., {\bf 9}, no.~1, (1977), 38--41. \

\bibitem{Satriano} M. Satriano, {\sl de Rham Theory for Tame Stacks and Schemes with
  Linearly Reductive Singularities\/},  Annales de l'institut Fourier, {\bf 62},
  no.~6, (2012), 2013--2051. \ 

\bibitem{Schlichting} M. Schlichting, {\sl Negative $K$-theory of derived
  categories \/}, Math. Z., {\bf 253}, 1, 2006, 97-134.  \

\bibitem{Schwede-Shipley} S. Schwede, B. Shipley, {\sl Stable model categories are
  categories of modules\/}, Topology, {\bf 42}, no.~1, (2003), 103--153. \
  
\bibitem{Segal} G. Segal, {\sl The representation-ring of a compact Lie group\/},
Publications math{\'e}matiques de l'I.H.{\'E}.S., {\bf 34}, (1968), 113--128. \  

  
\bibitem{Serre} J.-P. Serre, {\sl Groupes de Grothendieck des sch{\'e}mas en groupes
  r{\'e}ductifs d{\'e}ploy{\'e}s\/},
  Publications math{\'e}matiques de l'I.H.{\'E}.S., {\bf 34}, (1968), 37--52. \

  
\bibitem{Sumihiro} H. Sumihiro, {\sl Equivariant completion II\/},
J. Math. Kyoto, {\bf 15}, no.~3, (1975), 573-605. \
  

\bibitem{Tabuada-Bergh-IMRN} G. Tabuada, M. Van Den Bergh,
  {\sl Motivic Atiyah-Segal
  Completion Theorem\/}, Int. Math. Res. Not. IMRN, no.~4, (2024), 3497--3550. \

  \bibitem{Tauvel} P. Tauvel, R. Yu, {\sl Lie Algebras and Algebraic Groups\/},
   Springer Monographs in Mathematics, Springer-Verlag, Berlin Heidelberg, 2005. \ 
    
  \bibitem{Thomason-ENS} R. Thomason, {\sl Algebraic $K$-theory and {\'e}tale
    cohomology\/}, Annales scientiﬁques de l'{\'E}.N.S., 4e s{\'e}rie, {\bf 18},
    no.~3, (1985), 437--552. \
  
\bibitem{Thomason-Inv} R. Thomason,  {\em Lefschetz-Riemann-Roch theorem and 
coherent trace formula\/}, Invent. Math., {\bf 85}, (1986), 515--543. \ 

 \bibitem{Thomason-Duke-1} R. Thomason, {\em Comparison of equivariant algebraic and 
topological $K$-theory\/}, Duke Math. J., {\bf 53}, no.~3, (1986), 795--825. \

  \bibitem{Thomason-Orange} R. Thomason, {\sl Algebraic $K$-theory of group scheme
  actions\/}, In: `Algebraic Topology and Algebraic $K$-theory',
  Ann. Math. Stud., Princeton, {\bf 113}, (1987), pp. 539--563.\
   
 \bibitem{Thomason-Adv} R. Thomason, {\sl Equivariant resolution, linearization, and
   Hilbert’s fourteenth problem over arbitrary base schemes\/}, Adv. Math,
   {\bf 65}, no.~1, (1987), 16--34. \
   
 \bibitem{Thomason-Duke-2} R. W. Thomason, {\em Equivariant algebraic vs. topological 
$K$-homology Atiyah-Segal-style\/},  Duke Math. J., no.~3,  {\bf 56}, (1988), 
589--636. \
 
\bibitem{TT} R. W. Thomason, T. Trobaugh, {\em Higher Algebraic
$K$-Theory of Schemes and of Derived Categories\/},
The Grothendieck Festschrift III, Progress in Math., 
Birkh{\"a}user, {\bf 88}, (1990), 247--435. \


\bibitem{Toen-DAG} B. To{\"e}n, {\sl Derived algebraic geometry\/},
  EMS Surv. Math. Sci., {\bf 1}, (2014), no. 2, 153--240. \


\bibitem{Toen-Vezzosi} B.  To{\"e}n,  G. Vezzosi, {\sl Homotopical algebraic
  geometry II: Geometric stacks and applications\/}, Mem. Amer. Math. Soc.,
  {\bf 193}, (2008), pp.~x+224. \
  
  
\bibitem{Totaro1} B. Totaro, {\sl The Chow ring of a classifying space\/},
Algebraic $K$-theory (Seattle, WA, 1997), 
Proc. Sympos. Pure Math., Amer. Math. Soc., {\bf 67}, (1999), 249-281. \ 

\bibitem{VV} G. Vezzosi, A. Vistoli, {\sl  Higher algebraic $K$-theory of 
group actions with finite stabilizers\/}, Duke Math. J., no.~1, 
  {\bf 113}, (2002), 1--55. \

\bibitem{VV-Inv} G. Vezzosi, A. Vistoli, {\sl Higher algebraic $K$-theory for
  actions of diagonalizable group actions}, Invent. Math., {\bf 153}, (2003), 1--44.

\bibitem{Vistoli} A. Vistoli, {\sl Higher equivariant $K$-theory for finite
  group actions\/}, Duke Math J., {\bf 63}, no.~2, (1991), 399--419. \

\bibitem{Voev-ICM} V. Voevodsky, {\sl $\A^1$-homotopy theory\/},
Proceedings of the International Congress of Mathematicians, {\bf 1},
(Berlin, 1998), Doc. Math., (1998), 579-604. \


  

\bibitem{Weibel-HALG} C. Weibel, {\sl An introduction to homological algebra\/},
  Cambridge University Press, Cambridge, 1994. \
  





\end{thebibliography}
\end{document}